\documentclass{amsart}
\usepackage{amssymb,mathtools}
\usepackage[margin=0.9in]{geometry}
\usepackage{graphics}
\usepackage{hyperref}
\usepackage[capitalize]{cleveref}
\usepackage{tikz}
\usepackage{tikz-cd}
\usepackage{bm}
\usetikzlibrary{decorations.markings}
\usepackage{nccmath}

\makeatletter
\def\part{%
  \@startsection{part}
  {0}
  {\z@}
  {\linespacing\@plus\linespacing}
  {.5\linespacing}
  {\let\@secnumfont\relax\normalfont\Large\bfseries\raggedright}%
}
\makeatother

\newtheorem{conjecture}{Conjecture}
\newtheorem{theorem}[conjecture]{Theorem}
\newtheorem{lemma}[conjecture]{Lemma}
\newtheorem{proposition}[conjecture]{Proposition}
\newtheorem{corollary}[conjecture]{Corollary}
\newtheorem{definition}[conjecture]{Definition}
\newtheorem{defn}[conjecture]{Definition}

\newtheorem{problem}[conjecture]{Problem}

\newtheorem{remark}[conjecture]{Remark}

\newtheorem{example}[conjecture]{Example}
\def\bT{{\overline{\T}}}
\def\dL{{\mathfrak L}}

\def\tSigma{\Sigma'}
\def\one{{\bf 1}}
\def\Trop{{\rm Trop}}
\def\O{{\mathcal{O}}}
\def\Res{{\rm Res}}
\def\res{{\rm res}}
\def\P{{\mathbb P}}
\def\Q{{\mathbb Q}}
\def\Z{{\mathbb Z}}
\def\C{{\mathbb C}}

\def\build{{\mathcal{G}}}
\def\Gr{{\rm Gr}}
\def\R{{\mathbb R}}

\def\A{{\mathcal{A}}}
\def\bA{{\bar {\mathcal{A}}}}
\def\M{{\mathcal{M}}}
\def\F{{\mathcal{F}}}
\def\bU{{\bar U}}
\def\N{N}
\def\0{{\hat 0}}

\def\T{{\mathcal{T}}}

\def\B{{\mathcal{B}}}
\def\tilM{{\widetilde{M}}}
\def\tM{{\widetilde\M}}
\def\L{{\mathcal{L}}}
\def\Frac{{\rm Frac}}
\def\dlog{{\rm dlog}}
\def\reg{{\rm reg}}
\def\dR{{\rm dR}}
\def\Spec{{\rm Spec}}

\def\I{{\mathcal{I}}}

\def\tE{{\widetilde E}}

\def\Ker{{\rm Ker}}
\def\sp{{\rm span}}
\def\rk{{\rm rk}}
\def\del{\kern-0.8pt{\setminus}\kern-0.8pt}
\def\v{{\mathbf{v}}}
\def\Fl{{\rm Fl}}
\def\A{{\mathcal{A}}}
\def\rOS{{\bar\A}}
\def\B{{\mathcal{B}}}
\def\I{{\mathcal{I}}}
\newcommand\ip[1]{\langle #1 \rangle}

\def\x{{\mathbf{x}}}
\def\y{{\mathbf{y}}}
\def\v{{\mathbf{v}}}
\def\u{{\mathbf{u}}}
\def\bchi{{\bar \chi}}

\def\nbc{{\textbf{nbc}}}
\def\a{{\mathbf{a}}}
\def\b{{\mathbf{b}}}

\def\bE{{\bar E}}
\def\bM{{\overline{\M}}}
\def\be{{\bar e}}
\def\OS{{\rm OS}}
\def\lf{{\rm lf}}
\def\bOmega{{\overline{\Omega}}}
\def\bP{{\overline{P}}}
\def\bQ{{\overline{Q}}}
\def\tb{{\tilde b}}

\def\codim{{\rm codim}}

\def\bG{{\overline{G}}}
\def\rOS{{\bar A}}
\def\OS{A}
\def\Int{{\rm Int}}
\def\an{{\rm an}}
\def\At{{\mathfrak{A}}}
\def\atom{c}
\def\Hom{{\rm Hom}}
\def\be{{\bar e}}
\def\bomega{{\bar \omega}}
\def\minL{\hat {\mathfrak o}}
\def\pFl{{\Delta}}
\def\sep{{\rm sep}}
\def\flip{\tilde}
\def\emptyflag{\varnothing}

\def\image{{\rm image}}
\def\z{{\mathbf{z}}}
\def\F{{\mathcal{F}}}
\def\tF{{\tilde \F}}

\newcommand\bip[1]{{\overline{\langle #1 \rangle}}}
\newcommand\gBip[1]{{\langle #1 \rangle}_{\L}}
\newcommand\gDBip[1]{{\langle #1 \rangle}^{\L}}

\newcommand\dRip[1]{\langle #1 \rangle^{\dR}}
\newcommand\gdRip[1]{\langle #1 \rangle^{\nabla}}
\newcommand\gDdRip[1]{\langle #1 \rangle_{\nabla}}
\newcommand\dRipp[1]{\langle #1 \rangle^{\dR'}}
\newcommand\bdRip[1]{\overline{\langle #1 \rangle}^{\dR}}
\newcommand\DdRip[1]{\langle #1 \rangle_{\dR}}

\newcommand\halfip[1]{\langle #1 \rangle}

\numberwithin{conjecture}{section}
\numberwithin{equation}{section}

\author{Thomas Lam}
\address{Department of Mathematics, University of Michigan, 2074 East Hall, 530 Church Street, Ann Arbor, MI 48109-1043, USA}
\email{\href{mailto:tfylam@umich.edu}{tfylam@umich.edu}}

\begin{document}
\begin{abstract}
In the 1990s, Kita--Yoshida and Cho--Matsumoto introduced intersection forms on the twisted (co)homologies of hyperplane arrangement complements.  We give a closed combinatorial formula for these intersection pairings.  We show that these intersection pairings are obtained from (continuous and discrete) Laplace transforms of subfans of the Bergman fan of the associated matroid.  We compute inverses of these intersection pairings, allowing us to identify (variants of) these intersection forms with the contravariant form of Schechtman--Varchenko, and the bilinear form of Varchenko.

Building on parallel joint work with C. Eur, we define a notion of scattering amplitudes for matroids.  We show that matroid amplitudes satisfy locality and unitarity, and recover biadjoint scalar amplitudes in the case of the complete graphic matroid.  We apply our formulae for twisted intersection forms to deduce old and new formulae for scattering amplitudes.
\end{abstract}
\title{Matroids and amplitudes}
\maketitle

\setcounter{tocdepth}{1}
\tableofcontents
\section{Introduction}
The theory of hyperplane arrangements is one of the central topics in algebraic combinatorics and combinatorial algebraic geometry.
Let $\bA = \{H_1,H_2,\ldots,H_n\} \subset \P^d$ denote an arrangement of hyperplanes in complex projective space, and let $\bU:= \P^d \setminus \bA$ denote the hyperplane arrangement complement.  Brieskorn \cite{Brie}, following ideas of Arnold, showed that the cohomology ring $H^*(\bU)$ is generated by the classes of the 1-forms $df_j/f_j$, where $f_j$ is a linear function cutting out the hyperplane $H_j$.  Orlik and Solomon \cite{OS} subsequently described the ring $H^*(\bU)$ by generators and relations, defining the \emph{Orlik-Solomon algebra} $\OS^\bullet(M)$, where $M$ denotes the matroid of $\bA$.

Motivated by connections to the theory of hypergeometric functions, attention turned to twisted cohomologies of hyperplane arrangement complements.
Gauss's hypergeometric function is distinguished by being the solution to a second-order linear differential equation with three regular singular points.  The Aomoto-Gelfand generalized hypergeometric functions \cite{Aom,Gel} generalize Gauss's hypergeometric function by allowing singularities along hyperplanes in $\P^d$.  These generalized hypergeometric functions can be viewed as twisted periods, pairings between algebraic deRham twisted cohomology $H^*(\bU, \nabla_\a)$ and Betti twisted homology $H^*(\bU,\L_\a)$ groups.  Esnault, Schechtman, and Viehweg \cite{ESV} and Schechtman, Terao, and Varchenko \cite{STV} showed that under a genericity hypothesis \eqref{eq:Mon}, elements of the twisted cohomologies $H^*(\bU, \nabla_\a)$ could again be represented by global algebraic logarithmic forms.  Thus $H^*(\bU, \nabla_\a)$ can be identified with the cohomology of the \emph{Aomoto complex} $(\OS^\bullet(M),\omega)$; see \eqref{eq:Aomotocomplex}.

In the 1990s, Cho and Matsumoto \cite{CM} and Kita and Yoshida \cite{KY} introduced intersection pairings on these twisted (co)homologies, which we denote
\begin{align*}
\gdRip{\cdot,\cdot}&: H^*(\bU, \nabla_\a) \otimes H^*(\bU, \nabla_{-\a}) \to \C, \\
\gBip{\cdot,\cdot}&: H_*(\bU, \L_\a) \otimes H_*(\bU, \L_{-\a}) \to \C. 
\end{align*}
The first goal of this work is to give a closed formula for these intersection pairings, which we call the \emph{(twisted) deRham cohomology} (resp. \emph{(twisted) Betti homology}) intersection forms, following the terminology of \cite{BD}.  Explicit formulae for these pairings were previously known, for example, in the one-dimensional case \cite{CM,KY}, the case of a generic arrangement \cite{Matgen}, and the braid arrangement \cite{MHhom, Miz}.  A general method to compute $\gBip{\cdot,\cdot}$ is given in \cite{KY2}, and this approach is further studied in \cite{Tog}.

Our explicit formulae for $\gdRip{\cdot,\cdot}$ and $\gBip{\cdot,\cdot}$ reveal new connections between existing constructions.  The Bergman fan $\Sigma_{\bU}$ of $\A$ is a polyhedral fan \cite{Bergman,FS,AK}, in modern language the \emph{tropical variety} associated to the very affine variety $\bU$.
First, we show that $\gdRip{\cdot,\cdot}$ and $\gBip{\cdot,\cdot}$ can be expressed as a Laplace transform and a discrete Laplace transform of various subfans of the Bergman fan.  In particular, this gives an interpretation of the Cho-Matsumoto twisted period relations as interpolating between continuous and discrete Laplace transforms.  Second, we prove that the twisted deRham cohomology intersection form $\gdRip{\cdot,\cdot}$ is essentially equal to the ``contravariant form" of Schechtman and Varchenko \cite{SV}, and the Betti homology intersection form $\gBip{\cdot,\cdot}$ is essentially equal to the inverse of Varchenko's bilinear form on a real hyperplane arrangement \cite{Var}.  

Our work is heavily motivated by the theory of scattering amplitudes from physics.  Cachazo-He-Yuan \cite{CHYarbitrary} introduced around a decade ago a new approach to tree-level scattering amplitudes in various theories: biadjoint scalar, Yang-Mills, gravity, ...  This approach relies on the \emph{scattering equations} on the configuration space $M_{0,n+1}$ of $n+1$ points on $\P^1$ to produce rational functions on kinematic space.  Mizera \cite{Miz} first observed that the Cachazo-He-Yuan formalism could be interpreted in terms of the twisted intersection forms of \cite{CM,KY} for the hyperplane arrangement complement $\bU = M_{0,n+1}$, and he showed that $\gdRip{\cdot,\cdot}$ and $\gBip{\cdot,\cdot}$ described \emph{biadjoint scalar amplitudes} and the \emph{inverse string theory KLT kernel} respectively.   Scattering potentials and scattering equations had appeared earlier in the mathematical literature, especially in work of Varchenko \cite{Varcrit,Varbook,VarBethe,Varquantum} where they were called \emph{master functions} and critical point equations.

One of the starting points of our work is to replace the space $M_{0,n+1}$ with an arbitrary (oriented) matroid.  We rely on the concurrent parallel work \cite{EL} joint with C. Eur, where we construct \emph{canonical forms for oriented matroids}.  The work \cite{EL} shows that topes of oriented matroids can be viewed as positive geometries \cite{ABL,LamPosGeom}, and in the current work we use their canonical forms as an input to the CHY construction of scattering amplitudes.  More precisely, the construction of \cite{EL} replaces the \emph{Parke-Taylor forms} from physics, allowing us to systematically use the formalism of matroids in our theory.  An eventual goal of this work is to understand the twisted periods of hyperplane arrangement complements in matroid-theoretic terms \cite{Lamstring}.

In the last part of this work, we give some immediate applications of our results to physics: a new formula for biadjoint scalar amplitudes using \emph{temporal Feynman diagrams}, a construction of scattering forms for matroids, and new formulae for various determinants of amplitudes.  Further applications to physics will be pursued in separate future work.

\section{Main results}
Let $M$ be the matroid associated to the hyperplane arrangement $\bA$, defined on the ground set $E$, and let $\M$ be the corresponding oriented matroid.  Thus $M$ has rank $r = d+1$ where $d$ is the dimension of the projective hyperplane arrangement $\bA$.  The lattice of flats of $M$ is denoted $L(M)$, with minimum $\hat 0$ and maximum $\hat 1$.  Let $\OS^\bullet(M)$ denote the Orlik-Solomon algebra of $M$, and $\rOS^\bullet(M)$ the reduced Orlik-Solomon algebra; see \cref{sec:OS}.  Thus $\rOS^\bullet(M)$ is isomorphic to the cohomology ring $H^*(\bU)$ of the projective hyperplane arrangement complement $\bU$.  We always assume that the hyperplane arrangement $\bA$ is essential.  We let $\OS(M) = \OS^r(M)$ denote the top homogeneous component of $\OS^\bullet(M)$.

Fix $0 \in E$.  Let $\T, \T^+, \T^\star, \T^0$ denote the set of topes, the set of topes $P$ satisfying $P(0) = +$, the set of topes bounded with respect to a general extension $\star$, and the set of bounded topes with respect to $0$, respectively.  See \cref{sec:matroids}.

\subsection{Canonical forms for oriented matroids}
A \emph{positive geometry} is a semialgebraic subset $X_{\geq 0}$ of a projective algebraic variety $X$ \cite{ABL,LamPosGeom} satisfying certain axioms.  Any positive geometry is equipped (by definition) with a rational top-form $\Omega(X_{\geq 0})$ on $X$, called the \emph{canonical form} of the positive geometry $X_{\geq 0}$.  We will not need the full definition of positive geometry in this work.  Instead, we note that every full-dimensional (oriented) projective polytope $P \subset \P^d$ is a positive geometry and is thus equipped with a distinguished top-form $\Omega_P$, satisfying the recursion: 

\noindent
(a) if $P$ is a point then $\Omega_P = \pm 1$ depending on orientation, and 

\noindent
(b) if $\dim(P) > 0$, then all the poles of $\Omega_P$ are simple and along facet hyperplanes, and we have the recursion $\Res_F \Omega_P = \Omega_F$, for any facet $F$ of $P$.  

In \cite{EL}, Eur and the author generalize canonical forms to oriented matroids, showing the existence of distinguished elements in the Orlik-Solomon algebra that play the role of canonical forms.

\begin{theorem}[see \cref{thm:EL}]\label{thm:ELintro}
To each tope $P \in \T$, there is a distinguished canonical form $\Omega_P \in \OS(M)$, satisfying the recursions of canonical forms.  Furthermore, the collection $\{\Omega_P \mid P \in \T^\star\}$ give a basis of $\OS(M)$.
\end{theorem}
For the case that $P$ is a chamber of a real hyperplane arrangement, the canonical form $\Omega_P$ is the usual one associated to a projective polytope.  Canonical forms play a special role in our computations: we will compute our intersection pairings with respect to the basis of \cref{thm:ELintro}. 

\subsection{Matroid intersection forms}
Let $R := \Z[\a] = \Z[a_e \mid e \in E]$ and $S := \Z[\b] = \Z[b_e \mid e \in E]$ be the polynomial rings in variables $a_e$ (resp. $b_e$), and let $Q = \Frac(R)$ and $K = \Frac(S)$ be their fraction fields.  When the parameters are specialized to complex numbers, the variables $a_e,b_e$ are related by $b_e = \exp(- \pi i a_e)$ (see \cref{sec:twistedco}).
Our main objects of study are two combinatorially defined bilinear forms
\begin{align*}
\dRip{\cdot,\cdot}&: \OS(M) \otimes \OS(M) \to Q, \\
\halfip{\cdot,\cdot}_B&:  \Z^{\T^+} \otimes \Z^{\T^+} \to K,
\end{align*}
called the \emph{deRham cohomology twisted intersection form} and \emph{Betti homology twisted intersection form} respectively.  We remark that $\dRip{\cdot,\cdot}$ is defined for an arbitrary matroid while $\halfip{\cdot,\cdot}_B$ is only defined in the setting of an oriented matroid.  The bilinear form $\dRip{\cdot,\cdot}$ is defined (\cref{def:dR}) by using \emph{residue maps} on the Orlik-Solomon algebra, and the bilinear form $\halfip{\cdot,\cdot}_B$ is defined (\cref{def:Bettipair}) directly using the combinatorics of the Las Vergnas lattice of flats.

\subsection{Laplace transforms of Bergman fan}
We explain the combinatorics of $\dRip{\cdot,\cdot}$ and $\halfip{\cdot,\cdot}_B$ in the language of \emph{Bergman fans}.

Bergman \cite{Ber} defined the logarithmic limit-set of an algebraic variety, with the aim of studying the behavior of the variety at infinity.  We view Bergman's construction as a \emph{tropical variety}: the set of valuations of points of the variety defined over the field of Puiseux series.  When the variety is a linear space, the Bergman fan depends only on the matroid of that linear space.  The Bergman fan $\Sigma_M$ of a matroid $M$ was further studied by Ardila and Klivans \cite{AK} and Feichtner and Sturmfels \cite{FS}.

We shall consider a particular fan structure on $\Sigma_M$: the maximal cones $C_{F_\bullet}$ are $d$-dimensional cones indexed by $F_\bullet \in \Fl(M)$, where $\Fl(M)$ denotes the set of complete flags of flats of $M$.  Other \emph{nested fan structures} on $\Sigma_M$ are considered in \cref{sec:building}.  Associated to a tope $P \in \T$, the \emph{Bergman fan of $P$}, $\Sigma_M(P)$, is the subfan of $\Sigma_M$ consisting of all cones $C_{F_\bullet}$ where $F_\bullet \in \Fl(P)$; see \cite{AKW}.

In \cref{prop:noover}, we introduce a canonical decomposition of the intersection of positive Bergman fans: for $P,Q \in \T$, we introduce a collection $G^{\pm}(P,Q)$ of partial flags of lattices, and we have
$$
\Sigma_M(P) \cap \Sigma_M(Q) = \bigsqcup_{G_\bullet \in G^{\pm}(P,Q)} \Sigma_M(P,G_\bullet),
$$
where both sides of the equality are viewed as collections of $d$-dimensional cones.

In \cref{sec:Bergman}, we introduce two integral operators $\L$ and $\dL$ called the \emph{continuous Laplace transform} and \emph{discrete Laplace transform} respectively.  These operators are defined as an integral and as a sum over lattice points respectively, and produce rational functions in $\a$ and $\b$ respectively when applied to subfans of $\Sigma_M$.  

\begin{theorem}[\cref{thm:deRhamfan} and \cref{thm:Bettifan}] \label{thm:fan}
Let $P,Q \in \T$ be topes.  Then
\begin{align*}
\dRip{\Omega_P,\Omega_Q} &= \sum_{G_\bullet \in G^{\pm}(P,Q)} (\pm)^r (-1)^{\sum_{i=1}^s \rk(G_i)} \L(\Sigma_M(P,G_\bullet)) \\
\halfip{P,Q}_B&= (-1)^d \sum_{G_\bullet \in G^{\pm}(P,Q)} (\pm)^r b(G_\bullet) \dL(\Sigma_M(P,G_\bullet)).
\end{align*}
In particular, $\dRip{\Omega_P,\Omega_P} = \L(\Sigma_M(P))$ and $\halfip{P,P}_B = (-1)^d \dL(\Sigma_M(P,G_\bullet))$.
\end{theorem}
The sign $(\pm)^r$ is explained in  \cref{thm:deRhamfan}, and the quantity $b(G_\bullet)$ is a signed monomial in the $b$-variables, defined in \cref{def:Bettipair}.  We show in \cref{prop:nondeg} and \cref{thm:Bettinondeg} that two bilinear forms are non-degenerate.  In \cref{sec:building}, we show that \cref{thm:fan} is compatible with other \emph{nested fan structures} on $\Sigma_M$.

\begin{example}\label{ex:3pt}
Let $\A$ be the arrangement of three points $\{z_1,z_2,z_3\}$ in $\P^1(\R)$.  Thus $M = M(\A) = U_{2,3}$ is the uniform matroid of rank $2$ on three elements $E = \{1,2,3\}$.  The Bergman fan $\Sigma_M$ consists of three rays (see \cref{fig:posBerg}), which we draw in $\R^E/\one$.  Let $P,Q,R$ be the three topes (modulo negation) given by the intervals $P = [z_1,z_2]$, $Q = [z_2,z_3]$, and $R = [z_3,z_1]$.  The intersection $\Sigma_M(P) \cap \Sigma_M(R)$ consists of the single cone $C_{F_\bullet}$ where $F_\bullet = (\hat 0 \subset \{1\} \subset \hat 1)$.  By \cref{thm:fan}, we have 
$$
\dRip{\Omega_P,\Omega_R} = - \frac{1}{a_1}, \qquad \halfip{P,R}_B = -\frac{b_1}{b_1^2-1} = b_1(1+b_1^2 + b_1^4 + \cdots).
$$
On the other hand, $\Sigma_M(P)$ is the union of two cones, $C_{F_\bullet}$ and $C_{F'_\bullet}$ where $F'_\bullet = (\hat 0 \subset \{2\} \subset \hat 1)$.  By \cref{thm:fan}, we have 
$$
\dRip{\Omega_P,\Omega_P} = \frac{1}{a_1} + \frac{1}{a_2}, \qquad \halfip{P,P}_B = 1 + \frac{1}{b_1^2-1} + \frac{1}{b_2^2-1} = -\left(1 + (b_1^2 + b_1^4 + \cdots) + (b_2^2 + b_2^4+ \cdots) \right).
$$

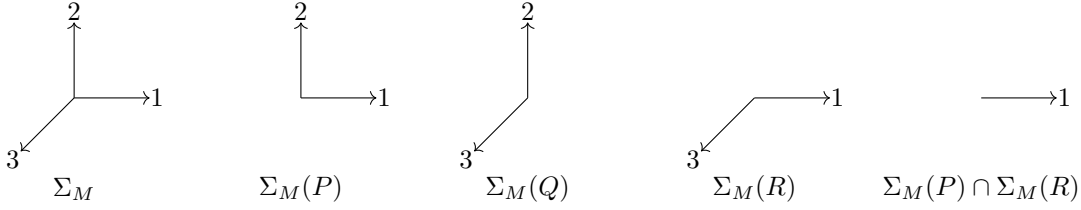
\begin{figure}
\begin{center}
\begin{tikzpicture}
\draw[->] (0:0) -- (0:1);
\node (A1) at (0:1.1) {$1$};
\draw[->] (0:0) -- (90:1);
\node (A2) at (90:1.15) {$2$};
\draw[->] (0:0) -- (225:1);
\node (A3) at (225:1.15) {$3$};
\node (AA) at (270:1.2) {$\Sigma_M$};
\begin{scope}[shift={(3,0)}]
\draw[->] (0:0) -- (0:1);
\node (A1) at (0:1.1) {$1$};
\draw[->] (0:0) -- (90:1);
\node (A2) at (90:1.15) {$2$};
%\draw[->] (0:0) -- (225:1);
%\node (A3) at (225:1.15) {$3$};
\node (AA) at (270:1.2) {$\Sigma_M(P)$};
\end{scope}
\begin{scope}[shift={(6,0)}]
%\draw[->] (0:0) -- (0:1);
%\node (A1) at (0:1.1) {$1$};
\draw[->] (0:0) -- (90:1);
\node (A2) at (90:1.15) {$2$};
\draw[->] (0:0) -- (225:1);
\node (A3) at (225:1.15) {$3$};
\node (AA) at (270:1.2) {$\Sigma_M(Q)$};
\end{scope}
\begin{scope}[shift={(9,0)}]
\draw[->] (0:0) -- (0:1);
\node (A1) at (0:1.1) {$1$};
%\draw[->] (0:0) -- (90:1);
%\node (A2) at (90:1.15) {$2$};
\draw[->] (0:0) -- (225:1);
\node (A3) at (225:1.15) {$3$};
\node (AA) at (270:1.2) {$\Sigma_M(R)$};
\end{scope}
\begin{scope}[shift={(12,0)}]
\draw[->] (0:0) -- (0:1);
\node (A1) at (0:1.1) {$1$};
%\draw[->] (0:0) -- (90:1);
%\node (A2) at (90:1.15) {$2$};
%\draw[->] (0:0) -- (225:1);
%\node (A3) at (225:1.15) {$3$};
\node (AA) at (270:1.2) {$\Sigma_M(P) \cap \Sigma_M(R)$};
\end{scope}
\end{tikzpicture}
\end{center}
\caption{Positive Bergman fans and their intersections.}
\label{fig:posBerg}
\end{figure}
\end{example}

\subsection{Twisted intersection forms}
We recall the definition of the intersection forms on twisted (co)homology due to Cho and Matsumoto \cite{CM} and Kita and Yoshida \cite{KY}.  For more details, see \cref{sec:twistedco}.  
Let $\bA$ be a projective hyperplane arrangement, and let $E$ be the indexing set for hyperplanes given by $\{f_e = 0\}$, with $0 \in E$ the hyperplane at infinity.  Let $a_e$, $e \in E$ be complex parameters. 
Consider the meromorphic 1-form 
$$
\omega = \omega_\a = \sum_e a_e \dlog f_e = \sum_{e \in E \setminus 0} a_e \dlog(f_e/f_0) \in \Omega^1(\bU)
$$
on $\bU$, where we assume that $\sum_{e \in E} a_e = 0$, or equivalently, $a_0 = - \sum_{e \in E \setminus 0} a_e$.  We have a logarithmic connection
$(\O_\bU,\nabla_\a := d + \omega \wedge)$ on the trivial rank one vector bundle $\O_\bU$ on $\bU$.  The flat (analytic) sections of $\nabla_\a$ define a complex rank one local system $\L_\a$ on $\bU$.  Up to isomorphism, the local system $\L_\a$ is determined by a representation of the fundamental group $\pi_1(\bU)$; the natural generators $\gamma_e, e \in E$ of $\pi_1(\bU)$ are sent to the monodromy values $b_e = \exp(-\pi i a_e)$. 

When the genericity hypothesis
\begin{equation}\label{eq:Mon}
a_F = \sum_{e \in F} a_e \notin \Z \mbox{ for all connected }F \in L(M) \setminus \{ \hat 0, \hat 1\}
\end{equation}
is satisfied, a theorem of Kohno \cite{Koh} (see \cref{thm:Koh}) states that we have \emph{regularization} isomorphisms
$$
\reg: H^{\lf}_k(\bU,\L_\a) \stackrel{\cong}{\longrightarrow} H_k(\bU,\L_\a), \qquad
\reg: H^k(\bU,\nabla_\a) \stackrel{\cong}{\longrightarrow} H^k_c(\bU,\nabla_\a)
$$
between locally-finite (or Borel-Moore) twisted homology and usual twisted homology, and between twisted cohomology and compactly supported twisted cohomology.  These isomorphisms are inverse to the natural maps between these (co)homologies.  The intersection forms $\gdRip{\cdot,\cdot}$ and $\gBip{\cdot,\cdot}$ are defined by composing the Poincar\'e-Verdier duality pairings with the regularization isomorphism:
%\stackrel{{\rm id} \otimes \reg}{\longrightarrow} 
% \stackrel{\text{Poincar\'e-Verdier}}
\begin{align*}
\gdRip{\cdot,\cdot}&: H^d(\bU,\nabla^\vee_\a) \otimes H^d(\bU,\nabla_\a) \xrightarrow{{\rm id} \otimes \reg} H^d(\bU,\nabla^\vee_\a) \otimes H^d_c(\bU,\nabla_\a) \xrightarrow{\text{Poincar\'e-Verdier}} \C, \\
\gBip{\cdot,\cdot}&: H^{\lf}_d(\bU,\L^\vee_\a) \otimes H^{\lf}_d(\bU,\L_\a)  \xrightarrow{{\rm id} \otimes \reg}H^{\lf}_d(\bU,\L^\vee_\a) \otimes H_d(\bU, \L_\a)  \xrightarrow{\text{Poincar\'e-Verdier}} \C.
\end{align*}
In the deRham case $\gdRip{\cdot,\cdot}$, we view this as a bilinear form on the Aomoto cohomology $\rOS(M,\omega)$ of the Orlik-Solomon algebra, using the result \cref{thm:ESV} of Esnault--Schechtman--Viehweg \cite{ESV}.  In the Betti case $\gBip{\cdot,\cdot}$, we choose a basis of twisted cycles with the \emph{standard loading}, and obtain a bilinear form on $\Z^{\T^0}$.  In both cases, somewhat surprisingly, the bilinear form turns out to be symmetric.
 
 It has long been expected that the intersection forms $\gdRip{\cdot,\cdot}$ and $\gBip{\cdot,\cdot}$ have explicit combinatorial formulae.  For instance, we may quote Matsumoto and Yoshida \cite[p. 228]{MYrecent}: ``We expect that these intersection numbers can be expressed combinatorially in a closed form."  In \cref{thm:dRpairmain} and \cref{thm:Bettipairmain} we resolve this question in the affirmative. 
 
 \begin{theorem}\label{thm:combgeom}
 In the case of a projective hyperplane arrangement, the geometrically defined intersection forms $\gdRip{\cdot,\cdot}$ and $\gBip{\cdot,\cdot}$ agree with the combinatorially defined intersection forms $\dRip{\cdot,\cdot}$ and $\halfip{\cdot,\cdot}_B$ when the parameters satisfy $\sum_{e\in E} a_e =0$ (resp. $\prod_{e \in E} b_e =1$).
 \end{theorem}
 
The basic approach to the computation of the intersection forms is the same as in the original works \cite{CM,KY}, and carried out in various cases in, for example, \cite{MOY,MY,Goto,Tog,MHcoh,MHhom}.  Our key novelty lies in the systematic use of the wonderful compactification $X_{\max}$ of $\bU$ associated to the maximal building set.  

\begin{remark}\label{rem:descent}
For generic parameters, the bilinear form $\dRip{\cdot,\cdot}$ is non-degenerate on $\OS(M)$, but in \cref{sec:Aomoto} we show that when $\sum_e a_e = 0$ is satisfied, the bilinear form $\dRip{\cdot,\cdot}$ descends to the Aomoto cohomology $\rOS(M,\omega)$.  Similarly, for generic parameters the bilinear form $\halfip{\cdot,\cdot}_B$ is non-degenerate on $\Z^{\T^+}$, but when $\prod_{e \in E} b_e =1$, the rank drops, and it restricts to a non-degenerate bilinear form on $\Z^{\T^0}$ (see \cref{thm:Bettinondeg}).

We view the bilinear forms $\dRip{\cdot,\cdot}$ and $\halfip{\cdot,\cdot}_B$ with generic parameters as the ``correct" combinatorial objects, as they lead to the most elegant combinatorics.  We expect these bilinear forms can be geometrically interpreted as {\bf local} twisted intersection forms for the corresponding central hyperplane arrangement.
\end{remark}

Recall that a very affine variety $U$ is a closed subvariety of a complex torus.  The description of the intersection forms in terms of the Bergman fan (\cref{thm:fan}) is especially attractive because of the following natural problem.
 
\begin{problem}\label{prob:Bergman}
Generalize \cref{thm:combgeom} to arbitrary very affine varieties $U$ by replacing the Bergman fan $\Sigma_M$ with the tropicalization $\Trop(U)$.
\end{problem}
We point the reader to \cite[Section 6]{LamModuli} for more discussion in this direction.  In the case that $U$ is the uniform matroid stratum of the Grassmannian $\Gr(k,n)$, \cref{prob:Bergman} is related to the study of the generalized biadjoint scalar amplitudes of Cachazo-Early-Guevara-Mizera \cite{CEGM,CEZ,CEZ24}.

\subsection{deRham homology intersection form}
For a subset $B \subseteq E$, denote
$$a^B:= \prod_{b \in B} a_b.$$ 
For two bounded topes $P,Q \in \T^\star$, we define in \cref{def:DdR} the set $\B(P,Q)$, consisting of all bases $B \in \B(M)$ such that both topes $P$ and $Q$ belong to the \emph{bounded simplex} cut out by $B$.  The \emph{deRham homology intersection form} on $\Z^{\T^\star}$ is defined to be
$$
\DdRip{P,Q} := \sum_{B \in \B(P,Q)} a^B.
$$

\begin{theorem}
The bilinear form $\frac{1}{a_E}\DdRip{\cdot,\cdot}$ is the inverse of the bilinear form $\dRip{\cdot,\cdot}$ with respect to the basis $\{\Omega_P \mid P \in \T^\star\}$.
\end{theorem}

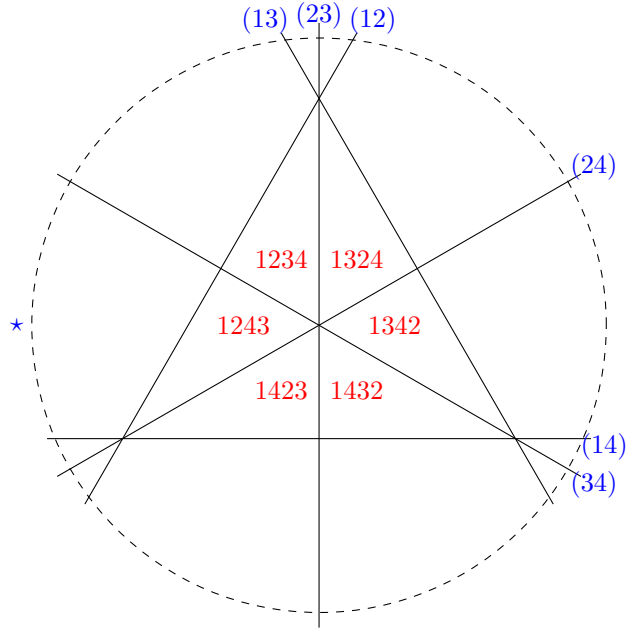
\begin{figure}
\begin{center}
$$
\begin{tikzpicture}[extended line/.style={shorten >=-#1,shorten <=-#1},
 extended line/.default=1cm]
 \draw[fill=none,dashed](0,0) circle (3.8);
 \draw[extended line] (90:3) -- (210:3);
  \draw[extended line] (90:3) -- (330:3);
   \draw[extended line] (330:3) -- (210:3);
 \draw[extended line] (90:3) -- (270:3);   
 \draw[extended line] (210:3) -- (30:3);   
  \draw[extended line] (330:3) -- (150:3);   
  \node[color=blue] at (100:4.1) {$(13)$};
  \node[color=blue] at (90:4.1) {$(23)$};
    \node[color=blue] at (80:4.1) {$(12)$};

  \node[color=blue] at (330:4.2) {$(34)$};
    \node[color=blue] at (337:4.1) {$(14)$};

  \node[color=blue] at (30:4.2) {$(24)$};
  \node[color=blue] at (180:4) {$\star$};
\node[color=red] at (120:1) {$1234$};
\node[color=red] at (60:1) {$1324$};
\node[color=red] at (0:1) {$1342$};
\node[color=red] at (-60:1) {$1432$};
\node[color=red] at (-120:1) {$1423$};
\node[color=red] at (-180:1) {$1243$};
\end{tikzpicture}
$$
\end{center}
\caption{The configuration space of $5$ point on $\P^1$, drawn with a general extension $\star$ at infinity.}
\label{fig:M05star}
\end{figure}
\begin{example}\label{ex:KLTexample}
In \cref{fig:M05star} we have drawn the hyperplane arrangement associated to the configuration space $M_{0,5}$, with a general extension $\star$ drawn as the ``circle at infinity".  The set $\T^\star$ consists of the six labeled regions bounded with respect to $\star$ and which are labeled by the permutations $w \in S_4$ satisfying $w(1) = 1$.  Two simplices contain both $1234$ and $1342$, namely $B = \{(12),(13),(14)\}$ and $\{(12),(13),(34)\}$.  One additional simplex $B = \{(12),(13),(24)\}$ contains both $1234$ and $1324$.  We obtain
$$
\DdRip{1234,1342} = a_{12}a_{13}( a_{14} + a_{34}), \qquad \DdRip{1234,1324} = a_{12}a_{13}( a_{14} +a_{24}+ a_{34}).
$$
\end{example}

\begin{remark}
The elegance of the deRham homology intersection form, and in particular the fact that it is positive, suggests that there is a direct geometric interpretation of this form, without relying on the duality with the deRham cohomology intersection form.
\end{remark}

\subsection{Betti cohomology intersection form}
Given $P,Q \in \T^+$, define the \emph{separating set}
$$
\sep(P,Q) := \{ e \in E \setminus 0 \mid P(e) \neq Q(e)\} \subset E.
$$
In the case of an affine hyperplane arrangement, these are the set of hyperplanes, not including the plane at infinity, that separate $P$ from $Q$.  The Betti cohomology intersection form on $\Z^{\T^+}$ is defined to be
$$
\ip{P,Q}^B := b_{\sep(P,Q) }+ (-1)^r b_{E \setminus \sep(P,Q)} = \ip{Q,P}^B
$$
for $P,Q \in \T^+$. In fact, $\ip{P,Q}^B$ is actually defined for $P,Q \in \T$, and $\ip{P,Q}^B= (-1)^r \ip{P,-Q}^B$.  The following result is \cref{thm:Bettiinverse}.

\begin{theorem}
The $\T^+ \times \T^+$ matrices $(-1)^{r-1}(1- b_E)^{-1}\ip{\cdot,\cdot}^B_{\T^+}$ and $\ip{\cdot,\cdot}^{\T^+}_B$ are inverse.
\end{theorem}

\begin{example}
Consider the hyperplane arrangement of \cref{fig:M05star} and take $0$ to be the hyperplane $(12)$.  Then we have
$\halfip{1234,1324}^B = b_{23} - b_{12}b_{13}b_{14}b_{24}b_{34}$ and $\halfip{1234,1423}^B = b_{24}b_{34}-b_{12}b_{13}b_{14}b_{23}$.
\end{example}

\begin{remark}
The elegance of the Betti cohomology intersection form suggests that there is a direct geometric interpretation of this form, without relying on the duality with the Betti homology intersection form.
\end{remark}

\subsection{Relation to the bilinear forms of Schechtman--Varchenko and Varchenko}
In \cite{SV}, motivated by the study of Knizhnik-Zamolodchikov equations, Schechtman and Varchenko introduced a \emph{contravariant form} $\ip{\cdot,\cdot}^{SV}$ on the Orlik-Solomon algebra $\rOS(M)$ of a hyperplane arrangement.  Their bilinear form is an analogue of the Shapovalov form of a highest weight representation of a Kac-Moody algebra.  The contravariant form is generalized to an arbitrary matroid by Brylawski and Varchenko \cite{BV}, and the restriction of the form to ``singular vectors" (corresponding to the Aomoto cohomology of the Orlik-Solomon algebra) was studied by Falk and Varchenko \cite{FalkVar}.  The following result is proved as \cref{cor:SVform}; see also \cref{rem:a0infinity}.

\begin{corollary}\label{cor:SV}
The Schechtman--Varchenko contravariant form $\ip{\cdot,\cdot}^{SV}$ for a central hyperplane arrangement is equal to the deRham intersection form $\dRip{\cdot,\cdot}$ up to an overall factor of $a_E$.  For an affine arrangement, the Schechtman--Varchenko contravariant form $\ip{\cdot,\cdot}^{SV}$ is obtained from the deRham intersection form $\dRip{\cdot,\cdot}$ by evaluating at $a_0 = \infty$.
\end{corollary}

Schechtman and Varchenko \cite[(4.7.4)]{SV} relate the contravariant form to twisted (co)homology via an asymptotic formula.  As described in \cref{rem:BBM} below, Belkale, Brosnan, and Mukhopadhyay \cite{BBM} show that the twisted deRham cohomology intersection form $\gdRip{\cdot,\cdot}$ can be obtained from $\ip{\cdot,\cdot}^{SV}$.  This should be compared to our \cref{thm:combgeom} and \cref{cor:SV}. 

\begin{remark}\label{rem:BBM}
Let $\bU$ be a projective hyperplane arrangement with matroid $M$, and let the $a_e$ be generic.
View the Schechtman--Varchenko contravariant form as a map $S:\rOS(M)^* \to \rOS(M)$ (\cref{prop:Fk} and \eqref{eq:RS}).  Then \cite[(2.7)]{BBM} show that the composition
\begin{equation}\label{eq:BBM}
\rOS(M, \omega)^* \to \rOS(M)^* \stackrel{S}{\longrightarrow} \rOS(M) \longrightarrow \rOS(M, \omega)
\end{equation}
can be identified with $\gdRip{\cdot,\cdot}$, after composing with the isomorphism $\rOS(M,\omega) \cong H^*(\bU, \nabla_\a)$.
Note that in \eqref{eq:BBM} the bilinear form $\ip{\cdot,\cdot}^{SV}$ (giving rise to the map $S:\rOS(M)^* \to \rOS(M)$) has full rank on $\rOS(M)$, in contrast to our description of $\gdRip{\cdot,\cdot}$ (\cref{rem:descent}).

We thank Prakash Belkale for explaining the results of \cite{BBM} to us.
\end{remark}

In \cite{Var}, Varchenko introduces a bilinear form $\ip{\cdot,\cdot}^V$ on a real configuration of hyperplanes.  As Varchenko observes, the contravariant form $\ip{\cdot,\cdot}^{SV}$ is the quasiclassical limit of $\ip{\cdot,\cdot}^V$.  The bilinear form $\ip{\cdot,\cdot}^V$ was generalized to the setting of oriented matroids in \cite{HV,Ran}.

\begin{corollary}[{\cref{cor:Var}}]
Varchenko's bilinear form $\ip{\cdot,\cdot}^V$ is obtainted from the Betti cohomology intersection form $\ip{\cdot,\cdot}^B$ by evaluating at $b_0 = 0$.  Equivalently, Varchenko's bilinear form is the inverse of the Betti homology intersection form $\halfip{\cdot,\cdot}_B$, after evaluating at $b_0 = 0$.
\end{corollary}
This appears to be the first geometric interpretation of Varchenko's bilinear form $\ip{\cdot,\cdot}^V$.

Among the deep properties of their contravariant form $\ip{\cdot,\cdot}^{SV}$, Schechtman--Varchenko \cite{SV} proved a formula for its determinant (recalled in \cref{thm:SVdet}), and an analogous determinant for $\ip{\cdot,\cdot}^V$ is given in \cite{Var}.  We give variants of these results: in \cref{thm:Aomotodet} we compute the determinant of $\bdRip{\cdot,\cdot}$ on Aomoto cohomology, and in \cref{thm:Bettihomdet}, we compute the determinant of $\halfip{\cdot,\cdot}_B$ on the lattice $\Z^{\T^+}$.

\subsection{Scattering amplitudes}
Our work is motivated by the theory of scattering amplitudes in physics, and especially the scattering equations of Cahcazo-He-Yuan \cite{CHYarbitrary}.  For a survey intended for mathematicians, we refer the reader to \cite{LamModuli}.  In the CHY formalism for the scattering of $n+1$ particles, \emph{kinematic space} $K_{n+1}$ (roughly, the space of momentum vectors of $n$ particles) is coupled with the \emph{worldsheet}, the moduli space $M_{0,n+1}$ by \emph{scattering equations} (S.E.).  Various scattering amplitudes can then be obtained via the CHY ansatz:
$$
{\rm amplitude} = \sum_{\text{solns } p \text{ to S.E.}} f(p) 
$$
where $f(p)$ is a rational function on $M_{0,n+1}$ evaluated at the solution $p$ to the scattering equations.  The choice of function $f(p)$ depends on the specific quantum field theory: biadjoint scalar, Yang-Mills, gravity, and so on.  As explained in \cite{LamModuli} and reviewed in \cref{sec:veryaffine}, the \emph{biadjoint scalar} amplitudes can be viewed as functions $A(\Omega,\Omega')$ that depend on the choice of two rational top-forms $\Omega,\Omega'$, and this definition extends the CHY formalism to the setting of very affine varieties.  Here, the very affine variety $U$ takes the role of the worldsheet, replacing the moduli space $M_{0,n+1}$.

In \cref{sec:amplitude}, we define amplitudes for matroids using the deRham intersection form $\dRip{\cdot,\cdot}$ and the canonical forms of \cref{thm:EL}.  We show in \cref{thm:AP} the basic properties of ``locality" and ``unitarity" for matroid amplitudes.  This result exposes a surprising parallel between the dichotomy of deletion-contraction in matroid theory and factorization phenomena in quantum field theory.

In the case of $U = M_{0,n+1}$, the relationship between twisted cohomology and CHY amplitudes was first observed by Mizera \cite{Miz}, and this equality was proven in a general setting by Matsubara-Heo \cite[Corollary 2.7]{MHcoh}.  In \cref{sec:scatform}, we give a new proof of this equality in the case that $U$ is a hyperplane arrangement complement.  Our approach relies on the definition of a scattering correspondence \cref{def:scatcorr}, which has appeared in the setting of hyperplane arrangements \cite{CDFV} and in likelihood geometry \cite{Huh,HS}.  

In \cref{sec:M0n}, we spell out some of our results in the case $U = M_{0,n+1}$, which is the case of the complete graphic matroid $M = M(K_n)$.  We obtain a new formula (\cref{thm:temporal}) for biadjoint scalar amplitudes in terms of objects we coin \emph{temporal Feynman diagrams}.  We show (\cref{thm:Frost}) that the celebrated field-theory KLT (Kawai-Lewellen-Tye) matrix \cite{BDSV} can be obtained from our results in a form that is different to the existing literature.  In \cref{cor:det1} and \cref{cor:det2}, we give new formulae for determinants of matrices of partial amplitudes.  We summarize the basic analogies between matroids and quantum field theory in the following table.
\begin{center}
\begin{tabular}{|c|c|}
\hline

worldsheet & matroid \\
\hline
kinematic space & dual of Lie algebra of intrinsic torus \\
\hline
\# of solutions to scattering equations & beta invariant \\
\hline
Parke-Taylor form & canonical form of a tope \\
\hline
biadjoint scalar partial amplitude & Laplace transform of Bergman fan \\
\hline
inverse string KLT matrix & discrete Laplace transform of Bergman fan\\
\hline
physical poles & connected flats \\
\hline
factorization & deletion-contraction \\
\hline
Feynman diagram & flag of flats \\
\hline
\end{tabular}
\end{center}

\subsection{Matroids and motives}
We have largely excluded from this work a discussion of the generalized hypergeometric functions 
\begin{equation}\label{eq:AG}
\int_{[P]} \varphi_P \; \Omega
\end{equation}
studied by Aomoto \cite{Aom} and Gelfand \cite{Gel}.  These integral functions are a main motivation for the study of twisted (co)homologies of hyperplane arrangement complements.  Indeed, the integrals \eqref{eq:AG} are given by pairings between twisted cocycles $[\Omega] \in H^d(U,\nabla_\a)$ and twisted cycles $[P \otimes \varphi_P] \in H_d(U,\L^\vee_\a)$.  As noted in the original work of Cho and Matsumoto \cite{CM}, the computation of the intersection forms $\ip{\cdot,\cdot}^\nabla$ and $\ip{\cdot,\cdot}_{\L}$ leads to explicit period relations for the twisted periods \eqref{eq:AG}.  See for example \cite{MOY,MY,Goto}.  We briefly discuss twisted period relations in \cref{sec:beta}.

The relation to scattering amplitudes suggests one to focus on the special case when $\Omega = \Omega_P$ is a canonical form in \eqref{eq:AG}.  The resulting integral functions, which we call \emph{string amplitudes for hyperplane arrangements}, will be studied in the work \cite{Lamstring}.  In the special case that $U = M_{0,n+1}$, these functions are the open string theory amplitudes at tree-level; see \cite{AHLstringy,BD,Miz}. 

Let us explicitly articulate one of the main directions that our work opens up.

\begin{problem}\label{prop:motives}
For an oriented matroid $\M$, define and study the space of all twisted period matrices $\mathbf{P}^\a$ (as in \cref{sec:beta}) compatible with $\M$.
\end{problem}

We view \cref{prop:motives} as a step towards \emph{(twisted) motives} for matroids.  We have seen that the intersection forms $\dRip{\cdot,\cdot}, \DdRip{\cdot,\cdot}, \halfip{\cdot,\cdot}^B,\halfip{\cdot,\cdot}_B$ exist even for matroids not arising from hyperplane arrangements.  A fundamental tension is the question: do the twisted period matrices $\mathbf{P}^\a$ exist when $M$ is a nonrealizable matroid?
\subsection*{Acknowledgements}
We acknowledge support from the National Science Foundation under grants DMS-1953852 and DMS-2348799.  We thank the Simons Foundation for support under a Simons Fellowship.  We are grateful to the Institute for Advanced Study, Princeton for supporting a visit during which part of this manuscript was completed.
We thank Chris Eur for our parallel joint work on canonical forms for matroids.  We thank Hadleigh Frost, June Huh, Sebastian Mizera, Oliver Schlotterer, Bernd Sturmfels, and Simon Telen for stimulating discussions.  We thank Prakash Belkale, Nick Early, and Alexander Varchenko for helpful comments on an earlier version of this manuscript.

\part{Combinatorics}

\section{Matroids}\label{sec:matroids}
We denote $[n]:=\{1,2,\ldots,n\}$.
\subsection{Conventions for matroids}
Let $M$ be a matroid of rank $r = d+1$ with ground set $E$.  We use the notation
\begin{align*}
\rk = \rk_M  &= \mbox{rank function of $M$,} \\ 
\B(M) &= \mbox{set of bases of $M$,} \\
\I_k(M) &= \mbox{$k$-element independent sets of $M$.} 
\end{align*}

An element $e \in E$ is a \emph{loop} if it belongs to no bases, and a \emph{coloop} if it belongs to all bases.  Two elements $e, e' \in E$ are called parallel if they belong to the same bases.  An element $e \in E$ is in \emph{general position} if $\rk(S \cup e) = \min(\rk(S) + 1,r)$ for any $S \subseteq E \setminus e$.

A matroid $M$ is called \emph{simple} if it has no loops and no parallel elements.
If $M,M'$ are matroids on the ground sets $E,E'$ with ranks $r, r'$, then the \emph{direct sum} $M\oplus M'$ is the rank $(r+r')$ matroid on the ground set $E \sqcup E'$ with bases $\B(M\oplus M') = \{B \sqcup B' \mid B \in \B(M), B' \in \B(M')\}$.  A matroid $M$ is called \emph{connected} or \emph{indecomposable} if it cannot be expressed as a non-trivial direct sum $M = M|_{E_1} \bigoplus M|_{E_2}$ where $E = E_1 \sqcup E_2$.

Let $L(M)$ denote the lattice of flats of $M$, and let $L^k(M)$ denote the set of flats of rank $k$.  Each flat $F \in L(M)$ is viewed as a subset of $E$.  By convention $L(M)$ has minimal element $\hat 0$ (consisting of all the loops) and maximal element $\hat 1 = E$.  We use $\vee$ and $\wedge$ to denote the join and meet operations of $L(M)$.  A flat $F$ is called \emph{connected} if the restriction $M^F$ (see \cref{ssec:extensions}) is connected.  An atom $a \in L(M)$ is a flat of rank one and we let $\At(M)$ denote the set of atoms of $M$.  An atom in a loopless matroid consists of an equivalence class of parallel elements of $M$.  We say that an atom $a\in \At(M)$ is a coloop if any of the elements in $a$ is a coloop.  For an example of $L(M)$, see \cref{fig:5line}.

An \emph{affine matroid} $(M,0)$ is a matroid $M$ together with a distinguished element $0 \in E$.  In terms of hyperplane arrangements, $0$ indexes the hyperplane at infinity.  We say that an affine matroid $(M,0)$ is generic at infinity if $0 \in E$ is in general position.  

\subsection{Some invariants}
We will be interested in the following invariants of a matroid $M$:
\begin{align*}
\chi_M(t) &= \mbox{characteristic polynomial}\\
\bchi_M(t) &= \mbox{reduced characteristic polynomial}\\
\mu^+(M) &= \mbox{unsigned M\"obius invariant} \\
\beta(M) &= \mbox{beta invariant} \\
w_\Sigma(M) = |\bchi_M(-1)| &= \mbox{(reduced) total Whitney invariant}
\end{align*}

Let $\mu = \mu_{L(M)}(x,y)$ denote the Mobius function of $L(M)$, where $[x,y]$ is an interval in $L$.  For $x \in L$, we set $\mu(x) := \mu(\hat 0, x)$.  Let $\mu(M):= \mu(\hat 1)$ denote the \emph{Mobius invariant} of $M$, and let $\mu^+(M) = |\mu(M)|$ denote the unsigned Mobius invariant.  Let $\chi_M(t)$ (resp. $\bchi_M(t)$) denote the \emph{characteristic polynomial} (resp. reduced characteristic polynomial) of $M$, given by
$$
\chi_M(t):= \sum_{F \in L(M)} t^{r - \rk(F)} \mu(F), \qquad \text{and} \qquad \bchi_M(t) := \chi_M(t)/(t-1).
$$
The \emph{beta invariant} $\beta(M)$ of $M$ is given by 
$$
\beta(M) := (-1)^{r+1} \left.\frac{d}{dt} \chi_M(t) \right|_{t=1}.
$$ 
If $e \in E$ is neither a loop nor a coloop, then we have the recursion
\begin{equation}\label{eq:betaeq}
\beta(M) = \beta(M/e) + \beta(M\setminus e)
\end{equation}
We have $\beta(M) = 0$ if and only if $M$ is disconnected, or a loop, or empty ($|E|=0$).

\subsection{Extensions and liftings}\label{ssec:extensions}
For a flat $F \in L(M)$, we have the matroids 
\begin{align*}
M^F &:= \text{restriction of $M$ to $F$} = \text{deletion of $E \setminus F$ from $M$}  \\
M_F&:= \text{contraction of $M$ by $F$}.
\end{align*} 
The lattice $L(M^F)$ of flats of $M^F$ (resp. $L(M_F)$ of flats of $M_F$) is isomorphic to the lower order ideal $[\hat 0, F] \subset L(M)$ (resp. upper order ideal $[F, \hat 1]\subset L(M)$).  For an element $e \in E$, we denote by $M\backslash e $ the deletion of $e$, and by $M/e = M_e$ the contraction of $M$ by $e$.  We call $(M, M' = M\backslash e, M'' = M/e)$ a deletion-contraction triple.  More generally, we have a deletion-contraction triple for any atom $a \in \At(M)$.

An \emph{extension} (resp. \emph{lifting}) $\tilM$ of $M$ is a matroid $\tilM$ on $\tE = E \cup \star$ such that the deletion $\tilM \backslash\star$ (resp. contraction $\tilM/\star$) is equal to $M$.   The extension or lifting $\tilM$ is called general if the element $\star$ is in general position in $\tilM$.  Given a matroid $M$ on $E$ , we often let $(\tilM, \star)$ denote an affine matroid on $\tilde E = E \cup \star$ which is a general extension of $M$ by an element $\star$.

\begin{lemma}\label{lem:betageneric}
Suppose that $(\tilM,\star)$ is a general extension of a non-loop matroid $M$.  Then $\mu^+(M) = \beta(\tilM)$.
\end{lemma}
\begin{proof}
We may assume that $M$ is simple.  Then we have 
$$
\chi_M(t) = \sum_{A \subset E} (-1)^{|A|} t^{r- \rk(A)}, \qquad \chi_{\tilM}(t) = \sum_{A \subset E \cup \star} (-1)^{|A|} t^{r- \rk(A)}.
$$
By genericity, if  $\star \notin A$ and $\rk(A) < d$ then $\rk(A \cup \star) = \rk(A) + 1$.  Also,  if $\rk(A) = r$ then $\rk(A \cup \star) = r$.  It follows that
$$
 \chi_{\tilM}(t) = (\chi_M(t)-\chi_M(0)) (1 - 1/t).
$$
Thus,
$$
(-1)^{r+1}  \beta(\tilM) = \left. \frac{d}{dt}  \chi_{\tilM}(t) \right|_{t=1} =\left(\chi'_M(t)(1-1/t)-(\chi_M(t)-\chi_M(0))(1/t^2)\right)|_{t=1} = \chi_M(0)-\chi_M(1) = \mu(M),
$$
where for the last equality we have used $\chi_M(1) = 0$.
\end{proof}

\begin{lemma}\label{lem:genericlift}
Suppose that $\overline{M}$ is a general lifting of a matroid $M$.  Then $\chi_{\overline{M}}(t) = (t-1) \chi_M(t)$.
\end{lemma}
\begin{proof}
We may assume that $M$ is simple and of rank $r$.  Then we have 
\begin{align*}
 \chi_{\bM}(t) &= \sum_{A \subset E \cup \star} (-1)^{|A|} t^{r+1- \rk_{\overline{M}}(A)} \\
 &= \sum_{A \subset E}(-1)^{|A|} t^{r+1- \rk_{M}(A)} + \sum_{A \cup \star \subset E \cup \star}(-1)^{|A|+1} t^{r+1- \rk_{M}(A)-1} \\
 &= t \chi_M(t) - \chi_M(t) = (t-1)\chi_M(t). \qedhere
\end{align*}
\end{proof}

\subsection{Flags of flats}
The order complex $\Delta(Q)$ of a poset $Q$ is the simplicial complex whose vertices are the elements of $Q$ and whose simplices are the chains of $Q$.  Define
$$\Delta(M) := \Delta(L(M)-\{\hat 0, \hat 1\}),$$ %$- \{\hat 0 , \hat 1\})$, 
the order complex of the (reduced) lattice of flats in $M$.  The faces $E_\bullet \in \Delta(M)$ can be identified with partial flags of flats
$$
E_\bullet = \{\hat 0 = E_0 \subset E_1 \subset E_2 \subset \cdots \subset E_{s} \subset E_{s+1}= E = \hat 1\}
$$
which start at $\hat 0$ and end at $\hat 1 = E$, and have $s = s(E_\bullet)$ intermediate flats.  The facets, or maximal simplices, of $\Delta(M)$ can be identified with complete flags of flats
$$
F_\bullet = \{\hat 0 = F_0 \subset F_1 \subset F_2 \subset \cdots \subset F_{r-1} \subset F_r = E = \hat 1\}
$$
where $F_i$ is a flat of rank $i$.  We denote by $\Fl(M)$ the set of complete flags in $L(M)$, or equivalently, the set of facets of $\Delta(M)$.  Let $\Fl^k$ denote the set of saturated flags $F_\bullet = \{\hat 0 = F_0 \subset F_1 \subset \cdots \subset F_k \mid \rk(F_i) = i\}$ of length $k$ starting at $\hat 0$.  Let $\Fl^\bullet(M) = \bigcup_k \Fl^k(M)$ denote the set of all saturated flags in $L(M)$ starting at $\hat 0$.

\subsection{Oriented matroids}\label{sec:OM}

Let $\M$ be an oriented matroid with underlying matroid $M$.   We typically view $\M$ as a collection of \emph{signed covectors}, certain sign sequences $X: E \to \{+,0,-\}$ satisfying a collection of axioms \cite{OMbook}.  For a signed covector $X$, the zero set $X_0 \subset E$ is given by $X_0:= \{e \in E \mid X(e) = 0\} \in L(M)$, and is a flat.  The negative $-X$ %or $X^-$ 
of a signed covector $X$ is always a signed covector.  Given two signed covectors $X,Y$ of $\M$, the composition $X \circ Y$ is also a signed covector of $\M$ and is defined by
\begin{equation}\label{eq:compo}
(X \circ Y)(e) = \begin{cases} X(e) & \mbox{if $X(e) \neq 0$} \\
Y(e) & \mbox{if $X(e) = 0$.}
\end{cases}
\end{equation}
Oriented matroids can also be axiomatized using \emph{chirotopes}: a function $\chi: \B(M) \to \{+,-\}$ satisfying a collection of axioms.  The choice of $\M$ is equivalent to the choice of a pair $\chi,-\chi$ of opposite chirotopes.  We typically assume that a choice of chirotope has been fixed, omitting it from the notation.  

Let $\L = \L(\M)$ denote the lattice of signed covectors of $\M$.  We have $X \leq Y$ in $\L(M)$ if $Y$ is obtained from $X$ by setting some entries to 0.  By convention, $\L$ has a minimal element $\minL$ and a maximal element $\hat 1 = (0,0,\ldots,0)$.  In the poset $\L \setminus \{\minL,\hat 1\}$, the maximal elements are signed cocircuits, and the minimal elements are \emph{topes}.  We let $\T = \T(\M)$ denote the set of topes of $\M$.  The oriented matroid $\M$ is \emph{acyclic} if there is a tope $P \in \T$ with $P(e) = +$ for all $e \in E$.  

There is a surjective map of posets
$$
\phi: \L(\M)\setminus \minL \to L(M), \qquad X \mapsto X_0 = \{e \in E \mid X(e) = 0\}
$$
sending a signed covector to its zero set.  The rank $\rk(X)$ is defined to be $\rk(X) = \rk(\phi(X))$.

For a tope $P \in \T$, we let $\L(P):=[P,\hat 1]$ denote the closed interval between $P$ and $\hat 1$.  The lattice $\L(P)$ is known as the \emph{Las Vergnas face lattice}.  The restriction of $\phi$ to $\L(P)$ is injective, with image equal to $L(P) \subset L(M)$.  We will often identify $\L(P)$ and $L(P)$ via this map.  The elements of $\L(P)$ or $L(P)$ are called the \emph{faces} of $P$.  Rank one faces are called \emph{facets}.  Corank one faces are called \emph{vertices}.  If $\M$ is acyclic and $P$ is the positive tope, then $L(P)$ is the set of zero sets of the nonnegative signed covectors of $\M$.  We let $\Fl(P) \subset \Fl(M)$ be the set of flags of flats that belong to $L(P)$.  Similarly, define $\Delta(P):= \Delta(L(P) - \{\hat 0,\hat 1\})$ to be the order complex of the reduced part of $L(P)$.

\subsection{Affine oriented matroids}\label{sec:AOM}
An \emph{affine oriented matroid} is a pair $(\M,0)$ where $0 \in E$ is a distinguished element.  We let $\T^+= \T^+(\M)$ denote the set of topes $P\in \T$ satisfying $P(0) = +$.  Thus, $\T^+$ can be identified with the orbits of $\T$ under negation.

\begin{defn}
 Given an affine oriented matroid $(\M,0)$, we define the \emph{bounded complex} by
$$
\L^0 := \{\minL\} \cup \{X \in \L \setminus \minL \mid Y(0) = + \text{ for all } Y \geq X\} \subset \L.
$$
The set of bounded topes $\T^0(\M)$ of $(\M,0)$ are the minimal elements of $\L^0 \setminus \minL$.
\end{defn}
By definition, we have $\T^0 \subset \T^+$.

Now let $\tM$ be an extension of $\M$ by an element labelled $\star$.  Given a sign sequence $X: E \to \{+,0,-\}$, we denote by $(\epsilon, X)$ the sign sequence $\widetilde X$ on $\widetilde E = \{\star\} \sqcup E$ defined by $\widetilde X(\star) = \epsilon$ and $\widetilde X(e) = X(e)$ for all $e\in E$.  The pair $(\tM, \star)$ is an affine oriented matroid, and we let 
$$
\T^\star = \T^\star(\widetilde\M) := \{\mbox{topes $P \in \T(\M)$ such that $(+,P)$ is bounded in }(\tM,\star)\}.
$$ 
 
If $\tM$ is a general extension of $\M$, then there is a simpler description of the set of bounded topes, not requiring one to check all faces of $P$.

\begin{lemma}
Suppose that $\tM$ is a general extension of $\M$.  Then we have
\begin{equation}\label{eq:Tstar}
\T^\star = \T^\star(\widetilde\M) = \{P  \in \T(M) \mid (+,P) \text{ is a tope of $\tM$ but $(-,P)$ is not}\}.
\end{equation}
\end{lemma}

The following result appears in classical work of Greene, Las Vergnas, Zaslavsky \cite{GZ, LV}.
\begin{proposition}\label{prop:numbertopes}
Let $(\M,0)$ be an affine matroid and let $(\tM,\star)$ be a general extension of $\M$.  We have
\begin{align*}
|\T^+| &= w_\Sigma(M),  \qquad
|\T^\star| = \mu^+(M), \qquad
|\T^0| = \beta(M).
\end{align*}
\end{proposition}

\section{Orlik-Solomon algebra and canonical forms}\label{sec:OS}
Let $M$ be a matroid on ground set $E$.
\subsection{Orlik-Solomon algebra}
Let $\Lambda^\bullet(E)$ denote the exterior algebra over $\Z$ generated by elements $e \in E$.  If $S = \{s_1,\ldots,s_k\} \subset E$ is an ordered set, then we write $e_S := e_{s_k} \wedge \cdots \wedge e_{s_1}$.  (We caution the reader that this convention is the reverse of that of \cite{EL}.)  Define the linear map $\partial: \Lambda^\bullet(E) \to \Lambda^{\bullet-1}(E)$ by
$$
\partial(e_{1} \wedge e_2 \wedge \cdots \wedge e_k) = \sum_{i=1}^k (-1)^{k-i} e_1 \wedge \cdots \wedge \widehat{e_i} \wedge \cdots \wedge e_k.
$$
We have $\partial^2 = 0$.  

\begin{definition}
The \emph{Orlik-Solomon algebra} $\OS^\bullet(M)$ is the quotient of the exterior algebra $\Lambda^\bullet(E)$ over $\Z$ by the ideal 
$$
I = (\partial e_S \mid S \subseteq E \text{ is dependent}).
$$
\end{definition}

When $E = [n]$, we denote the generators of $\OS(M)$ by $e_1,e_2,\ldots,e_n$ for clarity.  The Orlik-Solomon algebra is supported in degrees $0,1,\ldots,r$.

\begin{proposition}[\cite{OS,OTbook,SV}]
For each $k = 0,1,\ldots,r$, $\OS^k(M)$ is a free $\Z$-module with rank equal to the absolute value of the coefficient of $t^k$ in the characteristic polynomial $\chi_M(t)$.
\end{proposition}
In particular, $\OS(M)$ is a free $\Z$-module with rank $\mu^+(M)$.  
We have $\OS^0(M) \cong \Z$,  the isomorphism given by identifying the basis element $e_\emptyset \in \OS^0(M)$ with $1 \in \Z$.
Let 
\begin{align*}
\OS(M) &:= \OS^r(M) \mbox{ denote the top degree component of the Orlik-Solomon algebra.} 
\end{align*}
By convention, for the empty matroid $M_\emptyset$, we have $\OS(M_\emptyset) = \OS^0(M_\emptyset) \cong \Z$.  For a flat $F \in L(M)$ of rank $k$, define the subspace $\OS_F(M) \subset \OS^k(M)$ by
$$
\OS_F(M) = {\rm span}(e_S \mid S \in \I_k(M) \text{ and } \overline{S} = F) \cong \OS(M^F).
$$
\begin{proposition}[\cite{OTbook,SV}]\label{prop:OSsum}
We have a direct sum decomposition
$$
\OS^\bullet(M) = \bigoplus_{F \in L(M)} \OS_F(M).
$$
\end{proposition}

\subsection{Broken circuits}\label{sec:nbc}
A basis of $\OS^\bullet(M)$ can be constructed from the broken circuit complex, dating back to work of Wilf and Brylawski.  
Fix a total ordering $\prec$ on $E$.  A broken circuit is a set $C' = C \setminus \min(C)$ where $C$ is a circuit, and the minimum $\min(C)$ is taken with respect to $\prec$.  An independent set $S \subset E$ is called \nbc~if it does not contain any broken circuits.  A basis $B \in \B(M)$ is called a \nbc-basis if it does not contain any broken circuits.  
For the following, see \cite{OTbook, Yuz}.
\begin{theorem}
The set $\{e_S \mid \mbox{S is \nbc~and } S \in \I_k(M)\}$ is a basis of $\OS^k(M)$.
\end{theorem}

\subsection{Reduced Orlik-Solomon algebra}
We assume now that we have an affine matroid $(M,0)$.  For clarity, the element of $\Lambda^\bullet(E)$ that corresponds to $0 \in E$ is denoted $e_0$.  Since $\partial^2 = 0$, the map $\partial$ descends to a map $\partial: \OS^\bullet(M) \to \OS^{\bullet-1}(M)$.  We let 
$$
\rOS^\bullet(M) := \partial(\OS^{\bullet}(M)) \subset \OS^\bullet(M)
$$ 
denote the \emph{reduced Orlik-Solomon algebra}.  The subalgebra $\rOS^\bullet(M)$ is generated by $\be:= e- e_0$ for $e \in E \setminus 0$, and it is also equal to the kernel of $\partial$ on $\OS^\bullet(M)$.  The reduced Orlik-Solomon algebra is supported in degrees $0,1,\ldots,r-1$.  For the next results, see \cite[Section 2.7]{Yuz01} and \cite[Proposition 3.2]{Dim}.
\begin{proposition}
For each $k = 0,1,\ldots,r-1$, $\rOS^k(M)$ is a free $\Z$-module with rank equal to the absolute value of the coefficient of $t^k$ in the reduced characteristic polynomial $\bchi_M(t)$.
\end{proposition}

Let $d:= r-1$ and
\begin{align*}
\rOS(M)&:= \rOS^d(M) \mbox{ denote the top degree component of the reduced Orlik-Solomon algebra.}
\end{align*}

\begin{proposition}\label{prop:OSrOS}
Let $M$ be a matroid with rank $r \geq 1$.  We have an isomorphism $\partial: \OS(M) \stackrel{\cong}{\longrightarrow} \rOS(M)$.
\end{proposition}

\subsection{Canonical forms}
In this section we assume that an orientation $\M$ of $M$, together with a general extension $(\tM,\star)$ of $\M$, has been given.  Let $B \in \B(M)$ be a basis.  We now define topes that are in the bounded part of $B$; see \cite{EL}.  Let $C_B$ be the signed fundamental circuit on $B \cup \star$ of $\tM$ with $C_B(\star) = -$.  The circuit necessarily has support $B\cup \star$ by genericity of the extension $\tM$.  Note that $(+, C_B|_B)$, i.e.\ the sign sequence on $B\cup \star$ with $+$ at $\star$ and $C_B(i)$ at $i\in B$, is a tope in the restriction $\widetilde\M|_{B\cup \star}$ but $(-, C_B|_B)$ is not.  We say that a tope $P \in \T$ is in the \emph{bounded part} of $B$ if we have $P|_B = C_B|_B$.  Write
$$
\T^{B} = \{P \in \T \mid \mbox{$P$ is in the bounded part of $B$}\}.
$$

\begin{lemma}
For any basis $B$, we have $\T^B \subset \T^\star$.
\end{lemma}
\begin{proof}
If a tope $(+,P)$ satisfies $P|_B = C_B|_B$ for some basis $B$ of $\M$, then $P$ is automatically bounded in $\tM$ since $(-,P)$ cannot be orthogonal to $C_B$.
\end{proof}

For an \emph{unordered} basis $B \in M$, we say that an ordering $(b_1,b_2,\ldots,b_r)$ of $B$ is positive if $\chi(b_r,b_{r-1},\ldots,b_1) = +$, where $\chi$ is the chirotope of $\M$.  
We define an element
$$
e_B := \chi(b_r,b_{r-1},\ldots,b_1)  e_{b_r} \wedge \cdots \wedge e_{b_1}
\in A(M)$$
where $(b_1,b_2,\ldots,b_r)$ is any ordering of $B$.

In the following result we will use the residue maps between Orlik-Solomon algebras, reviewed in \cref{sec:residue}; see \cite{EL} for further details.
\begin{theorem}[{\cite[Theorem 2.10]{EL}}] \label{thm:EL}
For each $P \in \T(\M)$, there exists a distinguished element $\Omega_P \in \OS(M)$ satisfying the following properties:
\begin{enumerate}
\item
The \emph{canonical form} $\Omega_P$ is invariant under simplification of matroids, satisfies $\Omega_{-P} = (-1)^r \Omega_P$, and is uniquely characterized by the following recursion.  If $\M$ is the rank $0$ empty matroid with chirotope $\chi$, then $\Omega_P = \chi(\emptyset) \in \OS^0(M)$.  If $r \geq 1$, then for any atom $\atom \in \At(M)$, we have
$$
\Res_\atom \Omega_P = \begin{cases} P(e)\, \Omega_{P/\atom} \in \OS(M/\atom) &\mbox{if $\atom \in L(P)$,} \\
0 & \mbox{otherwise.}
\end{cases}
$$
Here, $P/\atom = P_\atom \in \T(\M/\atom)$ is the tope given by $P/\atom = P|_{E \setminus \atom}$, and the chirotope of $\M/\atom$ is fixed by choosing $e \in \atom$ and setting $\chi_{\M/\atom}(e_1,\ldots,e_{r-1}) := \chi_{\M}( e_1,\ldots,e_{r-1},e)$.
\item
For a general extension $(\tM, \star)$ of $\tM$, the elements
$$
\{\Omega_P \mid P \in \T^\star\}
$$
form a basis of $\OS(M)$, and for any basis $B \in \B(M)$, we have
\begin{equation}\label{eq:cone}
(-1)^{|C^{-1}_B(-)|-1}e_B = \sum_{P \in \T^B} \Omega_P.
\end{equation}
\end{enumerate}
\end{theorem}

By \cref{prop:OSrOS}, the set $\{\bOmega_P := \partial \Omega_P \mid P \in \T^\star\}$ is a basis of $\rOS(M)$.

\begin{example}
Let $\M$ be the oriented matroid of rank 2 associated to the arrangement of three points on $\P^1$, as in \cref{ex:3pt}.  Then $\OS^\bullet(M)$ is generated by $e_1,e_2,e_3$ with the relation $e_2 e_1 - e_3 e_1 + e_3e_2 = 0$.  The canonical forms of $P,Q,R$ are $\Omega_P = e_2 e_1, \Omega_Q = e_3 e_2, \Omega_R =  e_1 e_3$.  Any two of these give a basis of $\OS(M)$.  The reduced canonical forms are $\partial \Omega_P = e_2 - e_1, \partial \Omega_Q = e_3 - e_2, \partial \Omega_R = e_1 - e_3$.  Any two of these give a basis of $\rOS(M)$.
\end{example}

\begin{remark}
In the case that $M$ arises from a real hyperplane arrangement $\bA$, the canonical forms of \cref{thm:EL}, are the usual canonical forms of polytopes \cite{ABL,LamPosGeom}.  These forms have also appeared in the work of Yoshinaga \cite{Yos} where they are referred to as the ``chamber basis".
\end{remark}

\section{DeRham cohomology intersection form}
\subsection{Residue maps}\label{sec:residue}
\begin{proposition}[{\cite[Proposition 2.2]{EL}}]\label{prop:OSexact}
For every atom $\atom \in \At(M)$, we have a short exact sequence 
\[
0\longrightarrow \OS^\bullet(M\backslash \atom) \overset{\iota_{\atom}}\longrightarrow \OS^\bullet(M) \overset{\Res_{\atom}}\longrightarrow \OS^{\bullet-1}(M/\atom) \to 0
\]
where $\iota_{\atom}(e_I) = e_I$ for $I \subseteq E \setminus c$, and $\Res_{\atom}(e_I) = e_{I\setminus e}$ if $I = (e \in \atom,i_1, \dots, i_{k-1})$ and $\Res_{\atom}(e_I) = 0$ if $I \cap \atom = \emptyset$.
These maps restrict to give the short exact sequence
\[
0\longrightarrow \rOS^\bullet(M\backslash \atom) \longrightarrow \rOS^\bullet(M) \longrightarrow \rOS^{\bullet-1}(M/\atom) \to 0.
\]
\end{proposition}

Now let $F_\bullet = (\hat 0 = F_0 \subset F_1 \subset \cdots \subset F_k) \in \Fl^k(M)$ be a saturated flag of flats.  Then $F_1$ is an atom in $L(M)$, and for each $i = 1,2,\ldots,k-1$, we have that the contraction $F_{i+1}/F_i$ of $F_{i+1}$ is an atom in the lattice of flats $L(M/F_i)$ of the contraction $M/F_i$.  Thus the following definition makes sense.

\begin{definition}
Let $F_\bullet = (\hat 0 = F_0 \subset F_1 \subset \cdots \subset F_k) \in \Fl^k(M)$ be a saturated flag of length $k$.  The \emph{residue map} $\Res_{F_\bullet}: \OS^\bullet(M) \to \OS^{\bullet - k}(M/F_k)$ of the flag $F_\bullet$ is the $k$-fold composition
$$
\Res_{F_\bullet} =  \Res_{F_k/F_{k-1}} \circ \cdots \circ \Res_{F_2/F_1} \circ \Res_{F_1}: \OS^\bullet(M) \to \OS^{\bullet-k}(M/F_k).
$$
For an element $x \in \OS^k(M)$, we view the residue $\Res_{F_\bullet}(x)$ of $x$ at $F_\bullet$ as an integer via the identification $\OS^0(M/F_k) \cong \Z$. 
\end{definition}

By \cref{prop:OSexact}, $\Res_{F_\bullet}$ restricts to a residue map $\Res_{F_\bullet}: \rOS^\bullet(M) \to \rOS^{\bullet - k}(M/F_k)$.

\begin{example}
Let $M = U_{2,3}$ be the uniform matroid of rank 2 on $\{e_1,e_2,e_3\}$.  Let $F_\bullet = (\hat 0 \subset \{e_1\} \subset \hat 1)$.  Then 
$$
\Res_{F_\bullet} e_2 \wedge e_1 = \Res_{F_\bullet} e_3 \wedge e_1 = 1, \qquad \text{and} \qquad \Res_{F_\bullet} e_3 \wedge e_2 = 0.
$$
This is consistent with the relation $e_2e_1 - e_3e_1 + e_3 e_2 = 0$ in $\OS(M) = \OS^2(M)$.
\end{example}

For $S \subset E$, let $L(S) \subseteq L(M)$ be the sublattice of $L(M)$ generated by the atoms in $S$.  Equivalently, $L(S) = L(M \backslash S)$ where $M \backslash S$ is the matroid obtained by deleting all elements not in $S$.

Now let $S \in \I_k(M)$ be an independent set of size $k$ and let $F_\bullet \in \Fl^k$ be a saturated flag of length $k$.
We say that $F_\bullet$ is \emph{generated} by $S$ if $F_\bullet$ is a maximal chain in $L(S)$.  In other words, each $F_i$ is spanned by a subset of $S$.  Given a pair $(S,F_\bullet)$ where $F_\bullet$ is generated by an \emph{ordered} independent set $S=(s_1,s_2,\ldots,s_k)$, we define a permutation $\sigma = \sigma(S,F_\bullet) \in S_k$ by
\begin{equation}\label{eq:sigma}
F_i = {\rm span}(s_{\sigma(1)},s_{\sigma(2)},\ldots,s_{\sigma(i)}), \qquad \text{for } i = 1, 2, \ldots,k.
\end{equation}

\begin{definition}\label{def:rSF}
Let $F_\bullet \in \Fl^k(M)$ be a saturated flag and $S = (s_1,\ldots,s_k)$ be an ordered independent set. Define the \emph{residue $r(S, F_\bullet) \in \{0,1,-1\}$ of $S$ at $F_\bullet$} as follows.  If $F_\bullet$ is not generated by $S$ then we set $r(S,F_\bullet)= 0$.  If $F_\bullet$ is generated by $S$, then we set $r(S,F_\bullet) = (-1)^{\sigma(S,F_\bullet)}$ to be the sign of the permutation $\sigma(S,F_\bullet)$.  
\end{definition}

The following comparison follows immediately from the definitions.  
\begin{lemma}\label{lem:rSF}
Let $F_\bullet \in \Fl^k(M)$ be a saturated flag and $S = (s_1,\ldots,s_k)$ be an ordered independent set.  Then
$$
\Res_{F_\bullet}(e_S) = r(S, F_\bullet).
$$
\end{lemma}
\subsection{Definition of intersection form}
Let $R := \Z[\a]= \Z[a_e: e \in E]$ be the polynomial ring in variables $a_e$ indexed by $e$ and let $Q = \Frac(A) = \Q(a_e: e \in E)$ be the fraction field of rational functions.  For a subset $S \subset E$, define
$$
a_S:= \sum_{e \in S} a_e.
$$
For a flag $F_\bullet \in \Fl^k(M)$, define
$$
 \frac{1}{a_{F_\bullet}} := \prod_{i=1}^{k-1} \frac{1}{a_{F_i}} \in Q, \qquad \frac{1}{a'_{F_\bullet}} := \prod_{i=1}^{k} \frac{1}{a_{F_i}} \in Q.
 $$

\begin{definition}\label{def:dR}
The $Q$-valued \emph{deRham cohomology intersection form} on $\OS^k(M)$ is given by
$$
\dRip{x, y}:= \sum_{F_\bullet \in \Fl^k(M)} \Res_{F_\bullet}(x) \frac{1}{a_{F_\bullet}} \Res_{F_\bullet}(y).
$$
\end{definition}
We shall also use the slight modification 
$$
\dRipp{x,y} :=  \sum_{F_\bullet \in \Fl^k(M)} \Res_{F_\bullet}(x) \frac{1}{a'_{F_\bullet}} \Res_{F_\bullet}(y).
$$
It is clear from the definition that $\dRip{\cdot,\cdot}$ is a symmetric bilinear form.
We view $\dRip{\cdot,\cdot}$ both as a $Q$-valued form on $\OS^k(M)$, and as a $Q$-valued form on $\OS^k(M)_Q := \OS^k(M) \otimes_\Z Q$.

\begin{proposition}\label{prop:dRind}
Let $S, S'$ be two ordered independent sets of size $k$.  Then
$$
\dRip{e_S,e_{S'}}= \sum_{F_\bullet \in \Fl^k(M)} r(S, F_\bullet) \frac{1}{a_{F_\bullet}} r(S', F_\bullet) .
$$
\end{proposition}
\begin{proof}
Follows immediately from \cref{lem:rSF}.
\end{proof}

The formula in \cref{prop:dRind} will be improved in \cref{thm:localBF}.

\begin{example}\label{ex:boolean}
Let $M$ be the boolean matroid of rank $d$ on $E = \{e_1,\ldots,e_d\}$.  The flats of $M$ consists of all the subsets of $E$.  The complete flags of flats $F_\bullet$ are in bijection with saturated chains of subsets $F_\bullet(w) = \{ \emptyset \subset \{e_{w_1}\} \subset \{e_{w_1},e_{w_2}\} \subset \cdots \}$, or equivalently with permutations $w = w_1w_2 \cdots w_d$ of $\{1,2,\ldots, d\}$.  The only basis is $E$ and $\OS(M)$ is one-dimensional, spanned by $e_E$.  We have
$$
\dRip{e_E, e_E} = \sum_{w \in S_d} \frac{1}{a_{F_\bullet(w)}} = \sum_{w \in S_d}  \prod_{i=1}^{d-1} \frac{1}{a_{w_1} + \cdots + a_{w_d}} = \frac{a_E}{a_1 \cdots a_d}.
$$
\end{example}

\begin{proposition}\label{prop:dRdirectsum}
The bilinear form $\dRip{\cdot,\cdot}$ on $\OS^k(M)$ is compatible with the direct sum decomposition $\OS^k(M) = \bigoplus_{F \in L^k(M)} \OS_F(M)$ of \cref{prop:OSsum}.  That is, for distinct $F,F' \in L^k(M)$ and $x \in \OS_F(M)$, $x' \in \OS_{F'}(M)$, we have $\dRip{x,x'} =0$.
\end{proposition}
\begin{proof}
We may assume that $x = e_S$ and $x' = e_{S'}$ where $\overline{S} = F$ and $\overline{S'} = F'$.  Let $F_\bullet \in \Fl^k(M)$.  We have $r(S,F_\bullet) = 0$ unless $F_k = F$, and $r(S',F_\bullet) = 0$ unless $F_k = F'$.  Thus $r(S,F_\bullet) r(S', F_\bullet) = 0$ for all $F_\bullet \in \Fl^k(M)$, and hence $\dRip{x,x'}=0$.
\end{proof}

The following result states that $\dRip{\cdot,\cdot}$ is compatible with restriction to flats.
\begin{proposition}\label{prop:restrictF}
Let $F \in L(M)$.  The restriction of $\dRip{\cdot,\cdot}$ to $\OS_F(M)$ is equal to $\dRip{\cdot,\cdot}$ for $\OS(M^F)$.
\end{proposition}
\begin{proof}
The interval $[\hat 0, F]$ in $L(M)$ is isomorphic to $L(M^F)$.
\end{proof}

\subsection{Bilinear form on reduced Orlik-Solomon algebra}
\begin{proposition}\label{prop:dRpartial}
Let $x,y \in \OS_F(M)$.  Then
$$
\dRip{x,y} = \dRipp{\partial x, \partial y}.
$$
\end{proposition}
\begin{proof}
Let $x = e_S$ and $y = e_T$ for ordered independent sets $S = (s_1,\ldots,s_k),T = (t_1,\ldots,t_k)$ such that $\bar S = \bar T = F$ for some flat $F$.  Let $F_\bullet \in \Fl^k(M)$.  Since $x,y \in \OS_F(M)$, we have $\Res_{F_\bullet}(x) = \Res_{F_\bullet}(y) = 0$ unless $F_k = F$. 
We calculate
\begin{align*}
\dRipp{\partial x, \partial y} &= \sum_{F_\bullet \in \Fl^{k-1}(M)} \Res_{F_\bullet}(\partial x) \frac{1}{a'_{F_\bullet}} \Res_{F_\bullet}(\partial y) \\
&= \sum_{i,j=1}^k (-1)^{i-1}(-1)^{j-1} \sum_{F_\bullet \in   \Fl^{k-1}(M) \mid F_{k-1} = \overline{S \setminus i} = \overline{T \setminus j}} r(S \setminus s_i, F_\bullet) \frac{1}{a'_{F_\bullet}} r(T \setminus t_j, F_\bullet) \\
&= \sum_{G_\bullet \in \Fl^k(M) \mid G_k = F} r(S, G_\bullet) \frac{1}{a_{G_\bullet}} r(T, G_\bullet) = \dRip{x,y}. \qedhere
\end{align*}
\end{proof}

Recall the reduced Orlik-Solomon algebra $\rOS^\bullet(M) \subset \OS^\bullet(M)$ from \cref{sec:OS}.

\begin{corollary}\label{cor:same}
The bilinear form $\dRipp{\cdot,\cdot}$ on $\rOS(M)$ agrees with the bilinear form $\dRip{\cdot,\cdot}$ on $\OS(M)$ under the isomorphism $\partial: \OS(M) \to \rOS(M)$ of \cref{prop:OSrOS}.
\end{corollary}

\subsection{Intersection form on topes}\label{sec:pFl}
For a tope $P \in \T(\M)$ and a flag $F_\bullet \in \Fl(M)$, define 
$$
r(P, F_\bullet) := \Res_{F_\bullet}(\Omega_P).
$$

\begin{lemma}
For any $P \in \T(\M)$ and $F_\bullet \in \Fl(M)$, we have $r(P,F_\bullet) \in \{-1,0,1\}$.
\end{lemma}
\begin{proof}
By \cref{thm:EL}, the residue $\Res_{F_1} \Omega_P$ is either 0, or it equals to another canonical form $\Omega_{P/F_1}$.  The result then follows from induction on the rank $r$, with the case $r = 1$ being trivial.
\end{proof}

Recall that
$$
\pFl(M) := \{E_\bullet = (\hat 0  \subset E_1 \subset \cdots \subset E_s \subset E = \hat 1)\}
$$ denotes the set of partial flags of flats in $L(M)$.  We always assume that a partial flag starts at $\hat 0$ and ends at $\hat 1$.  We let $s = s(E_\bullet)$ denote the number of flats in $E_\bullet$ that belong to the proper part $L(M) \setminus \{\hat0,\hat1\}$.

Let $L(P)$ denote the Las Vergnas face lattice of a tope $P \in \T(\M)$ (see \cref{sec:OM}), viewed as a subposet of $L(M)$.  Note that $L(P) = L(-P)$.  A \emph{wonderful face} of $P$ is a partial flag $G_\bullet = \{\hat 0 \subset G_1 \subset G_2 \cdots \subset G_s \subset \hat 1\}$ where $G_i \in L(P)$.  We let $\pFl(P) = \Delta(L(P) - \{\hat 0,\hat1\})$ denote the set of wonderful faces of $P$, viewed as a subcomplex of $\pFl(M)$.  The closure $\bG_\bullet \subset \pFl(P)$ of a wonderful face $G_\bullet$ is the set of all partial flags $G'_\bullet$ of wonderful faces that refine $G_\bullet$.  A \emph{wonderful vertex} of $P$ is a complete flag $F_\bullet \in \Fl(P)$.  Equivalently, $F_\bullet$ is a facet of $\pFl(P)$.  In particular, a wonderful vertex $F_\bullet$ is contained in the closure of a wonderful face $G_\bullet$ if every flat in $G_\bullet$ also appears in $F_\bullet$.  We endow $\pFl(P)$ with the poset structure $G'_\bullet \leq G_\bullet$ if and only $G'_\bullet \in \bG_\bullet$.  Write $\emptyflag \in \pFl(P)$ for the trivial flag $\{\hat 0 <  \hat 1\}$.

The relation between $\pFl(P)$ and the wonderful compactification is explained in \cref{prop:wonderfulface}.

\begin{lemma}\label{lem:fliptope}
Let $P$ be a tope and $F \in L(P)$.  Then there is a unique tope $P_{\flip F}$ on the antipodal side of $F$.  More precisely, we have
$$
P_{\flip F}(e) = \begin{cases} - P(e) & \mbox{if $e \in F$,} \\
P(e) & \mbox{if $e \notin F$.}
\end{cases}
$$
\end{lemma}
\begin{proof}
Viewing $F$ as a signed covector, the tope $P_{\flip F}$ is given by the composition $F \circ (-P)$ (see \eqref{eq:compo}).  
\end{proof}

\begin{proposition}\label{prop:flipflag}
Let $G_\bullet \in \pFl(P)$.  Then there exists a tope $P_{\flip G_\bullet} \in \T$ satisfying 
\begin{equation}\label{eq:flip}
P_{\flip G_\bullet}(e) = P(e) (-1)^{\#\{1 \leq i \leq s \mid e \in G_i\}} 
\end{equation}
for all $e \in E$.  We have $(P_{\flip G_\bullet})_{\flip G_\bullet} = P$.
\end{proposition}
\begin{proof}
Apply \cref{lem:fliptope} to $P$ and the flat $G_1 \in L(P)$ to obtain $P_{\flip G_1}$.  We have $G_2 \in L(P_{\flip G_1})$ since $G_1 \subset G_2$, so we may apply \cref{lem:fliptope} again to $P_{\flip G_1}$ and the flat $G_2 \in L(P_{\flip G_1})$.  Continuing in this manner, we obtain the tope $P_{\flip G_\bullet}$.
%Begin with the tope $P$ and apply \cref{lem:fliptope} with the flat $G_1$, then the flat $G_2$, and so on until the flat $G_s$.  
\end{proof}

For $P,Q \in \T$, define 
$$
G(P,Q):= \{G_\bullet \in \pFl(P) \mid Q = P_{\flip G_\bullet}\}, \qquad \text{and} \qquad G^{\pm}(P,Q):= G(P,Q) \cup G(P,-Q) .
$$

\begin{lemma}\label{lem:closurePQ}
Suppose that $G_\bullet \in G(P,Q)$.  Then the closure $\bG_\bullet \subset \pFl(M)$ is the same regardless of whether it is taken in $\pFl(P)$ or $\pFl(Q)$.
\end{lemma}

\begin{proof}
Let $E_\bullet \in \bG_\bullet$, where the closure is taken in $\pFl(P)$.  For each $E \in E_\bullet \setminus G_\bullet$, let $X \in \L(P)$ be a signed covector lifting $E$.  Similarly to \cref{prop:flipflag}, the formula $X_{\flip G_\bullet}(e) = X(e) (-1)^{\#\{1 \leq i \leq s \mid e \in G_i\}}$ determines a signed covector $X_{\flip G_\bullet}$, and $X_{\flip G_\bullet} \in \L(Q)$.  It follows that $E \in L(Q)$, and thus $E_\bullet \in \pFl(Q)$.
\end{proof}

\begin{proposition}\label{prop:noover}
\
\begin{enumerate}
\item We have $G(P,P) = \{\emptyflag\}$ consisting only of the trivial flag, and $G(P,-P) = \emptyset$.
\item We have $G(P,Q) = G(Q,P)$ and $G(-P,-Q) = G(P,Q)$.  We have $G^{\pm}(P,Q) = G^{\pm}(Q,P)$.
\item For distinct $G_\bullet, G'_\bullet \in G^{\pm}(P,Q)$, we have $\bG_\bullet \cap \overline{G'_\bullet} = \emptyset$.
\item We have $\bigsqcup_{G_\bullet \in G^{\pm}(P,Q)} \{F_\bullet \in (\bG_\bullet \cap \Fl(M))\} = \Fl(P) \cap \Fl(Q)$.
\end{enumerate}
\end{proposition}
\begin{proof}
(1) is clear from the definitions.
For (2), the equality $G(P,Q)= G(Q,P)$ follows from the last statement of \cref{prop:flipflag} and the equality $G(-P,-Q) = G(P,Q)$ is clear from the definitions.  The last equality $G^{\pm}(P,Q) = G^{\pm}(Q,P)$ also follows.

For (3), suppose that $F_\bullet \in \bG_\bullet \cap \bG'_\bullet$ for some wonderful vertex $F_\bullet$ and $G_\bullet \neq G'_\bullet$.  Then $P_{\flip G_\bullet}$ and $P_{\flip G'_\bullet}$ are both obtained from $P$ by flipping the signs of some subset of $\{F_1,\ldots, F_{r-1}\}$.  Suppose that $P_{\flip G_\bullet} = P_{\flip G'_\bullet}$.  Then \eqref{eq:flip} shows that $\{e \in E \mid P(e) = P_{\flip G_\bullet}(e)\}$ uniquely determines $G_\bullet$ (once $F_\bullet$ has been fixed), forcing the contradiction $G_\bullet = G'_\bullet$.  However, it is not possible to have $P_{\flip G_\bullet} = Q$ and $P_{\flip G'_\bullet} = -Q$ because $P_{\flip G_\bullet}(e) = P_{\flip G'_\bullet}(e)$ for any $e \in E \setminus F_{r-1}$.  It follows that if $G_\bullet \neq G'_\bullet$ then $P_{\flip G_\bullet} \neq P_{\flip G'_\bullet}$.

For (4), the union is disjoint by (3).  The inclusion $\subseteq$ is clear from \cref{lem:closurePQ}.  To prove the inclusion $\supseteq$, we proceed by induction.  Assume that $r > 1$, and let $F_\bullet \in  \Fl(P) \cap \Fl(Q)$.  Then by induction, $F_\bullet/F_1 \in \bG'_\bullet \cap \Fl(M_{F_1})$ for some $G'_\bullet \in G^{\pm}(P_{F_1},Q_{F_1})$, where $P_{F_1} = P|_{E \setminus F_1}$ and $Q_{F_1} = Q|_{E \setminus F_1}$.  After possibly replacing $Q$ by $-Q$, we may suppose that $(P_{F_1})_{\flip G'_\bullet} = Q_{F_1}$.  If $P|_{F_1} = Q|_{F_1}$, then $P_{\flip G_\bullet} = Q$ for $G_\bullet$ the natural lift of $G'_\bullet$ (adding no additional flats so that $s(G_\bullet) = s(G'_\bullet)$).  If $P|_{F_1} = -Q|_{F_1}$, then instead we lift $G'_\bullet$ to a partial flag in $\pFl(M)$ and then add $F_1$ to it to obtain $G_\bullet$ (so that $s(G_\bullet) = s(G'_\bullet)+1$).  In both cases, we have shown that $F_\bullet \in \bG_\bullet$ for some $G_\bullet \in G^{\pm}(P,Q)$.
\end{proof}

In ``big" examples, we typically have $|G^{\pm}(P,Q)| \in \{0,1\}.$  

\begin{example}\label{ex:3pttope}
We give an example where $|G^{\pm}(P,Q)| > 1$.  Consider the two-dimensional arrangement of two lines $\ell_1,\ell_2$ in $\R^2$, and let $\ell_0$ denote the line at infinity.  
$$
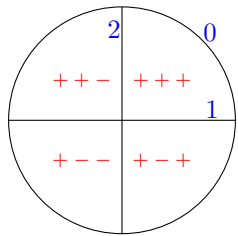
\begin{tikzpicture}
\draw (0:1.5)--(180:1.5);
\draw (90:1.5)--(270:1.5);
\draw (0,0) circle (1.5);
\node[color=blue] at (45:1.65) {$0$};
\node[color=blue] at (7:1.2) {$1$};
\node[color=blue] at (95:1.2) {$2$};

\node[color=red] at (45:0.75) {\scriptsize $+++$};
\node[color=red] at (135:0.75) {\scriptsize$++-$};
\node[color=red] at (225:0.75) {\scriptsize$+--$};
\node[color=red] at (-45:0.75) {\scriptsize$+-+$};
\end{tikzpicture}
$$

The corresponding matroid $M$ is the boolean matroid of rank three on three elements $E = \{0,1,2\}$.  The set $\T^+$ consists of four topes: $(+,+,+),(+,-,+),(+,-,-),(+,+,-)$.  Then 
\begin{align*}
G^{\pm}((+,+,+),(+,+,+)) &= \{(\hat 0 \subset \hat 1)\}, \\
G^{\pm}((+,+,+),(+,-,+)) &= \{(\hat 0 \subset \{1\} \subset \hat 1), (\hat 0 \subset \{2\} \subset \{1,2\} \subset \hat 1), (\hat 0  \subset \{0\} \subset \{0,1\} \subset \hat 1), (\hat 0  \subset \{0,2\} \subset \hat 1)\}, \\
G^{\pm}((+,+,+),(+,-,-)) &= \{(\hat 0 \subset \{1,2\} \subset \hat 1), (\hat 0 \subset \{0\} \subset \hat 1), (\hat 0 \subset \{1\} \subset \{0,1\} \subset \hat 1), (\hat 0 \subset \{2\}\subset \{0,2\} \subset \hat 1)\}, \\
G^{\pm}((+,+,+),(+,+,-)) &= \{(\hat 0 \subset \{2\} \subset \hat 1), (\hat 0 \subset \{1\} \subset \{1,2\} \subset \hat 1), (\hat 0  \subset \{0\} \subset \{0,2\} \subset \hat 1), (\hat 0  \subset \{0,1\} \subset \hat 1)\}.
\end{align*}
\end{example}

\begin{lemma}\label{lem:FlP}
Let $F_\bullet \in \Fl(M)$ and $P \in \T$.  We have $r(P,F_\bullet)  \neq 0$ if and only if $F_\bullet \in \Fl(P)$.
\end{lemma}
\begin{proof}
By \cref{thm:EL}, we have $\Res_{F_1}(\Omega_P) \neq 0$ if and only if $F_1 \in L(P)$ is a facet of $P$.  In this case, $\Res_{F_1}(\Omega_P) = \Omega_{P/F_1}$, and $L(P/F_1)$ is isomorphic to the interval $[F_1, \hat 1] \subset L(P)$.  The result then follows by induction.
\end{proof}

\begin{lemma}\label{lem:Gsign}
Let $P, Q \in \T$ and $G_\bullet \in G(P,Q)$.  Suppose $F_\bullet \in \bG_\bullet \cap \Fl(M)$.  Then 
$$
r(P,F_\bullet) r(Q,F_\bullet) = (-1)^{\sum_{i=1}^s \rk(G_i)}.
$$
\end{lemma}
\begin{proof}
We proceed by induction on $s$.  If $s = 0$ then $G_\bullet = \emptyflag$ and $P = Q$ and the claim is clear.  
Suppose $s \geq 1$, and let $p = \rk(G_1)$.  Pick $f_1,f_2,\ldots,f_p$ so that $F_i = \sp(f_1,\ldots,f_i)$ and fix the chirotope of $\M_{G_1}$ by 
$$
\chi_{\M_{G_1}}(e_1,\ldots,e_{r-p}) := \chi_{\M}(e_1,\ldots,e_{r-p},f_p,f_{p-1},\ldots,f_1).
$$
Then by \cref{thm:EL}, we have
$$
\Res_{F_p = G_1} \circ \cdots \circ \Res_{F_1} \Omega_P = \prod_{i=1}^p P(f_i) \Omega_{P_{G_1}},
$$
where $P_{G_1} = P|_{E \setminus G_1} \in \T(\M_{G_1})$, and similarly for $Q$.  It follows from the definitions that $G_\bullet/G_1 = (\hat 0 = G_1/G_1,G_2/G_1,\ldots,) \in G(P_{G_1},Q_{G_1})$.  By the inductive hypothesis, we have 
$$
r(P_{G_1},F_\bullet/G_1) r(Q_{G_1},F_\bullet/G_1) = (-1)^{\sum_{i=2}^s \rk(G_i) - \rk(G_1)}.
$$
By \cref{prop:flipflag}, we have $\prod_{i=1}^p P(f_i) Q(f_i) = (-1)^{sp}$.  Thus 
\begin{align*}
r(P,F_\bullet) r(Q,F_\bullet) &= (-1)^{\sum_{i=2}^s \rk(G_i) - p} \prod_{i=1}^p P(f_i) Q(f_i) = (-1)^{\sum_{i=2}^s (\rk(G_i) - p) + sp } =  (-1)^{\sum_{i=1}^s \rk(G_i)}. \qedhere
\end{align*}
\end{proof}

\begin{theorem}\label{thm:dRtope}
Let $P,Q \in \T$.  Then 
$$
\dRip{\Omega_P,\Omega_Q} = \sum_{G_\bullet \in G^{\pm}(P,Q)} (\pm)^r (-1)^{\sum_{i=1}^s \rk(G_i)} \sum_{F_\bullet \in \bG_\bullet \cap \Fl(M)} \frac{1}{a_{F_\bullet}},
$$
where the sign $(\pm)^r$ is equal to $1$ or $(-1)^r$ depending on whether $G_\bullet$ belongs to $G(P,Q)$ or $G(P,-Q)$.  In particular, 
$$
\dRip{\Omega_P,\Omega_P} = \sum_{F_\bullet \in \Fl(P)} \frac{1}{a_{F_\bullet}}.
$$
\end{theorem}
\begin{proof}
By \cref{lem:FlP} and \cref{prop:noover}(3),
$$
\dRip{\Omega_P,\Omega_Q} = \sum_{F_\bullet \in \Fl(P) \cap \Fl(Q)} \pm \frac{1}{a_{F_\bullet}} = \sum_{G_\bullet \in G^{\pm}(P,Q)} \sum_{F_\bullet \in \bG_\bullet \cap \Fl(M)} \pm \frac{1}{a_{F_\bullet}}.
$$
Since $\Omega_{-Q} = (-1)^r \Omega_Q$, by \cref{lem:Gsign}, the sign $\pm$ is equal to $(\pm)^r (-1)^{\sum_{i=1}^s \rk(G_i)}$.  
The last statement follows from \cref{prop:noover}(1).
\end{proof}

\begin{example}
Continue \cref{ex:3pttope}.  We have
\begin{align*}
\dRip{\Omega_{(+,+,+)},\Omega_{(+,+,+)}} &= \frac{1}{a_1 a_2} + \frac{1}{a_0 a_1} + \frac{1}{a_0 a_2}, \\
\dRip{\Omega_{(+,+,+)},\Omega_{(+,-,-)}} &= \frac{1}{a_1 a_2} + \frac{1}{a_0 a_1} + \frac{1}{a_0 a_2}, \\
\dRip{\Omega_{(+,+,+)},\Omega_{(+,+,-)}} &= -\frac{1}{a_1 a_2} - \frac{1}{a_0 a_1} - \frac{1}{a_0 a_2}. 
\end{align*}
\end{example}

\section{DeRham homology intersection form}
By \cref{prop:dRdirectsum} and \cref{prop:restrictF}, to understand the bilinear form $\dRip{\cdot,\cdot}$ it suffices to consider the form on the top homogeneous component $\OS(M) = \OS^r(M)$ of the Orlik-Solomon algebra.  We henceforth focus on this case.  In this section, we investigate the dual $\DdRip{\cdot,\cdot}$ of the symmetric bilinear form $\dRip{\cdot,\cdot}$.  We discover remarkable combinatorics when we compute $\DdRip{\cdot,\cdot}$ on the basis dual to the canonical forms in \cref{thm:EL}.

\subsection{Non-degeneracy}
\begin{proposition}\label{prop:nondeg}
The symmetric bilinear form $\dRip{\cdot,\cdot}$ is non-degenerate on $\OS(M)_Q := \OS(M) \otimes_{\Z} Q$.
\end{proposition}

\cref{prop:nondeg} can also be deduced from the results of \cite{SV}.    In \cref{thm:dRmain}, we will sharpen \cref{prop:nondeg} by explicitly inverting the bilinear form matrix. 
We prove \cref{prop:nondeg} using residue maps, which will be useful in the sequel.

Assume that $M$ is a simple matroid.  Let $\atom \in \At(M)$ be an atom, which we view as both an element of $L(M)$ and as an element of $\OS^1(M)$.  Let $R_{\atom} := R/(a_\atom)$.  
Let $M' = M \backslash \atom = M^{E \backslash \atom}$ and $M'' = M/\atom = M_\atom$.
Let $\theta_{\atom}: R \to R_{\atom}$ be the quotient map that sends $a_{\atom}$ to $0$. 
 
 \begin{lemma}\label{lem:deleteform}
 For $x,y \in A(M')$, we have
$$
\dRip{x,y}_{M'} = \theta_{\atom} \dRip{\iota_{\atom} x, \iota_{\atom} y}_{M}.
$$

\end{lemma}
\begin{proof}
It suffices to show that for two bases $B,B' \in \B(M')$, we have
$$
\dRip{e_B, e_{B'}}_{M'} = \theta_{\atom} \dRip{\iota_{\atom} e_B, \iota_{\atom} e_{B'}}_{M}.
$$
For any flag $F_\bullet \in \Fl(M')$, by \cref{def:rSF} we have that the residues $r(B,F_\bullet)$ and $r(B',F_\bullet)$ are the same regardless of whether they are calculated inside $M$ or $M'$.  Let $F'_{\bullet}$ be a flag in $L(M')$ generated by $B$.  Since we have an injection $\iota: L(M') \hookrightarrow L(M)$, the flag $F'_\bullet$ can also be identified with a flag $F_\bullet = \iota_{\atom}(F'_\bullet)$ in $L(M)$ generated by $B$.  We have
$$
\frac{1}{a_{F'_\bullet}} =\theta_{\atom} \frac{1}{a_{F_\bullet}}
$$
and the result follows from \cref{prop:dRind}.
\end{proof}

\begin{lemma}\label{lem:contractform}
For $x,y \in \OS(M'')$, we have
$$
\res_{\atom=0} \dRip{x \wedge \atom, y \wedge \atom}_{M} =  \dRip{x,y}_{M''},
$$
where $\res_{\atom=0}: Q \to Q_\atom = \Frac(R_\atom)$ is the map that sends $f(x)$ to $\theta_{\atom}(x_{\atom} f(x))$, if this is well-defined.
\end{lemma}
\begin{proof}
It suffices to show that for two bases $B,B' \in \B(M)$, we have
$$
\res_{\atom=0} \dRip{e_B\wedge \atom, e_{B'}\wedge \atom}_{M} =  \dRip{e_B,e_{B'}}_{M''}.
$$
For a flag $F_\bullet \in \Fl(M)$ with $F_1 = \atom$, we let $(F/\atom)_\bullet \in \Fl(M'')$ be the flag defined by $(F/\atom)_i = F_{i+1} \backslash \atom$. 
The pairing $ \dRip{e_B\wedge \atom,  e_{B'}\wedge \atom}_{M}$ is a sum of terms $\pm \frac{1}{a_{F_\bullet}}$ for various flags $F_\bullet$.  We have 
$$
\res_{\atom = 0}  \frac{1}{a_{F_\bullet}} = \begin{cases}  \frac{1}{a_{(F/\atom)_\bullet}} & \mbox{if $F_1 = \atom$,} \\
 0 & \mbox{otherwise.} 
 \end{cases}
 $$
 Thus $\Res_{\atom = 0}  \dRip{e_B \wedge \atom, e_{B'}\wedge \atom}_{M}$ can be expressed as a sum over flags in $\Fl(M'')$, and comparing with \cref{prop:dRind} we see that it equals to $\dRip{e_B,e_{B'}}_{M''}$.
 \end{proof}
 
 \begin{lemma}\label{lem:Resa0}
 For any $x \in \iota_{\atom}(\OS(M'))$ and $y \in \OS^{r-1}(M)$, we have $\res_{\atom=0} \dRip{x,y \wedge \atom}= 0$.
 \end{lemma}
 \begin{proof}
 The operation $\res_{\atom=0}$ will annihilate $\dRip{x, y \wedge \atom}$ unless there are terms that involve $1/a_\atom$.  These terms appear in the summands of \cref{prop:dRind} for flags $F_\bullet$ with $F_1 = \atom$.  But if $F_\bullet$ is a flag with $F_1 = \atom$, then $\Res_{F_1}(\iota_{\atom}(\OS(M'))) = 0$, so $\Res_{F_\bullet}(\iota_{\atom}(\OS(M'))) = 0$.  It follows that $\res_{\atom=0} \dRip{x, y \wedge \atom}= 0$.
 \end{proof}

\begin{proof}[{Proof of \cref{prop:nondeg}}]

The statement reduces to the case that $M$ is simple which we assume.  Suppose that $0 \neq \eta \in \OS(M)_Q$ belongs to the kernel of $\dRip{\cdot,\cdot}$.  By clearing denominators, we may assume that $\eta \in \OS(M)_R:=\OS(M) \otimes_\Z R$.  Since the pairing $\dRip{\cdot,\cdot}$ is homogeneous of degree $-d$, we may assume that $\eta$ is a homogeneous element, that is $\eta =\sum_{B\in M} p_B(\a) e_B$ where $p_B(\a) \in R$ all have the same degree.  We assume that $\eta \neq 0$ has been chosen to have minimal degree.

Pick an atom $\atom$.  Write
$$
\eta = \eta' + \eta'' \wedge \atom
$$
for $\eta'$ and $\eta''$ not depending on $\atom$.  Note that $\eta'$ and $\eta''$ are not uniquely determined by $\eta$.  For example, if $e_1,e_2,e_3$ are dependent, then $e_2 e_1 - e_3 e_1 + e_3 e_2 = 0$, so $(e_2-e_3)e_1= - e_3 e_2$, where both $e_2-e_3$ and $-e_3e_2$ do not involve $e_1$.  The map $\Res_{\atom}: \OS(M) \to \OS(M'')$ can be extended to a map $\Res_{\atom}:\OS(M)_R \to \OS(M'')_{R}$.  The map $\theta_{\atom}: R \to R_{\atom}$ can be applied to coefficients to give a map $\theta_{\atom}:\OS(M)_R \to \OS(M)_{R_{\atom}}$.  By composition we obtain a map $\theta_{\atom} \Res_{\atom}: \OS(M)_R \to \OS(M'')_{R_{\atom}}$. 

Consider $\theta_{\atom} \Res_{\atom} \eta = \theta_{\atom} \eta''$.  By \cref{lem:contractform}, we deduce that
$$
\dRip{\theta_{\atom} \eta'', \tau''}_{M''} = \res_{\atom=0} \dRip{\theta_{\atom} \eta'' \wedge \atom,  \tau''  \wedge \atom}_M= \res_{\atom=0} \dRip{\eta, \tau'' \wedge \atom}_M =0
$$
for any $\tau'' \in \OS(M'')$.  In the second equality, we used $\res_{\atom=0} \dRip{\eta', \tau'' \wedge \atom}_M = 0$ which holds by \cref{lem:Resa0}, and $\res_{\atom=0} \dRip{\eta''' \wedge \atom, \tau'' \wedge \atom} = 0$ if $\eta''' \in \Ker(\theta_{\atom})$ allowing us to replace $\theta_{\atom} \eta''$ by $\eta''$.

By induction we may assume that $\dRip{\cdot,\cdot}_{M''}$ is non-degenerate, and so we have $\theta_{\atom} \eta'' = 0$ inside $\OS(M'')_{R_{\atom}}$, or equivalently, $\Res_{\atom}(\eta) = \eta'' \in \Ker(\theta_{\atom})$ as an element of $\OS(M'')_R$.  Thus, 
$$
\eta \in \iota_{\atom}(\OS(M')_R) + \Ker(\theta_{\atom}).
$$
Let $\eta = \iota_{\atom}(\nu) \mod \Ker(\theta_{\atom})$.  Then by \cref{lem:deleteform}, we deduce that $ \dRip{\theta_{\atom} \nu, \OS(M')}_{M'} =  \theta_{\atom}\dRip{ \nu, \OS(M')}_{M'}  =0$ and by induction, we must have $\theta_{\atom}\nu = 0$.  Thus $\eta \in \Ker(\theta_{\atom})$, or equivalently, $\eta = a_\atom \mu$ for some homogeneous element $\mu \in \OS(M)_R$.  This contradicts our assumption that $\eta$ was chosen to have minimal degree.  
\end{proof}

\subsection{deRham homology pairing}
In this section we work with a general extension $(\tM, \star)$ of $\M$ by $\star$, viewed as an affine oriented matroid.  Let $\T^\star$ denote the corresponding set of bounded topes \eqref{eq:Tstar}.  We have $\tE = E \cup \star$. 

We define a symmetric bilinear form $\DdRip{\cdot,\cdot}$ on $\Z^{\T^\star}$ with values in $R = \Z[\a] =  \Z[a_e \mid e \in E]$.  For $P \in \T^\star$, we write $P$ to also denote the corresponding basis element of $\Z^{\T^\star}$.  Denote
$$a^B:= \prod_{b \in B} a_b, \qquad \mbox{for $B \subseteq E$.}$$ 

\begin{definition}\label{def:DdR}
For two bounded topes $P,Q \in \T^\star$, define 
$$
\B(P,Q) = \{B \in \B(M) \mid P, Q \in \T^{B}\}.
$$
Define the $R$-valued \emph{deRham homology intersection form} on $\Z^{\T^\star}$ by
$$
\DdRip{P,Q} := \sum_{B \in \B(P,Q)} a^B.
$$ 
\end{definition}
By definition $\DdRip{\cdot,\cdot}$ is a symmetric bilinear form, homogeneous of degree $r$.

\begin{theorem}\label{thm:dRmain}
The bilinear form $\frac{1}{a_E}\DdRip{\cdot,\cdot}$ (resp. $\DdRip{\cdot,\cdot}$)  is the inverse of the bilinear form $\dRip{\cdot,\cdot}$ (resp. $\dRipp{\cdot,\cdot}$) with respect to the basis $\{\Omega_P \mid P \in \T^\star\}$ of $\OS(M)$.
\end{theorem}

\begin{corollary}
Viewing the $a_e$ as complex parameters, the bilinear form $\dRip{\cdot,\cdot}$ on $\OS(M)$ is non-degenerate when $a_E \neq 0$ and \eqref{eq:Mon} is satisfied.
\end{corollary}
\begin{proof}
In \cref{cor:denom}, we will show that the matrix entries of $\dRip{\cdot,\cdot}$ only have the linear forms $a_F$ in the denominator, where $F$ varies over connected flats.  Since $\DdRip{\cdot,\cdot}$ has polynomial entries, we obtain the stated result from \cref{thm:dRmain}.
\end{proof}

\begin{example}
Consider the line arrangement with five lines labeled $E = \{a,b,c,d,e\}$ and five regions labeled $1,2,3,4,5$ as in \cref{fig:5line}.  
We use the five parameters $a,b,c,d,e$ in place of $a_e, e \in E$.

\begin{figure}
\begin{center}
$$
\begin{tikzpicture}[extended line/.style={shorten >=-#1,shorten <=-#1},
 extended line/.default=1cm]
 \useasboundingbox (0,-0.3) rectangle (12,2);
\draw (0,0) -- (5,0);
\draw[extended line] (1,0) --(3,1);
\draw[extended line=0.4cm] (1.5,-0.3) --(3,2);
\draw[extended line=0.7cm] (3,0) --(3,1.6);
\draw[extended line] (4,0) --(3,1);
\node[color=blue] at (-0.2,0) {$a$};
\node[color=blue] at (0,-0.5) {$b$};
\node[color=blue] at (1.5,-0.5) {$c$};
\node[color=blue] at (3.2,-0.5) {$d$};
\node[color=blue] at (4.7,-0.5) {$e$};

\node[color=red] at (2.85,1.45) {$1$};
\node[color=red] at (2.55,1.02) {$2$};
\node[color=red] at (1.6,0.15) {$5$};
\node[color=red] at (2.5,0.35) {$3$};
\node[color=red] at (3.3,0.3) {$4$};
\begin{scope}[shift={(7,0.5)}]
\node (h0) at (2,-0.8) {$\hat 0$};
\node (a) at (0,0) {$a$};
\node (b) at (1,0) {$b$};
\node (c) at (2,0) {$c$};
\node (d) at (3,0) {$d$};
\node (e) at (4,0) {$e$};
\node (ab) at (-1.5,1) {$ab$};
\node (ac) at (-0.5,1) {$ac$};
\node (ad) at (0.5,1) {$ad$};
\node (ae) at (1.5,1) {$ae$};
\node (bc) at (2.5,1) {$bc$};
\node (bde) at (3.5,1) {$bde$};
\node (cd) at (4.5,1) {$cd$};
\node (ce) at (5.5,1) {$ce$};
\node (abcde) at (2,1.8) {$abcde$};
\draw (a)--(ab)--(b);
\draw (a)--(ac)--(c);
\draw (a)--(ad)--(d);
\draw (a)--(ae)--(e);
\draw (b)--(bc)--(c);
\draw (b)--(bde)--(d);
\draw (e)--(bde);
\draw (c)--(cd)--(d);
\draw (c)--(ce)--(e);
\draw (h0)--(a);
\draw (h0)--(b);
\draw (h0)--(c);
\draw (h0)--(d);
\draw (h0)--(e);
\draw (abcde)--(ab);
\draw (abcde)--(ac);
\draw (abcde)--(ad);
\draw (abcde)--(ae);
\draw (abcde)--(bc);
\draw (abcde)--(bde);
\draw (abcde)--(cd);
\draw (abcde)--(ce);
\end{scope}
\end{tikzpicture}
$$
\end{center}
\caption{Left: a line arrangement in $\P^2$ consisting of 5 lines.  The line at infinity is the general extension $\star$ and not one of the hyperplanes of the arrangement. Right: the lattice of flats $L(M)$.}
\label{fig:5line}
\end{figure}
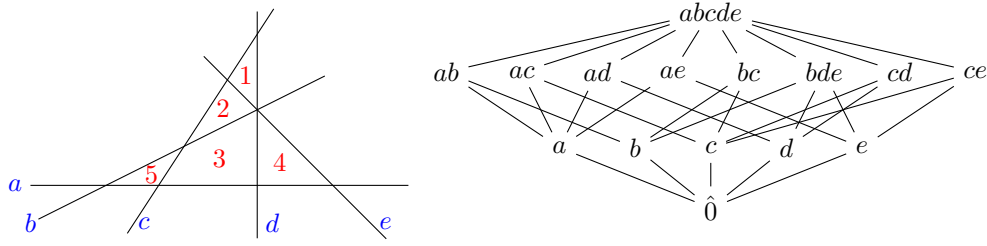

\noindent
The deRham cohomology intersection form $\dRip{\cdot,\cdot}$ is given by

\scalebox{0.75}{\hspace*{-0.8cm}
$
\begin{bmatrix}
 \frac{1}{d (b+d+e)}+\frac{1}{e (b+d+e)}+\frac{1}{c d}+\frac{1}{c e} & -\frac{1}{e (b+d+e)}-\frac{1}{c e} & -\frac{1}{d (b+d+e)} & \frac{1}{e (b+d+e)}+\frac{1}{d (b+d+e)} & 0 \\
 -\frac{1}{e (b+d+e)}-\frac{1}{c e} & \frac{1}{b c}+\frac{1}{b (b+d+e)}+\frac{1}{e (b+d+e)}+\frac{1}{c e} & -\frac{1}{b c}-\frac{1}{b (b+d+e)} & -\frac{1}{e (b+d+e)} & \frac{1}{b c} \\
 -\frac{1}{d (b+d+e)} & -\frac{1}{b c}-\frac{1}{b (b+d+e)} & \frac{1}{a c}+\frac{1}{a d}+\frac{1}{b c}+\frac{1}{b (b+d+e)}+\frac{1}{d (b+d+e)} & -\frac{1}{a d}-\frac{1}{d (b+d+e)} & -\frac{1}{a c}-\frac{1}{b c} \\
 \frac{1}{e (b+d+e)}+\frac{1}{d (b+d+e)} & -\frac{1}{e (b+d+e)} & -\frac{1}{a d}-\frac{1}{d (b+d+e)} & \frac{1}{a d}+\frac{1}{a e}+\frac{1}{d (b+d+e)}+\frac{1}{e (b+d+e)} & 0 \\
 0 & \frac{1}{b c} & -\frac{1}{a c}-\frac{1}{b c} & 0 & \frac{1}{a b}+\frac{1}{a c}+\frac{1}{b c}
\end{bmatrix}.
$}
For example, the $(1,3)$-entry is equal to $-1/(d (b+d+e))$ because there is a single flag $F_\bullet = (\hat 0 \subset \{d\} \subset \{b,d,e\} \subset \hat 1)$ for which both residues $\Res_{F_\bullet} \bOmega_{P_1}$ and $\Res_{F_\bullet}  \bOmega_{P_3}$ are non-zero.  This can be deduced from \cref{thm:EL}.

\noindent
The deRham homology intersection form $\DdRip{\cdot,\cdot}$ is given by 
$$
\begin{bmatrix}
 a c d+b c d+c d e & a c d+b c d & a c d & 0 & 0 \\
 a c d+b c d & a c d+a c e+b c d+b c e & a c d+a c e & a c e & 0 \\
 a c d & a c d+a c e & a b d+a b e+a c d+a c e & a b e+a c e & a b d+a b e \\
 0 & a c e & a b e+a c e & a b e+a c e+a d e & a b e \\
 0 & 0 & a b d+a b e & a b e & a b c+a b d+a b e 
\end{bmatrix}.
$$
For example, the $(1,2)$-entry is equal to $acd+bcd$ because the two simplices bounded by $a,c,d$ and $b,c,d$ contain both of the chambers $1$ and $2$.
\end{example}

\cref{thm:dRmain} can be proven by induction in a direct combinatorial manner.  We instead proceed indirectly, using the flag space of \cite{SV}.  This has the advantage of directly connecting our constructions to \cite{SV}.

\subsection{Flag space}\label{sec:flagspace}
Let $\tF^k$ denote the free abelian group on elements $[F_\bullet]$ for $F_\bullet \in \Fl^k$.  Let $G_\bullet = (G_0 \subset G_1 \subset \cdots \subset G_{j-1} \subset G_{j+1} \subset \cdots \subset G_k)$ be a partial flag with a single jump, where $\rk(G_i) = i$.  For $L \in L(M)$ satisfying  $G_{j-1} < L <G_{j+1}$, let $G^L_\bullet := (G_0 \subset \cdots \subset G_{j-1} \subset L \subset G_{j+1} \subset \cdots \subset G_k) \in \Fl^k$.  
\begin{definition}
The \emph{flag space} $\F^k$ is the quotient of $\tF^k$ by the submodule generated by the elements
$$
\sum_{L \in (G_{j-1},G_{j+1})} [G^L_\bullet]
$$
for all $0 < j < k$ and all partial flags $G_\bullet$ with a single jump.  
\end{definition}

Define a map $\eta: \tF^k \to \OS^k(M)^* = \Hom(\OS^k(M),\Z)$ by the formula
\begin{equation}\label{eq:etadef}
(\eta([F_\bullet]), x) =  \Res_{F_\bullet} x
\end{equation}
for $x \in \OS^k(M)$.  Abusing notation, we may also write $\Res_y: \OS^k(M) \to \OS^k(M)$ for an arbitrary $y \in \tF^k$. 

\begin{lemma}\label{lem:Resdes}
The action of $\tF^k$ descends to $\F^k$.
\end{lemma}
\begin{proof}
We need to show that for any partial flag $G_\bullet$ with a single jump, we have that $\sum_L \Res_{G^L_\bullet}$ acts by zero on $\OS^k(M)$.  Since $\Res_{G^L_\bullet}$ is a composition of residue maps, we reduce immediately to the case $j = 1$.  We may further assume that $M$ is simple.  Let $e_S \in \OS^k(M)$ for $S \subset E$.  If $|S \cap G_2| < 2$ then $ \Res_{G^L_\bullet} e_S = 0$ for any $L$.  If $|S \cap G_2| > 2$ then $S$ is not independent and $e_S = 0$.  If $S \cap G_2 = \{e,e'\}$, then 
\begin{equation*}
\sum_L \Res_{G^L_\bullet} e_S = \Res_{G_p} \cdots \Res_{G_3} (\Res_{e'} \Res_{e} e_S + \Res_{e} \Res_{e'} e_S) = 0.  \qedhere
\end{equation*}
\end{proof}

Suppose that $k = r$.  By \cref{thm:EL}, $\OS(M)$ has basis $\{\Omega_P \mid P \in \T^\star\}$.  Let $\{\delta_P \mid P \in \T^\star\}$ denote the dual basis of $\OS(M)^*$.  In this basis, the homomorphism $\eta: \F^r \to \OS(M)^*$ is given by
\begin{equation}\label{eq:etadeltaP}
\eta([F_\bullet]) = \sum_{P \in \T^\star} r(P,F_\bullet) \delta_P.
\end{equation}
Let $\delta_{F_\bullet} \in (\tF^k)^*$ be the linear functional taking the value $1$ on $[F_\bullet]$ and $0$ on all other flags.
Define a map $\nu: \OS^k(M) \to (\tF^k)^*$ by 
$$
\nu(x) := \sum_{F_\bullet \in \Fl^k} \Res_{F_\bullet}(x) \delta_{F_\bullet}.
$$
By the proof of \cref{lem:Resdes}, $\nu$ has image in the subspace $(\F^k)^* \subset (\tF^k)^*$.

\begin{proposition}
The two maps $\eta: \F^k \to \OS^k(M)^*$ and $\nu:\OS^k(M) \to (\F^k)^*$ are transpose to each other.
\end{proposition}
\begin{proof}
Let $S \in \I_k(M)$ and $F_\bullet \in \Fl^k$.  We have
\begin{align*}
([F_\bullet], \nu(e_S)) &= ([F_\bullet],  \sum_{F'_\bullet \in \Fl^k} \Res_{F'_\bullet}(e_S) \delta_{F'_\bullet}) = \Res_{F_\bullet}(e_S) \stackrel{\eqref{eq:etadef}}{=} (\eta([F_\bullet]), e_S). \qedhere
\end{align*}
\end{proof}

A fundamental property of the flag space $\F^k$ is the duality with $\OS^k(M)$.
\begin{proposition}[{\cite[Theorem 2.4]{SV}}] \label{prop:Fk}
The maps $\eta: \F^k \to \OS^k(M)^*$ and $\nu: \OS^k(M) \to (\F^k)^*$ are isomorphisms.  
\end{proposition}

\begin{remark}
The flag spaces $\F^k$ form a complex $(\F^\bullet, d)$ where the differential $d$ is defined in \cite[(2.2.1)]{SV}.  The cohomology of this complex is naturally isomorphic to the reduced cohomology of the order complex of $L(M) \setminus \{\hat 0, \hat 1\}$.  See \cite[Remark 3.8]{FT}.
\end{remark}

\subsection{Proof of \cref{thm:dRmain}}
In this section, we extend coefficients of $\OS^\bullet(M)$ and $\F^\bullet$ from $\Z$ to $Q$.  
Following \cite{SV}, define linear maps $R^k: \OS^k(M)_Q \to \F^k_Q$ and $S^k: \F^k_Q \to \OS^k(M)_Q$ by
\begin{align}\label{eq:RS}
\begin{split}
R^k(x) &:=\sum_{F_\bullet \in \Fl^k} \Res_{F_\bullet}(x) \frac{1}{a'_{F_\bullet}} [F_\bullet], \\
S^k([F_\bullet]) &:= \sum_{S \in \I_k(M)} r(S,F_\bullet) a^S e_S.
\end{split}
\end{align}

\begin{proposition}[{\cite[Lemma 3.4.4]{SV}}]\label{prop:SVinverse}
For any $k$, we have $S^k \circ R^k = {\rm id}$.
\end{proposition} 
\begin{proof}
Proceed by induction on $k$.  The case $k=1$ is straightforward.
Let $S = \{s_1,\ldots,s_k\}$ be an ordered independent set with closure $F:=\bar S $.   Then for $e \in F \setminus S$, the set $S \cup e$ is dependent, giving
$$
e \wedge (\sum^k_{i=1} (-1)^{k-i} e_{s_k} \wedge \cdots \widehat{e_{s_i}} \cdots \wedge e_{s_1}) = e_S.
$$
Thus we have
\begin{equation}\label{eq:SV}
 (\sum_{e \in F} a_e e) \wedge(\sum^k_{i=1} (-1)^{k-i} e_{s_k} \wedge \cdots \widehat{e_{s_i}} \cdots \wedge e_{s_1})= a_F e_S.
\end{equation}
Fix an independent set $T \in \I_k$ and let $F=\bar T \in L^k(M)$.  We have
$$
S^k \circ R^k(e_T) = \sum_{F_\bullet} \frac{1}{a'_{F_\bullet}} r(T,F_\bullet) \sum_{Z} r(Z,F_\bullet) a^Z e_{Z} = \sum_{Z,F_\bullet} \frac{a^Z}{a'_{F_\bullet}}  r(T,F_\bullet) r(Z,F_\bullet) e_Z
$$
where the summation can be restricted to pairs $(Z, F_\bullet) \in \I_k \times \Fl^k$ such that both $Z$ and $T$ generate $F_\bullet$, and in particular $F_k = F$.  For each such pair $(Z,F_\bullet)$, there exists a unique $ b\in Z$ such that $b \notin F_{k-1}$ and a unique $t_i \in T = \{t_1,t_2,\ldots,t_k\}$ such that $t_i \notin F_{k-1}$.  We may rewrite the sum as
$$
S^k\circ R^k(e_T) = \frac{1}{a_F}\sum_{b \in F}  \sum_{i =1}^k (-1)^{k-i} \frac{a_b}{a_G} e_b \left(\sum_{Z^-, F^-_\bullet} \frac{1}{a_{F^-_\bullet}} r(Z^-, F^-_\bullet) r(T^-, F^-_\bullet) a^{Z^-} e_{Z^-} \right)
$$
where $G = \overline{T \setminus t_i}$, and $Z^- = Z \setminus b$, and $T^- = T \setminus t_i$, and $F^-_\bullet \in \Fl^{k-1}$ is obtained by dropping $F_k$ from $F_\bullet$.  We compute, using the inductive hypothesis,
\begin{align*}
S^k \circ R^k (e_T)  &= \frac{1}{a_F}  \left(  \sum_{b\in F}a_b e_b  \right)\sum_{i =1}^k (-1)^{k-i}  \left((S^{k-1} \circ R^{k-1})(e_{T \setminus t_i})\right)\\
&= \frac{1}{a_F}  \left(  \sum_{b\in F}a_b e_b  \right) \wedge \left(\sum_{i =1}^k (-1)^{k-i} e_{T \setminus t_i} \right) & \mbox{by inductive hypothesis}\\
& = e_T &\mbox{by \eqref{eq:SV}.} & \qedhere
\end{align*}
\end{proof}

Define two $\T^\star \times \T^\star$ matrices
$$
V(P,Q):=\frac{1}{a_E} \dRip{\Omega_P,\Omega_Q} = \dRipp{\Omega_P,\Omega_Q}, \qquad W(P,Q):= \DdRip{P,Q}.
$$

\begin{lemma}\label{lem:V}
The matrix $V$ is the matrix of $R^r: \OS(M)_Q \to \F^r_Q$ with respect to the basis $\{\Omega_P \mid P \in \T^\star\}$ of $\OS(M)$ and $\{\delta_P \mid P \in \T^\star\}$ of $\eta:\F^r \cong \OS(M)^*$.
\end{lemma}
\begin{proof}
Follows from the definitions.
\end{proof}

\begin{lemma}\label{lem:W}
The matrix $W$ is the matrix of the linear map $S^r: \F^r \to \OS(M)$
with respect to the basis $\{\delta_P \mid P \in \T^\star\}$ of $\F^r \cong \OS(M)^*$ and $\{\Omega_P \mid P \in \T^\star\}$ of $\OS(M)$.
\end{lemma}
\begin{proof}
Define $S'(\delta_P) = \sum_{Q \in \T^\star} W(P,Q) \Omega_Q$.  Then 
\begin{align*}
S'([F_\bullet]) &= S'(\sum_{P \in \T^\star} r(P,F_\bullet) \delta_P) \\
&= \sum_{P \in \T^\star} r(P,F_\bullet) \sum_{Q \in \T^\star} \Omega_Q  \sum_{B \in \B(P,Q)} a^B & \mbox{by \cref{def:DdR}} \\
&= \sum_B a^B \left(\sum_{P \in \T^B} \Res_{F_\bullet}(\Omega_P) \right) \left(\sum_{Q \in \T^B} \Omega_Q\right) \\
&= \sum_B a^B  r(B,F_\bullet) e_B & \mbox{by \eqref{eq:cone}}.
\end{align*}
Comparing with the definition of $S^r$, we find that $S' = S^r$.
\end{proof}

\cref{thm:dRmain} is equivalent to the matrix identity $VW = {\rm Id}$, which follows from \cref{prop:SVinverse}, \cref{lem:V} and \cref{lem:W}.

\subsection{Comparison to Schechtman--Varchenko contravariant form}

The following result compares our definition with the ``contravariant form'' of Schechtman and Varchenko \cite{SV} defined in the setting of affine hyperplane arrangements.  This form is extended to the setting of matroids by Brylawski and Varchenko \cite{BV}.

Let $\ip{\cdot,\cdot}_{SV}$ be the form on $\OS^k(M)$ induced by the map $R^k: \OS^k(M) \to (\F^k)^*$.  More precisely,
$$
\ip{x,y}_{SV} := (\eta(R^k(x)), y),
$$
where $(\cdot,\cdot)$ is the natural evaluation map on $\OS^k(M)^* \otimes \OS^k(M)$.

\begin{corollary}\label{cor:SVform}
Suppose that $x, y \in \OS^k(M)$.  Then 
$$
\dRipp{x,y} = \ip{x,y}_{SV} =\ip{y,x}_{SV}.
$$
\end{corollary}
\begin{proof}
For two independent sets $S,S' \in \I_k(M)$, we compute:
\begin{align*}
\ip{e_S,e_{S'}}_{SV} &= (\eta(R^k(e_S)), e_{S'}) = \sum_{F_\bullet} r(S,F_\bullet) \frac{1}{a'_{F_\bullet}} (\eta([F_\bullet]), e_{S'}) =  \sum_{F_\bullet} r(S,F_\bullet) \frac{1}{a'_{F_\bullet}} r(S',F_\bullet)  = \dRipp{e_S, e_{S'}}. \qedhere
\end{align*}
\end{proof}

\begin{remark}\label{rem:a0infinity}
Our symmetric bilinear form $\dRip{\cdot,\cdot}$ agrees with that of \cite{SV} in the case of a central hyperplane arrangement, and to that of \cite{BV}.  In the case of an affine hyperplane arrangement $\A$, the symmetric bilinear form $\ip{\cdot,\cdot}_{SV,\A}$ of \cite{SV} is obtained from our $\dRip{\cdot,\cdot}$ by ``removing contributions from infinity".  More precisely, for an affine matroid $(M,0)$ associated to an affine arrangement $\A$, we have
$$
\ip{\cdot,\cdot}_{SV,\A} = \dRip{\cdot,\cdot}|_{a_0 = \infty}.
$$  
The substitution $a_0 = \infty$ sends $1/a_F$ to 0 for any flat $F \ni 0$ containing $0$.
\end{remark}

\subsection{Schechtman-Varchenko determinant}

The main result of Schechtman and Varchenko \cite{SV} (in the hyperplane arrangement case) and Brylawski and Varchenko \cite{BV} (in the general matroid case) is the following determinantal formula.

\begin{theorem}[\cite{SV,BV}]\label{thm:SVdet}
The determinant of the form $\dRipp{\cdot,\cdot}$ on the free $\Z$-module $\OS(M)$ is equal to 
$$
\Delta' = \frac{1}{\prod_{F \in L(M)\setminus \hat 0} a_F^{\beta(M^F) \mu^+(M_F)}}.
$$
The determinant of the form $\dRip{\cdot,\cdot}$ on $\OS(M)$ is equal to 
$$
\Delta = \frac{a_E^{\mu^+(M)-\beta(M)}}{\prod_{F \in L(M)\setminus \{\hat 0,\hat 1\}} a_F^{\beta(M^F) \mu^+(M_F)}}.
$$
\end{theorem}

For $F$ an atom, we have $\beta(M^F) = 1$, so the exponent $\beta(M^F) \mu^+(M_F)$ is equal to $\mu^+(M_F)$.  For $F = E$, we have $\mu^+(M_F) = 1$, so the exponent $\beta(M^F) \mu^+(M_F)$ is equal to $\beta(M)$.

\section{Aomoto complex intersection form}
In this section, we consider an affine oriented matroid $(\M,0)$, and study the situation when the parameters $a_e \in \C$ are specialized to complex numbers satisfying 
\begin{equation}\label{eq:sumto0}
a_E = \sum_{e \in E} a_e = 0,
\end{equation}
or equivalently, $a_0 = - \sum_{e \in E \setminus 0} a_e$.  In this section, we always assume that \eqref{eq:Mon} is satisfied. 
By \cref{cor:denom}, $\dRip{\cdot,\cdot}$ is defined when \eqref{eq:Mon} is satisfied.

\begin{remark}
Falk and Varchenko \cite{FalkVar} study the Schechtman-Varchenko contravariant form on the \emph{subspace of singular vectors} within the flag space $\F^r$, which is dual to the setting of this section.
\end{remark}

\begin{remark}
Instead of taking $a_e, e \in E$ to be complex parameters, we could alternatively work in the ring $R_0 = R/(a_E)$ and its fraction field $Q_0 = \Frac(F_0)$.
\end{remark}

\subsection{Aomoto complex}\label{sec:Aomoto}
Let $a_e$, $e \in E$ be complex parameters.  Consider the element 
$$
\omega = \sum_e a_e e \in \OS^1(M) \otimes_{\Z} \C.
$$
Since $\omega \wedge \omega = 0$, multiplication by $\omega$ gives a chain complex, the \emph{Aomoto complex}:
\begin{equation}\label{eq:Aomotocomplex}
\OS^0(M) \otimes_\Z \C \stackrel{\omega}{\longrightarrow} \OS^1(M) \otimes_\Z \C \stackrel{\omega}{\longrightarrow} \cdots \stackrel{\omega}{\longrightarrow} \OS^r(M) \otimes_\Z \C,
\end{equation}
denoted $(\OS^\bullet(M), \omega)$.  When $\sum_e a_e = 0$, we have $\omega \in \rOS^1(M)$, and we obtain a subcomplex $(\rOS^\bullet(M), \omega) \subset (\OS^\bullet(M),\omega)$.  We let $\OS^\bullet(M,\omega)$ (resp. $\rOS^\bullet(M,\omega)$) denote the cohomologies of the Aomoto complex.  %Note that $\rOS^\bullet(M,\omega)$ is defined only if $\sum_e a_e = 0$.

The cohomology of the Aomoto complex was initially considered in the study of the topology of hyperplane arrangement complements; see \cref{sec:twistedco}.  Yuzvinsky \cite{Yuz} studied the cohomology from the abstract perspective of the Orlik-Solomon algebra. 

\begin{theorem}[{\cite[Proposition 2.1 and Theorem 4.1]{Yuz}}]\label{thm:Yuz}\
\begin{enumerate}
\item Suppose that $\sum_e a_e \neq 0$.  Then we have $\OS^\bullet(M,\omega) = 0$.
\item  
Suppose that \eqref{eq:sumto0} and \eqref{eq:Mon} hold.  Then we have $\rOS^k(M,\omega) = 0$ unless $k = d$, and $ \dim \rOS^{d}(M,\omega) = \beta(M)$.
\end{enumerate}
\end{theorem}
Denote $\rOS(M,\omega):=  \rOS^{d}(M,\omega)$ for the non-vanishing cohomology group of the complex $(\rOS^\bullet(M),\omega)$.  Henceforth, we always assume that $\sum_e a_e = 0$ when considering the cohomology $\rOS(M,\omega)$.  We have the following comparison (cf.  \cite[Theorem 4.1]{Yuz}).

\begin{proposition}\label{prop:OSrOStwisted}
Suppose that \eqref{eq:sumto0} and \eqref{eq:Mon} hold.
The isomorphism $\partial: \OS^r(M) \otimes \C \to  \rOS^{r-1}(M) \otimes \C$ of \cref{prop:OSrOS} descends to an isomorphism $\partial: \OS^r(M,\omega) \to  \rOS^{r-1}(M,\omega) = \rOS(M,\omega)$.
\end{proposition}
\begin{proof}
For any two elements $\alpha, \beta$ of $A^\bullet$, we have the Leibniz rule:
$$
\partial( \alpha \wedge \beta) = \pm \partial(\alpha) \wedge \beta + \alpha \wedge \partial(\beta)
$$
which holds generally for the contraction of a differential form $\alpha \wedge \beta$ against a vector field $\partial$.
Now, let $\alpha = \omega$ and $\beta \in A^\bullet(M)$.  Then $\partial(\omega) = \sum_{e \in E} a_e = 0$, so
\begin{equation}\label{eq:partialomega}
\partial( \omega \wedge \beta) =\omega \wedge \partial(\beta).
\end{equation}
It follows that $\partial$ sends the subspace $\omega \OS^{r-1}(M) \subset \OS^r(M)$ isomorphically to the subspace $\omega \rOS^{r-2}(M) \subset \rOS^{r-1}(M)$.  Thus $\partial$ descends to an isomorphism $\partial: \OS^r(M,\omega) \cong \rOS^{r-1}(M,\omega)$.
\end{proof}

\begin{lemma}\label{lem:AMgeneric}
Let $(\tilM,\star)$ be a general extension of $M$ by $\star$.
Then $\rOS(M)_\C := \rOS(M)\otimes_\Z \C \cong \rOS(\tilM, \omega)$.
\end{lemma}
\begin{proof}
Let $\tE = E \cup \star$.  There is an inclusion $\iota_0: \rOS(M)_\C \to \rOS(\tilM)_\C$, and therefore a map $\kappa: \rOS(M)_\C \to \rOS(\tilM, \omega)$.  We show that this map is surjective.  Clearly any $\partial e_B$ where $B \in \B(M)$ is in the image of $\kappa$.  Suppose that $\star \cup B' \in \B(\tilM)$.  Let us consider $\partial(\star \wedge e_{B'} )\in \rOS(\tilM, \omega)$.  By \eqref{eq:partialomega}, we have
$$
 \omega \wedge \partial \left( \frac{1}{a_\star} e_{B'} \right) = 
\partial \left(\frac{1}{a_\star}\omega \wedge e_{B'} \right)= \partial(\star \wedge e_{B'}) + \text{ terms in the image of } \kappa,
$$
so $\partial(\star \wedge e_{B'} )$ lies in the image of $\kappa$ and we conclude that the map $\kappa$ is surjective.
However, by \cref{lem:betageneric}, we have $|\mu(M)| = \beta(\tilM)$, so $\kappa$ is an isomorphism.
\end{proof}

\subsection{Canonical forms for Aomoto cohomology}
For $P \in \T$, the \emph{reduced canonical form} $\bOmega_P \in \rOS(M)$ is
$$
\bOmega_P:= \partial \Omega_P,
$$
where $\Omega_P$ is the canonical form of \cref{thm:EL}.
Recall that $\T^0 \subset \T(\M)$ denotes the set of topes bounded with respect to $0 \in E$.  
\begin{theorem}[\cite{EL}]\label{thm:ELtwisted}
Assume that the $a_e \in \C$ are generic, and \eqref{eq:sumto0}.  The canonical forms
$$
\{\Omega_P \mid P \in \T^0\}, \qquad \text{and} \qquad \{\bOmega_P \mid P \in \T^0\}
$$
give bases of $\OS(M,\omega)$ and $\rOS(M,\omega)$ respectively.
\end{theorem}

In \cref{cor:Aomotobasis} below, we shall strengthen \cref{thm:ELtwisted} by weakening the genericity assumption.

\subsection{Descent of intersection form}\label{sec:descent}
According to \cref{thm:SVdet}, when $a_E = 0$, the symmetric form $\dRip{\cdot, \cdot}$ is degenerate.

\begin{theorem}\label{thm:descent}
Suppose \eqref{eq:sumto0} holds.
The symmetric pairing $\dRip{\cdot,\cdot}$ on $\OS(M)_\C$ descends to a symmetric pairing $\bdRip{\cdot, \cdot}$ on $\OS(M,\omega)$.
\end{theorem}
\begin{proof}
Let $B \in \B(M)$ be a basis, and $\tau\in \I_{r-1}(M)$ be an independent set of size $r-1$.  We shall check that
$$
\dRip{ e_\tau \wedge \omega, e_B}= 0.
$$
Let $F_\bullet = (F_0 \subset F_1 \subset \cdots \subset F_r)$ be generated by $B$.  Let $L(\tau) \subset L$ be the sublattice of the lattice of flags generated by $\tau$.  Since $F_r \notin L(\tau)$, there is a minimal $\alpha = \alpha(F_\bullet)$ such that $F_\alpha \notin L(\tau)$.  We say that $F_\bullet$ is \emph{nearly generated} by $\tau$ if $F_\bullet$ is generated by $B' = \tau \cup f$ for some $f \in E$.  Let 
$$
F(\tau,B) := \{F_\bullet \mid F_\bullet \mbox{ is generated by } B \mbox{ and nearly generated by } \tau\}.
$$
We define a simple graph $\Gamma(\tau,B)$ with vertex set $F(\tau,B)$.

For $i = 1,2,\ldots,r-1$, let $\mu_i(F_\bullet) = \mu_{i,B}(F_\bullet) =  (F_0 \subset F_1 \subset \cdots \subset F'_i \subset \cdots \subset F_r)$ be the unique flag differing from $F_\bullet$ in the $i$-th position and such that $\mu_{i,B}(F_\bullet)$ is still generated by $B$.  If $B = \{b_1,\ldots,b_r\}$ is ordered so that $F_k = b_{1} \vee \cdots \vee b_{k}$ then we have the explicit formula
$$
F'_i = b_{1} \vee \cdots \vee b_{{i-1}} \vee b_{{i+1}}.
$$
Let $F_\bullet \in F(\tau,B)$ and $\alpha = \alpha(F_\bullet)$.  Then $F_{\alpha-1} \in L(\tau)$ and $F_{\alpha} = F_{\alpha-1} \vee b \notin L(\tau)$ for some $b \in B$.  Since $F_\bullet$ is nearly generated by $\tau$, it follows that $F_\bullet$ is generated by the basis $B_{F_\bullet} := \tau \cup b$.  We note that if $\alpha > 1$, then
$$
B_{\mu_{\alpha-1}(F_\bullet)} = B_{F_\bullet} \qquad \text{and} \qquad \alpha(\mu_{\alpha-1}(F_\bullet)) = \alpha-1
$$
and if $\alpha < r$ then $\alpha(\mu_\alpha(F_\bullet))\in \{\alpha,\alpha+1\}$ (using that $F_\bullet$ is generated by $\tau \cup b$), and
$$
B_{\mu_\alpha(F_\bullet)} = \begin{cases} B_{F_\bullet} &\mbox{if $\alpha(\mu_\alpha(F_\bullet)) = \alpha+1$,}\\
B_{F_\bullet}\cup b' - b \text{ for some } b' \in B& \mbox{if $\alpha(\mu_\alpha(F_\bullet)) = \alpha$.}\
\end{cases}
$$
It follows that both $\mu_{\alpha-1}(F_\bullet)$ and $\mu_{\alpha}(F_\bullet)$ belong to $F(\tau,B)$.

For each $F_\bullet$, we add the edge $(F_\bullet, \mu_{\alpha(F_\bullet)}(F_\bullet))$ whenever $\alpha < r$, and add the edge $(F_\bullet, \mu_{\alpha(F_\bullet)-1}(F_\bullet))$ whenever $\alpha > 1$.  (If $\alpha = 1$, we only add $(F_\bullet,\mu_1(F_\bullet))$, and if $\alpha = r$, we only add $(F_\bullet,\mu_{r-1}(F_\bullet))$.). This defines the graph $\Gamma(\tau,B)$.  

For $F_\bullet \in F(\tau,B)$, define 
$$
E(F_\bullet) := \{f \mid \tau \cup f \text{ generates } F_\bullet\} = F_{\alpha}\setminus F_{\alpha-1}  \subset E.
$$
We compute that 
\begin{align*}
\dRip{e_\tau  \wedge \omega, e_B}&= \sum_{F_\bullet \in F(\tau,B)} h(F_\bullet) \prod_{i=1}^{r-1} \frac{1}{a_{F_i}} \sum_{E(F_\bullet)} a_f \\
&=\sum_{F_\bullet \in F(\tau,B)}h(F_\bullet)\prod_{i=1}^{r-1} \frac{1}{a_{F_i}} \sum_{F_{\alpha}\setminus F_{\alpha-1}} a_f \\
&=\sum_{F_\bullet \in F(\tau,B)} h(F_\bullet) \prod_{i=1}^{r-1} \frac{1}{a_{F_i}} (a_{F_\alpha} - a_{F_{\alpha-1}})\\
&=\sum_{F_\bullet \in F(\tau,B)} h(F_\bullet) \left(\prod_{i \neq \alpha} \frac{1}{a_{F_i}} - \prod_{i \neq \alpha-1} \frac{1}{a_{F_i}}\right)
\end{align*}
where the first term is omitted if $\alpha = r$ (using $a_E = 0$), and the second term is omitted if $\alpha = 1$.  The sign $h(F_\bullet) \in \{+,-\}$ is given by the formula
$$
h(F_\bullet) = r(B,F_\bullet) r(\{f,\tau_{1},\ldots,\tau_{d-1}\},F_\bullet),
$$
where $e_\tau = e_{\tau_{d-1}} \wedge \cdots \wedge e_{\tau_{1}}$ and $f$ is any element of $E(F_\bullet)$.  Let $(F_\bullet,F'_\bullet)$ be an edge of $\Gamma(\tau,B)$.  In the case $\alpha(F_\bullet) \neq \alpha(F'_\bullet)$, we have $r(B,F_\bullet) =- r(B,F'_\bullet) $ and the factor $r(\{f,\tau_1,\ldots,\tau_{r-1}\},F_\bullet)$ changes sign, so we have $h(F_\bullet) = h(F'_\bullet)$.  In the case $\alpha(F_\bullet) = \alpha(F'_\bullet)$, we have $r(B,F_\bullet) = - r(B,F'_\bullet)$ but the factor $r(\{f,\tau_1,\ldots,\tau_{r-1}\},F_\bullet)$ does not change sign, so we have $h(F_\bullet) = -h(F'_\bullet)$.

The (at most) two terms in the $F_\bullet$ summand cancel out with the corresponding terms (depending on whether $\alpha$ changes) for $F'_\bullet$ and $F''_\bullet$ where the (at most) two edges incident to $F_\bullet$ in $\Gamma(\tau,B)$ are $(F_\bullet,F'_\bullet = \mu_{\alpha(F_\bullet)}(F_\bullet))$ and $(F_\bullet,F''_\bullet = \mu_{\alpha(F_\bullet)-1}(F_\bullet))$.  We conclude that $\dRip{ e_\tau  \wedge \omega, e_B } = 0$.
\end{proof}

\begin{example}
Let $U_{2,n}$ denote the uniform matroid of rank $2$ on $[n]$.  Let $\tau = \{1\}$ and $B = \{i,j\}$.  Then $e_\tau  \wedge \omega= \sum_{k=2}^n a_k e_1 \wedge e_k$.  The flags that potentially contribute to $\dRip{ e_\tau  \wedge \omega, e_B = e_i \wedge e_j}$ are $(\hat 0 \subset \{i\} \subset \hat 1)$ and $(\hat 0 \subset \{j\} \subset \hat 1)$, and we obtain
$$
\dRip{ e_\tau  \wedge \omega, e_B} = \begin{cases} \frac{1}{a_i} a_i - \frac{1}{a_j} a_j = 0 & \mbox{if $i,j \neq 1$,} \\
\frac{1}{a_1} \left(- \sum_{k=2}^n a_k \right)  - \frac{1}{a_i} a_i= 0&\mbox{if $i = 1$ and $j >1$,} 
\end{cases}
$$
using \eqref{eq:sumto0}.
\end{example}

By \cref{prop:dRpartial} and \cref{prop:OSrOStwisted}, the symmetric form $\dRipp{\cdot, \cdot}$ on $\rOS(M)$ also descends to a symmetric form on $\rOS(M,\omega)$, and we use $\bdRip{\cdot, \cdot}$ to denote the symmetric forms on both $\OS(M,\omega) = \OS^r(M,\omega)$ and $\rOS(M,\omega) = \rOS^{r-1}(M,\omega)$.  The assumption \eqref{eq:sumto0} is always in place when we use the notation $\bdRip{\cdot, \cdot}$.

\subsection{$\beta$\nbc-basis}
\def\Bnbc{\B_{\mathbf{nbc}}}
We continue to assume that $(M,0)$ is an affine matroid.  Recall that in \cref{sec:nbc} we have defined \nbc-bases with respect to a fixed total order $\prec$ on $E$.  We assume that $0$ is the minimum of $\prec$.  Then every $\nbc$-basis $B$ of $(M,0, \prec)$ contains the element $0$.  

\begin{definition}
A \nbc-basis $B$ is called a $\beta$\nbc-basis if for any $i \in B \setminus 0$ there exists $j \prec i$ such that $B \setminus i \cup j \in \B(M)$.
\end{definition}

Let $\Bnbc = \Bnbc(\M)$ denote the set of $\beta$\nbc-bases $B$, where we always assume that $B = (b_1 \succ b_{2} \succ \cdots \succ b_r)$ is reversely ordered according to $\prec$.

\begin{theorem}[\cite{Zie}]\label{thm:Bnbc}
The cardinality of $\Bnbc$ is equal to $\beta(M)$.
\end{theorem}

$\beta$\nbc-bases behave well with respect to deletion-contraction of the largest element.  Suppose that $e_\prec = \max_\prec E$, and consider the deletion-contraction triple $(M,M' = M\backslash e_\prec,M'' =M/e_\prec)$.  

\begin{proposition}[{\cite[Theorem 1.5]{Zie}}]\label{prop:Zie}
Suppose that $e_\prec = \max_\prec E$ is not a loop and $|E| > 1$.  Then 
$$
\Bnbc(M) = \Bnbc(M') \sqcup \{(B \sqcup e_\prec) \mid B \in \Bnbc(M'')\}. 
$$
\end{proposition}

For each ordered basis $B \in \Bnbc$, we define a flag 
$$
F_\bullet(B) := (\hat 0 \subset \sp(b_1) \subset \sp(b_1,b_2) \subset \cdots \subset \sp(b_1,\ldots,b_{r-1})) \in \Fl^{r-1}(M).
$$

\subsection{$\beta$\nbc-basis determinant}
We now assume that an orientation $\M$ of $M$ has been fixed.
Let $(F^{(1)}_{\bullet},\ldots,F^{(\beta)}_{\bullet})$ be an ordering of $\{F_\bullet(B) \mid B \in \Bnbc\}$, and let $(P_1,\ldots,P_\beta)$ be an ordering of the set $\T^0(\M)$ of bounded topes.  Both sets have cardinality $\beta(M)$.  In the following, we declare that a $0 \times 0$ matrix has determinant $1$.
\begin{proposition}\label{prop:detnbc}
The $\beta(M) \times \beta(M)$ matrix 
$$
Z = \left(\Res_{F^{(i)}_{\bullet}} \bOmega_{P_j}\right)^{\beta(M)}_{i,j=1}
$$
has determinant $\pm 1$. 
\end{proposition}
\begin{proof}
We may suppose that $M$ is simple.  Let $D_M := \det(Z)$ denote the determinant.  We prove the statement by a deletion-contraction induction.  If $|E| = 0$ we have $\beta(M) = 0$ and if $|E| = 1$ we have $\beta(M) = 1$, and in both cases the claim is clear.  Let $e = e_\prec = \max_\prec(M)$, and consider the deletion-contraction triple $(M,M' = M\backslash e,M'' =M/e)$.  The set of flags $\Fl_\nbc = \{F_\bullet(B) \mid B \in \Bnbc\}$ decomposes into a disjoint union $\Fl'_\nbc \sqcup \Fl''_\nbc$ as in \cref{prop:Zie}.  On the other hand, let us write $\T^0(\M) = \T_1 \sqcup \T_2 \sqcup \T_3$ where 
\begin{align*}
\T_1 &= \mbox{topes in $\T^0(\M)$ that are also topes of $\T^0(\M')$} \\
\T_2 &= \mbox{topes in $\T^0(\M)$ that are cut into two topes in $\T^0(\M')$} \\
\T_3 &= \mbox{topes in $\T^0(\M)$ whose restriction to $E'$ do not belong to $\T^0(\M')$}.
\end{align*}

Each term of the determinant $D_M$ corresponds to a bijection $\tau: \Fl_\nbc \to \T^0(\M)$ between flags and topes.

Suppose that $\tau$ maps two distinct flags $F^{(a)}_{\bullet}, F^{(b)}_{\bullet} \in \Fl''_\nbc$ to two topes $P, P' \in \T_2$ respectively, where $P,P'$ are divided by $e$, i.e. $P(f) = P'(f)$ for all $f \in E \setminus e$ and $P(e) = - P'(e)$.  Then we obtain another bijection $\tau'$ by swapping $P,P'$, and since $\Res_e \bOmega_P = - \Res_e \bOmega_{P'}$, the contribution of $\tau$ and $\tau'$ to the determinant cancels out.  Furthermore, if $\tau(F_\bullet) \in \T_1$ for $F_\bullet \in \Fl''_\nbc$ then $\Res_{F_\bullet}(\tau(F_\bullet)) = 0$.

Let $\Z/2\Z$ act involutively on $\T_2$ by sending a tope $P$ to the adjacent tope on the other side of $e$.
Since $|\T_2/(\Z/2\Z)| + |\T_3| = |\T^0(\M'')| = \beta(M'')$, we reduce to summing over bijections $\tau$ that induce a bijection between $\Fl''_\nbc$ and $\T_2/(\Z/2\Z) \sqcup \T_3$.  For such $\tau$, we may restrict $\tau$ to $\Fl'_\nbc$ and obtain a bijection $\tau': \Fl'_\nbc \to \T^0(\M')$ by composing with the map that sends each tope in $\T_1 \cup \T_2$ to $\T^0(\M')$ by restricting topes to $E'$.  For $F_\bullet \in \Fl'_\nbc$ and $P, P' \in \T_2$ divided by $e$, we have that at least one of $\Res_{F_\bullet}(\bOmega_P), \Res_{F_\bullet}(\bOmega_{P'})$ vanishes, and the sum is equal to $\Res_{F_\bullet}(\bOmega_P+\bOmega_{P'})$.  It follows that for each non-vanishing term $\tau': \Fl'_\nbc \to \T^0(\M')$ in the determinant $D_{M'}$, there is a unique corresponding $\tau|_{\Fl'_{\nbc}}$ that gives rise to it.

So viewing $\tau|_{\Fl'_{\nbc}}$ as a bijection $\tau|_{\Fl'_{\nbc}}: \Fl'_{\nbc} \to \T^0(\M')$ and $\tau|_{\Fl''_\nbc}$ as a bijection $\tau|_{\Fl''_\nbc}: \Fl''_\nbc \to \T^0(\M'')$, we have a bijection
\begin{equation}\label{eq:tautau}
\tau \mapsto (\tau' = \tau|_{\Fl'_{\nbc}}, \tau'' = \tau|_{\Fl''_\nbc})
\end{equation}
that sends non-zero terms of the determinant $D_M$ to pairs of non-zero terms of the determinants $D_{M'}$ and $D_{M''}$.  It remains to show that the signs are correct.  Let $P,P'$ be divided by $e$.  If we swap $P$ and $P'$ in $\tau$ then $(-1)^\tau$ acquires a sign $(-1)$.  However, this is compensated for by the sign-change $\Res_e(\bOmega_P) = - \Res_e(\bOmega_{P'})$ (\cref{thm:EL}).  It follows that up to a single global sign, the map \eqref{eq:tautau} sends a term in $D_M$ to a product of terms in $D_{M'} D_{M''}$.  By induction, we conclude that $D_M = \pm D_{M'} D_{M''} = \pm 1$.
\end{proof}

For $F \subset E \setminus 0$, define
$$
\bomega(F) := \sum_{e \in F} a_e (e - e_0) \in \rOS^1(M).
$$
and
$$
S(F_\bullet(B)):= \bomega(F_{r-1}) \wedge \bomega(F_{r-2}) \wedge \cdots \wedge \bomega(F_{1}).
$$

\begin{lemma}\label{lem:Resnbc}
For $B \in \Bnbc$ and $P \in \T^0(\M)$, we have $\bdRip{S(F_\bullet(B)), \bOmega_P}  = \Res_{F_\bullet(B)} \bOmega_P$.
\end{lemma}
\begin{proof}
Recall the isomorphism $S^k: \F^k \to \OS^k(M)$ from \eqref{eq:RS}.  We have
$$
S^{r-1}(F_\bullet(B))=\omega(F_{r-1}) \wedge \omega(F_{r-2}) \wedge \cdots \wedge \omega(F_{1}), \qquad \omega(F):= \sum_{e \in F} a_e e.
$$
Thus
$$
S(F_\bullet(B)) = S^{r-1}(F_\bullet(B)) \mod e_0 \OS^\bullet(M).
$$
Since $P \in \T^0(\M)$ is bounded, we have that $\Res_{F_\bullet} \bOmega_P = 0$ for any $F_\bullet \in \Fl^{r-1}(M)$ such that $0 \in F_{r-1}$.  It follows that 
$\bdRip{S(F_\bullet(B)), \bOmega_P} =\bdRip{S^{r-1}(F_\bullet(B)), \bOmega_P}$.  By \cref{cor:SVform}, \cref{prop:SVinverse} and \eqref{eq:etadef}, we have for $y \in \OS^{r-1}(M)$,
$$
\bdRip{S^{r-1}(F_\bullet(B)), y} =  (\eta((R^{r-1} \circ S^{r-1})(F_\bullet(B))), y) = (\eta(F_\bullet(B)),y) = \Res_{F_\bullet(B)} y.
$$ 
Thus $\bdRip{S(F_\bullet(B)), \bOmega_P} =  \Res_{F_\bullet(B)} \bOmega_P$.
\end{proof}

By \cref{lem:Resnbc}, \cref{prop:detnbc} calculates the determinant of the $\bdRip{\cdot,\cdot}$-pairing between the two sets $\{S(F_\bullet(B)) \mid B \in \Bnbc(M)\}$ and $\{\bOmega_{P} \mid P \in \T^0(\M)\}$ in $\rOS(M,\omega)$.

\begin{corollary}\label{cor:Aomotobasis}
When \eqref{eq:Mon} is satisfied, the two sets $\{S(F_\bullet(B)) \mid B \in \Bnbc(M)\}$ and $\{\bOmega_{P} \mid P \in \T^0(\M)\}$ generate dual spanning lattices of $\rOS(M,\omega)$.  In particular, both $\{S(F_\bullet(B)) \mid B \in \Bnbc(M)\}$ and $\{\bOmega_{P} \mid P \in \T^0(\M)\}$ form bases of $\rOS(M,\omega)$.
\end{corollary}

The basis $\{S(F_\bullet(B)) \mid B \in \Bnbc(M)\}$ of $\rOS(M,\omega)$ was studied by Falk and Terao \cite{FT}.

\begin{corollary}\label{cor:twistnondeg}
When \eqref{eq:Mon} is satisfied, $\bdRip{\cdot,\cdot}$ is a non-degenerate symmetric bilinear form on $\rOS(M,\omega)$. 
\end{corollary}
Note that \cref{cor:twistnondeg} is only proven in the case that an orientation $\M$ of $M$ exists, though it is likely that it always holds.

\begin{example}\label{ex:npoint}
Let $(\M,0)$ be the affine oriented matroid of the arrangement $\bA$ of $n$ real points $1,2,\ldots,n$ in order on the real affine line, with $0$ the point at infinity.  Then the underlying matroid $M$ is isomorphic to the uniform matroid $U_{2,n+1}$ of rank 2 on the set $E = [n+1]$.  The space $\rOS(M) = \rOS^1(M)$ has basis $\be_1,\be_2,\ldots,\be_n$, where $\be_i = e_i - e_0$.  We take the total order on $E$ to be $0 \prec 1 \prec 2 \prec \cdots \prec n$, so that the $\beta$\nbc-basis is 
$$
\Bnbc = \{20,30,\ldots,n0\}.
$$
and $\beta(M) = n-1$.  We have 
$$
\{S(F_\bullet(B)) \mid B \in \Bnbc(M)\} = \{a_2 \be_2, a_3 \be_3,\ldots, a_n \be_n\}
$$
and 
$$
\{\bOmega_{P} \mid P \in \T^0(\M)\} = \{\be_2 - \be_1,\be_3- \be_2, \ldots, \be_n- \be_{n-1}\}.
$$
The matrix $Z$ of \cref{prop:detnbc} is given by 
$$
\begin{bmatrix} 1 & -1 & 0 & \cdots & 0 \\
0 &1 & -1 & \cdots & 0 \\
\vdots & \ddots & \ddots & \ddots & \vdots \\
0 & 0 & 0 & \cdots &-1\\
0 & 0 & 0 & \cdots &1
\end{bmatrix}
$$
which has determinant 1.  Let $\theta_i = \be_i- \be_1$ for $i =2,\ldots,n$, so that $\{\theta_2,\ldots,\theta_n\}$ generate the same lattice as $\{\bOmega_{P} \mid P \in \T^0(\M)\}$.  Then using the relation $\omega = 0$, we have 
$$
a_i (e_i - e_0) = a_i e_i + \frac{1}{a_0}(\sum_{i=1}^n a_i e_i) =  \frac{a_i}{a_0} \left( a_2 \theta_2 + \cdots + (a_i+a_0) \theta_i + \cdots + a_n \theta_n\right).
$$
The transition matrix from $\{S(F_\bullet(B)) \mid B \in \Bnbc(M)\}$ to $\{\theta_2,\ldots,\theta_n\}$ is, after multiplying the rows by $\frac{a_i}{a_0}$,
$$
\begin{bmatrix} a_0+a_2 & a_3 & a_4 & \cdots & a_n \\
a_2 &a_0+a_3 & a_4 & \cdots & a_n \\
\vdots & \ddots & \ddots & \ddots & \vdots \\
a_2 & a_3 & a_4 & \cdots &a_n\\
a_2 & a_3 & a_4 & \cdots &a_0+a_n
\end{bmatrix}
$$
which has determinant $-a^{n-2}_0 a_1$.  So the transition matrix between $\{S(F_\bullet(B)) \mid B \in \Bnbc(M)\}$ and $\{\bOmega_{P} \mid P \in \T^0(\M)\}$ has determinant 
$$
\det = \pm \frac{a_1a_2\cdots a_n}{a_0} =\pm R_M(\a)^{-1}, \mbox{where $R_M(\a)$ is defined in \cref{def:RM}.}
$$
\end{example}

\subsection{Determinant on bounded chambers}
Let $L_0  \subset L(M)$ consist of those flats containing $0$.  Thus $L_0 \cong L(M_0)$.
In the following, for $F \in L$ or $F \in L_0$, we write $\beta(F)$ to refer to the beta invariant of $M^F$. 

\begin{definition}\label{def:RM}
Let $(M,0)$ be a simple affine matroid.  Define
$$
R_M(\a):= \frac{\prod_{F \in L_0 \setminus \hat 1} a_{F}^{\beta(F) \beta(M_F)}}{\prod_{F \in L \setminus (L_0 \cup \hat 0)} a_F^{\beta(F) \beta(M_F)}}.
$$
\end{definition}

The following result is a variant of \cref{thm:SVdet} for $\rOS(M,\omega)$.

\begin{theorem}\label{thm:Aomotodet}
The determinant of $\bdRip{\cdot, \cdot}$ on the lattice spanned by $\{\bOmega_{P} \mid P \in \T^0(\M)\}$ is equal to
$$
\det \bdRip{\cdot, \cdot}_{\T^0}= \pm R_M(\a).
$$
\end{theorem}
If $\beta(M) = 0$ then the determinant is defined to be 1.

\begin{corollary}
The transition matrix between the two bases $\{S(F_\bullet(B)) \mid B \in \Bnbc(M)\}$ and $\{\bOmega_{P} \mid P \in \T^0(\M)\}$ has determinant equal to $\pm R_M(\a)^{\pm 1}$.
\end{corollary}
\begin{proof}
The determinant in question is equal to the ratio of the determinants in \cref{thm:Aomotodet} with \cref{prop:detnbc}.
\end{proof}

\begin{corollary}\label{cor:bnbcdet}
The determinant of $\bdRip{\cdot, \cdot}$ on the lattice spanned by $\{S(F_\bullet(B)) \mid B \in \Bnbc(M)\}$ is equal to $\pm R_M(\a)^{-1}$.
\end{corollary}

Our proof of \cref{cor:bnbcdet} depends on the existence of an orientation $\M$ of $M$, though it is likely that the result holds without this assumption.

\begin{example}
Consider the affine hyperplane arrangement $\bA$ in $\R^2$, pictured below.
$$
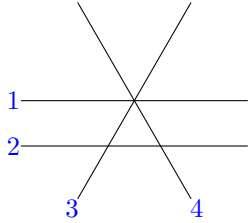
\begin{tikzpicture}
\draw (0:1.5) -- (180:1.5);
\draw (60:1.5) -- (240:1.5);
\draw (-60:1.5) -- (-240:1.5);
\draw (-1.5,-0.6)--(1.5,-0.6);
\node[color=blue] at (-1.6,0) {$1$};
\node[color=blue] at (-1.6,-0.6) {$2$};
\node[color=blue] at (240:1.65) {$3$};
\node[color=blue] at (-60:1.65) {$4$};
\end{tikzpicture}
$$
Let $(M,0)$ be the affine matroid of $\A$, with ground set $E = \{0,1,2,3,4\}$.  The characteristic polynomial of $\A$ (or the reduced characteristic polynomial of $M$) is $\bchi(t) = t^2- 4t + 4$.  The rank $2$ flats are $134,23,24,012,03,04$, of which $134$ and $012$ are connected.  The reduced Orlik-Solomon algebra $\rOS(M)$ is the exterior algebra on $\be_1,\be_2,\be_3,\be_4$ modulo the relations $\be_2 \be_1 = 0$, $\be_3 \be_1 - \be_4 \be_1 + \be_4 \be_3 =0$, and all cubic monomials vanish.  Thus $\dim(\rOS^2(M)) = 4$ with \nbc~basis 
$$
\be_3 \be_1, \be_4 \be_1, \be_3 \be_2, \be_4 \be_2.
$$
The intersection form $\dRip{\cdot,\cdot}$ on the \nbc~basis is

\scalebox{0.75}{\hspace*{-0.5cm}

$\begin{bmatrix}
 \frac{1}{a_1 a_{012}}+\frac{1}{a_0 a_{012}}+\frac{1}{a_0 a_3}+\frac{1}{a_1 a_{134}}+\frac{1}{a_3 a_{134}} & \frac{1}{a_1 a_{012}}+\frac{1}{a_0 a_{012}}+\frac{1}{a_1 a_{134}} & \frac{1}{a_0 a_{012}}+\frac{1}{a_0 a_3} & \frac{1}{a_0 a_{012}} \\
 \frac{1}{a_1 a_{012}}+\frac{1}{a_0 a_{012}}+\frac{1}{a_1 a_{134}} & \frac{1}{a_1 a_{012}}+\frac{1}{a_0 a_{012}}+\frac{1}{a_0 a_4}+\frac{1}{a_1 a_{134}}+\frac{1}{a_4 a_{134}} & \frac{1}{a_0 a_{012}} & \frac{1}{a_0 a_{012}}+\frac{1}{a_0 a_4} \\
 \frac{1}{a_0 a_{012}}+\frac{1}{a_0 a_3} & \frac{1}{a_0 a_{012}} & \frac{1}{a_2 a_{012}}+\frac{1}{a_0 a_{012}}+\frac{1}{a_0 a_3}+\frac{1}{a_2 a_3} & \frac{1}{a_2 a_{012}}+\frac{1}{a_0 a_{012}} \\
 \frac{1}{a_0 a_{012}} & \frac{1}{a_0 a_{012}}+\frac{1}{a_0 a_4} & \frac{1}{a_2 a_{012}}+\frac{1}{a_0 a_{012}} & \frac{1}{a_2 a_{012}}+\frac{1}{a_0 a_{012}}+\frac{1}{a_0 a_4}+\frac{1}{a_2 a_4} \\
\end{bmatrix}
$
}

\noindent
with determinant 
$$
\frac{a_{01234}^3}{a_0^2 a_1 a_2^2a_3^2 a_4^2 a_{012}a_{134}},
$$
agreeing with \cref{thm:SVdet}.  Taking $a_0 \to \infty$, we get
$$
\begin{bmatrix}
 \frac{1}{a_3 a_{134}}+\frac{1}{a_1 a_{134}} & \frac{1}{a_1 a_{134}}& 0 & 0 \\
 \frac{1}{a_1 a_{134}} &\frac{1}{a_1 a_{134}}+\frac{1}{a_4 a_{134}}& 0 & 0 \\
 0 & 0 & \frac{1}{a_2a_3} & 0 \\
 0 & 0 & 0 & \frac{1}{a_2a_4} \\
\end{bmatrix}
\qquad
\mbox{with determinant}
\qquad
\frac{1}{a_1 a_2^2 a_3^2 a_4^2 a_{134}}.
$$
This is the matrix of the Schechtman-Varchenko contravariant form \cite{SV}.

Now, let us consider bounded chambers.  We have $\beta(M) = \bchi(1) = 1$.  The bilinear form $\bdRip{\cdot,\cdot}$ on the basis $\{\bOmega_P \mid P \in \T^0(\M)\}$ is the single entry
$$
\frac{1}{a_2a_3} + \frac{1}{a_2a_4} + \frac{1}{a_3(a_1+a_3+a_4)} +  \frac{1}{a_4(a_1+a_3+a_4)}  = \frac{(a_3+a_4)(a_1+a_2+a_3+a_4)}{a_2a_3a_4 a_{134}}.
$$
The factors in the numerator are, up to sign, equal to $a_{012}$ and $a_0$, with $\{0,1,2\}$ and $\{0\}$ the connected flats in $L_0$, agreeing with \cref{thm:Aomotodet}.
\end{example}

\subsection{Proof of \cref{thm:Aomotodet}}

For $F \subset E$, define $\kappa_F := \beta(F) \beta(M_F)$ if $F$ is a flat and $0$ otherwise.  Let $e \in E$ be neither a loop or a coloop, and let $\kappa'_F, \kappa''_F$ and $\beta',\beta''$ be the corresponding functions for $M', M''$.  

\begin{lemma}\label{lem:kappaF}
For $F \subset E \setminus e$, we have $\kappa_F + \kappa_{F \cup e} = \kappa'_F + \kappa''_F$.
\end{lemma}
\begin{proof}
We may assume that $M$ is simple.  If $F$ and $F\cup e$ are both non-flats, we have $0 = 0$.  If both are flats then $F \cup e$ is decomposable, so $\kappa_{F \cup e} = 0$, and we have
$$
\kappa_F = \beta(F) \beta(M_F) = \beta(F) (\beta(M/F \backslash e) + \beta(M/(F\cup e))) = \kappa'_F + \kappa''_F,
$$
where we have used \eqref{eq:betaeq} and the fact that $e$ is not a loop or coloop in $M_F$.
If $F$ is a flat and $F \cup e$ is not, then $\kappa_{F \cup e} = \kappa''_F = 0$, and 
$$
\kappa_F = \beta(F) \beta(M_F) = \beta(F) \beta(M'_F) = \kappa'_F,
$$
because $M/F$ and $(M/F)\backslash e$ have the same lattice of flats (the element $e$ belongs to a non-trivial parallel class in $M_F$).
If $F$ is not a flat but $F\cup e$ is, then $\kappa_F = 0$ and
$$
\kappa_{F\cup e} = \beta(F \cup e) \beta(M_{F \cup e}) = \beta'(F) \beta(M'_{F }) +  \beta''(F) \beta(M''_{F}) = \kappa'_{F} + \kappa''_F,
$$
where in the second equality we have used \eqref{eq:betaeq} for $\beta(F \cup e)$ and the isomorphism $L(M_{F \cup e}) = L(M'_F)$.
\end{proof}

The statement of \cref{thm:Aomotodet} reduces to the case that $M$ is simple, which we assume.  We proceed by deletion-contraction induction.  When $\rk(M) = 1$, we have $\beta(M) = 1$, and the determinant is equal to $1$.  We henceforth assume that $\rk(M) > 1$.     If $M$ is not connected, then $\beta(M) = 0$, and the result holds by our convention.  We thus assume that $M$ is connected, and in particular has no coloops, and apply deletion-contraction to an element $e \in E \setminus 0$.

Since $a_E = 0$, we have $a_0 = - \sum_{e \in E \setminus 0} a_e$.  We use this substitution to work within the ring of rational functions in $a_e$, $e \in E \setminus 0$.
To begin the proof of the theorem, we note that by \cref{cor:denom}, all the pairings $\bdRip{\bOmega_P, \bOmega_{Q}}$ have denominators belonging to $\{a_F \mid F \text{ connected}\}$.  Also, according to \cref{cor:Aomotobasis}, $\{\bOmega_{P} \mid P \in \T^0(\M)\}$ is a basis of $\OS(M,\omega)$ when \eqref{eq:Mon} is satisfied, and thus the determinant in question can only vanish when one of the $a_F$ vanishes.  We thus have

\begin{lemma}
The determinant is of the form 
\begin{equation}\label{eq:gammaF}
D(\M) = C(\M) \cdot \prod_{F \text{ connected }\in L \setminus \{\hat 0,\hat 1\}} a_F^{\gamma_F}
\end{equation}
where $C(\M)$ is a constant and $\gamma_F \in \Z$.
\end{lemma}
The assumption that $M$ is connected implies that $F$ and $E \setminus F$ cannot simultaneously be flats.  It follows that there are no repetitions (even up to sign) among the linear forms in the product \eqref{eq:gammaF}.  In particular, the integers $\gamma_F$ are uniquely determined.

Recall the decomposition $\T^0(\M) = \T_1 \sqcup \T_2 \sqcup \T_3$ from the proof of \cref{prop:detnbc}.  Let $\T_2' \subset \T_2$ be a choice of a tope $P$ for each pair of topes $(P,P')$ divided by $e$.
Define 
\begin{align*}
Z_1 &=  \{\bOmega_P \mid P \in \T^0(\M')\} = \{ \bOmega_P \mid P \in \T_1\} \sqcup \{\partial (\Omega_P + \Omega_{P'}) \mid P,P' \in \T_2 \text{ divided by } e\}\\
Z_2 &= \{\bOmega_P \mid P \in \T'_2 \} \sqcup \{ \bOmega_P \mid P \in \T_3\}.
\end{align*}
It is easy to see that $Z_1 \sqcup Z_2$ is again a basis of of $\rOS(M,\omega)$ and spans the same lattice as $\{\bOmega_{P} \mid P \in \T^0(\M)\}$.  We compute the determinant with respect to $Z_1 \sqcup Z_2$, ordering $Z_1$ before $Z_2$. 

\begin{lemma}
Let $Y'$ be the matrix of $\bdRip{\cdot,\cdot}$ with respect to the basis $Z_1 \sqcup Z_2$ and let $Y$ be obtained from $Y'$ by multiplying the rows indexed by $Z_2$ by $a_e$, and then substituting $a_e = 0$ in the whole matrix.  Then $Y$ has the form
$$
Y = \begin{bmatrix} A & B \\ 0 & D \end{bmatrix}
$$
where $A$ is a matrix representing $\bdRip{\cdot,\cdot}_{M'}$ and $D$ is a matrix representing $\bdRip{\cdot,\cdot}_{M''}$.
\end{lemma}
\begin{proof}
The statement regarding $A$ follows immediately from \cref{lem:deleteform}.  The statement concerning $D$ follows from \cref{lem:contractform}.  Finally, we need to show that the bottom-left block of $Y$ is the zero matrix.  Similarly to the proof of \cref{lem:deleteform}, for $x \in Z_1$, we have $\Res_e(x) = 0$.  Thus for $x \in Z_1$ and $y \in Z_2$, none of the terms contributing to $\bdRip{x,y}$ have $a_e$ in the denominator.  It follows that those entries become $0$ after multiplying by $a_e$ and setting $a_e$ to $0$.
\end{proof}

The cardinality of $Z_2$ is equal to $\beta(M'')$.  It follows that 
$$
\left.\left(a_{e}^{\beta(M'')} D(\M)\right)\right|_{a_{e}=0} = \pm D(\M') D(\M'').
$$
We immediately obtain that the constant $C(\M)$ in \eqref{eq:gammaF} is equal to $\pm 1$.  We also deduce that for the flat $F = \{e\}$, the integer $\gamma_F$ is equal to $\beta(M'')$.  For this flat, $M_F = M''$ and $\beta(F) = 1$, so $\gamma_F = \kappa_F$.

Now let $F \subset E \setminus e$.  Then $a_F|_{a_e = 0} = a_{F \cup e}|_{a_e = 0}$.  So comparing the coefficient of $a_F|_{a_e = 0}$ on both sides and using the inductive hypothesis for $\M',\M''$ and \cref{lem:kappaF}, we see that it is consistent with 
$$
\gamma_F = \begin{cases}
- \kappa_F & \mbox{ if $F \in L \setminus (L_0 \cup \hat 0)$} \\
\kappa_F & \mbox{if $F \in L_0$.} 
\end{cases}
$$ 
Note that in the case that $F$ and $F\cup e$ are both flats, the latter is decomposable and $\gamma_{F \cup e} = 0$.  However, there is one possible ambiguity.  It is possible for $a_F|_{a_e=0}$ to equal $-a_{F'}|_{a_e=0}$.  This occurs in two situations: (a) when $F$ and $F'$ are flats such that $F \cup F' = E$ and $F \cap F' = \{e\}$, or (b) when $F$ and $F'$ are flats such that $F \cup F' = E \setminus e$ and $F \cap F' = \emptyset$.  Call such pairs of flats $(F,F')$ \emph{$e$-special pairs}.  Note that the situation $F \cup F' = E$ and $F \cap F' = \emptyset$ does not appear since in this case $M$ is not connected.

Thus for $F \notin \{\hat 0, e\}$ and connected, the integer $\gamma_F$ in $D(\M)$ is equal to the product of the corresponding exponents in $D(\M')$ and $D(\M'')$, except for flats belonging to $e$-special pairs.  For a special pair $(F,F')$, the integer $\gamma_F + \gamma_{F'} = \kappa_F + \kappa_{F'}$ is determined.  Since $\rk(M) \geq 2$ and $M$ is connected, we have $|E| \geq 3$, and thus there is $e' \in E \setminus \{0, e\}$.  Repeating the deletion-contraction argument with $e'$ uniquely determines $\gamma_F$ and $\gamma_{F'}$ for $e$-special pairs $(F,F')$.  This completes the proof of \cref{thm:Aomotodet}.

\subsection{Inverse}
It would be interesting to compute the inverse of the matrix $\bdRip{\bOmega_P,\bOmega_Q}$ for $P,Q \in \T^0$.  In the case that $0$ is generic, this follows from \cref{thm:dRmain}.

\begin{theorem}
Suppose that the affine matroid $(M,0)$ is generic at infinity.  Then the inverse of the matrix $\bdRip{\bOmega_P,\bOmega_Q}$ with $P,Q \in \T^0$ is given by the matrix $\DdRip{P,Q} := \sum_{B \in \B(P,Q)} a^B$ of \cref{def:DdR}.
\end{theorem}
\begin{proof}
Let $\M' = \M \setminus 0$, and identify $(\M')^\star = \M$ and $\star = 0$.  With these choices, $\T^\star(\M') = \T^0(\M)$.  Furthermore, the calculation of $\bdRip{\bOmega_P,\bOmega_Q}_{M}$ only involves flags of flats that do not contain $0$ (and only uses $a_e$, $e \in E \setminus 0$), so we have $\bdRip{\bOmega_P,\bOmega_Q}_{M} = \dRip{\Omega_P,\Omega_Q}_{M'}$.  The result follows by applying \cref{thm:dRmain} to $\M'$.
\end{proof}

\section{Betti homology intersection form}

In this section we consider an affine oriented matroid $(\M,0)$.  We use the notation and results in \cref{sec:pFl}. \subsection{Definition of Betti homology intersection form}
Let $S := \Z[\b] =  \Z[b_e \mid e \in E]$ and $K = \Frac(S)$.  When we specialize the parameters $b_e$ to complex numbers, they are related to the parameters $a_e$ by the formula
\begin{equation}\label{eq:ba}
b_e = \exp(-\pi i a_e),
\end{equation}
to be explained in \cref{ssec:twisted}.  For $e \in E$ and $S \subset E$, define
$$\tb_e:= b_e^2-1, \qquad b_S := \prod_{e \in S} b_e, \qquad \tb_S := b_S^2 - 1.$$

\begin{lemma}
We have $\tb_S = \sum_{\emptyset \subsetneq S' \subseteq S} \prod_{e \in S'} \tb_e$.
\end{lemma}

Recall that $\T^+$ denotes the set of topes $P \in \T$ satisfying $P(0) = +$.  Let $\Z^{\T^+}$ denote the free abelian group with basis $\{P \mid P \in \T(\M)\}$.  For clarity, we sometimes also write $[P] \in \Z^{\T^+}$ for the basis element indexed by $P$.  For $P \in \T^+$, we define $[-P] := (-1)^r [P] \in \Z^{\T^+}$, so that all topes $P \in \T$ index elements of $\Z^{\T^+}$.  We shall define a $K$-valued bilinear pairing on $\Z^{\T^+}$,
$$
\halfip{\cdot,\cdot}_B:\Z^{\T^+} \otimes \Z^{\T^+} \to K.
$$

For $E_\bullet \in \pFl(P)$, define
$$
\frac{1}{\tb_{E_\bullet}} := \prod_{i=1}^s \frac{1}{\tb_{E_i}} = \prod_{i=1}^s \frac{1}{b_{E_i}^2 - 1}.
$$

\begin{definition}\label{def:Bettipair}
For $G_\bullet =  \{\hat 0 \subset G_1 \subset G_2 \cdots \subset G_s \subset E\} \in \pFl(P)$, define
$$
\ip{G_\bullet}_B := b(G_\bullet) \sum_{E_\bullet \in \bG_\bullet} \prod_{i=1}^{s(E_\bullet)} \frac{1}{b_{E_i}^2 -1} =  b(G_\bullet) \sum_{E_\bullet \in \bG_\bullet} \prod^{s(E_\bullet)}_{i=1} \frac{1}{\tb_{E_i}} =  b(G_\bullet) \sum_{E_\bullet \in \bG_\bullet} \frac{1}{\tb_{E_\bullet}}
$$
where
$$
b(G_\bullet):=  \prod_{i=1}^s (-1)^{\rk(G_i)} b_{G_i}.
$$
Define the Betti homology intersection form on $\Z^{\T^+}$ by
\begin{equation}\label{eq:halfPQ}
\halfip{P,Q}_B := \sum_{G_\bullet \in G^{\pm}(P,Q)} (\pm)^r \ip{G_\bullet}_B,
\end{equation}
where the sign $(\pm)^r$ is equal to $1$ or $(-1)^r$ depending on whether $G$ belongs to $G(P,Q)$ or $G(P,-Q)$, and $P,Q \in \T$.  We write $\bip{P,Q}_B$ when we work with coefficients satisfying $b_E = \prod_{e\in E} b_e = 1$.  
\end{definition}
If we consider $\ip{P,Q}_B$ as a rational function in $\{b_e \mid e \in E\}$, then $\bip{P,Q}_B$ is the image of that rational function in the fraction field of the ring $\Z[b_e \mid e \in E]/(\prod b_e = 1)$.  

It follows from the definitions that \eqref{eq:halfPQ} is consistent with $[P] = (-1)^r[-P]$.

\begin{remark}
Since the formula for $\ip{G_\bullet}_B$ uses $\bG_\bullet$, the expression $\ip{G_\bullet}_B$ depends on $\pFl(P)$ and thus on $P$.  However, if $G_\bullet \in G(P,Q)$, then $\ip{G_\bullet}_B$ is the same whether we consider $G_\bullet \in \pFl(P)$ or $G_\bullet \in \pFl(Q)$; see \cref{lem:closurePQ}.
\end{remark}

\begin{remark}
While the deRham cohomology intersection form $\dRip{\cdot,\cdot}$ is defined for an unoriented matroid $M$, the Betti homology intersection form $\halfip{\cdot,\cdot}_B$ is defined with a choice of orientation $\M$ of $M$.  It would be interesting to define $\halfip{\cdot,\cdot}_B$ without choosing an orientation.
\end{remark}

By \cref{prop:noover}, each term $\frac{1}{\tb_{E_\bullet}}$ appears at most once in $\halfip{P,Q}_B$.
\begin{proposition}\label{prop:ipneg}
We have $\halfip{P,Q}_B = \halfip{Q,P}_B = \halfip{-P,-Q}_B = \halfip{-Q,-P}_B$.
\end{proposition}
\begin{proof}
Follows from \cref{prop:noover}(2).
\end{proof}

\begin{example}\label{ex:3pttopeB}
We calculate $\halfip{\cdot,\cdot}_B$ for the arrangement in \cref{ex:3pttope}.  We order $$\T^+ = \{(+,+,+), (+,-,+), (+,-,-), (+,+,-)\}.$$  We have the $4 \times 4$ matrix:
$$\halfip{\cdot,\cdot}^{\T^+}_B=
\begin{bmatrix}
 \frac{(b_1 b_2 b_0-1) (b_1 b_2 b_0+1)}{\left(b_1^2-1\right) \left(b_2^2-1\right) \left(b_0^2-1\right)} & -\frac{(b_1+b_2 b_0) (b_1 b_2 b_0-1)}{\left(b_1^2-1\right) \left(b_2^2-1\right) \left(b_0^2-1\right)} & \frac{(b_1 b_2+b_0) (b_1 b_2 b_0-1)}{\left(b_1^2-1\right) \left(b_2^2-1\right) \left(b_0^2-1\right)} & -\frac{(b_1 b_0+b_2) (b_1 b_2 b_0-1)}{\left(b_1^2-1\right) \left(b_2^2-1\right) \left(b_0^2-1\right)} \\
 -\frac{(b_1+b_2 b_0) (b_1 b_2 b_0-1)}{\left(b_1^2-1\right) \left(b_2^2-1\right) \left(b_0^2-1\right)} & \frac{(b_1 b_2 b_0-1) (b_1 b_2 b_0+1)}{\left(b_1^2-1\right) \left(b_2^2-1\right) \left(b_0^2-1\right)} & -\frac{(b_1 b_0+b_2) (b_1 b_2 b_0-1)}{\left(b_1^2-1\right) \left(b_2^2-1\right) \left(b_0^2-1\right)} & \frac{(b_1 b_2+b_0) (b_1 b_2 b_0-1)}{\left(b_1^2-1\right) \left(b_2^2-1\right) \left(b_0^2-1\right)} \\
 \frac{(b_1 b_2+b_0) (b_1 b_2 b_0-1)}{\left(b_1^2-1\right) \left(b_2^2-1\right) \left(b_0^2-1\right)} & -\frac{(b_1 b_0+b_2) (b_1 b_2 b_0-1)}{\left(b_1^2-1\right) \left(b_2^2-1\right) \left(b_0^2-1\right)} & \frac{(b_1 b_2 b_0-1) (b_1 b_2 b_0+1)}{\left(b_1^2-1\right) \left(b_2^2-1\right) \left(b_0^2-1\right)} & -\frac{(b_1+b_2 b_0) (b_1 b_2 b_0-1)}{\left(b_1^2-1\right) \left(b_2^2-1\right) \left(b_0^2-1\right)} \\
 -\frac{(b_1 b_0+b_2) (b_1 b_2 b_0-1)}{\left(b_1^2-1\right) \left(b_2^2-1\right) \left(b_0^2-1\right)} & \frac{(b_1 b_2+b_0) (b_1 b_2 b_0-1)}{\left(b_1^2-1\right) \left(b_2^2-1\right) \left(b_0^2-1\right)} & -\frac{(b_1+b_2 b_0) (b_1 b_2 b_0-1)}{\left(b_1^2-1\right) \left(b_2^2-1\right) \left(b_0^2-1\right)} & \frac{(b_1 b_2 b_0-1) (b_1 b_2 b_0+1)}{\left(b_1^2-1\right) \left(b_2^2-1\right) \left(b_0^2-1\right)} \\
\end{bmatrix}
$$
For instance, the $(1,2)$-entry is
\begin{align*}
\halfip{(+,+,+), (+,-,+)}_B = &- \frac{b_1}{b_1^2-1} \left(1 + \frac{1}{b_1^2b_2^2-1} + \frac{1}{b_1^2b_0^2-1}\right)  - \frac{b_1b_2^2}{(b_2^2-1)(b_1^2b_2^2-1)}\\
&- \frac{b_1b_0^2}{(b_0^2-1)(b_1^2b_0^2-1)}  - \frac{b_2b_0}{b_2^2b_0^2-1} \left( 1 +\frac{1}{b_2^2-1}+ \frac{1}{b_0^2-1}\right) ,
\end{align*}
the four terms corresponding to the four elements of 
$$
G^{\pm}((+,+,+),(+,-,+)) = \{(\hat 0 \subset \{1\} \subset \hat 1), (\hat 0 \subset \{2\} \subset \{1,2\} \subset \hat 1), (\hat 0  \subset \{0\} \subset \{0,1\} \subset \hat 1), (\hat 0  \subset \{0,2\} \subset \hat 1)\}.
$$
Note that the $4 \times 4$ matrix $\halfip{\cdot,\cdot}^{\T^+}_B$ has rank 4.  The corresponding matrix $\dRip{\cdot,\cdot}_{\T^+}$ has rank 1.
\end{example}

\subsection{Limit}
For $P,Q \in \T$, the intersection pairing $\dRip{\Omega_P,\Omega_Q}$ can be obtained from $\halfip{P,Q}_B$ by taking a limit.  In the following result, we view the intersection forms as rational functions in $\a$ and $\b$ respectively.

\begin{theorem}\label{thm:limit}
For $P,Q \in \T$, we have
$$
\dRip{\Omega_P,\Omega_Q} = \lim_{\alpha \to 0} \alpha^d \left.\halfip{P,Q}_B \right|_{b_e \to 1 + \alpha a_e/2}.
$$
\end{theorem}
\begin{proof} 
With $b_e = 1 + \alpha a_e/2$, we have $b_e^2 = 1 + \alpha a_e + O(\alpha^2)$, and $\tb_e = \alpha a_e + O(\alpha^2)$.  Let $G_\bullet \in \pFl(M)$ and $E_\bullet \in \bG_\bullet$.  Then
$$
 \lim_{\alpha \to 0} \alpha^d \left. b(G_\bullet) \frac{1}{\tb_{E_\bullet}} \right|_{b_e \to 1 + \alpha a_e/2} = \lim_{\alpha \to 0} \alpha^d \prod_{i=1}^{s(E_\bullet)} \frac{1}{\alpha a_{E_i}} = \lim_{\alpha \to 0} \alpha^{d-s} \frac{1}{a_{E_\bullet}}.
$$
So in the limit $\lim_{\alpha \to 0}$ only full flags $E_\bullet \in \bG_\bullet \cap \Fl(M)$ (with $d = s$) contribute.  The result follows from comparing \cref{thm:dRtope} with \cref{def:Bettipair}.
\end{proof}

\subsection{Non-degeneracy}
In this section, we consider specializations of the parameters $b_e$ to complex numbers.  We consider the following genericity assumption:
\begin{equation}\label{eq:bMon}
\tb_F = b_F^2-1 \neq 0 \mbox{ for all connected }F \in L(M) \setminus \{ \hat 0, \hat 1\}
\end{equation}
This assumption is implied by \eqref{eq:Mon} when $\a$ and $\b$ are related by \eqref{eq:ba}.

Recall from \cref{prop:numbertopes} that $|\T^+| = w_\Sigma(M)$ and $|\T^0| = \beta(M)$.
\begin{theorem}\label{thm:Bettinondeg}\
\begin{enumerate}
\item
Suppose that \eqref{eq:bMon} holds and $b^2_E \neq 1$.  Then $\halfip{\cdot,\cdot}_B$ is non-degenerate on $\Z^{\T^+}$.
\item
Suppose that \eqref{eq:bMon} holds and $b_E = 1$.  Then 
the restriction of $\halfip{\cdot,\cdot}_B$ to $\Z^{\T^0}$ is non-degenerate.
\end{enumerate}
\end{theorem}

\cref{thm:Bettinondeg}(1) follows from \cref{thm:Bettihomdet} below.  

\begin{proof}[Proof of \cref{thm:Bettinondeg}(2)]
It suffices to show that the $\T^0 \times \T^0$ matrix $\halfip{P,Q}_B$ is non-degenerate.  Applying \cref{thm:limit} to this matrix, we obtain the $\T^0 \times \T^0$ matrix in \cref{thm:Aomotodet}.  This matrix has non-vanishing determinant whenever \eqref{eq:sumto0} and \eqref{eq:Mon} are satisfied.  These two conditions follow from taking the limit of $b_E = 1$ and \eqref{eq:bMon} respectively.  It follows that the $\T^0 \times \T^0$ matrix $\halfip{P,Q}_B$ has a non-vanishing determinant.
\end{proof}

\subsection{Determinant}

\begin{theorem}\label{thm:Bettihomdet}
The determinant of $\halfip{\cdot,\cdot}_B$ on $\Z^{\T^+}$ is equal to 
$$
\det \halfip{\cdot,\cdot}_B^{\T^+} = (-1)^{(r-1)w_\Sigma(M)}\frac{(1-b_E)^{w_\Sigma(M)-\beta(M)}}{\prod_{F \in L(M) \setminus \{\hat0,\hat1\}} (1-b_F^2)^{\beta(F)w_\Sigma(M_F)}}.
$$
\end{theorem}

\cref{thm:Bettihomdet} will be proved in \cref{sec:Betticohomdet}.

\begin{conjecture}\label{conj:Bettidet}
The determinant of the $\T^0 \times \T^0$ matrix $\bip{\cdot,\cdot}_B$ is equal to 
$$
\frac{ \prod_{F \in L_0 \setminus \hat 1} (1-b_{E \setminus F}^2)^{\beta(F)\beta(M_F)}
}{\prod_{F \in L \setminus (L_0 \cup \hat 0)} (1-b_F^2)^{\beta(F) \beta(M_F)}}.
$$
\end{conjecture}

\begin{example}\label{ex:5pt1}
Let us consider the arrangement $\bA$ in \cref{ex:npoint} with $n = 4$, taking the points to be $z_1,z_2,z_3,z_4 \in \R$.  We have $\beta(M) = n-1= 3$ and $w_\Sigma(M) = n+1 = 5$.
The bounded topes $\T^0$ consists of the $3$ intervals
$$
P_1 = [z_1,z_2], \qquad P_2 = [z_2,z_3], \qquad P_{3}= [z_3,z_4].
$$
We write down the intersection forms.  
The intersection matrix $\halfip{P,Q}_B$ restricted to $(P,Q) \in \T^0 \times \T^0$ is
$$ \halfip{\cdot,\cdot}^{\T^0}_B = 
\begin{bmatrix}
\frac{1}{b^2_1-1} + \frac{1}{b^2_2-1} + 1 & -\frac{b_2}{b^2_2-1} & 0  \\
 -\frac{b_2}{b^2_2-1} & \frac{1}{b^2_2-1} + \frac{1}{b^2_3-1} + 1 & -\frac{b_3}{b^2_3-1}   \\
0 &  -\frac{b_3}{b^2_3-1}  & \frac{1}{b^2_3-1} + \frac{1}{b^2_4-1} +1\end{bmatrix}
$$
with determinant 
$$
\det \halfip{\cdot,\cdot}_B^{\T^0} = \frac{\tb_{1234}}{ \tb_1 \tb_2 \tb_3 \tb_4} = - \frac{1- (b_1b_2b_3b_4)^2}{(1-b_1^2)(1-b_2^2)(1-b_3^2)(1-b_4^2)}.
$$
If we consider $\T^+$, we have two additional topes $P_4 = [z_4,\infty]$ and $P_5 = [-\infty,z_1]$.  We have
$$ \halfip{\cdot,\cdot}^{\T^+}_B =
\begin{bmatrix}
\frac{1}{b^2_1-1} + \frac{1}{b^2_2-1} + 1 & -\frac{b_2}{b^2_2-1} & 0 & 0 & -\frac{b_1}{b^2_1-1}  \\
 -\frac{b_2}{b^2_2-1} & \frac{1}{b^2_2-1} + \frac{1}{b^2_3-1} + 1 & -\frac{b_3}{b^2_3-1} &0 &0 \\
0 &  -\frac{b_3}{b^2_3-1}  & \frac{1}{b^2_3-1} + \frac{1}{b^2_4-1} +1 & -\frac{b_4}{b^2_4+1} &0 \\
0 &0 & -\frac{b_4}{b^2_4-1} &\frac{1}{b^2_4-1} + \frac{1}{b^2_0-1} +1 &-\frac{b_0}{b^2_0-1}  \\
-\frac{b_1}{b^2_1-1} &0 &0 & -\frac{b_0}{b^2_0-1} &\frac{1}{b^2_0-1} + \frac{1}{b^2_1-1} +1   \\
\end{bmatrix}
$$
with determinant 
$$
\det \halfip{\cdot,\cdot}_B^{\T^+}= \frac{(b_E-1)^2}{\tb_0 \tb_1 \tb_2 \tb_3 \tb_4} = - \frac{ (1 - b_0b_1b_2b_3b_4)^2}{(1-b_0^2)(1-b_1^2)(1-b_2^2)(1-b_3^2)(1-b_4^2)}.
$$
When $b_E = b_0b_1b_2b_3b_4 = 1$, $\halfip{\cdot,\cdot}_B$ is degenerate on $\Z^{\T^+}$ with a two-dimensional kernel.  However, it restricts to a non-degenerate symmetric bilinear form on $\Z^{\T^0}$.
\end{example}

\section{Betti cohomology intersection form}
\def\id{{\rm id}}

In this section, we assume that $(\M,0)$ is an affine oriented matroid. 

\subsection{Definition of Betti cohomology intersection form}

Given $P,Q \in \T$, define the \emph{separating set}
$$
\sep(P,Q) := \{ e \in E \mid P(e) \neq Q(e)\} \subset E.
$$
If $P, Q \in \T^+$, then $\sep(P,Q) \subseteq E \setminus 0$.

\begin{definition}\label{def:Betticohpair}
The $S$-valued Betti cohomology intersection form on $\Z^{\T^+}$ is given by
$$
\ip{P,Q}^B := b_{\sep(P,Q) }+(-1)^r b_{E \setminus \sep(P,Q)} = \ip{Q,P}^B
$$
for $P,Q \in \T^+$.
\end{definition}

Note that $\ip{P,Q}^B$ can be extended to $P, Q \in \T$, with $\ip{P,Q}^B= (-1)^r \ip{P,-Q}^B$.  The sign $(-1)^r$ is parallel to the signs in \cref{def:Bettipair} and \cref{thm:EL}(1).  The main result of this section is the following.

\begin{theorem}\label{thm:Bettiinverse}
The $\T^+ \times \T^+$ matrices $(-1)^{r-1}(1- b_E )^{-1}\ip{\cdot,\cdot}^B_{\T^+}$ and $\ip{\cdot,\cdot}^{\T^+}_B$ are inverse.
\end{theorem}

\subsection{Varchenko's bilinear form}
Define \emph{Varchenko's bilinear form} \cite{Var} on $\Z^{\T^+}$ by 
$$
\ip{P,Q}^V := b_{\sep(P,Q)}.
$$
It is immediate from the definition that $\ip{P,Q}^V= \ip{P,Q}^B_0 := \ip{P,Q}^B|_{b_0 = 0}$.  Evaluating \cref{thm:Bettiinverse} at $b_0 = 0$ gives the following corollary.
\begin{corollary}\label{cor:Var}
The inverse of the $\T^+ \times \T^+$ matrix $\halfip{\cdot, \cdot}_{0,B}:=\halfip{\cdot,\cdot}_B|_{b_0 = 0}$ is equal to $(-1)^{r-1} \ip{\cdot,\cdot}^V$.
\end{corollary}
\cref{cor:Var} generalizes \cite[Theorem 5.2]{Var} which describes the possible denominators that can appear in the inverse of $\ip{\cdot,\cdot}^V$.

\begin{remark}
Our bilinear form $\ip{P,Q}^{B}$ can be obtained from Varchenko's bilinear form as follows.  Let $(\bM,\star)$ be a general lifting of $\M$ by a new element $\star$.  Then $\bT^+ = \T^+(\bM)$ is naturally in bijection with $\T$.  Let $\bP,\bQ \in  \bT^+$ lift the topes $P,Q \in \T$ .  Then we have the equality $\ip{\bP,\bQ}_{\bM}^B = \ip{P,Q}^V + \ip{P,-Q}^V$.  
\end{remark}

\subsection{Determinant}\label{sec:Betticohomdet}

\begin{theorem}\label{thm:Bettidet}
The bilinear form $\ip{\cdot,\cdot}^B$ on $\Z^{\T^+}$ has determinant 
$$
\det \ip{\cdot,\cdot}^B_{\T^+} = (1 - b_E )^{\beta(M)} \prod_{F \in L(M) \setminus \{\hat0,\hat1\}} (1-b_F^2)^{\beta(F)w_\Sigma(M_F)}.
$$
\end{theorem}

\begin{proof}[Proof of \cref{thm:Bettihomdet}]
By \cref{prop:numbertopes}, we have $|\T^+| = w_\Sigma(M)$.  Combine \cref{thm:Bettidet} and \cref{thm:Bettiinverse} to obtain
\begin{align*}
\det \halfip{\cdot,\cdot}_B^{\T^+} &= (-1)^{(r-1) w_\Sigma(M)} (1- b_E)^{w_\Sigma(M)}\frac{1}{\det \ip{\cdot,\cdot}^B_{\T^+}}. \qedhere
\end{align*}
\end{proof}

\begin{example}\label{ex:5pt2}
The inverse of the matrix $\halfip{\cdot,\cdot}_B^{\T^+}$ from \cref{ex:5pt1} is $1/(b_E- 1)$ times
$$
 \ip{\cdot,\cdot}^B_{\T^+} =\begin{bmatrix}
 b_0 b_1 b_2 b_3 b_4+1 & b_0 b_1 b_3 b_4+b_2 & b_0 b_1 b_4+b_2 b_3 & b_0 b_1+b_2 b_3 b_4 & b_0 b_2 b_3 b_4+b_1 \\
 b_0 b_1 b_3 b_4+b_2 & b_0 b_1 b_2 b_3 b_4+1 & b_0 b_1 b_2 b_4+b_3 & b_0 b_1 b_2+b_3 b_4 & b_0 b_3 b_4+b_1 b_2 \\
 b_0 b_1 b_4+b_2 b_3 & b_0 b_1 b_2 b_4+b_3 & b_0 b_1 b_2 b_3 b_4+1 & b_0 b_1 b_2 b_3+b_4 & b_0 b_4+b_1 b_2 b_3 \\
 b_0 b_1+b_2 b_3 b_4 & b_0 b_1 b_2+b_3 b_4 & b_0 b_1 b_2 b_3+b_4 & b_0 b_1 b_2 b_3 b_4+1 & b_0+b_1 b_2 b_3 b_4 \\
 b_0 b_2 b_3 b_4+b_1 & b_0 b_3 b_4+b_1 b_2 & b_0 b_4+b_1 b_2 b_3 & b_0+b_1 b_2 b_3 b_4 & b_0 b_1 b_2 b_3 b_4+1 
\end{bmatrix},
$$
which has determinant 
$$\det \ip{\cdot,\cdot}^B_{\T^+} = (1-b_0^2)(1-b_1^2)(1-b_2^2)(1-b_3^2)(1-b^2_4) (1 - b_0b_1b_2b_3b_4)^3.$$
Setting $b_0 = 0$, we obtain Varchenko's matrix
$$
\ip{\cdot,\cdot}^V_{\T^+} = \begin{bmatrix}
1 & b_2 & b_2 b_3 & b_2 b_3 b_4 & b_1 \\
 b_2 & 1 & b_3 & b_3 b_4 & b_1 b_2 \\
b_2 b_3 & b_3 & 1 & b_4 & b_1 b_2 b_3 \\
b_2 b_3 b_4 & b_3 b_4 & b_4 & 1 & b_1 b_2 b_3 b_4 \\
b_1 & b_1 b_2 &b_1 b_2 b_3 & b_1 b_2 b_3 b_4 & 1 
\end{bmatrix}
$$
which has determinant 
$$\det \ip{\cdot,\cdot}^V_{\T^+} = (1-b_1^2)(1-b_2^2)(1-b_3^2)(1-b_4)^2.$$
\end{example}

\begin{example}
We continue \cref{ex:3pttopeB}.  The inverse of the matrix $\halfip{\cdot,\cdot}^{\T^+}_B$ is equal to $1/(1-b_0b_1b_2)$ times the matrix
$$ 
\ip{\cdot,\cdot}^B_{\T^+} = 
\begin{bmatrix}
 1-b_0 b_1 b_2 & b_1-b_0 b_2 & b_1 b_2-b_0 & b_2-b_0 b_1 \\
 b_1-b_0 b_2 & 1-b_0 b_1 b_2 & b_2-b_0 b_1 & b_1 b_2-b_0 \\
 b_1 b_2-b_0 & b_2-b_0 b_1 & 1-b_0 b_1 b_2 & b_1-b_0 b_2 \\
 b_2-b_0 b_1 & b_1 b_2-b_0 & b_1-b_0 b_2 & 1-b_0 b_1 b_2 \\
\end{bmatrix}
$$
which has determinant $\det \ip{\cdot,\cdot}^B_{\T^+} = (1-b_0^2)^2 (1-b_1^2)^2  (1-b_2^2)^2$.  Note that in this case $\beta(M) = 0$.
\end{example}

\begin{example}
Consider the line arrangement $\bA$ in $\P^2$, pictured below.  The parameters $b_e, e \in E$ are taken to be $a,b,c,d$, where $d = b_0$ corresponds to the line at infinity.  
$$
\begin{tikzpicture}
\draw (0.1,-0.1) -- (2.4,2.9);
\draw (-0.5,0.8) -- (3.5,0.8);
\draw (2.9,-0.1) -- (0.6,2.9);
\draw (1.5,1.25) circle (1.7);

\node[color=blue] at (0,-0.1) {$b$};
\node[color=blue] at (-0.6,0.8) {$a$};
\node[color=blue] at (3,-0.15) {$c$};
\node[color=blue] at (1.5,-0.6) {$d$};

\node[color=red] at (0.2,0.55) {$1$};
\node[color=red] at (0.5,1.5) {$2$};
\node[color=red] at (1.5,2.2) {$3$};
\node[color=red] at (2.5,1.5) {$4$};
\node[color=red] at (2.8,0.55) {$5$};
\node[color=red] at (1.5,0.4) {$6$};
\node[color=red] at (1.5,1.15) {$7$};
\end{tikzpicture}
$$
The intersection matrices are
$$
\ip{\cdot,\cdot}^B_{\T^+} =\begin{bmatrix}
 1-a b c d & a-b c d & a b-c d & a b c-d & b c-a d & c-a b d & a c-b d \\
 a-b c d & 1-a b c d & b-a c d & b c-a d & a b c-d & a c-b d & c-a b d \\
 a b-c d & b-a c d & 1-a b c d & c-a b d & a c-b d & a b c-d & b c-a d \\
 a b c-d & b c-a d & c-a b d & 1-a b c d & a-b c d & a b-c d & b-a c d \\
 b c-a d & a b c-d & a c-b d & a-b c d & 1-a b c d & b-a c d & a b-c d \\
 c-a b d & a c-b d & a b c-d & a b-c d & b-a c d & 1-a b c d & a-b c d \\
 a c-b d & c-a b d & b c-a d & b-a c d & a b-c d & a-b c d & 1-a b c d \\
\end{bmatrix}
$$
with determinant $\det \ip{\cdot,\cdot}^B_{\T^+} = (1-a^2)^3 (1-b^2)^3 (1-c^2)^3 (1-d^2)^3 (1-abcd)$ and 
\def\acd{[acd]}
\def\bcd{[bcd]}
\def\abc{[abc]}
\def\abd{[abd]}
\def\cd{[cd]}
\def\ac{[ac]}
\def\ad{[ad]}
\def\ab{[ab]}
\def\bd{[bd]}
\def\bc{[bc]}
$$\halfip{\cdot,\cdot}^{\T^+}_B = 
\begin{bmatrix}
 \frac{\acd}{\tilde a \tilde c \tilde d} & - \frac{a \cd}{\tilde a \tilde c \tilde d} & -\frac{c d}{\tilde c \tilde d} & \frac{d \ac}{\tilde a \tilde c \tilde d} & -\frac{a d}{\tilde a \tilde d} & -\frac{c \ad}{\tilde a \tilde c \tilde d} & \frac{a c}{\tilde a \tilde c} \\
-\frac{a \cd}{\tilde a \tilde c \tilde d} & \frac{\ab \cd}{\tilde a \tilde b \tilde c \tilde d} &- \frac{b \cd}{\tilde b \tilde c \tilde d} & \frac{\tilde a b c \tilde d-a b^2 \tilde c d+a \tilde c d}{\tilde a \tilde b \tilde c \tilde d} & \frac{d \ab}{\tilde a \tilde b \tilde d} & \frac{-a^2 b \tilde c d+a \tilde b c \tilde d+b \tilde c d}{\tilde a \tilde b \tilde c \tilde d} & -\frac{c \ab}{\tilde a \tilde b \tilde c} \\
 -\frac{c d}{\tilde c \tilde d} & -\frac{b \cd}{\tilde b \tilde c \tilde d} & \frac{\bcd}{\tilde b \tilde c \tilde d} & -\frac{c\bd}{\tilde b \tilde c \tilde d} & -\frac{b d}{\tilde b \tilde d} & \frac{d \bc}{\tilde b \tilde c \tilde d} & \frac{b c}{\tilde b \tilde c} \\
 \frac{d \ac }{\tilde a \tilde c \tilde d} & \frac{\tilde a b c \tilde d-a b^2 \tilde c d+a \tilde c d}{\tilde a \tilde b \tilde c \tilde d} & -\frac{c\bd}{\tilde b \tilde c \tilde d} & \frac{\ac \bd}{\tilde a \tilde b \tilde c \tilde d} & -\frac{a\bd}{\tilde a \tilde b \tilde d} & \frac{-a^2 \tilde b c d+a b \tilde c \tilde d+\tilde b c d}{\tilde a \tilde b \tilde c \tilde d} & -\frac{b\ac}{\tilde a \tilde b \tilde c} \\
 -\frac{a d}{\tilde a \tilde d} & \frac{d \ab }{\tilde a \tilde b \tilde d} & -\frac{b d}{\tilde b \tilde d} & -\frac{a\bd}{\tilde a \tilde b \tilde d} & \frac{\abd}{\tilde a \tilde b \tilde d} & -\frac{b\ad}{\tilde a \tilde b \tilde d} & \frac{a b}{\tilde a \tilde b} \\
 -\frac{c\ad}{\tilde a \tilde c \tilde d} & \frac{-a^2 b \tilde c d+a \tilde b c \tilde d+b \tilde c d}{\tilde a \tilde b \tilde c \tilde d} & \frac{d \bc}{\tilde b \tilde c \tilde d} & \frac{-a^2 \tilde b c d+a b \tilde c \tilde d+\tilde b c d}{\tilde a \tilde b \tilde c \tilde d} & -\frac{b \ad}{\tilde a \tilde b \tilde d} & \frac{\ad\bc}{\tilde a \tilde b \tilde c \tilde d} & -\frac{a\bc}{\tilde a \tilde b \tilde c} \\
 \frac{a c}{\tilde a \tilde c} & -\frac{c\ab}{\tilde a \tilde b \tilde c} & \frac{b c}{\tilde b \tilde c} & -\frac{b\ac}{\tilde a \tilde b \tilde c} & \frac{a b}{\tilde a \tilde b} & -\frac{a\bc}{\tilde a \tilde b \tilde c} & \frac{\abc}{\tilde a \tilde b \tilde c} \\
\end{bmatrix}
$$
where for clarity we have used the notation $\tilde a := a^2-1$, $[acd] := a^2c^2d^2 -1$, and so on.  We have $\beta(M) = 1$ and $w_\Sigma(M) = 7$, and
$$\det \halfip{\cdot,\cdot}^{\T^+}_B = \frac{(1 - a b c d)^6}{(1 - a^2)^3 (1 - b^2)^3 (1 - c^2)^3 (1 - d^2)^3}.$$
\end{example}

\subsection{Proof of \cref{thm:Bettidet}}

\begin{lemma}\label{lem:Bettifactor}
The determinant $\det \ip{\cdot,\cdot}^B_{\T^+}$  is a constant times a rational function whose irreducible factors belong to 
$$
\{\tb_F \mid F \in L(M) \setminus \{\hat 0,\hat 1\} \} \cup \{(1 - b_E)\}.
$$ 
\end{lemma}
\begin{proof}
Follows from the form of $\halfip{\cdot,\cdot}_B$ and \cref{thm:Bettiinverse}.
\end{proof}

We deduce \cref{thm:Bettidet} from the following theorem of Varchenko \cite{Var}, generalized to the setting of oriented matroids in \cite{HV,Ran}.  Recall that $L_0 \subset L(M)$ denotes the subposet of flats $F$ that contain $0$.

\begin{theorem}[\cite{Var,HV,Ran}]\label{thm:BettiVar}
The $\T^+ \times \T^+$ matrix $\ip{\cdot,\cdot}^V_{\T^+}$ has determinant
$$
\det \ip{P,Q}^V_{\T^+} = \prod_{F \in L(M) \setminus \{L_0 \cup \hat 0\}} (-\tb_F)^{\beta(F)w_\Sigma(M_F)} =\prod_{F  \in L(M) \setminus  \{L_0 \cup \hat 0\}} (1-b_F^2)^{\beta(F)w_\Sigma(M_F)}.
$$
\end{theorem}

We apply \cref{thm:BettiVar} to the affine matroid $(\bM,\star)$, a general lifting of $\M$.  We denote the groundset of $\bM$ by $\bE = E \sqcup \star$.  We consider the topes $\bT^+$ of $\bM$ positive with respect to $\star$.  Then $\bT^+$ is in bijection with $\T = \T(\M)$ under the map $\bP \mapsto \bP|_E$.  For $P \in \T$, write $\bP \in \bT^+$ for the corresponding positive tope. 

We have
$$
\ip{\bP,\bQ}^V = \begin{cases} b_{\sep(P,Q)} &\mbox{ if $P,Q \in \T^+$,} \\
b_{\sep(-P,-Q)} & \mbox{ if $-P,-Q \in \T^+$,} \\
 b_{E \setminus \sep(P,-Q)} &\mbox{ if $P \in \T^+$ and $-Q \in \T^+$,} \\
b_{E \setminus \sep(-P,Q)} & \mbox{ if $-P \in \T^+$ and $Q \in \T^+$.} 
\end{cases}
$$
Here, $\ip{\cdot,\cdot}^V$ is calculated with respect to the affine matroid $(\bM,\star)$ while $\sep(\cdot,\cdot)$ is calculated with respect to $(\M,0)$.  Note that $b_\star$ does not appear in these formulae.

We extend the bilinear form $\ip{\cdot,\cdot}^V$ on $\Z^{\bT^+}$ to $\R^{\bT^+}$, and calculate the determinant with respect to the basis 
$$
\left\{ \frac{1}{\sqrt{2}}([\bP]+[-\bP]) \mid P \in \T^+ \right\} \cup \left\{ \frac{1}{\sqrt{2}}([\bP]-[-\bP]) \mid P \in \T^+ \right\}.
$$
We have that 
\begin{align*}
\ip{\frac{1}{\sqrt{2}}([\bP]+[-\bP]), \frac{1}{\sqrt{2}}([\bQ]+[-\bQ])}^V &= b_{\sep(P,Q)}+ b_{E \setminus \sep(P,Q)}  \\
\ip{\frac{1}{\sqrt{2}}([\bP]+[-\bP]), \frac{1}{\sqrt{2}}([\bQ]-[-\bQ])}^V&=0  \\
\ip{\frac{1}{\sqrt{2}}([\bP]-[-\bP]), \frac{1}{\sqrt{2}}([\bQ]-[-\bQ])}^V&=  b_{\sep(P,Q)}- b_{E \setminus \sep(P,Q)}.
\end{align*}
Thus the intersection matrix of $\ip{\cdot,\cdot}^V$ with respect to this basis is block diagonal.  One block is identical to $\ip{\cdot,\cdot}^B$ on the basis $\T^+$ and the other is identical to $ \ip{\cdot,\cdot}^B|_{b_0 \mapsto -b_0}$ on the same basis.  

\begin{lemma}\label{lem:wsigma}
Let $F \in L(\overline{M}) \setminus \hat 0$ be such that $\star \notin F$.  Then 
$$
w_\Sigma(\overline{M}_F) = \begin{cases} 1 & \mbox{if $F = E$} \\
2 w_\Sigma(M_F) & \mbox{if $F \subsetneq E$.}
\end{cases}
$$
\end{lemma}
\begin{proof}
When $F = E$, the matroid $\overline{M}_F$ is a rank 1 matroid on the single element $\star$.  The reduced characteristic polynomial is equal to $1$, so $w_\Sigma = 1$.  When $F \subsetneq E$, the matroid $\overline{M}_F$ is a general lifting of $M_F$.  By \cref{lem:genericlift}, we have $\chi_{\overline{M}_F}(t) = (t-1)\chi_{M_F}(t)$ and it follows that $w_\Sigma(\overline{M}_F) = 2w_{\Sigma}(M_F)$.
\end{proof}
Using \cref{lem:wsigma}, we conclude that 
\begin{align*}
\det \ip{\cdot,\cdot}^B_{\T^+}  (\det \ip{\cdot,\cdot}^B_{\T^+})  |_{b_0 \mapsto -b_0} &= \pm \prod_{F \in L(\overline{M}) \setminus \{L_\star \cup \hat 0\}} (-\tb_F)^{\beta(F)w_\Sigma(\overline{M}_F)} \\
&= \pm \prod_{F \in L(M) \setminus \{\hat 0, \hat 1\}}(1-b_F^2)^{2\beta(F)w_\Sigma(M_F)}(1-b_E^2)^{\beta(M)}.
\end{align*}
Comparing this with \cref{lem:Bettifactor}, we deduce that for $L \neq \hat 1$, the factor $(1-b_F^2)^{2\beta(F)w_\Sigma(M_F)}$ factors up to sign as $(1-b_F^2)^{\beta(F)w_\Sigma(M_F)} (1-b_F^2)^{\beta(F)w_\Sigma(M_F)}$ in $\det \ip{\cdot,\cdot}^B_{\T^+}  (\det \ip{\cdot,\cdot}^B_{\T^+})  |_{b_0 \mapsto -b_0} $.  However, the factor $(1-b_E^2)^{\beta(M)}$ factors as $(1-b_E)^{\beta(M)} (1+b_E)^{\beta(M)}$.  This proves \cref{thm:Bettidet} up to sign.  To fix the sign, we substitute $b_0 = 0$ into the intersection matrix of $\ip{\cdot,\cdot}^B$, obtaining an instance of $\ip{\cdot,\cdot}^V$ for $(\M,0)$.  Applying \cref{thm:BettiVar} we see that the sign gives the stated formula.

\subsection{Proof of \cref{thm:Bettiinverse}}

Let $P \in \T$ and $L(P) \subset L(M)$ denote the face lattice of the tope $\T$.  By convention, $L(P)$ has minimal element $\hat 0 = \emptyset$ and maximal element $\hat 1 = E$.  For a graded poset $L$, we let $\chi(L):= \sum_{x \in L} (-1)^{\rk(x)}$ be its Euler characteristic.

\begin{theorem}[see {\cite[Theorem 4.3.5]{OMbook}}] \label{thm:sphere}
The lattice $L(P)$ is the (augmented with a $\hat 0$ and $\hat 1$) face lattice of regular cell decomposition of a $(d-1)$-dimensional sphere.  In particular, $L(P)$ is a graded poset with Euler characteristic equal to 0.  The subposet $L(P) \setminus \hat 1$ has Euler characteristic $(-1)^{d}$.
\end{theorem}

We shall show that 
$$
\sum_{Q \in \T^+} \ip{P,Q}_B \ip{Q,R}^B = (-1)^{r-1}(1 - b_E) \delta_{P,R}.
$$

For $P, R \in \T^+$, let us define 
\begin{equation}\label{eq:U}
U = U(P,R):= \sum_{G_\bullet \in \pFl(P)}  \ip{G_\bullet}_B b_{\sep(R,Q_{G_\bullet})}, \qquad V = V(P,R):=(-1)^r \sum_{G_\bullet \in \pFl(P)}  \ip{G_\bullet}_B b_{\sep(R,(-Q)_{G_\bullet})}.
\end{equation}
Note that if $Q_{G_\bullet} \in \T^+$ then $\sep(R,Q_{G_\bullet}) \subseteq E \setminus 0$, but if $(-Q)_{G_\bullet} \in \T^+$, then $0 \in \sep(R,Q_{G_\bullet})$.  The following identity follows from the definitions.

\begin{lemma}
We have
$$
\sum_{Q \in \T^+} \ip{P,Q}_B \ip{Q,R}^B = U(P,R) + V(P,R).
$$
\end{lemma}
\begin{proof}
Follows from the definitions.  The sign $(\pm)^r$ in \cref{def:Bettipair} of $\ip{P,Q}_B$ cancels out with the sign $(-1)^r$ in \cref{def:Betticohpair}.
\end{proof}

\begin{example} We continue \cref{ex:5pt1} and \cref{ex:5pt2}.  Let $P = P_1 = R$.  Then
\begin{align*}
U(P,R) &= \left(1 + \frac{1}{b_1^2-1} + \frac{1}{b_2^2-1} \right) \cdot 1 + \left( - \frac{b_2}{b_2^2-1} \right) \cdot b_2 + \left( - \frac{b_1}{b_1^2-1} \right) \cdot b_1 = (-1)^1, \\
V(P,R) &=  \left(1 + \frac{1}{b_1^2-1} + \frac{1}{b_2^2-1} \right) \cdot b_0b_1b_2b_3b_4b_5 + \left( - \frac{b_2}{b_2^2-1} \right) \cdot b_0b_1b_3b_4b_5 + \left( - \frac{b_1}{b_1^2-1} \right) \cdot b_0b_2b_3b_4b_5 = b_E.
\end{align*}
\end{example}

We shall show that $U(P,R) = (-1)^{r-1} \delta_{P,R}$.  The equality $V(P,R) = (-1)^r b_E \delta_{P,R}$ follows similarly.
We note that 
$$
 \sep(P,G_\bullet) := \sep(P,Q_{G_\bullet}) = (G_s \setminus G_{s-1}) \sqcup (G_{s-2} \setminus G_{s-3}) \sqcup \cdots.
$$  

\subsubsection{The case $P=R$}
Let us first assume that $P = R$.  Let $G_\bullet = \{\hat 0 < G_1 < G_2 < \cdots < G_s < \hat 1\}$.  Then
\begin{equation}\label{eq:sepG}
b(G_\bullet) b_{\sep(P,Q_{G_\bullet})} = (-1)^{\sum_i \rk(G_i)} b_{G_s}^2 b_{G_{s-2}}^2 \cdots,
\end{equation}
where $\sum_i \rk(G_i) := \sum_{i=1}^{s(G_\bullet)} \rk(G_i)$.
Thus with $\beta(G_i) = 1/\tb_{G_i}$ and $\beta(G_\bullet) = \prod_{i=1}^s \beta(G_i)$, we have
\begin{align}\label{eq:bipG}
\begin{split}
 \ip{G_\bullet}_B b_{\sep(P,Q_{G_\bullet})} &= (-1)^{\sum_i \rk(G_i)} \prod_{i \equiv s} \frac{b_{G_i}^2}{b_{G_i}^2-1} \prod_{j \not \equiv s} \frac{1}{b_{G_j}^2-1} \sum_{E_\bullet \in \overline{G}_\bullet} \prod_{E_k \notin G_\bullet} \frac{1}{\tb_{E_k}} \\ 
 &=  (-1)^{\sum_i \rk(G_i)} \prod_{i \equiv s}(1+\beta(G_i)) \prod_{j \not \equiv s} \beta(G_j) \sum_{E_\bullet \in \overline{G}_\bullet} \prod_{ E_k \notin G_\bullet} \beta(E_k).
 \end{split}
\end{align}
Here, $i \equiv r$ means equal parity.  Expanding the factors $(1+\beta(G_i))$ in \eqref{eq:bipG} and summing over $G_\bullet \in \pFl(P)$, we may write $U=U(P,P)$ as a sum
$$
U= \sum_{E_\bullet} u_{E_\bullet} \beta(E_\bullet)
$$
for some integer coefficients $u_{E_\bullet}$.

\def\co{\kappa}
Let us compute the coefficients $u_{E_\bullet}$.  We say that $E_\bullet$ is \emph{compatible} with $G_\bullet$ if $\beta(E_\bullet)$ appears in the expansion \eqref{eq:bipG} for $G_\bullet$.  Write $\co(E_\bullet) \subseteq \pFl(P)$ for the set of $G_\bullet$ such that $E_\bullet$ is compatible with $G_\bullet$.  Then we have 
$$
u_{E_\bullet} = \sum_{G_\bullet \in \co(E_\bullet)}  (-1)^{\sum_i \rk(G_i)}.
$$

\begin{lemma}\label{lem:compat}
$E_\bullet$ is compatible with $G_\bullet$ if and only if 
\begin{enumerate}
\item Every pair $(E_i,G_j)$ is comparable, that is, $E_i \leq G_j$ or $G_j \leq E_i$.
\item If $G_i$ is missing from $E_\bullet$ then $i \equiv s(G_\bullet)$.
\end{enumerate}
\end{lemma}

Now, let us first consider the case $E_\bullet = \emptyflag$.  Then by \cref{lem:compat}, $E_\bullet$ is compatible with $G_\bullet$ if and only if $s(G_\bullet) \in \{0,1\}$.
By \cref{thm:sphere}, we have $\chi(L(P)\setminus\hat 1) = (-1)^{r-1}$, so we conclude that the coefficient of $u_{\emptyflag}$ is $(-1)^{r-1}$.

Now suppose that $E_\bullet \neq \emptyflag$ is compatible with $G_\bullet \in \co(E_\bullet)$.  We have that $E_1 \subseteq G_2$.  Consider a non-empty subset $\co' \subset \co(E_\bullet)$ where 
$$
\co' := \{G'_\bullet \in \co(E_\bullet) \mid G'_{s'-i} = G_{s-i} \text{ for all $i$ such that } E_1 \subsetneq G_{s-i}\}.
$$
In other words, the higher rank flats in $G'_\bullet$ are the same as those of $G_\bullet$.  Then $\co'$ is in bijection with the lower interval $[\hat0,E_1]$.  The Euler characteristic of this lower interval is equal to 0 by \cref{thm:sphere} applied to the restriction $\M^{E_1}$.  It follows that the contribution of $\co'$ to the coefficient of $\beta(E_\bullet)$ in $ \ip{G_\bullet}_B b_{\sep(P,Q_{G_\bullet})}$ is equal to $0$.  We conclude that the coefficient of $u_{E_\bullet}$ for $E_\bullet \neq \emptyflag$ is equal to $0$.  This completes the proof for the case $P =R$.

\subsubsection{The case $P\neq R$}
Now we move on to the case that $P \neq R$.  We will proceed by induction on the rank $r$ of $\M$ and number of elements $|E|$ of the ground set.

Let $S := \sep(P,R)$.  In the following, for the empty flag $G_\bullet = \emptyflag$, we denote $G_s = \emptyset$.
Divide $\pFl(P)$ into three subsets:
\begin{align*}
G^1(P)&:= \{G_\bullet \mid G_s \cap S = \emptyset\} \\
G^2(P)&:= \{G_\bullet \mid \emptyset \neq G_s \cap S \subsetneq G_s\} \\
G^3(P)&:= \{G_\bullet \mid \emptyset \neq G_s \subseteq S\}.
\end{align*}

For a fixed $F \in L(P)$, write $G^F(P) :=  \{G_\bullet \in \pFl(P) \mid G_s = F\}$.
For $G_\bullet \in \pFl(P)$, define 
$$\sep_-:= \sep(P,G_\bullet) \cap \sep(P,R), \qquad \text{and} \qquad \sep_+:= \sep(P,R) \setminus \sep(P,G_\bullet).$$

\begin{lemma}\label{lem:G2}
The contribution of $G_\bullet \in G^2(P)$ to the summation \eqref{eq:U} is 0.
\end{lemma}
\begin{proof}
Consider $G^F(P)$ for a fixed $F \in L(P)$ such that $\emptyset \neq F \cap S \subsetneq F$.  Then $G^F(P)$ is in natural bijection with $G(P^F)$ where $P^F \in \T(\M^F)$ is the tope $P^F=P|_{ F}$ in the restriction of $\M$ to $F$.  Furthermore, 
$\sep(R^F,(G_\bullet)^F) = \sep(R,G_\bullet)\cap F$, and by assumption that $F \not \subset S$, we have $R^F \neq P^F$.  By the inductive hypothesis applied to the two topes $P^F, R^F \in \T(\M^F)$ we deduce that the contribution sums to $0$.
\end{proof}

\begin{lemma}\label{lem:G3}
Suppose that $\emptyset \neq F \subseteq S$.  Then the contribution of $G^F(P) \subset G^3(P)$ to the summation \eqref{eq:U} is equal to 
$$
 \sum_{G_\bullet \in G^F(P)} \ip{G_\bullet}_B b_{\sep(R,Q_{G_\bullet})} = 
- b_S \beta(F)\sum_{E_\bullet \in \pFl(\M_F)} \beta(E_\bullet).
$$
\end{lemma}
\begin{proof}
Let $G_\bullet \in G^F(P)$.  We have $\sep(P,G_\bullet) \subset F \subset S$.  Thus $\sep(R,G_\bullet) = S \setminus \sep(P,G_\bullet)$.  We write
$$
\frac{\sep_+}{\sep_-} b_{G_s}^2 b_{G_{s-2}}^2 \cdots = \frac{b_S}{b_{\sep(P,G_\bullet)}^2} b_F^2 b_{G_{s-2}}^2 \cdots = b_Sb_{G_{s-1}}^2 b_{G_{s-3}}^2 \cdots,
$$
where in the last equality we used \eqref{eq:sepG}.
We may thus reduce the calculation of the contribution of $G^F(P)$ to a sum over $G(P^F)$.  We deduce that the contribution is equal to $(-1)^{\rk(F)} b_S U(P^F,P^F)  \beta(F)\sum_{E_\bullet \in \pFl(\M_F)} \beta(E_\bullet)$, where $U(P^F,P^F)$ is calculated in $\M^F$.  Since we showed the previous case that $U(P^F,P^F)= (-1)^{\rk(F)-1}$, our result holds.
\end{proof}

By \cref{lem:G3}, the contribution of $G^3(P)$ is
\begin{equation}\label{eq:G3}
 \sum_{G_\bullet \in G^3(P)} \ip{G_\bullet}_B b_{\sep(R,Q_{G_\bullet})} = - b_S \sum_{F \subseteq S} \sum_{E_\bullet \mid E_1 = F} \beta(E_\bullet) = -\sum_{E_\bullet \mid E_1 \subseteq S}  b_S \beta(E_\bullet).
\end{equation}

\begin{lemma}\label{lem:G1}
The coefficient of $\beta(E_\bullet)$ in the summation $\sum_{G_\bullet \in G^1(P)} \ip{G_\bullet}_B b_{\sep(R,Q_{G_\bullet})}$ vanishes if $E_1 \not \subset S$.
\end{lemma}
\begin{proof}
We have
$$
\sum_{G_\bullet \in G^1(P)} \ip{G_\bullet}_B b_{\sep(R,G_\bullet)} = b_S \sum_{G_\bullet \in G^1(P)} \ip{G_\bullet}_B b_{\sep(P,G_\bullet)}.
$$
Suppose that $E_1 \not \subset S$.  Let $D = E_1 \setminus S$.  We can fix the rest of $G_\bullet$ and let $G_1$ vary over $[\hat 0, D]$.  (The case $G_1 = \hat 0$ means that we do not include any $G_1$ in the interval $(\hat 0, D]$.).  As in the proof of \cref{lem:G2}, these contributions cancel out.
\end{proof}

Now, fix $E_\bullet$ satisfying $\emptyset \neq E_1 \subseteq S$ and consider the coefficient of $\beta(E_\bullet)$.  The contribution from $G^3(P)$ is given by \eqref{eq:G3}, and is simply $-b_S$.  For $\emptyflag \neq G_\bullet \in G^1(P)$, we have $G_1 \cap S = \emptyset$ so $G_1 \cap E_1 = \emptyset$, and thus $\ip{G_\bullet}_B b_{\sep(R,G_\bullet)}$ does not contribute to the coefficient of $\beta(E_\bullet)$.  However, for the term $\ip{\emptyflag}_B b_{\sep(R,\emptyflag)} = \ip{\emptyflag}_B b_S$ we get a coefficient of $b_S$.  The contributions cancel and the coefficient of $\beta(E_\bullet)$ vanishes.

Finally, let us consider the coefficient of $\beta(E_\bullet)$ where $E_\bullet = \emptyflag$.  There is no contribution from $G^3(P)$.  The contribution from $G^1(P)$ is equal to
$$
b_S \sum_{G_1 \in L(P) \mid G_1 \cap S = \emptyset} (-1)^{\rk(G_1)}.
$$
The subposet $X(S) = \{G_1 \in L(P) \mid G_1 \cap S = \emptyset\} \subset L(P)$ is obtained from the ball $L(P) \setminus \hat 1$ as follows.  Let $H \subset E$ denote the facets of $P$.  We remove from $L(P) \setminus \hat 1$ the upper order ideal generated by $H \cap S = H \cap \sep(P,R)$.  

We show that the Euler characteristic of $X(S)$ is $0$ by using standard results on shellings of $L(P)$.  We find a shelling order on the facets $H$ of $P$ by considering a shortest path from $P$ to the negative tope $-P$.  By picking this path to pass through $R$, we can arrange the facets belonging to $H \cap \sep(P,R)$ to come first.  Finally, it is known that the union of any proper initial subset of facets in this shelling order is a contractible subcomplex; see \cite[Proposition 4.3.1 and Lemma 4.7.28]{OMbook}.  Since $X(S)$ is the complement of a contractible complex in a contractible complex, it has Euler characteristic $0$.  We have shown that $U(P,R) = 0$.  This completes the inductive step, and the proof of the theorem.

\section{Bergman fan and Laplace transform}\label{sec:Bergman}
\subsection{Bergman fan}
Let $\R^E$ be the vector space with basis $\epsilon_e$, $e \in E$.  For $S \subset E$, let $\epsilon_S \in \R^E$ denote the vector $\epsilon_S:= \sum_{s \in S} \epsilon_s$.  Then $\epsilon_E = \one$, the all $1$-s vector.
For each partial flag $G_\bullet \in \pFl(M)$, let $C'_{G_\bullet}$ denote the simplicial $s+1$-dimensional cone 
$$
C'_{G_\bullet}:= \sp_{\R_{\geq 0}}(\epsilon_{G_1}, \epsilon_{G_2}, \ldots, \epsilon_{G_s}, \epsilon_{G_{s+1}}=\one),
$$
and let $C_{G_\bullet}$ denote the simplicial $s$-dimensional cone
$$
C_{G_\bullet}:= \sp_{\R_{\geq 0}}(\epsilon_{G_1}, \epsilon_{G_2}, \ldots, \epsilon_{G_{s}}).
$$
Each cone $C'_{G_\bullet}$ is the direct product of $C_{G_\bullet}$ with $\R_{\geq 0} \cdot \one$.  When $G_\bullet = F_\bullet \in \Fl(M)$ is a complete flag, the cone $C'_{G_\bullet}$ is $r$-dimensional, and the cone $C_{F_\bullet}$ is $d = (r-1)$-dimensional.

\begin{remark}
In the literature, $C_{G_\bullet}$ is usually defined as the image of $C'_{G_\bullet}$ in the quotient space $\R^E/\one$.  The result is that the Laplace transform (resp. discrete Laplace transform) of $C_{G_\bullet}$ below would be defined modulo $\sum_e a_e = 0$ (resp. $\prod_e b_e = 1$).  \cref{thm:deRhamfan} and \cref{thm:Bettifan} would then give the forms $\bdRip{\cdot,\cdot}$ and $\bip{\cdot,\cdot}_B$ (instead of $\dRip{\cdot,\cdot}$ and $\ip{\cdot,\cdot}_B$).
\end{remark}

\begin{definition}
The \emph{Bergman fan} $\tSigma_M$ (resp. $\Sigma_M$) is the $r$-dimensional (resp. $(r-1)$-dimensional) fan in $\R^E$ given by the union of the cones $C'_{G_\bullet}$ (resp. $C_{G_\bullet}$) for $G_\bullet \in \pFl(M)$.  We denote by $|\tSigma_M|$ (resp. $|\Sigma_M|$) the support of the Bergman fan.  
\end{definition}
The maximal cones of $\tSigma_M$ (resp. $\Sigma_M$) are exactly the cones $C'_{F_\bullet}$ (resp. $C_{F_\bullet}$) for $F_\bullet \in \Fl(M)$.

There are other, coarser, fan structures on $|\Sigma_M|$ which we discuss in \cref{sec:building}; see \cite{FS}.  A simplicial cone $C \subset \R^E$ is \emph{unimodular} if the primitive integer vector generators of the cone can be extended to a $\Z$-basis of $\Z^E$.  The following result is easy to see directly.

\begin{lemma}\label{lem:unimodular}
For any $G_\bullet \in \pFl(M)$, the simplicial cone $C'_{G_\bullet}$ (resp. $C_{G_\bullet}$) is unimodular.  \end{lemma}

\subsection{Laplace transform}
The \emph{Laplace transform} is the integral transform defined as follows.  Given a function $f(\x)$ on $\R^n$, the Laplace transform $\L(f) = \L(f)(\y)$ is the function on $\R^n$
$$
\L(f)(\y) = \int_{\R_{>0}^n} f(\x) \exp(-\x \cdot \y) d^n \x,
$$
which is defined on the domain $\Gamma \subset \R^n$ of convergence of the integral.

Recall that a cone $C$ is \emph{pointed} if it does not contain a line.  If $C = \sp_{\R_{\geq 0}}(\v_1,\ldots,\v_m)$ is a pointed cone in $\R^n$, we consider the analogous integral transform of the characteristic function of $C$:
\begin{equation}\label{eq:Lapint}
\L(C)(\y) = \int_C \exp(-\x \cdot \y) d^n \x.
\end{equation}
Define the dual cone 
$$
C^*:= \{\y \in \R^n \mid \x \cdot \y \geq 0 \text{ for all } \x \in \R^n\}.
$$
The integral \eqref{eq:Lapint} converges absolutely for $\y \in \Int(C^*)$ in the interior of the dual cone $C^*$  to a rational function.  We declare this rational function to be the Laplace transform $\L(C)$ of $C$, ignoring the domain of convergence of the integral.  If $C$ has dimension less than $n$, then we instead set $\L(C)(\y) = 0$.

\begin{lemma}\label{lem:Lsimplicial}
Suppose that $C = \sp_{\R_{\geq 0}}(\v_1,\ldots,\v_n)$ is a $n$-dimensional simplicial cone in $\R^n$.  Then the Laplace transform of $C$ is the rational function
$$
\L(C)(\y) =  \prod_{i=1}^n \frac{|\det(\v_1,\ldots,\v_n)|}{\v_i \cdot \y},
$$
and the integral \eqref{eq:Lapint} converges absolutely for $\y \in \Int(C^*)$.
\end{lemma}
\begin{proof}
By rescaling, we may assume that $\det(\v_1,\ldots,\v_n) = 1$.  Let $\{\u_1,\ldots,\u_n\}$ be the dual basis.  Then we may write 
$$
\x = z_1 \v_1 + z_2 \v_2 + \cdots + z_n \v_n, \qquad \y = w_1 \u_1 + \cdots + w_n \u_n
$$
and
$$
\int_C \exp(-\x \cdot \y) d^n \x = \int_{\R_{>0}^n} \exp(-(z_1w_1+\cdots +z_nw_n)) dz_1 \cdots dz_n = \prod_{i=1}^n \int_0^\infty \exp(-z_iw_i) dz_i = \prod_{i=1}^n \frac{1}{w_n},
$$
where for the last equality we have assumed that $w_i >0$, or equivalently, $\y \in \Int(C^*)$.
Finally, we note that $w_i = \v_i \cdot \y$. 
\end{proof}

\begin{lemma}
Let $C$ be a $n$-dimensional pointed cone in $\R^n$.  Then the integral \eqref{eq:Lapint} converges absolutely for $\y \in \Int(C^*)$ to a rational function.
\end{lemma}
\begin{proof}
Triangulate $C$ with simplicial cones $C_1,C_2,\ldots,C_k$.  Since $C_i \subset C$, it follows that $C^* = \bigcap C_i^*$, so by \cref{lem:Lsimplicial}, the integral converges absolutely when $\y \in \Int(C^*)$, and equals the rational function $\L(C) = \sum_{i=1}^k \L(C_i)$.
\end{proof}

Now, if $C$ is a cone that contains a line then we define $\L(C):=0$.
These definitions are consistent in the following sense.

\begin{proposition}[{\cite[p.341]{Bar}}]\label{prop:Lval}
Suppose that $C_1,\ldots,C_k \subset \R^n$ are polyhedral cones, and we have the identity
\begin{equation}\label{eq:indicator}
\sum_i \alpha_i [C_i] = 0
\end{equation}
of indicator functions $[C_i]$.  Then
$$
\sum_i \alpha_i \L(C_i) = 0.
$$
\end{proposition}
In other words, $\L$ is a valuation on the algebra of indicator functions of cones.    Note that $\L(C_i) = 0$ if $C_i$ is not full-dimensional, so in \eqref{eq:indicator}, we can replace indicator functions by measures of cones.  See also \cite{GLX} for a study of the closely related dual volume rational functions.

\subsection{Discrete Laplace transform}

The discrete Laplace transform is defined as follows.  Given a closed rational polyhedral cone $C$, we define
$$
\dL(C) := \sum_{\y \in C \cap \Z^n} \b^{2\y},
$$
where $\b^{2\y} = b_1^{2y_1} b_2^{2y_2} \cdots b_n^{2y_n}$.  Note the unusual factor of $2$, which is used to match geometric parameters that will be introduced later.  We similarly define $\dL(C^\circ)$ for a (relatively) open rational polyhedral cone $C^\circ$.  Then $\dL$ can be viewed as a linear operator on the indicator functions of the cones.  In particular,
$$
\dL(C^\circ) = \dL(C) + \sum_{F} (-1)^{\codim(F)} \dL(F)
$$
where the summation is over all proper faces of $C$, including the origin.  For example, for $C$ equal to a cone over a quadrilateral, we would have
$$
\dL(C^\circ) = \dL(C) - (\dL(F_1)+\dL(F_2)+\dL(F_3)+\dL(F_4)) +(\dL(F_{12})+\dL(F_{23})+\dL(F_{34})+\dL(F_{14})) - \dL(0),
$$
where $F_1,F_2,F_3,F_4$ are the four facets of $C$, and $F_{ij} = F_i \cap F_j$ are the codimension two faces.

\begin{lemma}\label{lem:discLap}
Suppose that $C$ is a unimodular simplicial cone generated by primitive integer vectors $\v_1,\v_2,\ldots,\v_n$.  Then
\begin{align*}
\dL(C) &= \prod_{i=1}^n \frac{1}{1 - \b^{2\v_i}}, \qquad
\dL(C^\circ) = \b^{2(\v_1+\v_2+\cdots+\v_n)} \prod_{i=1}^n \frac{1}{1 - \b^{2\v_i}}.
\end{align*}
We have the identity
$$
\dL(C^\circ) = \sum_{I \subseteq [n]} (-1)^{n-|I|}\prod_{i\in I}^n \frac{1}{1 - \b^{2\v_i}} = (-1)^n \sum_{I \subseteq [n]} \prod_{i\in I}^n \frac{1}{\b^{2\v_i}-1}.
$$
\end{lemma}
\begin{proof}
The set $C \cap \Z^n$ is equal to $\{\alpha_1 \v_1 + \cdots + \alpha_n \v_n \mid \alpha_i \in \Z_{\geq 0}\}$.  We have
$$
\dL(C) = \sum_{\alpha \in \Z_{\geq 0}^n} \b^{2 (\alpha_1 \v_1 + \cdots + \alpha_n \v_n)} = \prod_{i=1}^n (1+ \b^{2\v_i} + \b^{4\v_i} + \cdots) = \prod_{i=1}^n \frac{1}{1 - \b^{2\v_i}}.
$$
The map $\y \mapsto \y + \v_1+ \v_2 + \cdots + \v_n$ sends $C \cap \Z^n$ bijectively to $C^\circ \cap \Z^n$, giving the second equation.
\end{proof}

For an arbitrary rational polyhedral cone $C$, we may define $\dL(C)$ by writing $[C]$ in terms of indicator functions of unimodular simplicial cones.  Again, we have $\dL(C) = 0$ if $C$ contains a line, but unlike the case of $\L(C)$, the rational function $\dL(C)$ does not vanish when $C$ is not full-dimensional.
We have the following analogue of \cref{prop:Lval}; see \cite[Theorem 3.3]{Bar}.

\begin{proposition}\label{prop:dLval}
Suppose that $C_1,\ldots,C_k \subset \R^n$ are rational polyhedral cones, and we have the identity \eqref{eq:indicator} of indicator functions.  Then
$$
\sum_i \alpha_i \dL(C_i) = 0.
$$
\end{proposition}
In other words, $\dL$ is a valuation on the algebra of indicator functions of cones.  

We will only need the following result for unimodular simplicial cones, but we state it in the natural generality.
\begin{proposition}
The Laplace transform and discrete Laplace transform are related by
$$
\L(C) = (-1)^n \lim_{\alpha \to 0} \alpha^n \dL(C)|_{b_i \mapsto 1 + \alpha a_i/2},
$$
where $C$ is a rational polyhedral cone in $n$-dimensions.
\end{proposition}
\begin{proof}
Suppose first that $C$ has dimension less than $n$.  Then $\L(C)=0$ and it is not hard to see that the limit vanishes as well.

Next suppose that $C$ is a unimodular simplicial full-dimensional cone.  Then substituting $1- \b^{2\v}|_{b_i \mapsto 1 + \alpha a_i/2} = \alpha(\v \cdot \a) + O(\alpha^2)$ into \cref{lem:Lap} and \cref{lem:discLap}, we obtain the stated formula.  

Now let $C$ be an arbitrary full-dimensional rational polyhedral cone.  We may write (the indicator function of) $C$ as an alternating sum of unimodular simplicial cones; see \cite{BaPo}.  The stated formula then follows from the result for unimodular simplicial cones and \cref{prop:Lval,prop:dLval}.
\end{proof}

\subsection{Laplace transforms of the Bergman fan}

In the following, we shall utilize the $k$-dimensional Laplace transform of $k$-dimensional subfans of the Bergman fan $\tSigma_M$ or $\Sigma_M$.  The lattice $\Z^E \subset \R^E$ defines a measure in $\R^E$ by declaring that the unit cube of the lattice $\Z^E \subset \R^E$ has volume $1$.  A $k$-dimensional subspace $V \subset \R^E$ is \emph{unimodular} if the abelian group $V \cap \Z^E$ has a $\Z$-basis that can be extended to a $\Z$-basis of $\Z^E$.  We define a measure in $V$ by declaring that the unit cube in $V \cap \Z^E$ has volume $1$.  Let $C$ be a $k$-dimensional cone spanning a unimodular subspace $V \subset \R^E$.  We define the Laplace transform $\L(C) = \L^k(C)$ by using this measure.   It follows from \cref{lem:unimodular} that the subspace spanned by any cone $C \subset \Sigma_M$ (resp. $C' \subset \tSigma_M$) of the Bergman fan is a unimodular subspace.  

\begin{lemma}\label{lem:Lap}
Let $F_\bullet = \{F_0 \subset F_{1} \subset F_{2} \subset \cdots \subset F_{r-1} \subset F_r = E\}$ be a complete flag.
Let $C'_{F_\bullet}$ (resp. $C_{F_\bullet}$) be the corresponding cone of the Bergman fan.  Then the Laplace transform of $C'_{F_\bullet}$ (resp. $C_{F_\bullet}$) is equal to 
$$
\L(C') = \prod_{i=1}^r \frac{1}{a_{F_i}} = \frac{1}{a'_{F_\bullet}}, 
\qquad \L(C) = \prod_{i={1}}^{r-1} \frac{1}{a_{F_i}} = \frac{1}{a_{F_\bullet}}.
$$
\end{lemma}

\begin{lemma}\label{lem:dLap}
Let $E_\bullet = \{\emptyset = E_0 \subset E_{1} \subset E_{2} \subset \cdots \subset E_{s} \subset E_{s+1} = E\}$ be a partial flag.
Let $C'_{E_\bullet}$ (resp. $C_{E_\bullet}$) be the corresponding cone of the Bergman fan.  Then the discrete Laplace transform of $C'_{E_\bullet}$ (resp. $C_{E_\bullet}$) is equal to 
$$
\dL(C') = (-1)^{s+1} \prod_{i=1}^{s+1} \frac{1}{\tb_{E_i}}, 
\qquad \dL(C) =  (-1)^s \prod_{i=1}^s \frac{1}{\tb_{E_i}}= (-1)^s \frac{1}{\tb_{E_\bullet}}.
$$
\end{lemma}

\subsection{Local Bergman fan}\label{sec:localBF}
Henceforth, we only work with the $(r-1)$-dimensional Bergman fan $\Sigma_M$.  For a basis $B \in \B(M)$, the \emph{local Bergman fan} $\Sigma_M(B) \subset \Sigma_M$ is the subfan given by the union of cones $C_{F_\bullet}$ where $F_\bullet$ is generated by $B$.  The support $|\Sigma_M(B)|$ is equal to the intersection of $|\Sigma_M|$ with the normal cone $C(e_B)$ to the vertex $e_B$ in the normal fan of the \emph{matroid polytope} $P_M$ \cite[Section 4]{FS}.  For $B,B' \in \B(M)$, define 
$$
\Sigma_M(B,B') := \Sigma_M(B) \cap \Sigma_M(B') \subset \Sigma_M
$$ 
to be the intersection of the two subfans $\Sigma_M(B)$ and $\Sigma_M(B')$.

\begin{proposition}
The support $|\Sigma_M(B,B')|$ of the fan $\Sigma_M(B,B')$ is the intersection of $|\Sigma_M|$ with the normal cone $C(G)$ where $G$ is the smallest face of $P_M$ containing $e_B$ and $e_{B'}$. 
\end{proposition}
\begin{proof}
We have 
$$
|\Sigma_M(B,B')| = |\Sigma_M(B)| \cap |\Sigma_M(B')| = (|\Sigma_M| \cap C(e_B)) \cap (|\Sigma_M| \cap C(e_{B'})) = |\Sigma_M| \cap C(e_B) \cap C(e_{B'}).$$
The two cones $C(e_B)$ and $C(e_{B'})$ are (maximal) cones in the normal fan $\N(P_M)$ of the matroid polytope $P_M$ of $M$.  The intersection $C(e_B) \cap C(e_{B'})$ is again a cone.  It is equal to $C(G)$, where $G$ is the smallest face of $P_M$ containing both vertices $e_B$ and $e_{B'}$.
\end{proof}

The fan $\Sigma_M(B,B')$ is a subfan of the pure $(r-1)$-dimensional fan $\Sigma_M$.  The $(r-1)$-dimensional part of $\Sigma_M(B,B')$ is the union of the cones $C_{F_\bullet}$ where $F_\bullet \in \Fl(M)$ is generated by both $B$ and $B'$.  Note that $|\Sigma_M(B,B')|$ is always non-empty since it contains the origin.  However, it may have dimension less than $r-1$.  Recall the definition of the permutation $\sigma(B,F_\bullet)$ from \eqref{eq:sigma}.

\begin{proposition}[{\cite[Proposition 4.5]{BV}}] \label{prop:BV}
Suppose that $F_\bullet \in \Fl(M)$ is generated by both $B$ and $B'$.  Then the permutation $\sigma(B,F_\bullet)\sigma(B',F_\bullet)^{-1}$ depends only on $B,B'$.
\end{proposition}
We define 
$$
(-1)^{B,B'} := (-1)^{\sigma(B,F_\bullet)} (-1)^{\sigma(B',F_\bullet)} \in \{+1,-1\}.
$$
By \cref{prop:BV}, $(-1)^{B,B'}$ does not depend on the choice of $F_\bullet$, as long as it is generated by both $B$ and $B'$.  The following result gives a tropical interpretation of the symmetric bilinear form $\dRip{\cdot, \cdot}$.
\begin{theorem}\label{thm:localBF}
We have $\dRip{e_B,e_{B'}} = (-1)^{B,B'} \L(|\Sigma_M(B,B')|) = (-1)^{B,B'} \sum_{F_\bullet} \frac{1}{a_{F_\bullet}}$, where the summation is over flags $F_\bullet$ generated by both $B$ and $B'$. 
\end{theorem}
\begin{proof}
If no flags $F_\bullet$ are generated by both $B$ and $B'$, then both sides are $0$.  Suppose otherwise.
The Laplace transform of a union of $(r-1)$-dimensional cones, intersecting only in lower-dimensional cones, is equal to the sum of the Laplace transform of the corresponding $(r-1)$-dimensional cones.  We thus have $\L(|\Sigma_M(B,B')|) = \sum_{F_\bullet} \frac{1}{a_{F_\bullet}}$, summed over $F_\bullet \in \Fl(M)$ generated by both $B$ and $B'$.  Comparing with \cref{prop:dRind}, we see that the result follows from the equalites
$$
r(B,F_\bullet) r(B',F_\bullet) = (-1)^{\sigma(B,F_\bullet)} (-1)^{\sigma(B',F_\bullet)} = (-1)^{B,B'},
$$
for any $F_\bullet$ generated by both $B$ and $B'$.
\end{proof}

\subsection{Positive Bergman fan}\label{sec:posBF}

In \cref{sec:localBF}, we only considered the matroid $M$.  We now work with an oriented matroid $\M$ lifting $M$, and let $P \in \T(\M)$ be a tope.  Recall that we have defined the Las Vergnas face lattice $L(P) \subset L(M)$ of the tope $P$, and $\Fl(P)$ denotes the set of complete flags of lattices belonging to $L(P)$.  The \emph{positive Bergman fan} \cite{AKW} $\Sigma_M(P)$ of the tope $P$ is the subfan of $\Sigma_M$ obtained by taking the union of all cones $C_{F_\bullet}$ for $F_\bullet \in \Fl(P)$, together with all the faces of these cones.

Recall that for a tope $P$, we denote by $\pFl(P)$ the set of all partial flags of flats belonging to $L(P)$.

\begin{definition}
Let $G_\bullet \in \pFl(P)$.  Define 
$$
\Sigma_M(P,G_\bullet):= \bigcup_{F_\bullet \in \bG_\bullet \cap \Fl(M)} C_{F_\bullet},
$$
to be the union of maximal cones (together with all subcones) in the Bergman fan that are indexed by complete flags in the closure of $G_\bullet$.
\end{definition}

See \cref{fig:posBerg} for an example.  The cones in $\Sigma_M(P,G_\bullet)$ are exactly the cones of the Bergman fan that have the cone $C_{G_\bullet}$ as a face.  In other words, $\Sigma_M(P,G_\bullet)$ is the star of $C_{G_\bullet}$ inside $\Sigma_M(P)$.

Let $P, Q \in \T$.  Define $\Sigma_M(P,Q)$ to be the subfan of $\Sigma_M$ that consists of all maximal cones $C_{F_\bullet}$ that belong to both $\Sigma_M(P)$ and $\Sigma_M(Q)$, and all faces of these cones.

\begin{proposition}
The subfan $\Sigma_M(P,Q) \subset \Sigma_M$ is given by
$$
\Sigma_M(P,Q) = \bigsqcup_{G_\bullet \in G^{\pm}(P,Q)} \Sigma_M(P,G_\bullet) =  \bigsqcup_{G_\bullet \in G^{\pm}(P,Q)} \Sigma_M(Q,G_\bullet).
$$
\end{proposition}
\begin{proof}
Follows from \cref{prop:noover}(3),(4).
\end{proof}

\begin{theorem}\label{thm:deRhamfan}
Let $P,Q \in \T$ be topes.  Then
\begin{equation}\label{eq:deRhamfan}
\dRip{\Omega_P,\Omega_Q}= \sum_{G_\bullet \in G^{\pm}(P,Q)} (\pm)^r (-1)^{\sum_{i=1}^s \rk(G_i)} \L(\Sigma_M(P,G_\bullet)),
\end{equation}
where the sign $(\pm)^r$ is as in \cref{thm:dRtope}.  In particular, we have
$$\dRip{P,P} = \L(\Sigma_M(P)).$$
\end{theorem}
\begin{proof}
Apply \cref{lem:Lap} to the right hand side \eqref{eq:deRhamfan}, and compare with \cref{thm:dRtope}.
\end{proof}

\begin{theorem}\label{thm:Bettifan}
Let $P,Q \in \T$ be topes.  For any $G_\bullet \in \pFl(P)$, we have
\begin{equation}\label{eq:Bettifan1}
\halfip{G_\bullet}_B = (-1)^d b(G_\bullet) \dL(\Sigma_M(P,G_\bullet)),
\end{equation}
and thus
\begin{equation}\label{eq:Bettifan2}
\halfip{P,Q}_B= (-1)^d \sum_{G_\bullet \in G^{\pm}(P,Q)} (\pm)^r b(G_\bullet) \dL(\Sigma_M(P,G_\bullet)).
\end{equation}
In particular, we have
$$\halfip{P,P}_B = (-1)^d \dL(\Sigma_M(P)).$$

\end{theorem}
\begin{proof}
Since by definition $\halfip{P,Q}_B = \sum_{G_\bullet \in G^{\pm}(P,Q)} (\pm)^r \halfip{G_\bullet}$, we have that \eqref{eq:Bettifan1} implies \eqref{eq:Bettifan2}.
For \eqref{eq:Bettifan1}, we need to show that
$$
(-1)^d \sum_{\y \in \Sigma_M(P,G_\bullet) \cap \Z^{E}/\Z} \b^{2\y}= \sum_{E_\bullet \in \bG_\bullet} \frac{1}{\tb_{E_\bullet}}.
$$
By \cref{lem:dLap}, we have 
\begin{equation}\label{eq:yE}
\frac{1}{\tb_{E_\bullet}} = (-1)^s \dL(C_{E_\bullet}) = (-1)^s\sum_{ \y \in C_{E_\bullet} \cap \Z^E} \b^{2\y}.
\end{equation}
Let $\y$ be a point in the relatively open cone $C^\circ_{E_\bullet}$.  Then $\b^{2\y}$ appears in the summation of \eqref{eq:yE} for every cone $C_{E'_{\bullet}} \supseteq C_{E_\bullet}$, or equivalently, for every partial flag $E'_\bullet$ that refines $E_\bullet$.  Viewing $\pFl(P)$ as a poset, the coefficient of $\b^{2\y}$ in $\sum_{E_\bullet \in \bG_\bullet} \frac{1}{\tb_{E_\bullet}}$ is up to sign equal to the Euler characteristic of the upper order ideal $\{E'_\bullet \geq E_\bullet\} \subset \pFl(P)$.  By \cref{thm:sphere}, $L(P)$ is the face lattice of a regular cell decomposition of a $d$-dimensional ball.  It follows that the upper order ideal of $E_\bullet$ in $\pFl(P) \cup \hat 1$ is also the face lattice of a regular cell decomposition of a ball.  After removing $\hat 1$, we see that the Euler characteristic is equal to $\pm 1$, depending on dimension.
We deduce that the coefficient of $\b^{2\y}$ in $\sum_{E_\bullet \in \bG_\bullet} \frac{1}{\tb_{E_\bullet}}$ is equal to $(-1)^d$.
\end{proof}

\subsection{Building sets and nested triangulation}\label{sec:building}
\def\bT{{\bar T}}
\def\dec{{\rm dec}}
We assume that $M$ is simple in this section.  A subset $\build \subseteq L(M) \setminus \hat0$ is a \emph{building set} if $\hat 1 \in \build$ and for any $F \in L(M)-\hat 0$ and $\max \build_{\leq F}= \{F_1,\ldots,F_k\}$, there is an isomorphism of posets
\begin{equation}\label{eq:facF}
\prod_{i=1}^k [\0, F_i] \longrightarrow [ \0, F], \qquad (x_1,\ldots,x_k) \mapsto x_1 \vee \cdots \vee x_k,
\end{equation}
obtained by taking joins.  This isomorphism holds if and only if we have a disjoint union of flats $F = \bigsqcup_i F_i$.  For $F \in L - \hat 0$, we define $\dec(F) = \{F_1,F_2,\ldots, F_k\}$ via \eqref{eq:facF}.

There is a maximal choice of building set, which is $ \build_{\max} = L \setminus \hat 0$.  There is also the minimal building set 
\begin{equation}\label{eq:Bmin}
\build_{\min} := \{F \in L \setminus \{\hat 0,\hat 1\} \mid F \mbox{ is connected}\} \cup \{\hat 1\}.
\end{equation}

A subset $T_\bullet$ of a building set $\build \setminus \{\hat 1\}$ is called \emph{nested} if for any $k \geq 2$ incomparable elements $T_1,\ldots,T_k$ of $T_\bullet$, the join $T_1 \vee T_2 \vee \cdots \vee T_k$ does not belong to $\build$.  The $\build$-nested collections form a simplicial complex $N(\build)$, called the \emph{nested set complex}.  The subset of maximal (under inclusion) nested sets are denoted $N_{\max}(\build)$.  Every maximal nested collection $S_\bullet \in N_{\max}(\build)$ contains exactly $d=r-1$ elements.  For a nested collection $T_\bullet$, let $|T_\bullet| = \bigcup_i T_i$ be the union of all elements in $T_\bullet$.  If $T_\bullet$ has a unique maximum then the maximum is equal to $|T_\bullet|$, otherwise $|T_\bullet|$ is a flat that does not belong to $\build$.  Define the closure $\bT_\bullet \subset N(\build)$ to be the set of all nested collections that contain $T_\bullet$.  In the case that $\build = \build_{\max}$, a subset $T_\bullet \subset \build_{\max}$ is nested if and only if it is a flag of flats.  We have $N(\build_{\max}) = \pFl(M)$. 

Now let $\build$ be an arbitrary building set.  For a nested set $T_\bullet = \{T_1,\ldots, T_s\} \in N(\build)$, define the cone 
$$
C_{T_\bullet} := \sp(\epsilon_{T_1},\ldots,\epsilon_{T_s}).
$$
The cones $C_{S_\bullet}$ where $S_\bullet \in N_{\max}(\build)$ form a triangulation of $\Sigma_M$; see \cite{FS}.

\begin{lemma}\label{lem:SF}
Let $S_\bullet \in N_{\max}(\build)$ and $F_\bullet \in \Fl(M)$.  Then the cone $C_{S_\bullet}$ contains $C_{F_\bullet}$ if and only if every $F_i, i=1,2,\ldots,d$ is a disjoint union of some of the $S_j \in S_\bullet$.
\end{lemma}
\begin{proof}
The ``if" part is clear.  For the ``only if" part, suppose that $C_{F_\bullet} \subseteq C_{S_\bullet}$.  Since $C_{S_\bullet}$ is simplicial, each $\epsilon_{F_i}$ has a unique expression as a linear combination of the generators of $C_{S_\bullet}$, and this expression must be a nonnegative linear combination.  It follows from this that $\epsilon_{F_i}$ is the sum of $\epsilon_{S_j}$ over all $S_j \in S_\bullet$ such that $S_j \subseteq F_i$ and $S_j$ is maximal under inclusion with this property.  In particular, we deduce that $F_i$ is the disjoint union of these $S_j$.
\end{proof}

The cones $F_\bullet$ described by \cref{lem:SF} are in bijection with total orderings of $S_\bullet$ (compatible with the usual ordering by inclusion).

We now repeat the constructions from \cref{sec:pFl} in the nested setting.  Let $P \in \T$.  Define $N(P,\build) \subset N(\build)$ to be the nested collections where every set belongs to $L(P) \cap \build$.
For $T_\bullet \in N(P,\build)$, define $P_{\flip T_\bullet} \in \T$ by (see \cref{prop:flipflag})
$$
P_{\flip T_\bullet}(e) = (-1)^{\#\{i \mid e \in T_i\}}P(e).
$$
For $P,Q \in \T$, we define $T(P,Q) \subset N(P,\build)$ by 
$$
T(P,Q) := \{T_\bullet \in N(P,\build) \mid P_{\flip T_\bullet} \in \{Q\} \} = \{T_\bullet \in N(Q,\build) \mid Q_{\flip T_\bullet} \in \{P\}\}, 
$$
and $T^{\pm}(P,Q) := T(P,Q) \cup T(P,-Q)$.  For $T_\bullet \in N(\build)$, define
$$b(T_\bullet):= (-1)^{\sum_i \rk(T_i)} \prod_{T_i \in T_\bullet} b_{T_i}.$$
Let $G_\bullet \in \pFl(M)$.  The \emph{decomposition} $\dec(G_\bullet)$ of $G_\bullet$ (with respect to $\build$) is
$$
\dec(G_\bullet) := \{F \in \build \mid F \text{ appears in } \dec(G_1),\dec(G_2),\ldots,\dec(G_s) \text{ an odd number of times}\}.
$$

\begin{lemma}\label{lem:GT} Let $G_\bullet \in \pFl(M)$ and $T_\bullet = \dec(G_\bullet)$.
\begin{enumerate} 
\item
We have $\dec(G_\bullet) \in N(\build)$.
\item
If $G_\bullet \in \pFl(P)$, then $T_\bullet \in N(P,\build)$.
\item 
For $G_\bullet \in \pFl(P)$, we have $P_{\flip G_\bullet} = P_{\flip T_\bullet}$.
\item
We have  $(-1)^{\sum_i \rk(T_i)} = (-1)^{\sum_i \rk(G_i)}$ and $b(G_\bullet) = b(T_\bullet) \b^{2\y}$ for some $\y \in \Z^E_{\geq 0} \setminus {\bf 0}$.
\end{enumerate}
\end{lemma}
\begin{proof}
For (1), suppose that $F \subseteq F'$.  Then the factorization \eqref{eq:facF} for $[\hat 0, F']$ induces one for $[\hat 0, F]$.  Thus for each $E_i \in \dec(F)$ we have $E_i \subseteq E'_j$ for some $E'_j \in \dec(F')$, and $E \cap E'_{\ell} = \emptyset$ for $\ell \neq j$.  Thus the union $\bigcup_i \dec(G_i)$ is a nested collection, and (1) follows.

For (2), suppose that $F \in L(P)$.  Then the decomposition \eqref{eq:facF} induces a decomposition $L(P^F) = L(P^{E_1}) \times \cdots \times L(P^{E_k})$.  In particular, $E_i \in L(P)$, from which (2) follows.

(3) and (4) follow from the definitions.
\end{proof}
Note, however, that the decomposition map $G_\bullet \mapsto \dec(G_\bullet)$ is not injective.

\begin{definition}
Let $T_\bullet \in N(P,\build)$.  Define 
$$
\Sigma_M(P,T_\bullet):= \bigcup_{S_\bullet \in \bT_\bullet \cap N_{\max}(P,\build)} C_{S_\bullet},
$$
to be the union of maximal cones (together with all subcones) in the nested Bergman fan that are indexed by nested collections in the closure of $T_\bullet$.
\end{definition}

\begin{lemma}\label{lem:ETcompat}
Let $E_\bullet \in \pFl(P)$ and $T_\bullet \in N(P,\build)$.  Suppose that $C_{E_\bullet} \subset \Sigma_M(P,T_\bullet)$.  Then for all $E_i \in E_\bullet$ and a collection $\{T_{j_1}, \ldots, T_{j_k}\} \subset T_\bullet$ that are pairwise disjoint and not comparable with $E_i$, we have $E_i \cap \bigsqcup_{a} T_{j_a} = \emptyset$ and $E_i \vee \bigsqcup_{a} T_{j_a}  \notin \build$.  In particular, $E_i \vee \bigsqcup_{a} T_{j_a}  = E_i \cup \bigsqcup_{a} T_{j_a} $.
\end{lemma}
\begin{proof}
Suppose that $C_{E_\bullet} \subset C_{S_\bullet}$ where $S_\bullet \in \bT_\bullet \cap N_{\max}(P,\build)$.  Then $C_{E_\bullet}$ is a face of some $C_{F_\bullet} \subset C_{S_\bullet}$.  Every $F_i \in F_\bullet$ is a disjoint union of some subset of $S_\bullet$.  Thus any $F_i \in F_\bullet$ and $S_j \in S_\bullet$ are nested.  In particular, any $E_i \in E_\bullet$ and $T_j \in T_\bullet$ are nested.  This implies the stated result. 
\end{proof}

\begin{lemma}\label{lem:Gdivide}
For each $E_\bullet \in \pFl(P)$ such that $C_{E_\bullet} \subset \Sigma_M(P,T_\bullet)$, there is a distinguished $G_\bullet$ satisfying $\dec(G_\bullet) = T_\bullet$ such that $C_{E_\bullet} \subset \Sigma_M(P, G_\bullet)$.  This $G_\bullet$ satisfies $b(G_\bullet)/b(T_\bullet) = \b^{2\y}$ where $\y$ is a multiplicity-free sum of the cone generators $\epsilon_{E_i}$ and is minimal in the following sense.  For any other $G'_\bullet$ satisfying $\dec(G'_\bullet) = T_\bullet$ such that $C_{E_\bullet} \subset \Sigma_M(P, G'_\bullet)$, we have $b(G'_\bullet) = b(G_\bullet) \b^{2\y}$, where $\y \in \Z_{\geq 0}^E/{\bf 0}$ does not lie in the cone $C_{E_\bullet}$.
\end{lemma}
\begin{proof}
We have $C_{E_\bullet} \subset \Sigma_M(P, G_\bullet)$ if and only if any pair $(E_i \in E_\bullet, G_j \in G_\bullet)$ is comparable.
We construct the minimal $G_\bullet$ algorithmically.  Let $E_s$ be the largest proper flat in $E_\bullet$.  
Let $s' = s(G_\bullet)$ so that $G_{s'}$ is the largest proper flat in $G_\bullet$.
\medskip

\noindent \emph{Case 1.} Suppose that $|T_\bullet| \not \subseteq  E_s$ and $E_s \cap T_\bullet = \emptyset$.  Then for \cref{lem:GT}(4) to hold, $G_\bullet$ must contain a flat that contains $E_s \vee |T_\bullet|$.  The minimal choice is $G_{s'}= E_s \vee |T_\bullet| = E_s \cup |T_\bullet|$, using \cref{lem:ETcompat}. We also pick $G_{s' -1} = E_s$.  Alter $T_\bullet$ by removing the maximal (under inclusion) elements of $T_\bullet$, and we remove $E_s$ from $E_\bullet$.  

\noindent \emph{Case 2.} Suppose that $|T_\bullet| \not \subseteq  E_s$ and $E_s \cap T_\bullet \neq \emptyset$.  Then as in Case 1, we set $G_{s'}= E_s \vee |T_\bullet| = E_s \cup |T_\bullet|$.  Alter $T_\bullet$ by removing the maximal (under inclusions) elements of $T_\bullet$ and adding $\dec(E_s \setminus |T_\bullet|)$ to the resulting nested collection, and we remove $E_s$ from $E_\bullet$.

\noindent \emph{Case 3.} Suppose that $|T_\bullet| = E_s$.  Then we take $G_{s'} = E_s$.  Alter $T_\bullet$ by removing the maximal elements, and we remove $E_s$ from $E_\bullet$.

\noindent \emph{Case 4.} Suppose that $|T_\bullet| \subsetneq E_s$.  Then we remove $E_s$ from $E_\bullet$ but leave $T_\bullet$ unchanged.

\medskip

In all four cases, we continue by running the algorithm for the new $T_\bullet$ and $E_\bullet$, working inside the matroid $M^{E_s}$.  This produces the desired $G_\bullet$.  One step of the algorithm produces all the flats in $G_\bullet$ that contain $E_s$ (and then $E_s$ is removed).  Each step contributes one or no factors of $b^2_{E_s}$ to the ratio $b(G_\bullet)/b(T_\bullet)$.  This proves the statement about $b(G_\bullet)/b(T_\bullet)$.

We observe that any other choice of $G'_\bullet$ for a particular step would involve strictly more subspaces that contain $E_s$, and involve additional elements that are not contained in $E_s$.  The claimed property $b(G'_\bullet) = b(G_\bullet) \b^{2\y}$ follows. 
\end{proof}

\begin{theorem}\label{thm:deRhamfannested}
Let $P,Q \in \T$ be topes.  Then
$$
\dRip{\Omega_P,\Omega_Q}= \sum_{T_\bullet \in T^{\pm}(P,Q)} (\pm)^r (-1)^{\sum_{i=1}^s \rk(T_i)} \L(\Sigma_M(P,T_\bullet)),
$$
where the sign $(\pm)^r$ is equal to $1$ or $(-1)^r$ depending on whether $T_\bullet$ belongs to $T(P,Q)$ or $T(P,-Q)$.
\end{theorem}
\begin{proof}
By \cref{prop:noover} and \cref{lem:Gdivide} applied to maximal cones, we have
$$
\bigcup_{G_\bullet \mid \dec(G_\bullet) = T_\bullet} \Sigma_M(P,G_\bullet) = \Sigma_M(P,T_\bullet)
$$
and the overlaps of the union are of dimension lower than $r-1$.  It follows that $$\L(\Sigma_M(P,T_\bullet)) = \sum_{G_\bullet \mid \dec(G_\bullet)=T_\bullet} \L(\Sigma_M(P,G_\bullet)).$$  Finally, we use \cref{lem:GT}(4) and substitute into \cref{thm:deRhamfan}.
\end{proof}

By definition, $\dRip{\cdot,\cdot}$ takes values in $R[a_F^{-1} \mid F \in L(M)]$.  From \cref{thm:deRhamfannested}, we have the following improvement.
\begin{corollary}\label{cor:denom}
The bilinear form $\dRip{\cdot,\cdot}$ on $\OS(M)$ takes values in $R[a^{-1}_F \mid F \in L \setminus \{\hat 0, \hat 1\} \text{ is connected }]$.
\end{corollary}

\begin{theorem}\label{thm:Bettifannested}
Let $P,Q \in \T$ be topes.  Then
\begin{equation}\label{eq:Tlattice}
\halfip{P,Q}_B= (-1)^d \sum_{T_\bullet \in T^{\pm}(P,Q)}(\pm)^r  b(T_\bullet) \dL(\Sigma_M(P,T_\bullet)).
\end{equation}
\end{theorem}
\begin{proof}
We claim that \begin{equation}\label{eq:latticepoint} \dL(\Sigma_M(P,T_\bullet)) = \sum_{G_\bullet \mid \dec(G_\bullet)=T_\bullet} \frac{b(G_\bullet)}{b(T_\bullet)} \dL(\Sigma_M(P,G_\bullet)).
\end{equation}
We view $ \frac{b(G_\bullet)}{b(T_\bullet)} \dL(\Sigma_M(P,G_\bullet))$ as the generating function of the lattice points in $|\Sigma_M(P,G_\bullet)|$ translated by the vector $\y$, where $ \frac{b(G_\bullet)}{b(T_\bullet)} = \b^{2\y}$; see \cref{lem:GT}(4).  Denote the translation $|\Sigma_M(P,G_\bullet)| + \y$ by $|\Sigma_M(P,G_\bullet)|'$.  Note that the vector $\y$ lies in $C_{G_\bullet}$, so $|\Sigma_M(P,G_\bullet)|' \subset |\Sigma_M(P,G_\bullet)|$.  Consider $E_\bullet \in \pFl(P)$ such that $C_{E_\bullet} \subset \Sigma_M(P,T_\bullet)$.  The translation vector for the minimal $G_\bullet$ from \cref{lem:Gdivide} sends the cone $C_{E_\bullet}$ to itself, surjective on lattice points in the interior.  For any other $G'_\bullet$ from \cref{lem:Gdivide}, the translation vector maps the cone $C_{E_\bullet}$ out of itself.  It follows that the lattice points in the interior of $C_{E_\bullet}$ appear in $|\Sigma_M(P,G_\bullet)|'$ and not in any other $|\Sigma_M(P,G'_\bullet)|'$.  We have proved the equality \eqref{eq:latticepoint}.  Substituting \eqref{eq:latticepoint} into \eqref{eq:Tlattice}, we obtain \cref{thm:Bettifan}, proving the theorem.
\end{proof}

By definition, $\halfip{\cdot,\cdot}_B$ takes values in $S[\tb_F^{-1} \mid F \in L(M)]$, where $S = \Z[\b]$.  From \cref{thm:Bettifannested}, we have the following improvement.
\begin{corollary}%\label{cor:denom}
The bilinear form $\halfip{\cdot,\cdot}_B$ on $\OS(M)$ takes values in $S[\tb^{-1}_F \mid F \in L \setminus \{\hat 0, \hat 1\} \text{ is connected}\;]$.
\end{corollary}

\begin{example}
Let $M$ be the boolean matroid on $E = \{1,2,3\}$ and $P$ be the positive tope.  Let $\build = \{ \{1\},\{2\},\{3\}\}$.  Let $T_\bullet = \{\{1\}\}$.  Let
$$
G_\bullet = \{ \hat 0 \subset \{1\} \subset \hat 1\}, \qquad G'_\bullet = \{\hat 0 \subset \{2\} \subset \{1,2\} \subset \hat 1\}.
$$
Then we have $\dec(G_\bullet) = \dec(G'_\bullet) = T_\bullet$, and $b(G_\bullet) = b(T_\bullet)$, and $b(G'_\bullet) = b_2^2 b(T_\bullet)$.  Now consider 
\begin{align*}
E^{(1)}_\bullet &=  \{ \hat 0 \subset \{1\} \subset \hat 1\}, \qquad E^{(2)}_\bullet =  \{\hat 0 \subset \{1\} \subset \{1,2\} \subset \hat 1\}, \qquad E^{(3)}_\bullet =  \{\hat 0 \subset \{2\} \subset \{1,2\} \subset \hat 1\}, \\\qquad E^{(4)}_\bullet &= \{\hat 0 \subset \{2\}\subset \hat 1\}, \qquad E^{(5)}_\bullet =  \{\hat 0 \subset \{1,2\} \subset \hat 1\}, \qquad E^{(6)}_\bullet =  \{\hat 0 \subset \hat 1\}.
\end{align*}
We have $C_{E^{(i)}_\bullet} \subset C_{T_\bullet}$ for each $i = 1,2,3,4,5$.
Let $G^{(i)}_\bullet \in \pFl(P)$ be the partial flag of \cref{lem:Gdivide} applied to $E^{(i)}_\bullet$.  We have
$$
G^{(1)}_\bullet = G^{(2)}_\bullet = G^{(5)}_\bullet = G^{(6)}_\bullet = G_\bullet, \qquad G^{(3)}_\bullet = G^{(4)}_\bullet = G'_\bullet
$$
and we may verify that $\epsilon_2 \notin C_{E^{(i)}_\bullet}$ for $i=1,2,5,6$.  The proof of \cref{thm:Bettifannested}, restricted to lattice points in $\sp(\epsilon_1,\epsilon_2)$, is the statement that 
$$
\{(x,y) \in \Z_{\geq 0}^2\} = (\sp(\epsilon_1,\epsilon_1+\epsilon_2) \cap \Z_{\geq 0}^2) \sqcup (\sp(\epsilon_2,\epsilon_1+\epsilon_2) \cap \Z_{\geq 0}^2 + \epsilon_2),
$$
where the first term are lattice points in $|\Sigma_M(P,G_\bullet)|$, while the second term are the lattice points in the translation $|\Sigma_M(P,G'_\bullet)| + \epsilon_2$, in both cases restricting to lattice points in $\Z_{\geq 0}^2$.

\end{example}

\subsection{Nested deRham cohomology intersection form}
We give one more description of $\dRip{\cdot,\cdot}$ using nested combinatorics.
Let $\build$ be a building set and $N(\build)$ denote the nested set complex and $N_{\max}(\build)$ denote the maximal nested collections.  Let $S_\bullet = (S_1,\ldots,S_{r-1}) \in N_{\max}(\build)$ be a maximal nested set.  For a basis $B \in \B(M)$, we say that $S_\bullet$ is generated by $B$ if $S_i \in L(B)$ for $i =1,2,\ldots, r-1$.  Pick an ordering of $S_\bullet$ and $B$.  Then there exists a permutation $\sigma$ such that
$$
b_{\sigma(i)} \in S_i \setminus \bigcup_{j \mid S_j \subsetneq S_i} S_j, \qquad \mbox{for $i =1,2,\ldots,r-1$.}
$$
In other words, $b_{\sigma(i)}$ belongs to $S_i$ but not to any smaller set in the nested collection.  Define $r(B,S_\bullet) := (-1)^{\sigma}$, which depends on both the orderings of $B$ and of $S_\bullet$.  Also define
\begin{equation}\label{eq:aS}
\frac{1}{a_{S_\bullet}} := \prod_{i=1}^{r-1} \frac{1}{a_{S_i}}.
\end{equation}
The following is the nested version of \cref{prop:dRind}, and follows in a similar way to \cref{thm:deRhamfannested}.
\begin{proposition}\label{prop:nesteddR}
We have $\dRip{e_B,e_{B'}} = \sum_{S_\bullet \in N_{\max}(\build)} r(B, S_\bullet) \frac{1}{a_{S_\bullet}} r(B',S_\bullet)$.
\end{proposition}

While $r(B, S_\bullet)$ depends on the ordering of $S_\bullet$, the product $r(B, S_\bullet)r(B', S_\bullet)$ does not.

\part{Geometry}
\section{Hyperplane arrangement complements}
\def\comp{{\rm comp}}

In this section, we work with central hyperplane arrangements in $\C^r$, or projective hyperplane arrangements in $\P^d$, or affine hyperplane arrangements in $\C^d$, where $d = r-1$.

\subsection{Hyperplane arrangements}
We denote by $\A =  \{H_e \mid e \in E\}  \subset \C^r$ a central hyperplane arrangement in $\C^r$.  Whenever we discuss chambers of $\A$, or the oriented matroid of $\A$, we assume that $\A$ is defined over the reals.  We call $\A$ \emph{essential} if $\bigcap_e H_e = (0)$.  Unless otherwise specified, the hyperplane arrangements we consider are assumed to be essential and non-empty.  We let $M = M(\A)$ denote the rank $r$ matroid associated to $\A$ and $\M = \M(\A)$ a choice of oriented matroid associated to $\A$.  We let $f_e$ be a linear function cutting out $H_e$.

Let $\bA \subset \P^{d}$ denote the projective hyperplane arrangement associated to $\A$.  The hyperplanes of $\bA$ are still denoted by $H_e$, $e \in E$, and the corresponding matroid $M(\bA) = M(\A)$ is the same as for $\A$.  Once a hyperplane at infinity $H_0$ has been chosen, we obtain an affine hyperplane arrangement $(\bA,0) = \{H_e \mid e \in E \setminus 0\} \subset \C^{d}$.  Thus we work with three types of hyperplane arrangements: central $\A$, projective $\bA$, and affine $(\bA,0)$, and the corresponding matroids are denoted $M$, $M$, and $(M,0)$.

Let $\bU = \P^d\setminus \bA = \C^d \setminus (\bA,0)$ denote the hyperplane arrangement complement, which is the same for $\bA$ and $(\bA,0)$.  We use matroid terminology to refer to various hyperplane arrangement notions.  For example, the lattice of flats $L(M)$ is the lattice of all intersections of hyperplanes in $\bA$, with $\hat 0 = \P^d$ and $\hat 1 = \emptyset$ the intersection of all hyperplanes (empty, since we are assuming that $\bA$ is essential).  The connected components of $\bU(\R)$ are called chambers and are identified with the positive topes $\T^+$.  The bounded chambers are identified with the bounded topes $\T^0$.  We use $P$ to both denote a tope of the oriented matroid, or a chamber of $\bU(\R)$.

\subsection{Wonderful compactifications}\label{sec:WC}
The compactification $\P^d$ of $\bU$ typically does not have a normal-crossing boundary divisor.  However, it can be blown up in various ways to obtain a smooth compactification with normal-crossing boundary divisor.

Recall the notion of building sets from \cref{sec:building}.  For a building set $\build$ of $L(M)$, let $X_{\build}$ be the corresponding De Concini-Procesi wonderful compactification \cite{DP}.  The wonderful compactification $X_{\build}$ is a smooth projective compactification of $\bU$ with a normal crossing boundary divisor.  Let $F_1,F_2,\ldots,F_s$ be any ordering of $\build$ such that $F_j > F_i$ implies that $j < i$.  Then $X_{\build}$ is obtained from $\P^d$ by first blowing up $F_1$, then blowing up the proper transform of $F_2$, and so on.  Geometrically, such an ordering can be obtained by blowing up low-dimensional flats before high-dimensional ones.  We write $X_{\max}:= X_{\build_{\max}}$ and $X_{\min}:= X_{\build_{\min}}$ for short.

The wonderful compactification $X_{\build}$ has a stratification induced by the intersection of boundary divisors (\cite[Theorem 3.2]{DP}).  The boundary divisors $D_F$ of $X_{\build}$ are indexed by flats $F \in \build$.  Intersections of the divisors $D_F$ endows $X_{\build}$ a stratification.  For $S_\bullet = (S_1,\ldots,S_k)$ the intersection  
$$
X_{S_\bullet} := D_{S_1} \cap \cdots \cap D_{S_k}
$$
is non-empty if and only if $S_\bullet \in \N(\build)$ is a nested collection, and if so, the intersection is transversal and irreducible.  In other words, the stratification of $X_{\build}$ has the combinatorics of the dual of the nested set complex.  In particular, the vertex strata of $X_{\build}$ are indexed by maximal nested subsets $S_\bullet \in N_{\max}(\build)$.

In the case of $\build = \build_{\max}$, the strata $X_{G_\bullet}$ of the compactification $X_{\max}$ are labeled by partial flags $G_\bullet \in \pFl(M)$.  We let 
$$\mathring{X}_{G_\bullet} := X_{G_\bullet} \setminus \bigcup_{G'_\bullet \leq G_\bullet} X_{G'_\bullet}
$$ denote the relatively open strata, so that $X_{G_\bullet} = \bigsqcup_{G'_\bullet \leq G_\bullet} \mathring{X}_{G'_\bullet}$. Given a tope $P \in \T^+$, the (analytic) closure $\bP \subset X_{\max}(\R)$ is stratified by the intersection with the strata $X_{G_\bullet}$.  Recall that $\pFl(P)$ denotes the set of wonderful faces of $P$, or equivalently, the set of partial flags of flats that belong to $L(P)$.

\begin{proposition}[{\cite[Theorem 4.5]{BEPV}}]\label{prop:wonderfulface}
A stratum $X_{G_\bullet} \cap \bP$ of $\bP$ in the wonderful compactification $X_{\max}$ is non-empty if and only if $G_\bullet \in \pFl(P)$ is a wonderful face of $P$.
\end{proposition}

\section{Twisted (co)homology}\label{sec:twistedco} 
In this section $\bU$ denotes a projective hyperplane arrangement complement of an essential hyperplane arrangement.  

\subsection{Cohomology}
We write $H^\bullet_\dR(\bU,\C)$ to denote the algebraic deRham cohomology of $\bU$, and let $H^\bullet(\bU,\Z)$ and $H_{\bullet}(\bU,\Z)$ denote the Betti (co)homology groups.  Brieskorn showed that $H^\bullet_\dR(\bU,\C)$ is generated by the forms $\dlog f_e = \frac{df_e}{f_e}$.

\begin{theorem}[\cite{Brie, OS}]\label{thm:Bri}
Let $\A \subset \C^d$ be a central and essential hyperplane arrangement with matroid $M$ and complement $U:= \C^d \setminus \A$.  We have an isomorphism
$$
\OS^\bullet(M) \otimes_\Z \C \cong H_{\dR}^\bullet(U,\C), \qquad e \mapsto [\dlog f_e].
$$
Let $\bA \subset \P^{d}$ be an essential projective hyperplane arrangement with matroid $M$.  We have an isomorphism
$$
\rOS^\bullet(M) \otimes_\Z \C \cong H_{\dR}^\bullet(\bU,\C), \qquad (e - e') \mapsto [\dlog (f_e/f_{e'})].
$$
\end{theorem}

The lattice $\rOS^\bullet(M) \subset \rOS^\bullet(M) \otimes_\Z \C$ spans a lattice inside $H^\bullet_\dR(\bU,\C)$.  Under the comparison isomorphism
$$
\comp: H^\bullet_\dR(\bU,\C) \to H^\bullet_B(\bU,\Z) \otimes_\Z \C,
$$
the lattice $\rOS^\bullet(M)$ is identified with the lattice $H^\bullet_B(\bU,2\pi i \Z) \subset H^\bullet_B(\bU,\Z) \otimes_\Z \C$.

\subsection{Twisted (co)homology}\label{ssec:twisted}
Henceforth, we assume we are in the case of a projective arrangement, or an affine arrangement.  Let $a_e$, $e \in E$ be complex parameters. 
Consider the meromorphic 1-form 
\begin{equation}\label{eq:omega}
\omega = \omega_\a = \sum_e a_e \dlog f_e = \sum_{e \in E \setminus 0} a_e \dlog(f_e/f_0) \in \Omega^1(\bU)
\end{equation}
on $\bU$, where we assume that $\sum_{e \in E} a_e = 0$, or equivalently, $a_0 = - \sum_{e \in E \setminus 0} a_e$.  We may also view $\omega$ as an element 
$$
\omega = \sum_e a_e e \in \rOS^1(M).$$
The formula
$$
\nabla_\a := d + \omega \wedge
$$
defines a logarithmic connection $(\O_\bU,\nabla_\a)$ on the trivial rank one vector bundle $\O_\bU$ on $\bU$.  The flat (analytic) sections of $\nabla_\a$ define a complex rank one local system $\L_\a$ on $\bU$.  The local sections of $\L_\a$ are branches of the multi-valued function
\begin{equation}\label{eq:varphi}
\varphi^{-1}:= \prod_e f_e^{-a_e} = \prod_{e \in E \setminus 0} (f_e/f_0)^{-a_e},
\end{equation}
which satisfies $\nabla_\a \varphi^{-1} = 0$.  Fixing a basepoint $u_0 \in \bU$, the local system $\L_\a$ determines a homomorphism
$$
\rho: \pi_1(\bU,u_0) \to \C^\times
$$
which determines the isomorphism class of the local system $\L_\a$.  Let $\gamma_e$ be a loop starting and ending at $u_0$ that goes once around $H_e$.  Then the monodromy of $\L_\a$ around $H_e$ is given by
$$
\rho(\gamma_e) = \exp(-2\pi i a_e) = b_e^2, \qquad \text{where} \qquad b_e := \exp(-\pi i a_e).
$$
The monodromies satisfy $1 = \prod_e \rho(\gamma_e) = \prod_e b_e^2$ which is implied by $\sum_e a_e =0$.  The dual local system $\L^\vee_\a = \L_{-\a}$ has local sections given by the branches of $\varphi = \prod_e f_e^{a_e}$ with $\b$ parameters given by $b^\vee_e = b_e^{-1}$.  We have $\dlog \varphi= \omega$.  

Let $\pi: X_{\build} \to \P^d$ be a wonderful compactification of $\bU$.  Recall that for a proper non-trivial flat $F$ belonging to the building set $\build$, we denote by $D_F$ the corresponding divisor of $X_{\build}$.  If $\rk(F) = 1$ then $D_F$ is the pre-image of a hyperplane $H= H_e \subset \P^d$ and maps birationally to $H$.  Otherwise,  $D_F$ is the exceptional divisor of the blowup of (the proper transform of) the flat $F$.

\begin{lemma}[see \cite{ESV}]\label{lem:monod}\ 
\begin{enumerate}
\item The residue of the connection $\pi^* \nabla_\a$ along $D_F$ is equal to $a_F$.
\item The monodromy of $\L_\a$ once anticlockwise around $D_F$ is equal to $b^2_F$.
\end{enumerate}
\end{lemma}
\begin{proof}
We prove (a).  This implies (b) by \cite[II.Th\'eor\`eme 1.17]{Del}.
If $\rk(F) =1$, then the result is clear since $\pi$ is an isomorphism locally near $D_F$ and by definition \eqref{eq:omega} the residue of $\nabla_\a$ along $H_e$ is $a_e$.  Otherwise, let $t$ be the local equation of $D_F$ in the blowup.  For $e \in F$, let the local equation of the proper transform of $H_e$ near the general point of $D_F$ be $g_e$.  Then $\pi^*(f_e) = t g_e$ for $e \in F$.
The pullback of $\omega = \sum_e a_e \dlog f_e$ to $X$ is thus locally given by
$$
\pi^*\omega = \sum_{e \in F} a_e (\dlog g_e+\dlog t) + \sum_{e \notin F} a_e \dlog f_e.
$$
Since $\dlog f_e$ for $e \notin F$ and $\dlog g_e$ for $e \in F$ have no pole along $D_F$, the residue of $\pi^* \nabla_\a = d + \pi^*\omega$ along $D_F$, is equal to the residue of $(\sum_{e \in F}a_e)\dlog t$ at $t = 0$, which is equal to $a_F = \sum_{e \in F} a_e$.
\end{proof}

We consider the four twisted Betti (co)homology groups
\begin{align*}
H_k(\bU,\L_\a) & = \mbox{homology with coefficients in the local system $\L_\a$,} \\
H^{\lf}_k(\bU,\L_\a) &= \mbox{locally-finite homology with coefficients in the local system $\L_\a$,} \\
H^k(\bU,\L_\a) &= \mbox{cohomology with coefficients in the local system $\L_\a$,}\\
H^k_c(\bU,\L_\a)  &= \mbox{compactly supported cohomology with coefficients in the local system $\L_\a$}.
\end{align*}
The locally-finite homology group is also often called Borel-Moore homology.  Duality between homology and cohomology gives canonical isomorphisms
\begin{equation}\label{eq:4dual}
H_k(\bU,\L^\vee_\a) \cong H^k(\bU, \L_\a)^\vee, \qquad H^{\lf}_k(\bU,\L^\vee_\a) \cong H^k_c(\bU, \L_\a)^\vee.
\end{equation}
The following well-known result is proved in \cite[Theorem 1]{Koh} with a stronger assumption.  See also \cite{MHcoh,EV,CDO}.
\begin{theorem}[see \cite{Koh}]\label{thm:Koh}
Under the assumption \eqref{eq:Mon}, the natural morphisms induce isomorphisms
\begin{align*}
H_k(\bU,\L_\a) &\stackrel{\cong}{\longrightarrow} H^{\lf}_k(\bU,\L_\a) \\
H^k_c(\bU,\L_\a) &\stackrel{\cong}{\longrightarrow} H^k(\bU,\L_\a)
\end{align*}
for each $k$.  Furthermore, we have the vanishing 
$$
H_k(\bU,\L_\a)= H^{\lf}_k(\bU,\L_\a)= H^k_c(\bU,\L_\a) = H^k(\bU,\L_\a) = 0 \qquad \mbox{when $k \neq d$}.
$$
\end{theorem}
\begin{proof}
Our argument is the same as that of \cite{Koh}.  See also \cite[Proposition 6.5]{BD} for a similar argument.
Let $Y \to \P^d$ be the minimal wonderful compactification of $\bA$ (see \cref{sec:WC}).  Thus $Y$ is a smooth compactification of $\bU$ with a normal-crossing boundary divisor $D:=Y \setminus U$, obtained by blowing up all the connected flats $F$.  Let $\iota: \bU \to Y$ be the inclusion.  To prove the first statement, we show that $\iota_! \L_\a \cong R\iota_* \L_\a$ in the derived category on $Y$.  It suffices to check that the stalk of $R\iota_* \L_\a$ vanishes on $D$.  Since by construction $D$ is a normal crossing divisor, it suffices to check that $\L_\a$ has non-trivial monodromy around every irreducible component of $D$.  The irreducible components $D_F$ of $D$ are labeled by the connected flats $F \in L(M)$.  The divisor $D_F$ is the exceptional divisor obtained when (the proper transform of) $F$ is blown up.

By \cref{lem:monod}, the monodromy of $\L_\a$ around $D_F$ is given by $b^2_F = \prod_{e \in F} b^2_e = \exp(-2\pi i a_F)$.  By assumption \eqref{eq:Mon}, this is not equal to 1, so $\L_\a$ has non-trivial monodromy around $D_F$.  We conclude that $\iota_! \L_\a \cong R\iota_* \L_\a$.  Taking derived global sections we obtain the isomorphism $H^k_c(\bU,\L_\a) \cong H^k(\bU,\L_\a)$, and by duality \eqref{eq:4dual} we obtain $H_k(\bU,\L_\a) \cong H^{\lf}_k(\bU,\L_\a)$.

Now we prove the last statement.  Since $\bU$ is affine, by Artin vanishing we have $H^k(\bU,\L_\a) = 0$ for $k > d$ and by Poincar\'e-Verdier duality we have $H^k_c(\bU,\L_\a) = 0$ for $k < d$.  Combining with $H^k_c(\bU,\L_\a) \cong H^k(\bU,\L_\a)$ and \eqref{eq:4dual}, we obtain the vanishing of all four groups when $k \neq d$.
\end{proof}

\subsection{Logarithmic description of twisted cohomology}
The analogue of Brieskorn's theorem (\cref{thm:Bri}) for twisted cohomology is due to Esnault--Schechtman--Viehweg \cite{ESV}, and was extended by Schechtman--Terao--Varchenko \cite{STV}.

Let $(\O^\an_\bU, \nabla^\an_\a)$ denote the analytification of $(\O_\bU,\nabla_\a)$.
Let $(\Omega^\bullet_\bU,\nabla_\a)$ (resp. $(\Omega^{\bullet,\an}_\bU, \nabla^\an_\a)$) denote the complex of algebraic (resp. analytic) differential forms on $\bU$, equipped with the differential $\nabla_\a$ (resp. $\nabla_\a^\an$).  The \emph{algebraic deRham cohomology} (resp. \emph{analytic deRham cohomology}) of the connection $\nabla_\a$ is 
$$
H^\bullet(\bU, \nabla_\a) := H^\bullet(\Gamma(\bU, \Omega_\bU^\bullet),\nabla_\a), \qquad H^\bullet(\bU, \nabla^\an_\a) := H^\bullet(\Gamma(\bU, \Omega_\bU^{\bullet,\an}),\nabla^\an_\a).
$$
In both cases, the cohomology groups are usually defined as hypercohomologies, and we may replace by global sections since $\bU$ is affine (resp. Stein); see \cite[II.6]{Del}.  We have a GAGA-type isomorphism between algebraic and analytic deRham cohomology of the connection
$$
H^\bullet(\bU, \nabla_\a) \cong H^\bullet(\bU, \nabla^\an_\a),
$$
see \cite[II Th\'eor\`eme 6.2]{Del}.  We may thus freely switch between algebraic and analytic deRham cohomologies.  The identification of flat sections of $\nabla_\a$ with the local system $\L_\a$ gives a comparison isomorphism
$$
\comp_\a: H^\bullet(\bU,\nabla_\a) \stackrel{\cong}{\longrightarrow} H^\bullet(\bU, \L_\a).
$$
Let $(\rOS^\bullet,\omega)$ denote the Aomoto complex of \cref{sec:Aomoto}.  Since $d$ acts trivially on $\rOS^\bullet$, we have $\nabla_\a = \omega \wedge$ on $\rOS^\bullet$, and thus $(\rOS^\bullet,\omega)$ is a subcomplex of $\Gamma(\bU, (\Omega_\bU^\bullet,\nabla_\a))$.

\begin{theorem}[\cite{ESV,STV}]\label{thm:ESV}
Under the assumption \eqref{eq:Mon}, the natural inclusion
$$
(\rOS^\bullet, \omega) \hookrightarrow \Gamma(\bU, (\Omega_\bU^\bullet,\nabla_\a))
$$
is a quasi-isomorphism.  In particular, we have
$$
H^\bullet(\bU,\L_\a) \cong H^\bullet(\bU, \nabla_\a) \cong H^\bullet(\rOS^\bullet,\omega).
$$
\end{theorem}
\begin{proof}
We work in the minimal wonderful compactification $\pi: Y = X_{\min} \to \P^d$.  Let $D:= Y \setminus \bU$.  By \cref{lem:monod} and assumption \eqref{eq:Mon}, the monodromy of the pullback $\pi^*\L_\a$ is non-trivial around any of the boundary divisors $D_F \subset Y$.  By \cite[II.3.13, 3.14]{Del}, this implies that in the calculation of $H^\bullet(\bU,\nabla_\a)$ we may replace the complex $(\Omega_\bU^\bullet,\nabla_\a)$ (or $(\Omega_\bU^{\bullet,\an},\nabla^\an_\a)$) by the complex $(\Omega^\bullet_Y(\log D), \nabla_\a)$ of logarithmic differential forms on $Y$, equipped with the differential $\nabla_\a$.  Namely, we have an isomorphism
\begin{equation}\label{eq:Lnabla}
H^\bullet(\bU,\L_\a) \cong {\mathbb H}^\bullet(\Omega^\bullet_Y(\log D),\nabla_\a).
\end{equation}
By \cite{Del}, the spectral sequence 
$$
E_1^{p,q} = H^q(Y, \Omega^p_Y(\log D)) \implies H^{p+q}(\bU,\C)
$$
degenerates at $E_1$ and by Brieskorn's theorem (\cref{thm:Bri}) we have an isomorphism $H^\bullet(\bU,\C) \cong \rOS^\bullet \cong \Gamma(Y,\Omega^\bullet_Y(\log D))$.  It follows that the higher cohomologies of the sheaves $\Omega^p_Y(\log D)$ vanish.  Combining with \eqref{eq:Lnabla}, the higher cohomology vanishing gives an isomorphism $H^\bullet(\bU,\L_\a) \cong H^\bullet(\Gamma(Y,\Omega^\bullet_Y(\log D)),\nabla_\a)$.  Composing with the isomorphism $(\rOS^\bullet,\omega) \stackrel{\cong}{\longrightarrow} (\Gamma(Y,\Omega^\bullet_Y(\log D)), \nabla_\a)$, we obtain the stated quasi-isomorphism.
\end{proof}

It follows from \cref{thm:Koh} and \cref{thm:ESV} that in this case $\rOS^k(M,\omega)$ vanishes unless $k =d$; compare with \cref{thm:Yuz}.  Furthermore, we have $\dim H^d(\bU,\L_\a) = \dim H^d(\bU,\nabla_\a) = \beta(M)$.

\subsection{Betti homology twisted intersection form}

The intersection form in Betti homology was first defined by Kita and Yoshida \cite{KY}.
\begin{definition}
Assume \eqref{eq:Mon}.  The \emph{Betti homology intersection form} 
$$
\gBip{\cdot, \cdot}: H^{\lf}_d(\bU,\L^\vee_\a) \otimes H^{\lf}_d(\bU, \L_\a) \to \C
$$
is the composition of the the regularization isomorphism $H^{\lf}_d(\bU,\L_\a) \cong H_d(\bU,\L_\a)$ of \cref{thm:Koh} with the Poincar\'e-Verdier duality perfect pairing
$$
H^{\lf}_d(\bU,\L^\vee_\a) \otimes H_d(\bU, \L_\a) \to \C.
$$
The \emph{Betti cohomology intersection form} $\gDBip{\cdot,\cdot}$ is similarly defined.
\end{definition}

The Betti homology and Betti cohomology intersection forms are related by duality.  Namely, the perfect pairing $\gBip{\cdot,\cdot}$ together with \eqref{eq:4dual} induces isomorphisms $\gamma: H_\bullet(\bU,\L_\a) \cong H^\bullet(\bU, \L_\a^\vee)$ and $\delta: H_\bullet(\bU,\L^\vee_\a) \cong H^\bullet(\bU, \L_\a)$, and we have
$$
\gDBip{a,b} = \gBip{\gamma^{-1}(b),\delta^{-1}(a)}
$$
for $a \in H^d(\bU, \L_\a)$ and $b \in H^d(\bU,\L^\vee_\a)$.

We now describe the twisted homology group $H_p(\bU,\L_\a)$ explicitly.  The $p$-th chain group $C_p(\bU,\L_\a)$ is spanned by twisted $p$-chains $[\Gamma \otimes \varphi^{-1}_\Gamma]$ where $\Gamma$ is a $p$-chain, that is, a map of the $p$-dimensional simplex $\Delta^p$ into $\bU$, and $\varphi^{-1}_\Gamma$ is a choice of a branch of $\varphi^{-1}$ from \eqref{eq:varphi} on $\Gamma$.  The boundary operator $\partial: C_p(\bU,\L_\a) \to C_{p-1}(\bU,\L_\a)$  is given by 
$$
\partial [\Gamma \otimes \varphi^{-1}_\Gamma] = \sum_i (-1)^i [\partial_i \Gamma \otimes (\varphi^{-1}_\Gamma)|_{\partial_i \Gamma}]
$$
where we write $\sum_i$ for the summation over the boundary components of a $p$-simplex.  The twisted $p$-chains in the kernel of $\partial$ are called \emph{twisted $p$-cycles}, and the twisted homology group is defined as
$$
H_p(\bU,\L_\a) := \frac{\ker( \partial: C_p(\bU,\L_a) \to C_{p-1}(\bU,\L_\a))}{\image ( \partial: C_{p+1}(\bU,\L_a) \to C_{p}(\bU,\L_\a))}.
$$
The twisted locally-finite homology group $H^{\lf}_p(\bU,\L_\a)$ is defined similarly, now allowing formal linear combinations of $[\Gamma \otimes \varphi^{-1}]$ that are locally-finite.  Twisted cycles are also called \emph{loaded} cycles in the literature.

\begin{definition}\label{def:standard}
Let $P \in \T$ be a tope of $\bA$.  The \emph{standard twisted $d$-cycle} $[P \otimes \varphi^{-1}_P] \in C_d^{\lf}(\bU,\L_\a)$ is the locally-finite $d$-dimensional cycle where $\varphi^{-1}_P$ is the scalar multiple of the branch of $\varphi^{-1}$ that takes positive real values on $P$ when the $\a$ are real.
\end{definition}
For each $e$, there is a ``real" branch of $f_e^{-a_e}$ that for $a_e \in \R$ takes real values on $\bU(\R) \cap \{f_e > 0\}$.  We have
\begin{equation}\label{eq:standardb}
\varphi^{-1}_P = b_P \varphi^{-1}, \qquad b_P := \prod_{e: P(e) = -}  b_e,
\end{equation}
where $\varphi^{-1}$ denotes the product of the real branches of $f^{-a_e}_e$.  The prefactor $b_P$ is chosen so that $\varphi^{-1}_P$ takes positive real values on $P$.  Indeed,
$$
b_P \varphi^{-1} = \prod_{e \neq 0} b_e (f_e/f_0)^{-a_e} = \prod_{e \neq 0}  \exp(-\pi i a_e) (f_e/f_0)^{-a_e} =  \prod_{e \neq 0}(-f_e/f_0)^{-a_e}.
$$
We remark that formally speaking $[P \otimes \varphi^{-1}_P]$ should be expressed as an infinite sum of chains that head towards the boundary of $P$.

Recall from \cref{thm:Koh} that under assumption \eqref{eq:Mon}, $H_\bullet^{\lf}(\bU,\L_\a) \cong H_\bullet(\bU,\L_\a)$ vanishes unless $\bullet = d$.  We now assume that we are working with an affine hyperplane arrangement $\bA \subset \C^d$, and let $\T^0$ denote the set of bounded topes (or chambers).  By \cref{prop:numbertopes}, we have $|\T^0| = \beta(M) = \dim H_d(\bU,\L_a)$.  The following result is due to Kohno \cite{Koh}; see also \cite[Proposition 3.14]{DT}.  It will also follow indirectly from our results (\cref{thm:Bettinondeg} and \cref{thm:Bettipairmain}).

\begin{proposition}
Assume \eqref{eq:Mon}.  The twisted cycles $C(P) := [P \otimes \varphi^{-1}_P]$ (resp. $C^\vee(P) := [P \otimes \varphi_P]$) for $P \in \T^0$ form a basis of $H_d^{\lf}(\bU,\L_\a)$ (resp. $H_d^{\lf}(\bU,\L^\vee_\a)$).  The regularized twisted cycles $\reg(C(P))$ (resp. $\reg(C^\vee(P))$) for $P \in \T^0$ form a basis of $H_d(\bU,\L_\a)$ (resp. $H_d(\bU,\L^\vee_\a)$).
\end{proposition}

We take $\{[P \otimes \varphi^{-1}_P]\}$ as a basis for $H_d^{\lf}(\bU,\L_\a)$, and $\{[P \otimes \varphi_P]\}$ as a basis for $H_d^{\lf}(\bU,\L^\vee_\a)$.

\begin{proposition}\label{prop:isoB}
There is a natural isomorphism $ H^{\lf}_\bullet(\bU,\L_\a) \cong H^{\lf}_\bullet(\bU, \L_\a^\vee)$ given by $[P \otimes \varphi^{-1}_P] \mapsto [P \otimes \varphi_P]$.
\end{proposition}
Compositing with the isomorphism of \cref{prop:isoB}, $\ip{\cdot,\cdot}_B$ can be viewed as a bilinear form on $ H^{\lf}_\bullet(\bU,\L_\a) $.  The pairing $\gBip{C^\vee, C}$ is given by the formula
\begin{equation}\label{eq:CC}
\gBip{C^\vee, C} = \langle C^\vee, \reg(C) \rangle := \sum_{p \in \Gamma^\vee \cap \Gamma} \varphi_{\Gamma^\vee}(p) \langle\Gamma^\vee,\Gamma \rangle_p \varphi^{-1}_\Gamma(p).
\end{equation}
where $\reg(C) = [\Gamma \otimes \varphi^{-1}_\Gamma] \in H_d(\bU,\L_\a)$ is a regularization of $C$ in general position with respect to $C^\vee = [\Gamma^\vee \otimes \varphi_{\Gamma^\vee}]$, and $\langle\Gamma^\vee,\Gamma \rangle_p$ denotes the topological intersection number of the transverse intersection $\Gamma^\vee \cap \Gamma$ at $p$.  With all the above notation defined, we can state our main theorem for the Betti intersection pairing.

\begin{theorem}\label{thm:Bettipairmain}
For $P,Q \in \T^+$, we have $\gBip{[P \otimes \varphi_P],[Q \otimes \varphi^{-1}_Q]}= \bip{P,Q}_B$, as defined in \cref{def:Bettipair}. 
\end{theorem}
The proof of \cref{thm:Bettipairmain} will be given in \cref{sec:Bettipairmain}.

Kita and Yoshida initially studied the Betti homology twisted intersection form in \cite{KY,KY2} and Yoshida \cite{Yos} studied determinantal formulae for the intersection form.  The intersection pairing is studied for the braid arrangement \cite{MOY}, for arrangements appearing in conformal field theory \cite{MY}, and in relation to Lauricella's hypergeometric function \cite{Goto}, and Togi \cite{Tog} gave some closed formulae.  In \cite{GM}, Goto and Matsubara-Heo study the intersection form for GKZ type systems, which involve arrangements of hypersurfaces of higher degree.

\begin{remark}
The symmetry of the bilinear form $\gBip{\cdot,\cdot}$ is specific to the choice of isomorphism in \cref{prop:isoB}, and the choice of standard loading in \cref{def:standard}.  Indeed, the original bilinear form computed in \cite{KY} did not have the symmetry property.
\end{remark}

\subsection{deRham cohomology twisted intersection form}

The intersection form in deRham cohomology was introduced and studied by Cho and Matsumoto \cite{CM}.  The dual connection $\nabla^\vee_\a$ to $\nabla_\a$ is given by the connection $\nabla_{-\a}$.

\begin{definition}
Assume \eqref{eq:Mon}.  The \emph{deRham cohomology intersection form} 
$$
\gdRip{\cdot, \cdot}: H^d(\bU,\nabla^\vee_\a) \otimes H^d(\bU, \nabla_\a) \to \C
$$
is the composition of the the regularization isomorphism $\reg: H^d(\bU,\nabla_\a) \cong H_c^d(\bU,\nabla_\a)$ of \cref{thm:Koh} with the Poincar\'e-Verdier duality perfect pairing
$$
H^d(\bU,\nabla^\vee_\a) \otimes H^d_c(\bU, \nabla_\a) \to \C.
$$
The \emph{deRham homology intersection form} $\gDdRip{\cdot, \cdot}$ is similarly defined.
\end{definition}

Just as in the Betti case, the deRham cohomology and deRham homology intersection forms are related by duality.  

Let $[\theta^\vee] \in H^d(\bU,\nabla^\vee_\a)$ and $[\theta] \in H^d(\bU,\nabla_\a)$ be cohomology classes represented by twisted $d$-forms $\theta^\vee,\theta$.
To compute $\dRip{\theta^\vee,\theta}$, we first find a cohomologous compactly supported form $\reg(\theta)$, representing a class in $[\reg(\theta)] \in H^d_c(\bU, \nabla_\a)$.  The Poincar\'e-Verdier duality pairing is then given by
\begin{equation}\label{eq:PV}
\gdRip{\theta^\vee,\theta} = \ip{\theta^\vee,\reg(\theta)} = \frac{1}{(2\pi i)^d} \int_\bU \theta^\vee \wedge \reg(\theta),
\end{equation}
where the integral is over all (complex) points of $\bU$.  We caution that the factor of $(2\pi i)^{-d}$ is not always included in the literature.

We note that $\omega_{-\a} = - \omega_{\a}$ and $\rOS(M,\omega) \cong \rOS(M,-\omega)$.  By \cref{thm:ESV}, the natural maps 
\begin{equation}\label{eq:rOSaa}
\rOS(M,\omega) \longrightarrow H^\bullet(\bU,\nabla^\vee_\a), \qquad \text{and} \qquad \rOS(M,\omega) \longrightarrow H^\bullet(\bU,\nabla_\a)
\end{equation}
are isomorphisms, so we may view the deRham cohomology intersection form as a bilinear form 
$$
\gdRip{\cdot,\cdot}: \rOS(M,\omega) \otimes \rOS(M,\omega) \to \C.
$$
We can now state our main theorem for the deRham intersection pairing.
\begin{theorem}\label{thm:dRpairmain}
For $x,y \in \rOS(M,\omega)$, we have $\gdRip{x,y} = \bdRip{x,y}$, as defined in \cref{sec:descent}. 
\end{theorem}
The proof of \cref{thm:dRpairmain} will be given in \cref{sec:dRpairmain}.  

Cho and Matsumoto \cite{CM} first studied the deRham cohomology twisted intersection pairing in the one-dimensional case.  Matsumoto \cite{Matgen} gave an explicit formula for the intersection pairing in the case of a generic hyperplane arrangement, and specific arrangements were considered, for instance, in \cite{Goto, MOY, MY}.  Goto and Matsubara-Heo \cite{GM} study the intersection form for GKZ type systems.

\begin{remark}
The symmetry of the bilinear form $\ip{\cdot,\cdot}^{\dR}$ has been apparent from the beginning \cite{CM}.  A formula due to Matsubara-Heo \cite[Theorem 2.2]{MHcoh} exhibits this symmetry in a general setting.
\end{remark}

\section{Proof of deRham cohomology pairing}\label{sec:dRpairmain}
We prove \cref{thm:dRpairmain}, following \cite{CM,Matgen}.  Let us first consider the one-dimensional case $\bU = \P^1 \setminus \{z_1,z_2,\ldots,z_t\}$ and let $\theta$ be a holomorphic $1$-form on $\bU$.  Let $g_i$ denote a smooth function that takes value $1$ in a small neighborhood $V_i \ni z_i$ and takes value $0$ outside a larger neighborhood $W_i \supset V_i$.  Also, let $\psi_i$ be a smooth function that is a solution to $\nabla_\a \psi_i = \theta$ in $\overline{W_i}$.  The existence of this solution follows from the assumption \eqref{eq:Mon} that the monodromy of $\L_\a$ is non-trivial around $z_i$, which implies
\begin{equation} \label{eq:DM}
H^\bullet(W_i \setminus z_i, \L_\a) \cong H^\bullet(S^1, \L_\a) = 0.
\end{equation}
See the comprehensive discussion in \cite[Section 2]{DM}.  We then take 
$$
\reg(\theta):= \theta - \nabla_\a(\sum_i \psi_i g_i) = \theta - \sum_i (\psi_i d g_i - g_i \theta) = \theta( 1- \sum_i g_i) - \sum_i \psi_i dg_i.
$$
The first formula shows that $\reg(\theta)$ is $\nabla_\a$-cohomologous to $\theta$, and the rightmost expression shows that $\reg(\theta)$ is compactly supported.  Now, for another holomorphic $1$-form $\theta^\vee$, we have
$$
\gdRip{\theta^\vee,\theta} = \frac{1}{2\pi i} \int_\bU \theta^\vee \wedge \reg(\theta) = - \frac{1}{2\pi i}  \sum_i \int_\bU \theta^\vee \wedge \psi_i dg_i
$$
since $\theta^\vee \wedge \theta = 0$ as there are no holomorphic $2$-forms on $\bU$.  Since $dg_i$ is supported on a small annulus $W_i \setminus V_i$ around $z_i$, the integral localizes to each of these points.  By a limiting procedure, one may replace $g_i$ by a step function, and $dg_i$ is a delta-function on a circle, that is
$$
- \frac{1}{2\pi i}  \sum_i \int_\bU \theta^\vee \wedge \psi_i dg_i = \sum_i \frac{1}{2\pi i} \int_{|z-z_i| = \epsilon} \psi_i \theta^\vee = \sum_i \Res_{z_i} (\psi \theta^\vee).
$$
Matsumoto \cite[Lemma 4.1]{Matgen} shows that 
\begin{equation}\label{eq:Mat}
\Res_{z_i}(\psi \theta^\vee) = \Res_{z_i} (\theta^\vee) \frac{1}{a_i} \Res_{z_i}(\theta), \qquad \text{giving} \qquad \gdRip{\theta^\vee, \theta} = \sum_i \Res_{z_i} (\theta^\vee) \frac{1}{a_i} \Res_{z_i}(\theta),
\end{equation} and further extends this calculation to higher-dimensional cases of generically intersecting collection of hyperplanes.  In that case, the integral localizes to the vertices of the arrangement, and since the hyperplanes are normal crossing at each vertex, the existence of solutions to $\nabla_\a \psi = \theta$ can be obtained by applying the Kunneth formula to \eqref{eq:DM}.  Since the computation is local, it can be applied whenever there is a compactification with normal-crossing divisors.  Matsumoto's local formula gives:

\begin{proposition}\label{prop:dRresidue}
Let $\pi: X_\build \to \P^d$ be a wonderful compactification of $\bU$.  Then for $x, y \in \rOS(M,\omega)$, we have
$$
\gdRip{x,y} = \sum_{S_\bullet \in N_{\max}(\build)} \Res_{S_\bullet}(x) \frac{1}{a_{S_\bullet}} \Res_{S_\bullet}(y),
$$
where the summation is over the maximal nested collections (see \cref{sec:building} and \eqref{eq:aS}), and $\Res_{S_\bullet}$ is the residue at a normal-crossing vertex, obtained as the integral over a multi-dimensional torus $(S^1)^d$:
\begin{equation}\label{eq:ResS}
\Res_{S_\bullet}(x) = \int_{|t_d|=\epsilon} \cdots \int_{|t_1|=\epsilon} x,
\end{equation}
where $t_1,t_2,\ldots,t_d$ are local coordinates cutting out the normal-crossing vertex $X_{S_\bullet}$.
\end{proposition}
The residue map $\Res_{S_\bullet}$ is defined up to a sign which appears twice and cancels out.
\begin{proof}
Let $S_\bullet = (S_1,\ldots,S_d)$.  The vertex $X_{S_\bullet}$ of $X_\build$ is the transverse intersection of $D_{S_1},\ldots,D_{S_d}$.  According to \cref{lem:monod}, the residue of $\pi^* \nabla_\a$ around $D_{S_i}$ is equal to $a_{S_i}$.  By Matsumoto's theorem \cite[Theorem 2.1]{Matgen}, or by applying the Kunneth formula to \eqref{eq:Mat}, we see that the local contribution around $X_{S_\bullet}$ is given by $\Res_{S_\bullet}(x) \frac{1}{a_{S_\bullet}} \Res_{S_\bullet}(y)$.
\end{proof}

Matsumoto's approach is generalized significantly by Matsubara-Heo \cite[Theorem 2.2]{MHcoh} from which \cref{prop:dRresidue} could also be obtained directly.  

In the case that $X_\build = X_{\max}$, we have $\N(\build) = \Fl(M)$, so comparing with \cref{def:dR}, the proof of \cref{thm:dRpairmain} is completed by the following.

\begin{lemma}\label{lem:ResF}
The residue $\Res_{F_\bullet}(x)$ agrees with the residue defined in \cref{sec:residue}.
\end{lemma}
\begin{proof}
The formula \eqref{eq:ResS} can be written as the composition of one-dimensional residue integrals.  We proceed by induction on the number of integrals $d$.  When $d=1$, we have $X_{\max} = \P^1$, so the result is clear.   Suppose $d > 1$ and let $F_1 = \{e\}$.  The divisor $D_{F_1} \subset X_{\max}$ is isomorphic to the maximal wonderful compactification $X''_{\max}$ of the contraction $\bA'' = \bA \cap H_e$, where the local equation of $H_e$ is $t_1 = 0$.  The following diagram commutes:
$$
\begin{tikzcd}
\rOS(M)\otimes \C  \arrow[r, "\Res_e"] \arrow[d]
& \rOS(M/e)\otimes \C \arrow[d] \\
\Omega^d(\bU) \arrow[r, "\Res_{H_e}"]
& \Omega^{d-1}(\P^{d-1} \setminus \bA'')
\end{tikzcd}
$$
where $\Res_{H_e}$ can be computed by a one-dimensional integral $\int_{|t_1|=\epsilon}$.   The result follows by induction on dimension.
\end{proof}

The analogue of \cref{lem:ResF} also holds for other building sets (see \cref{sec:building}), but the recursive structure in the proof is simpler for $X_{\max}$.

\section{Proof of Betti homology pairing}\label{sec:Bettipairmain}

\def\neg{{\rm neg}}

We work in an affine chart where $f_0$ is positive.  Then $\bU$ (resp. $\bU(\R)$) is identified with an open subset of $\C^d$ (resp. $\R^d$).  Each connected component $P$ of $\bU(\R)$ inherits an orientation from the ambient orientation of $\R^d$.  Viewing each $P$ as a signed covector, our assumption that $f_0$ is positive is that $P(0) = +$.

\subsection{Regularization}
We now discuss the regularization of our twisted cycles $C(P) = [P \otimes \varphi^{-1}_P]$, following \cite{KY,KY2}.  The construction is local.  We first discuss the one-dimensional case.  The regularization of a single interval $C = [[0,1] \otimes \psi]$ is taken to be the following finite sum of twisted $1$-chains:
$$
C_\reg = \frac{1}{b_0^2-1}[S(0,\epsilon)\otimes \psi] + [[\epsilon,1-\epsilon] \otimes \psi]-  \frac{1}{b_1^2-1}[S(1,1-\epsilon)\otimes \psi],
$$
pictured below in the punctured complex plane $\C^1 \setminus \{0,1\}$:
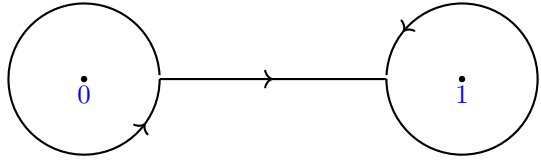
$$
\begin{tikzpicture}
\draw[thick,decoration={markings, mark=at position 0.9 with {\arrow{>}}},postaction={decorate}] ([shift=(3:1)]0,0) arc (3:360:1);
\draw[thick,decoration={markings, mark=at position 0.5 with {\arrow{>}}},postaction={decorate}] (1,0) -- (4,0);
\draw[thick,decoration={markings, mark=at position 0.9 with {\arrow{>}}},postaction={decorate}] (4,0) arc (-180:177:1);
\filldraw (0,0) circle (1pt);
\filldraw (5,0) circle (1pt);
\node[color=blue] at (0,-0.2) {$0$};
\node[color=blue] at (5,-0.2) {$1$};
\end{tikzpicture}
$$
We assume that the analytic continuation of $\psi$ once anticlockwise around $0$ (resp. $1$) produces a factor of $b^2_0$ (resp. $b^2_1$).  For instance,  if $\varphi = z^{-a_0}(1-z)^{-a_1}$ then we would have $b_i = \exp(-\pi i a_i)$.  Here, $S(0,\epsilon)$ (resp. $S(1,1-\epsilon)$ denotes an oriented circle starting at the point $\epsilon$ (resp. $1-\epsilon$) and going once around the point $0$ (resp. $1$) in the anti-clockwise direction.  (Strictly speaking, we should express $S(0,\epsilon)$ and $S(1,1-\epsilon)$ as the sum of two 1-chains, each of which is a half circle.). See \cite{KY,AKbook,Mat} for this formula.  That the factor $b_0^2-1$ is correct follows from the calculation:
$$
\partial [S(0,\epsilon) \otimes \psi] = [\epsilon \otimes b^2_0\psi(\epsilon)] - [\epsilon \otimes \psi(\epsilon)]= (b_0^2-1)[\epsilon \otimes \psi(\epsilon)],
$$
where $b_0^2$ (resp. $1$) is the value of $\psi$ at the end (resp. start) of the curve $S(0,\epsilon)$.

To show that $C_\reg$ is a regularization of $C$, we observe that the difference
$$
C - C_\reg = \sum_{i= 0,1} \partial(\text{punctured disk $D_i$ analytically continued anticlockwise from the interval})
$$
is the boundary of a twisted locally-finite $2$-chain.

In higher-dimensions, for a region $P$ where every vertex is a transversal intersection, we use a product construction.  More precisely, let us consider the locally-finite chain $C = [(0,1]^d \otimes \psi]$ in $\R^d$, equipped with the standard orientation of $\R^d$.  We imagine that locally $P$ is the product of intervals $[0,1]^d$, and the point $(0,0,\ldots,0)$ is a vertex of $P$, a transverse intersection of the hyperplanes $z_i = 0$.  Then we define the regularization $\reg(C)$ by
\begin{equation}\label{eq:regC}
\reg(C) = \left(\prod_{i=1}^d \frac{1}{b_i^2-1} S_i(0,\epsilon)+ [\epsilon,1]_i \right) \otimes \psi
\end{equation}
where the $i$-th factor of the product belongs to coordinate $z_i$ of $\C^d$.  That \eqref{eq:regC} is a regularization of $C$ follows by taking the product of the $1$-dimensional case.  Note that in \eqref{eq:regC}, $C$ and $\reg(C)$ are twisted $d$-chains and not cycles.  Gluing together such local regularizations, we obtain the regularization $\reg([P \otimes \varphi^{-1}_P])$.

\subsection{Intersection numbers}
Now let $C^\vee = C^\vee(Q)$ and $C= C(P)$.  In the following, we calculate the intersection number $\ip{C^\vee,C}$ from \eqref{eq:CC} by replacing $C$ with the regularization $\reg(C)$, and $C^\vee$ by a slight deformation.  Our calculation follows the local strategy of \cite{KY,KY2}.  Embed $\bU$ into the maximal wonderful compactification $X = X_{\max}$. 
\begin{proposition}\label{prop:Bettiint}
Let $P,Q \in \T^+$.  We have the formula
$$
\gBip{[Q \otimes \varphi_Q],  \reg([P \otimes \varphi^{-1}_P])}  = \sum_{E_\bullet} (\pm)^r \prod_{i \in I} (-1)^{\rk(E_i)} \frac{b_{E_i}}{\tb_{E_i}} \prod_{j \in J} \frac{1}{\tb_{E_j}} =  \sum_{E_\bullet} (\pm)^r \prod_{i \in I} (-1)^{\rk(E_i)} b_{E_i} \prod_{\ell \in I \cup J} \frac{1}{\tb_{E_\ell}}
$$
where the sum is over wonderful faces $E_\bullet \in \Delta(P)$ such that $\mathring{X}_{E_\bullet}$ intersects $\overline{P} \cap \overline{Q}$, and the first (resp. second) product over $i\in I$ (resp. $j \in J$) is over wonderful facets for which $P$ and $Q$ are on the opposite (resp. same) sides of.  The sign $(\pm)^r$ is equal to $(-1)^r$ if the line at infinity $0 \in E$ is crossed when going from $P$ to $Q$; otherwise, the sign is taken to be $1$.
\end{proposition}

\begin{proof}
Since $X$ is smooth, it is locally isomorphic to $\C^n$ with $X(\R)$ identified with $\R^d \subset \C^d$.  We have the decomposition $\C^d = \R^d \oplus i \R^d$, where $i \R^d$ can be identified with the fiber of the normal bundle of the manifold $X(\R)$ in $X(\C)$.

Let $\bQ$ be the (analytic) closure of $Q$ in the wonderful compactification $X_{\max}(\R)$.  Recall that the faces of $\bQ$ are defined to be the non--empty intersections of $\bQ$ with various strata $X_{E_\bullet}$ of $X$.  Fix a barycenter $\tau_K$ in the relative interior of each face $K$ of $\bQ$.  Define a continuous vector field $Z$ on $\bQ$ that such that if $p$ lies in the relative interior of a face $K$ of $\bQ$, then $Z(p)$ points towards the barycenter $\tau_K \in K$.  This is possible since each face $K$ of $\bQ$ is contractible.  By construction, the vector field $Z$ vanishes exactly at the points $\tau_K$.  If $f: W \to Q$ is a parametrization of $Q$, then we obtain a new (locally-finite) cycle $Q_Z$ by the parametrization 
$$
f_Y(w) = f(w) + i Z(f(w)) \in X(\C).
$$
(Formally, we should write $Q$ as the sum of infinitely many elementary chains, and apply this formula to each of them.).  Let $F_\bullet$ be a wonderful vertex of $\bQ$.  Locally near $X_{F_\bullet}$ we have that $\bQ$ can be identified with a positive orthant, and $Q_Z$ is the product of the picture 
$$
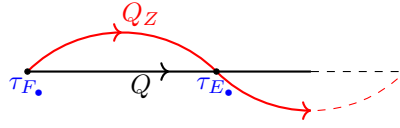
\begin{tikzpicture}
\draw[thick,decoration={markings, mark=at position 0.5 with {\arrow{>}}},postaction={decorate}] (0,0) -- (3.75,0);
\draw[dashed] (3.75,0) -- (5,0);
\draw[thick,color=red,decoration={markings, mark=at position 0.5 with {\arrow{>}}},postaction={decorate}] (0,0) to [bend left = 45] (2.5,0);
\begin{scope}
\clip (2.5,1) rectangle (3.75,-1);
\draw[thick,color=red,decoration={markings, mark=at position 0.5 with {\arrow{>}}},postaction={decorate}] (2.5,0) to [bend right = 45] (5,0);
\end{scope}
\begin{scope}
\clip (3.75,-1) rectangle (5.75,1);
\draw[dashed,color = red] (2.5,0) to [bend right = 45] (5,0);
\end{scope}
\filldraw (0,0) circle (1pt);
\filldraw (2.5,0) circle (1pt);
\node[color=blue] at (2.5,-0.2) {$\tau_{E_\bullet}$};
\node[color=blue] at (0,-0.2) {$\tau_{F_\bullet}$};
\node[color=black] at (1.5,-0.2) {$Q$};
\node[color=red] at (1.5,0.75) {$Q_Z$};
\end{tikzpicture}
$$
in each coordinate.
We replace the locally-finite twisted cycle $C^\vee(Q) = [Q \otimes \varphi_Q]$ with the homologous locally-finite twisted cycle 
$$
C^\vee(Q_Z) = [Q_Z \otimes \varphi_Q]
$$
and $[P \otimes \varphi_P^{-1}]$ by its regularization, which locally near any vertex $X_{F_\bullet}$ of $X$ has the form \eqref{eq:regC}.

Consider a term in the expansion of \eqref{eq:regC} where we take the factor $S_\ell(0,\epsilon)$ for $\ell \in L$ and $[\epsilon,1]_m$ for $m \in M$, and $L \sqcup M = [d]$.  
For sufficiently small $\epsilon$, any intersection of $Q_Y$ with such a chain will occur at a point $p \in \bU \subset X$ where 
$$
p|_{M} = \tau_M, \qquad \text{and} \qquad p_\ell \in S_\ell(0,\epsilon) \mbox{ for $\ell \in L$},
$$
and $\tau_M$ is the center of the face indexed by $M$.  (Here, we have written the coordinates of $p = (p_1,\ldots,p_d)$ using the same local coordinates as in \eqref{eq:regC}.)  Thus the intersection $Q_Z \cap \reg(C)$ consists of isolated transverse points, one close to each of the barycenters $\tau_M$.

We now calculate the intersection number $\varphi_{\Gamma^\vee}(p) \langle\Gamma^\vee,\Gamma \rangle_p \varphi^{-1}_\Gamma(p)$ at the point $p$ close to $\tau_M$.  We get a factor of
$$
\prod_{\ell \in L}\frac{1}{\tb_\ell}
$$
from the prefactors of $S_\ell(0,\epsilon)$, and a factor of 
$$
\varphi_{Q}(p) \varphi^{-1}_P(p) = \prod_{i \in I} b_i
$$
resulting from analytic continuation of $\varphi^{-1}_P$ along half circles (cf. the definition \eqref{eq:standardb} of the standard loading), where $I \subset L$ denotes the set of indices where $P$ and $Q$ are on the opposite sides of in the neighborhood of $X_{F_\bullet}$.  The set of indices where $P$ and $Q$ are on the same side of is $J = L \setminus I$.  In other words, $L = I \sqcup J$.

Finally, we calculate the sign.  We obtain a sign $(-1)^{|I|}$ from comparing the orientations of $P$ and $Q$ at the intersection point.  However, if we pass through the line at infinity to get from $P$ to $Q$, we additionally obtain a sign of $(\pm)^r$ since the analytic continuation of the natural orientation of $\R^d$ past the line at infinity gains a sign of $(-1)^{d+1} = (-1)^r$.  (This sign is the reason that $\P^d(\R)$ is an orientable manifold if and only if $d$ is odd.).  There is an additional sign coming from an orientation change in an even-codimensional blowup; see \cite[p.12]{MHhom} for details.  In our language, this contributes a factor $\prod_i (-1)^{\rk(F_i)+1}$, and together we get the sign $(\pm)^r (-1)^{|I|}\prod_i (-1)^{\rk(F_i)+1} = (\pm)^r \prod_i (-1)^{\rk(F_i)}$.  Putting everything together, we obtain the stated formula.
\end{proof}

\begin{example}
Let $\bU = \C \setminus \{0,1\}$ and $P = Q = (0,1)$.  Then $C^\vee(Q_Z)$ and $\reg(C(P))$ are pictured below:
$$
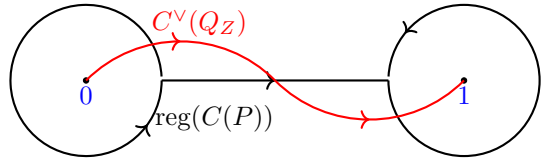
\begin{tikzpicture}
\draw[thick,decoration={markings, mark=at position 0.9 with {\arrow{>}}},postaction={decorate}] ([shift=(3:1)]0,0) arc (3:360:1);
\draw[thick,decoration={markings, mark=at position 0.5 with {\arrow{>}}},postaction={decorate}] (1,0) -- (4,0);
\draw[thick,decoration={markings, mark=at position 0.9 with {\arrow{>}}},postaction={decorate}] (4,0) arc (-180:177:1);
\filldraw (0,0) circle (1pt);
\filldraw (5,0) circle (1pt);
\node[color=blue] at (0,-0.2) {$0$};
\node[color=blue] at (5,-0.2) {$1$};
\draw[thick,color=red,decoration={markings, mark=at position 0.5 with {\arrow{>}}},postaction={decorate}] (0,0) to [bend left = 45] (2.5,0);
\draw[thick,color=red,decoration={markings, mark=at position 0.5 with {\arrow{>}}},postaction={decorate}] (2.5,0) to [bend right = 45] (5,0);
\node[color=red] at (1.5,0.75) {$C^\vee(Q_Z)$};
\node[color=black] at (1.7,-0.5) {$\reg(C(P))$};
\end{tikzpicture}
$$
The three intersection points (from left to right) give the three terms in
$$
\gBip{C^\vee(Q), C(P)} = \frac{1}{b_0^2 -1} + 1 + \frac{1}{b_1^2 - 1}.
$$
\end{example}

\begin{remark}
\cref{prop:Bettiint} also holds for any wonderful compactification $X = X_{\build}$.
\end{remark}

\subsection{Proof of \cref{thm:Bettipairmain}}
The faces in the intersection $\bP \cap \bQ$ in the maximal wonderful compactification $X$ is the union of $\bG_\bullet$ for $G_\bullet \in G(P,Q)$.  By \cref{prop:noover}(3), there is no overlap in $\bG_\bullet$ for $G_\bullet \in G(P,Q)$.  By definition, we have $Q = Q_{G_\bullet}$, so $P$ and $Q$ are on opposite sides of each of the facets $G_i$.  For any proper face $E_\bullet \in \bG_\bullet$, we have $Q \neq Q_{E_\bullet}$, so we must have that $P$ and $Q$ are on the same side of all the $E_i$ not belonging to $G_\bullet$.  Thus the set $I$ in \cref{prop:Bettiint} corresponds to $G_\bullet$ and the set $L= I \cup J$ corresponds to $E_\bullet$.  Comparing \cref{prop:Bettiint} with \cref{def:Bettipair}, we obtain \cref{thm:Bettipairmain}.

\section{Twisted period relations}\label{sec:beta}
\def\PP{{\bf P}}
We continue to assume \eqref{eq:Mon}.  Let $\comp_\a: H^\bullet(\bU,\nabla_\a) \to H^\bullet(\bU,\L_\a)$ denote the comparison map.  Then for $x \in H^\bullet(\bU,\nabla_\a)$ and $y \in H^\bullet(\bU,\nabla_{-\a})$, we have
\begin{equation}\label{eq:compip}
(2\pi i)^{d} \gdRip{x,y}= \gDBip{\comp_\a(x), \comp_{-\a}(y)}.
\end{equation}
The factor of $(2\pi i )^d$ comes from the choice of normalization in \eqref{eq:PV}, and agrees, for instance, with the choice in \cite{BD}. 

Twisted periods are the pairings
$$
(\comp_\a(\eta), [\Gamma \otimes \varphi_\Gamma])  = \int_\Gamma \varphi_\gamma \eta
$$
between $\eta \in H^\bullet(\bU,\nabla_\a)$ and $[\Gamma \otimes \varphi_\Gamma] \in H_\bullet(\bU,\L^\vee_\a)$ obtained via the comparison map, and the natural duality $(\cdot, \cdot)$ between $H^\bullet(\bU,\L_\a)$ and $H_\bullet(\bU,\L^\vee_\a)$.
Fixing bases of $H^\bullet(\bU,\nabla_\a)$ and $H_\bullet(\bU,\L^\vee_\a)$, we may express \eqref{eq:compip} as a matrix identity.  For concreteness, we assume that we take the basis $\{\bOmega_P \mid P \in \T^0\}$ of $H^\bullet(\bU,\nabla_\a)$ and $\{C^\vee(P) \mid P \in \T^0\}$ of $H_\bullet(\bU,\L^\vee_\a)$.  We furthermore identify $H^\bullet(\bU,\nabla_\a)$ with $H^\bullet(\bU,\nabla^\vee_\a)$ via \eqref{eq:rOSaa} and $H_\bullet(\bU,\L^\vee_\a)$ with $H_\bullet(\bU,\L_\a)$ via \cref{prop:isoB}.

Let $\PP^\a$ denote the $\T^0 \times \T^0$ period matrix
$$
\PP^\a_{P,Q} =  \int_Q \varphi_Q \Omega_P.
$$
Note that with our conventions $\PP^{-\a}$ is obtained from $\PP^\a$ by simply substituting $\a \mapsto -\a$.  Substituting into \eqref{eq:compip}, we obtain
the beautiful twisted period relations of Cho and Matsumoto~\cite{CM}:
\begin{equation}\label{eq:CM}
(2\pi i)^d \dRip{\cdot, \cdot}_{\T^0} = \PP^\a \ip{\cdot, \cdot}^B_{\T^0} (\PP^{-\a})^T =  \PP^{-\a} \ip{\cdot, \cdot}^B_{\T^0} (\PP^{\a})^T,
\end{equation}
where $\ip{\cdot, \cdot}^B_{\T^0}$ is simply the inverse matrix to $\ip{\cdot, \cdot}^{\T^0}_B$, and we have used that both $ \dRip{\cdot, \cdot}_{\T^0}$ and $\ip{\cdot, \cdot}^B_{\T^0}$ are symmetric matrices.  

Let 
$$
\beta(s,t):= \int_0^1 x^s (1-x)^t \frac{dx}{x(1-x)} = \frac{\Gamma(s) \Gamma(t)}{\Gamma(s+t)}
$$ 
denote the beta function.
For the case of $\bU = \C \setminus \{0,1\}$, and $\T^0 = \{P = [0,1]\}$, we have
$$
\gdRip{\bOmega_P,\bOmega_P} = \frac{1}{a_0} + \frac{1}{a_1}, \qquad \text{and} \qquad \gBip{P,P}= 1 + \frac{1}{b_0^2-1} + \frac{1}{b_1^2-1} = \frac{b^2_0b^2_1-1}{(b^2_0-1)(b^2_1-1)}, 
$$
$$
\PP^\a_{P,P} = \beta(a_0,a_1), \qquad \text{and} \qquad \PP^{-\a}_{P,P} = \beta(-a_0,-a_1)
$$
and we obtain from \eqref{eq:CM} the quadratic relation for the beta function:
$$
2 \pi i (\frac{1}{a_0} + \frac{1}{a_1}) = \beta(-a_0,-a_1) \beta(a_0,a_1) \frac{2}{i} \frac{\sin(\pi a_0) \sin(\pi a_1)}{\sin(\pi(a_0+a_1))},
$$
where we have used that 
$$
\frac{b}{b^2-1} = \frac{i}{2}\frac{1}{\sin(\pi a)} \qquad \text{ when } \qquad b = \exp(-\pi i a).
$$
Combining all our main theorems (\cref{thm:deRhamfan,thm:Bettifan,thm:dRpairmain,thm:Bettipairmain}), we obtain the following result.
\begin{corollary}
The twisted period relations of Cho and Matsumoto give an exact formula for the continuous Laplace transform in terms of the discrete Laplace transform (or for the discrete Laplace transform in terms of the continuous Laplace transform), with coefficients given by the period matrix (or its inverse). 
\end{corollary}

Note that we also have the limit formula \cref{thm:limit} relating the deRham and Betti intersection forms, but the twisted period relations are of a different nature.  

\begin{remark}
It would be interesting to understand the precise relation with the Todd class formulae relating lattice points in a polytope to the volume function; see for instance \cite[Theorem 8.9]{BaPo}.
\end{remark}

\part{Physics}
\def\kin{{\Lambda_\C}}
\def\Crit{{\rm Crit}}

\section{Scattering amplitudes on very affine varieties}\label{sec:veryaffine}
We refer the reader to \cite{LamModuli} for a more leisurely exposition of the material in this section.
A very affine variety $U$ is a closed subvariety of an algebraic torus.  The group $\C[U]^\times/\C^\times$ of units modulo scalars in the coordinate ring of $U$ is isomorphic to a lattice $\Lambda = \Lambda_U \cong \Z^{c}$.  The \emph{intrinsic torus} $T = T_U$ is defined to be the torus with character lattice $\Lambda$.  The very affine variety $U$ can always be embedded as a closed subvariety $U \hookrightarrow T$ of the intrinsic torus.  

We identify $\kin= \Lambda \otimes_\Z \C$ with the complex Lie algebra of the dual torus $T^\vee$.  Choose a natural embedding $T \hookrightarrow \P^{c}$, and let $\{f_e \mid e \in E\}$ be a choice of homogeneous coordinates on $\P^c$, where $|E| = c+1$.  A point $\a \in \kin$ is given by a collection of complex numbers $\{a_e, e\in E\}$ satisfying $\sum_{e} a_e =0$.  For $\a \in \kin$, we have the multi-valued function 
$$
\varphi := f^\a = \prod_{e \in E} f_e^{a_e}
$$
on $T$, which we may pullback to a multi-valued function on $U$.  In the case of a projective hyperplane arrangement complement $\bU$ with hyperplanes $\{H_e \mid e \in E\}$, we make choices so that $f_e$ is a linear form vanishing on $H_e$.

While the function $\varphi = f^\a$ is multi-valued, the 1-form 
$$
\omega = \dlog f^\a = \sum_{e\in E} a_e \dlog f_e = \sum_{e \in E} a_e \frac{d f_e}{f_e}
$$
is a well-defined algebraic $d$-form on $U$.  

\begin{definition}
The \emph{scattering equations} for $U$ are the critical point equations
$$
\omega = \sum_{e \in E} a_e \frac{d f_e}{f_e} = 0
$$
on $U$.  Let $\Crit(\omega) = \{\omega = 0\} \subset U$ denote the critical point locus.
\end{definition}

In the case of a hyperplane arrangement complement, the multi-valued function $f$ is known as the \emph{master function} by Varchenko \cite{Varbook}.  When $\a \in \Lambda_\C$ is generic and $U$ is smooth, Huh \cite{Huh} showed that the critical point equations consist of $|\chi(U)|$ reduced points.  For hyperplane arrangement complements, this was established by Varchenko \cite{Varcrit} in a special case, and by Orlik and Terao \cite{OT} in general.  

The following approach to scattering amplitudes was pioneered by Cachazo--He--Yuan \cite{CHYarbitrary}, who considered the case $U = M_{0,n+1}$, the configuration space of $n+1$ points on $\P^1$.  Relations to likelihood geometry were studied in \cite{ST}.

\begin{definition}\label{def:CHY}
Let $U$ be a $d$-dimensional very affine variety and $\a$ be generic.  Let $x_1,\ldots,x_d$ be local coordinates on $U$, and let
$$
\Omega = h(x_1,\ldots,x_d) d^d \x ,\qquad \Omega' = h'(x_1,\ldots,x_d) d^d \x
$$ be two rational $d$-forms on $U$.  Then the $(\Omega,\Omega')$-amplitude of $U$ is defined to be
$$
\A_U(\Omega,\Omega'):= \sum_p h(p) \det\left(\frac{\partial^2 \log \varphi}{\partial x_i \partial x_j}\right)_p^{-1}h'(p) 
$$
where the summation is over the critical points $p \in \Crit(\omega)$.
\end{definition}
While \cref{def:CHY} applies only when the parameters $\a$ are generic, the formula produces a rational function in $\a$ which we view as the amplitude.  \cref{def:CHY} is a sum over critical points, and sometimes called a \emph{stationary phase formula}.

For an arbitrary very affine variety $U$, we do not (currently) know of a natural candidate for rational forms $\Omega$ to use in \cref{def:CHY}.  However, for a hyperplane arrangement complement, the canonical forms $\bOmega_P$ are natural choices.  In the case $U = M_{0,n+1}$, the canonical forms $\bOmega_P$ are known as \emph{Parke-Taylor} forms; see \cite{LamModuli,AHLstringy,BD}.  For higher configuration spaces, see \cite{CEZ,CEZ24}.

\section{Amplitudes for matroids}\label{sec:amplitude}
Let $\M$ be an oriented matroid.  Recall that $R = \Z[\a] = \Z[a_e \mid e \in E]$ and $Q = \Frac(R)$.  We let $R_0 = R/(a_E)$ and $Q_0 = \Frac(R_0)$.  In this section, we view the deRham intersection forms as taking values in rational functions.

\begin{definition}\label{def:matroidamp} For a tope $P \in \T$, we define the \emph{amplitude}
$$
\A(P):= \bdRip{\Omega_P,\Omega_P} \in Q_0
$$
and for a pair of topes $P, Q \in \T$, we define the \emph{partial amplitude}
$$
\A(P,Q):=  \bdRip{\Omega_P,\Omega_Q} \in Q_0, 
$$
which is symmetric in $P$ and $Q$.
\end{definition}

\begin{theorem}\label{thm:CHY}
When $\M$ is the oriented matroid of a projective hyperplane arrangement and $U$ is the hyperplane arrangement complement, \cref{def:CHY} and \cref{def:matroidamp} agree, that is, $\A(P,Q) = \A(\bOmega_P,\bOmega_Q)$.
\end{theorem}

In the scattering amplitudes literature, the relation between the stationary phase formula \cref{def:CHY} and twisted cohomology was first observed by Mizera \cite{Miz}.  The formula was established rigorously in a general setting by Matsubara-Heo \cite{MHcoh}.  Thus \cref{thm:CHY} can be deduced from \cref{thm:dRpairmain} and the results of \cite{MHcoh}.  \cref{thm:CHY} can also be deduced from the results of Varchenko; for instance by using \cite[Theorem 3.1]{VarBethe}.  

\begin{remark}
It would be interesting to develop an analogue of \cref{def:CHY} for the other intersection forms $\DdRip{\cdot,\cdot}, \halfip{\cdot,\cdot}_B, \halfip{\cdot,\cdot}^B$.
\end{remark}
For $F \in L(M)$ and $q \in Q_0$, let
$$
\res_{a_F}(q) := (a_F q)|_{a_F = 0},
$$
if it is defined.  When $\res_F(q)$ is defined, it belongs to $\Frac(R_0/(a_F))$.  

Let $P \in \T$ be a tope and let $\A(P)$ be the amplitude.  We discuss some properties of the rational function $\A(P)$.  Recall that $L(P) \subset L(M)$ is the face lattice of the tope $P$.  For $F \in L(P)$, the tope $P$ restricts to the tope $P^F :=P|_F \in \T(\M^F)$, and contracts to the tope $P_F:=P|_{E \setminus F} \in \T(\M_F)$.

\begin{theorem}\label{thm:AP}
If $M$ is decomposable then $\A(P) = 0$ for all topes $P \in \T(\M)$.  Otherwise, the poles of the rational function $\A(P)$ are all simple and can only be along $\{a_F = 0\}$ for each connected $F \in L(P) - \{\hat 0, \hat 1\}$.  We have the recursion
$$
\res_{a_F} \A(P)= \A(P|_F) \A(P|_{E \setminus F}).
$$
\end{theorem}
\begin{proof}
If $M$ is decomposable, then $\beta(M) = 0$, and $\dim \OS(M,\omega) = 0$.  Thus $\bdRip{\cdot,\cdot}$ is the $0$ form, and $\A(P) = 0$.  Now, suppose that $M$ is connected.  By \cref{thm:dRtope}, $
\A(P) = \sum_{F_\bullet \in \Fl(P)} \frac{1}{a_{F_\bullet}}.
$
It is clear that the only possible poles of $\A(P)$ are along $a_F = 0$ for $F \in L(P)$.
Suppose $F \in L(P) - \{\hat 0, \hat 1\}$.  Then in each term of this sum, $\frac{1}{a_F}$ appears at most once as a factor.  Thus $\res_{a_F} \A(P)$ is defined and given by
$$
\res_{a_F} \A(P) = \sum_{F_\bullet \in \Fl(P) \mid F \in F_\bullet} \prod_{F_i \neq F} \frac{1}{a_{F_i}} = \sum_{F'_\bullet \in \Fl([\hat 0, F]), F''_\bullet \in \Fl([F,\hat 1])} \frac{1}{a_{F'_\bullet}} \frac{1}{a_{F''_\bullet}} =  \A(P|_F) \A(P|_{E \setminus F}).
$$
If $F$ is decomposable then $\A(P|_F)$ vanishes so $\res_{a_F} \A(P) =0$ and $\A(P)$ did not have a pole along $a_F$.  
\end{proof}

\begin{remark}
We do not know if $a_F$ is a pole of $\A(P)$ for every connected flat $F \in L(P)$.  
\end{remark}

\begin{remark}
The location of the poles of $\A(P)$ described in  \cref{thm:AP} is called ``locality" in the theory of scattering amplitudes.  The residue formula for $\A(P)$ is called ``unitarity".  Thus \cref{thm:AP} states that matroid amplitudes satisfy locality and unitarity.
\end{remark}

\section{Scattering forms for matroids}\label{sec:scatform}

In this section, we prove \cref{thm:CHY} using \emph{scattering forms}.  Let $M$ be a matroid of rank $r$ on $E$.

\subsection{Scattering forms on kinematic space}
Let $\Lambda := \Spec(R_0)$ be the free abelian group generated by $\{a_e \mid e \in E\}$, with the relation $\sum_e a_e = 0$.  In the case that $M$ is the matroid of a projective hyperplane arrangement complement $\bU$, this definition agrees with those in \cref{sec:veryaffine}.
\begin{definition}\label{def:kinematic}
The \emph{kinematic space} of $M$ is the complex vector space $\kin := \Lambda \otimes_\Z \C$.
\end{definition}

The vector space $\kin$ has dimension $|E|-1$.  For an ordered basis $B= \{b_1,b_2,\ldots,b_r\}$, let $\iota_B: W_B \hookrightarrow \kin$ be a $(r-1)$-dimensional \emph{affine} subspace satisfying $da_e = 0$ for $e \notin B$.  In other words, the functions $a_e$, $e \notin B$ are constant on $W_B$.  Thus the function $a_B = \sum_{b \in B} a_e$ is also constant on $W_B$.  The $1$-forms $\{da_b \mid b \in B\}$ span the cotangent space at every point of $W_B$, satisfying the single relation $\sum_{b \in B} da_b = 0$.  The volume form of $W_B$ is the $d=(r-1)$-form, given by 
$$
\mu_{W_B}:=  da_{b_d} \wedge \cdots \wedge da_{b_{1}}.
$$
The volume form $\mu_{W_B}$ changes sign when $B$ is reordered.  (The $b_i$ here should not be confused with the parameters $\b = (b_e, e \in E)$.)

For a flag $F_\bullet \in \Fl(M)$, define a differential form $\eta_{F_\bullet}$ on $\kin$ by
$$
\eta_{F_\bullet}:= \bigwedge_{i=1}^{r-1} \dlog a_F = \bigwedge_{i=1}^{r-1} \frac{d(a_{F_i})}{a_{F_i}} =  \frac{d(a_{F_{r-1}})}{a_{F_{r-1}}} \wedge \cdots \wedge  \frac{d(a_{F_1})}{a_{F_1}}  .
$$

\begin{definition}
For an element $x \in \rOS(M)$, the \emph{scattering form} $\eta_x$ on $\kin$ is defined to be
$$
\eta_x:= \sum_{F_\bullet \in \Fl(M)} \Res_{F^-_\bullet}(x) \eta_{F_\bullet}$$
where $F^-_\bullet \in \Fl^d(M)$ is obtained from $F_\bullet$ by ignoring the last step of the flag.
\end{definition}

\begin{proposition}\label{prop:iotaB}
For any ordered basis $B$, the pullback $\iota_B^*(\eta_{x})$ is equal to $\dRipp{x, \partial e_{B}} \mu_{W_B}$.  \end{proposition}
\begin{proof}
First we calculate the pullback $\iota_B^* \eta_{F_\bullet}$, which is a top-form on the $d$-dimensional affine space $W_B$.  Since $da_e = 0$ for $e \notin B$, we have
$$
da_{F_d} \wedge da_{F_{d-1}} \wedge \cdots \wedge da_{F_1} = da_{b_{\sigma(d-1)}} \wedge da_{b_{\sigma(d-2)}} \wedge \cdots \wedge da_{b_{\sigma(1)} }= r(B,F_\bullet) da_{b_{d-1}} \wedge \cdots \wedge da_{b_{1}}
$$
where $\sigma$ is the permutation satisfying $F_i = \sp\{b_{\sigma(1)},\ldots,b_{\sigma(i)}\}$, and in the second equality we have also used $\sum_{b_i \in B} db_i = 0$.  Thus $\iota_B^*(\eta_{F_\bullet}) = r(B,F_\bullet)\frac{1}{a_{F_\bullet}} \mu_{W_B}$.

Suppose that $x = \partial e_{B'}$.  Then using \cref{prop:dRind} and \cref{prop:dRpartial}, we have
\begin{align*}
\iota_B^*(\eta_{\partial e_{B'}}) &= \sum_{F_\bullet \in \Fl(M)} r(B',F_\bullet) \iota_B^*(\eta_{F_\bullet}) = \sum_{F_\bullet \in \Fl(M)} r(B',F_\bullet)  \frac{1}{a_{F_\bullet}}  r(B,F_\bullet) \mu_{W_B} = \dRipp{\partial e_{B'}, \partial e_{B}} \mu_{W_B}.\qedhere
\end{align*}
\end{proof}

\subsection{Scattering forms for topes}
Let $\M$ be an oriented matroid lifting $M$, and let $P \in \T(\M)$ be a tope.  Then we have
$$
\eta_P:= \eta(\bOmega_P) = \sum_{ F_\bullet \in \Fl(M)} r(P,F_\bullet) \eta_{F_\bullet}.
$$

\begin{lemma}\label{lem:etaP}
Let $P \in \T(\M)$ be a tope.  Then the poles of $\eta_P$ are (only) along $\{a_F = 0\}$ for connected flats $F \in L(P) \setminus \{\hat 0, \hat 1\}$, and we have
$$
\Res_{a_F = 0} \eta_P = (-1)^{\rk(F)} \eta_{P_F} \wedge \eta_{P^F} 
$$
as forms on $\{a_F = 0\}$.
\end{lemma}
\begin{proof}
By definition, the only possible poles in $\eta_P$ are along $\{a_F = 0\}$ for flats $F \in L(P)$.  
For a flag $F_\bullet$ passing through $F$, write $F_\bullet = (F'_\bullet < F < F''_\bullet)$ for the parts of the flag before and after $F$.
We have
$$
\Res_{a_F = 0} \eta_{F_\bullet} = \begin{cases} (-1)^{\rk(F)}  \eta_{F''_\bullet} \wedge \eta_{F'_\bullet} & \mbox{if $F_\bullet$ passes through $F$,} \\
0 & \mbox{otherwise.}
\end{cases}
$$
The formula for $\Res_{a_F = 0} \eta_P$ now follows in the same way as in the proof of \cref{thm:AP}.
\end{proof}

A tope $P \in \T(\M)$ is called \emph{simplex-like} if we have $\bOmega_P = \partial e_B$ for some basis $B \in \B(M)$.  Applying \cref{prop:iotaB} to $x = \bOmega_Q$, we have the following result.  We have used \cref{prop:dRpartial} to express the answer in terms of the canonical forms $\Omega_P$ instead of $\bOmega_P$.
\begin{corollary}\label{cor:simplexscat}
For a tope $Q$ and a simplex-like tope $P$ satisfying $\bOmega_P = \partial e_B$, we have 
$$
\iota^*_B(\eta_Q) =  \bdRip{\Omega_P,\Omega_Q}\mu_{W_B}
$$
for any subspace $\iota_B: W_B \hookrightarrow \kin$ satisfying $da_e = 0$ for $e \notin B$.  In particular, we have
$$
\iota^*_B(\eta_P) =  \bdRip{\Omega_P,\Omega_P}\mu_{W_B}.
$$
\end{corollary}
It would be interesting to remove the simplex-like condition in \cref{cor:simplexscat}.  

\begin{remark}
The terminology ``kinematic space" comes from the special case of the moduli space $M_{0,n+1}$ of $(n+1)$-pointed rational curves, which corresponds to the graphic matroid $M(K_n)$ on the complete graph with $n$ vertices.  We refer the reader to \cite{LamModuli} for an exposition targeted at mathematicians.
Scattering forms were first defined in the $M_{0,n+1}$ setting in \cite{ABHY}, and studied further in \cite{AHLstringy}.
\end{remark}

\subsection{Scattering correspondence}
We continue to assume that $\sum_e a_e = 0$.
Let $\bU$ be a projective hyperplane arrangement complement with matroid $M$.  The space $\Lambda_\C$ in \cref{sec:veryaffine} can be identified with the kinematic space in \cref{def:kinematic}.  On the variety $\bU \times \kin$ we have the twist $1$-form
$$
\omega = \sum_{e \in E} a_e \dlog f_e
$$
which we now view as depending on both the coordinates $a_e$ on $\kin$, and the functions $f_e$ on $\bU$.  

\begin{definition}\label{def:scatcorr}
The \emph{scattering variety} or \emph{critical point variety} $\I$ is the vanishing set in $\bU \times \kin$ of the twist $1$-form $\omega$ and fits into the scattering correspondence diagram
\begin{equation}\label{eq:scatteringmap} \begin{tikzcd}
&\I \arrow{rd}{q} \arrow[swap]{ld}{p} & \\%
\bU && \kin
\end{tikzcd}.
\end{equation}
\end{definition}

The space $\I$ was first studied in \cite{OT}.  In \cite[Proposition 2.5]{CDFV}, it is shown that $p: \I \to \bU$ is a vector bundle.  The scattering variety was studied in the setting of very affine varieties in \cite{HS, Huh}.

On $\bU$ we have the canonical form $\bOmega_P$ for a tope $P \in \T^+(\M)$.  In the following, we will use the notion of \emph{pushforward} of a rational form along a rational map $f: X \to Y$ of relative dimension $0$ between complex algebraic varieties.  We refer the reader to \cite[Section 3.9]{LamModuli} for more details on pushforwards.  The following result generalizes \cite[Theorem 3.26]{LamModuli}.

\begin{theorem}\label{thm:etaP}
We have the equality 
$
q_* p^* \bOmega_P = \eta_{P}
$
of rational $d$-forms on $\kin$.
\end{theorem}
\begin{proof}
\def\bI{{\overline{\I}}}
\def\bp{\bar p}
\def\bq{\bar q}
We proceed by induction on $d$ and $|E|$.  The claim reduces to that for simple matroids.  The base case is the rank $1$ matroid with a single element.  In this case, $\kin$ is a single point (a 0-dimensional vector space), and the result is trivial.  Now suppose $d \geq 1$.

Consider the maximal wonderful compactification $\pi: X= X_{\max} \to \P^d$ of $\bU$ and view $X \times \kin$ as a vector bundle over $X$.  Let $Z = X \setminus \bU$.  There is an evaluation map 
$$
\Psi: X \times \kin \to \Omega^1_X(\log Z), \qquad \Psi: (x,\a) \longmapsto \omega_\a(x),
$$
sending $X \times \kin$ to the vector bundle $\Omega^1_X(\log Z)$ of logarithmic one-forms on $X$.  Note that the statement that $\Omega^1_X(\log Z)$ is a vector bundle requires that $X$ is smooth with normal-crossing boundary divisor.  By \cite[Proof of Theorem 3.8]{Huh}, the kernel $\bI:= \ker \Psi$ is a vector bundle on $X$ that coincides with the closure of the scattering correspondence $\I \subset \bU \times \kin \subset X \times \kin$.  The diagram \eqref{eq:scatteringmap} sits inside the diagram
\[ \begin{tikzcd}
&\bI \arrow{rd}{\bq} \arrow[swap]{ld}{\bp} & \\
X && \kin
\end{tikzcd}
\]
and it suffices to show that $\bq_* \bp^* \bOmega_P = \eta_P$.  Let $\Theta:= \bp^* \bOmega_P$.  Since $\bI$ is a vector bundle over $X$, the poles of $\Theta$ are the divisors $\bI|_D \subset \bI$ for each polar divisor $D \subset X$ of $\bOmega_P$.  According to \cite{BEPV}, the poles of $\bOmega_P$ on $X$ are all simple and along the divisors $X_F$ for each flat $F \in L(P)$.  The image $\bq(X_F)$ is contained in the subspace $\{a_F = 0\} \subset \kin$.  Indeed, for the case $F = \{e\}$ is a single element, this is part of the ``geometric deletion-restriction formula" of Denham--Garrousian--Schulze \cite[Theorem 3.1]{DGS}.  For a non-atom flat, the result follows in the same way after using \cref{lem:monod}.

The map $\bq: \bI \to \kin$ is proper and surjective.  According to the ``residue commutes with pushforward" result of Khesin and Rosly \cite[Proposition 2.5]{KR}, we have
$$
\bq_* \Res_{\bI|_D} \Theta = \Res_{a_F = 0} \bq_* \Theta.
$$
Brauner-Eur-Pratt-Vlad \cite{BEPV} show that the (analytic) closure $\bP$ of $P$ in $X$ is a positive geometry \cite{ABL} with canonical form $\Omega(\bP) :=\pi^*\bOmega_P$.  The intersection of $X_F$ with $\bP$ is the product of corresponding (closures of) topes $ \bP_F \times \bP^F$ and has canonical form the product $ \Omega(\bP_F) \wedge \Omega(\bP^F)$ of the canonical forms, and \cite[Theorem 4.5]{BEPV} show that $\Res_{\bI|_D} \Theta = \bp^* \left(\bOmega_{P_F} \wedge \bOmega_{P^F}\right)$.  We have an isomorphism of bundles over $X_F = X(M_F) \times X(M^F)$: 
$$
\bI|_{D_F} \cong \bI(M_F) \times  \bI(M^F) .
$$
(Note that even if $M$ is simple, the matroid $M_F$ may not be simple, and the kinematic space $\kin(M_F)$ is typically of a higher dimension than the kinematic space of its simplification.)  By the inductive hypothesis applied to $M^F$ and $M_F$, we calculate
\begin{align*}
\bq_* \Res_{\bI|_D} \Theta  &= \bq_* \bp^* \left( \bOmega_{P_F} \wedge \bOmega_{P^F} \right)= \pm \eta_{P_F} \wedge  \eta_{P^F},
\end{align*}
where $\eta_{P^F}$ (resp. $\eta_{P_F}$) is a differential form on the subspace $\kin|_{F}$ (resp. $\kin|_{E \setminus F}$), direct summands of the vector space $\{a_F = 0\} \subset \kin$.  The sign $\pm$ can be calculated to be $(-1)^{\rk(F)}$ and is the same sign appearing in the proof of \cref{prop:Bettiint}.  These are all the possible poles of $\bq_* \bp^* \bOmega_P$.  By \cref{lem:etaP}, the poles and residues of $\bq_* \bp^* \bOmega_P$ and $\eta_P$ agree and it follows that they are equal since both forms are pullbacks of meromorphic forms on the projective space $\P(\kin)$.  This proves the induction step, and the theorem.
\end{proof}

\begin{example}
Let $d = 1$ and suppose $\bU = \C \setminus \{z_1,z_2,\ldots,z_n\}$, so that $E = \{1,2,\ldots,n\} \cup \{0\}$.  The canonical form for $P = [z_e,z_{e+1}]$ is $\bOmega_P = \dlog (z-z_e) - \dlog (z-z_{e+1})$.  We work in affine coordinates, using $a_1,a_2,\ldots,a_n$ as coordinates on $\kin$, eliminating $a_0 = -(a_1+\cdots+a_n)$.  Then 
$$
\omega = \sum_{e=1}^n a_e \frac{dz}{z-z_e}, \qquad \text{and} \qquad \I = \{(z,\a) \mid p(z) := \sum_{e=1}^n \frac{a_e}{z-z_e}  = 0\}.
$$
Write $(z-z_e) p(z) = a_e + q(z)$, so that on $\I$ we have 
$$
0 = da_e + q'(z) dz + \mbox{terms involving other $da_{e'}$}.
$$
Thus $dz = -da_e/q'(z) + \text{other terms}$, and the definition of pushforward gives 
$$
q_*p^* \frac{dz}{z-z_f} = - \sum_{z_* \in \Crit(\omega_\a)} \frac{1}{q'(z_*) (z_*-z_f)} da_e + \mbox{terms involving other $da_{e'}$}.
$$
Let $u(z):= (z-z_f)(a_e+q(z))$ which has zeroes at $z= z_f, \infty$ and $z= z_* \in \Crit(\omega_\a)$.  The global residue theorem\footnote{We thank Simon Telen for explaining to us the use of global residue theorems for computing pushforwards.} for the rational function $1/u(z)$ states that the sum of the residues vanishes, which is the identity
$$
\frac{1}{u'(z_f)} + \text{ residue at }\infty + \sum_{z_* \in \Crit(\omega_\a)} \frac{1}{u'(z_*)} = 0.
$$
Now, the residue $r_\infty$ at $\infty$ 
does not depend on the choice of $f \in E \setminus 0$.  Since $u'(z) = (a_e+q(z)) + (z-z_f)q'(z)$, we have
$$
r_\infty + \frac{1}{a_e+q(z_f)} + \sum_{z_* \text{ roots }} \frac{1}{(z_*-z_f)q'(z_*)} = 0, $$
giving
$$
\qquad - \sum_{z_* \text{ roots }} \frac{1}{(z_*-z_f)q'(z_*)} = r_\infty + \frac{1}{a_e + q(z_f)} = r_\infty + \delta_{e,f} \frac{1}{a_f},
$$
where $\delta_{e,f}$ is the Kronecker delta function.  Thus the pushforward of $dz/(z-z_f)-dz/(z-z_g)$ is equal to $da_f/a_f -da_g/a_g$, agreeing with \cref{thm:etaP}.  
\end{example}

\begin{remark}
\cref{thm:etaP} is a variant of the results of \cite{ABL,AHLstringy} where the pushforward of the canonical form along the algebraic moment map of a toric variety is computed.  
\end{remark}

\subsection{Proof of \cref{thm:CHY}}
Recall the description of $p: \I \to \bU$ as a (trivial) vector bundle from \cite[Proposition 2.5]{CDFV}.  Let $z_1,\ldots,z_d$ be coordinates on $\C^d$, and write 
$$
\omega = \sum_{i=1}^d h_i(\z) dz_i, \qquad \text{where} \qquad h_i(\z) = \sum_{e \in E\setminus 0} h_i^e(\z) a_e,
$$
and $h_i^e(\z)$ are rational functions in $\z$, well-defined on $\bU$.  The fiber $p^{-1}(\z)$ is the codimension $d$ subspace of $\kin$ cut out by the linear equations $h_1(\z) =0, h_2(\z)=0,\ldots,h_d(\z) = 0$.  

Now let $W \subset \kin$ be an affine subspace of dimension $d$.  The condition for $W \cap p^{-1}(\z)$ to have positive dimension is the zero-set of a rational function in $\z$.  Thus for $\z$ belonging to a dense subset $\bU' \subset \bU$, the intersection $W \cap p^{-1}(\z)$ will be a single point.  We conclude that $q: \I \to \kin$ restricts to a rational map $q_W := q|_W: \bU \to W$.  For a generic $W \subset \kin$, the rational map $q_W$ has degree equal to $\beta(M)$.  This is the \emph{scattering map}; see \cite[Section 3.10]{LamModuli}.  Since pushforwards and pullbacks of rational forms can be calculated on dense subsets, for a rational top-form $\Omega$ on $\bU$, we have
$$
\iota_W^* q_* p^* \Omega = (q_W)_* p^* \Omega = (q_W)_* \Omega, \qquad \text{where} \qquad q_W : \bU \to W
$$
and $\iota_W: W \hookrightarrow \kin$ denotes the inclusion.

We work in affine coordinates.  Let $B = \{b_1,b_2,\ldots,b_d,0\} \in \B(M)$ be a basis containing $0 \in E$, so that $\partial e_B = \be_{b_d} \wedge \be_{b_{d-1}} \wedge \cdots \wedge \be_{b_1}$.  Let $W = W_B$ be a generic $d$-dimensional affine subspace such that $da_e = 0$ for $e \notin B$.  We claim that
\begin{equation}\label{eq:pushres}
(q_W)_* \Omega = \A_\bU(\Omega,\partial e_B) \mu_{W_B} =  \A_\bU(\Omega,\partial e_B) da_{b_d} \wedge \cdots \wedge da_{b_1},
\end{equation}
where $\A_\bU(\Omega,\partial e_B)$ is defined in \cref{def:CHY}. 
At a point $\a \in W$, the left hand side is defined as the sum over pre-images $q_W^{-1}(\a) \in \bU$, which are exactly the critical points of $\omega_\a$ that the right hand side is defined as a sum over.  The contribution of each critical point $\z \in q_W^{-1}(\z)$ to the coefficient of $\mu_{W_B}$ is the evaluation of a rational function.  By a direct calculation, for both sides this rational function is equal to
$$
 \det\left(\frac{\partial^2 \log \varphi}{\partial f_{b_i} \partial f_{b_j}}\right)^{-1} \times \frac{\Omega}{\partial e_B}.
$$
See also \cite[Lemma 3.19 and Proposition 2.23]{LamModuli}.

Combining \eqref{eq:pushres} with \cref{cor:simplexscat} and \cref{thm:etaP}, we obtain $\bdRip{\Omega_P, e_B} =  \A_{\bU}(\bOmega_P,\partial e_B)$ for any tope $P$ and any basis $B$.  Extending by linearity, we see that $\A(P,Q) =  \A_\bU(\bOmega_P,\bOmega_Q)$, so the two definitions of amplitudes agree. 

\section{Configuration space of $n+1$ points on $\P^1$}\label{sec:M0n}
\subsection{Complete graphic matroid}
Let $M$ be the graphic matroid of the complete graph $K_n$, on the ground set $E = \{(i,j) \mid 1 \leq i < j < n\}$ of edges, where we identify $(j,i)$ with $(i,j)$.  The lattice of flats $L(M)$ is the partition lattice $\Pi_n$.  The elements of $\Pi_n$ are the set partitions of $[n]$.  For two set partitions $\pi, \pi'$, we have $\pi \leq\pi'$ if $\pi$ refines $\pi'$.  As a flat, the set partition $\pi$ corresponds to the following subset of $E$:
$$
\pi = \{(i,j) \in E \mid i \text{ and } j \text{ belong to the same block of }\pi\}.
$$
For example, the set partition $\pi = (145|26|3)$ of $\{1,2,3,4,5,6\}$ can be viewed as the collection of edges $\{(1,4),(1,5),(4,5),(2,6)\}$.

Let $\M$ be the oriented graphic matroid associated to the orientation where $(i,j)$ is oriented $i \to j$ for $i < j$.  The oriented matroid $\M$ arises from the braid arrangement $\B_n$, consisting of the hyperplanes $H_{i,j} = \{z_i - z_j = 0 \mid 1 \leq i < j < n\}$ in $\R^n$.  The braid arrangement is not essential, and instead we prefer to consider $\M$ as arising from the affine graphic hyperplane arrangement $\bA \subset \R^{n-2} = \{(z_2,\ldots,z_{n-1})\}$ with hyperplanes
$$
\{z_i = 0 \mid i = 2,3,\ldots,n-1\} \cup \{z_i = 1 \mid i=2,3,\ldots,n-1\} \cup \{z_i - z_j = 0 \mid 2 \leq i < j  \leq n-1\}.
$$
See \cref{fig:M05}.  In other words, the coordinate $z_1$ has been set to $0$ and the coordinate $z_n$ has been set to $1$.  The hyperplane at infinity should be thought of as $\{z_1 - z_n = 0\}$.  We still denote the $\binom{n}{2}$ hyperplanes as $H_{ij}$.  The parameters $a_e$ will be taken to be $\{a_{(i,j)} \mid 1 \leq i < j \leq n\}$.  The following result gives a standard description of $M_{0,n+1}$, the \emph{configuration space of $n+1$ points on $\P^1$}, also known as the \emph{moduli space of rational curves with $n+1$ marked points}; see for instance \cite{LamModuli}.

\begin{figure}
\begin{center}
$$
\begin{tikzpicture}[scale=1.5,extended line/.style={shorten >=-#1,shorten <=-#1},
 extended line/.default=1cm]
\draw (0,-1) -- (0,3);
\draw (1,-1) -- (1,3);
\draw (-1,0) -- (3,0);
\draw (-1,1) -- (3,1);
\draw(-1,-1) -- (3,3);
\node at (0,3.2) {$z_2 = 0$};
\node at (1,3.2) {$z_2 = 1$};
\node at (3,0.15) {$z_3 = 0$};
\node at (3,1.15) {$z_3 = 1$};
\node at (3,3.2) {$z_2 = z_3$};
\node[color=blue] at (0.2,2.8) { $(12)$};
\node[color=blue] at (1.2,2.8) { $(24)$};
\node[color=blue] at (3.1,2.8) { $(23)$};
\node[color=blue] at (2.8,-0.15) { $(13)$};
\node[color=blue] at (2.8,.85) { $(34)$};
\node[color=red] at (-0.5,2) {$2143$};
\node[color=red] at (0.5,2) {$1243$};
\node[color=red] at (1.5,2) {$1423$};
\node[color=red] at (2.5,2) {$1432$};

\node[color=red] at (2,0.5) {$1342$};
\node[color=red] at (0.6,0.25) {$1324$};
\node[color=red] at (0.3,0.65) {$1234$};
\node[color=red] at (-0.5,0.5) {$2134$};

\node[color=red] at (2,-0.5) {$3142$};
\node[color=red] at (0.5,-0.5) {$3124$};
\node[color=red] at (-0.4,-0.75) {$3214$};
\node[color=red] at (-0.7,-0.3) {$2314$};
\end{tikzpicture}
$$
\end{center}
\caption{The configuration space $M_{0,5}$ of $5$ points on $\P^1$ as a hyperplane arrangement complement.}
\label{fig:M05}
\end{figure}
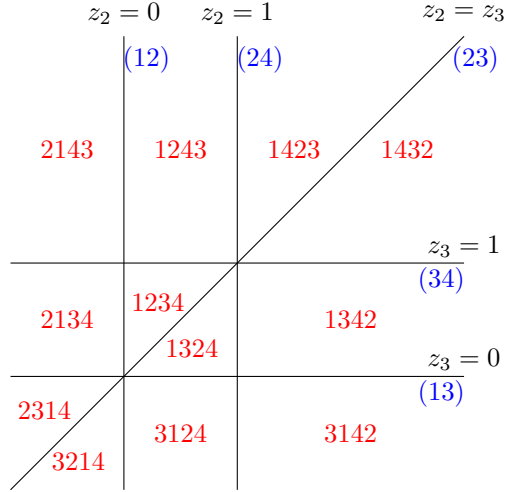

\begin{proposition}
For the affine graphic arrangement $\bA$, the complement $\bU$ is isomorphic to $M_{0,n+1}$.
\end{proposition}

The set $\T^+$ of positive topes has cardinality $n!/2$ and can be indexed by permutations $\sigma \in S_n$ satisfying $\sigma^{-1}(1) < \sigma^{-1}(n)$.  For example, for $n = 4$, we have 12 chambers, as shown in \cref{fig:M05}.

For a permutation $\sigma = \sigma(1) \sigma(2) \cdots \sigma(n)$, the tope $P_\sigma$ satisfies
$$
P_{\sigma}((i,j)) = \begin{cases} + & \mbox{if $\sigma^{-1}(i) < \sigma^{-1}(j)$,} \\
- & \mbox{if $\sigma^{-1}(i) > \sigma^{-1}(j)$.}
\end{cases}
$$
Geometrically, $P_\sigma$ is the region in $\R^{n-2}$ satisfying
$$
z_{\sigma(1)} < z_{\sigma(2)} < \cdots < z_{\sigma(n)}, \qquad \mbox{where $z_1 = 0$ and $z_n = 1$.}
$$
The facets of the tope $P_{\sigma}$ are the $n-1$ hyperplanes $H_{(\sigma(i),\sigma(i+1))}$, $i=1,2,\ldots,n-1$.  The lattice of positive flats $L(P)$ associated to $P_\sigma$ is isomorphic to the boolean lattice, and consists of the set partitions
\begin{equation}\label{eq:pisigma}
\pi_{c_1,c_2,\ldots, c_{r-1}} = \{\sigma(1),\ldots,\sigma(c_1)| \sigma(c_1+1),\ldots,\sigma(c_2)| \cdots | \sigma(c_{r-1}+1),\ldots,\sigma(n)\}
\end{equation}
for $1 \leq c_1 < c_2 < \cdots < c_{r-1} < n$. 

\begin{lemma}
The set of bounded topes $\T^0$ has cardinality $(n-2)!$ and are indexed by permutations $\sigma$ satisfying $\sigma(1) = 1$ and $\sigma(n) = n$.
\end{lemma}
\begin{proof}
It follows from \eqref{eq:pisigma} that the only way for $(1,n)$ not to belong to any of the flats in $L(P)\setminus \hat 1$ is to have $\sigma(1) = 1$ and $\sigma(n) = n$.
\end{proof}

Let us also pick a general extension $\star$, given by the hyperplane $H_\star = \{z_2 + z_3 + \cdots + z_{n-1} = - \epsilon < 0\}$, pictured in \cref{fig:M05star}.  With this choice of extension, we have
\begin{lemma}\label{lem:Knstar}
The set of bounded topes $\T^\star$ has cardinality $(n-1)!$ and is indexed by permutations $\sigma$ satisfying $\sigma(1) = 1$.
\end{lemma}
\begin{proof}
Suppose that $\sigma(1) = 1$.  Then in $P_\sigma$ we have $z_i > z_1 = 0$ for $i=2,\ldots,n-1$.  It follows that $P_\sigma$ does not intersect $H_\star$ and belongs to $\T^\star$.  Conversely, if $\sigma(1) > 1$, then there is clearly a point $z^+ \in P_\sigma$ such that $z^+_2 + z^+_3 + \cdots + z^+_{n-1} > - \epsilon$.  By decreasing $z^*_{\sigma(1)}$ a large amount, there is also a point $z^- \in P_\sigma$ such that $z^+_2 + z^+_3 + \cdots + z^+_{n-1} < - \epsilon$.  Thus $P_\sigma \notin \T^\star$ if $\sigma(1) > 1$.
\end{proof}

\subsection{Temporal Feynman diagrams}
We shall use a reindexing of variables that is common in physics: for $A \subset [n]$, write
$$
s_A := \sum_{(i<j) \in A} a_{(i,j)}
$$
so that, for example, $s_{ij} = a_{(i,j)}$ and $s_{ijk} = a_{(i,j)} + a_{(i,k)} + a_{(j,k)}$.

A \emph{planar tree} on $[n+1]$ is a planarly embedded tree $T$ with leaves labeled $1,2,\ldots,n+1$ in cyclic clockwise order.  A planar tree $T$ is called \emph{cubic} if all non-leaf vertices have degree three.  An edge $e$ of $T$ is called \emph{internal} if it is not incident to any of the leaves.  Let $\prec$ be the partial order on the internal edges $I(T)$ of $T$ such that $e \preceq e'$ if and only if the path from $e$ to $n+1$ passes through $e'$.  Denote by $\T_{n+1}$ the set of planarly embedded trees on $[n+1]$ and $\T^{(3)}_{n+1}$ the subset of cubic planar trees.  Define a partial order $<$ on $\T_{n+1}$ by $T \leq T'$ if $T$ can be obtained from $T'$ be contracting internal edges.  Then the maximal elements of $\T_{n+1}$ are the cubic planar trees and $\T_{n+1}$ is isomorphic to the dual of the face poset of the $(n-2)$-dimensional associahedron.

A \emph{temporal Feynman diagram} is a pair $(T,\leq)$ consisting of a cubic planar tree $T$ together with a total ordering $<$ of the internal edges such that if $e \preceq e'$ then $e \leq e'$.  It is easy to see that the number of temporal Feynman diagrams is equal to $(n-1)!$.

To each edge $e$ of a cubic planar tree let $S_e \subset E = \{1,2,\ldots,n\}$ denote the set of leaves in the component of $T \setminus e$ \emph{not} containing $n+1$.  Note that $S_e$ is always an interval in $[n]$.  Define
$$
F(S_e) = \{(i,j) \mid i < j  \text{ and } i,j \in S_e\}.
$$
To each temporal Feynman diagram $(T,\leq)$ and internal edge $e \in I(T)$ of $T$, let $I_{\leq e} := \{e' \in I \mid e' \leq e\}$ and
$$
D(e):= \{e' \in I_{\leq e} \mid e' \text{ is } \preceq \text{ maximal in } I_{\leq e}\}.
$$
Define 
\begin{equation}\label{eq:Fe}
F_\leq(e) := \bigcup_{e' \in D(e)} F(S_{e'}).
\end{equation}
It is easy to see that $F(e')$ for $e' \in D(e)$ are disjoint intervals, so that \eqref{eq:Fe} is a disjoint union.  For an internal edge $e$, we define the (usual) \emph{propagator} $X_e$ and the \emph{temporal propagator} $Y_e$ by
$$
X_e := \sum_{(i,j) \in F(S_e)} a_{(i,j)} = s_{S_e}, \qquad \text{ and } \qquad
Y_e := \sum_{(i,j) \in F_\leq(e)} a_{(i,j)} = \sum_{e' \in D(e)} X_{e'} = \sum_{e' \in D(e)} s_{S_{e'}}.
$$

\begin{remark}
We may interpret temporal Feynman diagrams as a particle scattering process.  Let $(T, \leq)$ be a temporal Feynman diagram.  We orient every edge of $T$ towards the leaf $n+1$, viewing $T$ as a scattering process where $n$ incoming particles $1,2,\ldots,n$ produce a single outgoing particle $n+1$.  At each cubic vertex of $T$, two incoming particles collide to produce a single outgoing particle.  The data of $\leq$ is a total ordering on when these cubic collisions occur, illustrated in \cref{fig:temporal}.
\begin{figure}
\begin{center}
\begin{tikzpicture}
\begin{scope}[decoration={
    markings,
    mark=at position 0.5 with {\arrow{>}}}
    ] 
\draw[decoration={markings, mark=at position 0.5 with {\arrow{>}}},postaction={decorate}] (0,5) -- (1,4.5);
\draw[decoration={markings, mark=at position 0.5 with {\arrow{>}}},postaction={decorate}] (0,4) -- (1,4.5);
\draw[decoration={markings, mark=at position 0.5 with {\arrow{>}}},postaction={decorate}] (1,4.5) -- (3,3.5);
\draw[decoration={markings, mark=at position 0.5 with {\arrow{>}}},postaction={decorate}] (0,3) -- (3,3.5);
\draw[decoration={markings, mark=at position 0.5 with {\arrow{>}}},postaction={decorate}] (3,3.5) -- (4,3);
\draw[decoration={markings, mark=at position 0.5 with {\arrow{>}}},postaction={decorate}] (4,3) -- (5,3);
\draw[decoration={markings, mark=at position 0.5 with {\arrow{>}}},postaction={decorate}] (0,2) -- (2,2);
\draw[decoration={markings, mark=at position 0.5 with {\arrow{>}}},postaction={decorate}] (0,1) -- (2,2);
\draw[decoration={markings, mark=at position 0.5 with {\arrow{>}}},postaction={decorate}] (2,2) -- (4,3);

\node[color=blue] at (-0.2,1) {$1$};
\node[color=blue] at (-0.2,2) {$2$};
\node[color=blue] at (-0.2,3) {$3$};
\node[color=blue] at (-0.2,4) {$4$};
\node[color=blue] at (-0.2,5) {$5$};
\node[color=blue] at (5.2,3) {$6$};

\node[color=red] at (2.1,4.15) {$a$};
\node[color=red] at (3,2.2) {$b$};
\node[color=red] at (3.5,3.4) {$c$};

\draw[dashed,color=purple] (3.3, 5)--(3.3,1);

\node[text width=7.5cm] at (10,4.2) {The three internal edges are ordered $a < b < c$ according to the order the corresponding particle is produced.};
\node[text width=7.5cm] at (10,1.8) {The dashed vertical line intersects the two internal edges $b$ and $c$.  The temporal propagator $Y_c$ is equal to $X_b+ X_c$.};
\end{scope}
\end{tikzpicture}
\end{center}
\caption{A temporal Fenyman diagram can be viewed as a recording of a particle scattering process.}
\label{fig:temporal}
\end{figure}
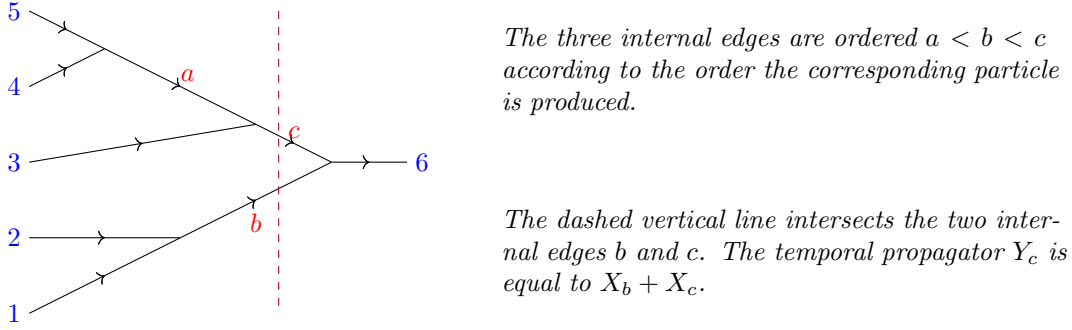
The usual propagator $X(e)$ has the physical interpretation as the square of the momentum along the internal edge $e$.  The temporal propagator $Y(e)$ is the sum of squares of momenta traveling along internal edges immediately after the particle traveling along $e$ has been produced.
\end{remark}

Let $P_{\id}$ denote the chamber indexed by the identity permutation.
\begin{definition}
The $(n+1)$-point planar $\phi^3$-amplitude is defined to be $A_{n+1}^{\phi^3} = \A(P_\id)$.
\end{definition}
The amplitude $\A(P_\id)$ is also commonly called the \emph{biadjoint scalar amplitude}.  By \cref{thm:CHY}, it can be computed using the formalism of scattering equations (\cref{def:CHY}), and is the simplest and most fundamental of the class of amplitudes originally considered by Cachazo-He-Yuan.

\begin{theorem}\label{thm:temporal}
The $(n+1)$-point planar $\phi^3$-amplitude is given by
$$
A_{n+1}^{\phi^3} = \sum_{(T,\leq)} \prod_{e \in I(T)} \frac{1}{Y_e},
$$
summed over temporal Feynman diagrams $(T,\leq)$ on $[n+1]$.
\end{theorem}
\begin{proof} 
We apply \cref{thm:dRtope} to \cref{def:matroidamp}.  The Las Vergnas face lattice $L(P_\id) \subset \Pi_n$ is the subposet consisting of set partitions $\pi$ where every block is an interval.  These are exactly the set partitions that can be written by inserting dividers into the identity permutation.  A flag of flats $F_\bullet \in \Fl(P_\id)$ then corresponds to a temporal Feynman diagram in a natural way: each step in the flag $F_\bullet$ produces a cubic vertex.  For example, the temporal Feynman diagram of \cref{fig:temporal} corresponds to the flag
\begin{equation*}
(1|2|3|4|5) < (1|2|3|45) < (12|3|45)< (12|345) < (12345). \qedhere
\end{equation*}
\end{proof}

\begin{example}
Consider $n=4$, corresponding to $M_{0,5}$.  The six temporal Feynman diagrams are illustrated here,, where ${\bf 1}, {\bf 2}$ are used to indicate orderings on internal edges:
$$
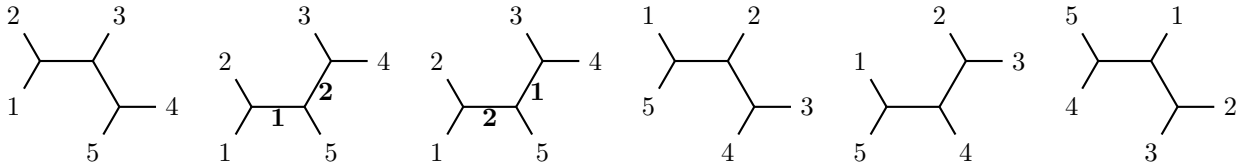
\begin{tikzpicture}[scale=0.7]
\coordinate (A) at (0,0);
\coordinate (B) at (1,0);
\coordinate (C) at (1.5,-0.866);
\node (L1) at (-0.5,-0.866) {$1$};
\node (L2) at (-0.5,+0.866) {$2$};
\node (L3) at (1.5,+0.866) {$3$};
\node (L4) at (2.5,-0.866) {$4$};
\node (L5) at (1,-2*0.866) {$5$};
\draw[thick] (A) --(B)--(C)--(L4);
\draw[thick] (L1)--(A)--(L2);
\draw[thick] (L3)--(B);
\draw[thick] (C)--(L5);

\begin{scope}[shift={(4,-0.866)}]
\coordinate (A) at (0,0);
\coordinate (B) at (1,0);
\coordinate (C) at (1.5,+0.866);
\node (L1) at (-0.5,-0.866) {$1$};
\node (L2) at (-0.5,+0.866) {$2$};
\node (L3) at (1,1.732) {$3$};
\node (L4) at (2.5,+0.866) {$4$};
\node (L5) at (1.5,-0.866) {$5$};
\draw[thick] (A) --(B)--(C)--(L4);
\draw[thick] (L1)--(A)--(L2);
\draw[thick] (L3)--(C);
\draw[thick] (B)--(L5);
\node[align=left] at (0.5,-0.2) {$\bf 1$};
\node[align=left] at (1.4,0.3) {$\bf 2$};
\end{scope}

\begin{scope}[shift={(8,-0.866)}]
\coordinate (A) at (0,0);
\coordinate (B) at (1,0);
\coordinate (C) at (1.5,+0.866);
\node (L1) at (-0.5,-0.866) {$1$};
\node (L2) at (-0.5,+0.866) {$2$};
\node (L3) at (1,1.732) {$3$};
\node (L4) at (2.5,+0.866) {$4$};
\node (L5) at (1.5,-0.866) {$5$};
\draw[thick] (A) --(B)--(C)--(L4);
\draw[thick] (L1)--(A)--(L2);
\draw[thick] (L3)--(C);
\draw[thick] (B)--(L5);
\node[align=left] at (0.5,-0.2) {$\bf 2$};
\node[align=left] at (1.4,0.3) {$\bf 1$};
\end{scope}

\begin{scope}[shift={(12,0)}]
\coordinate (A) at (0,0);
\coordinate (B) at (1,0);
\coordinate (C) at (1.5,-0.866);
\node (L5) at (-0.5,-0.866) {$5$};
\node (L1) at (-0.5,+0.866) {$1$};
\node (L2) at (1.5,+0.866) {$2$};
\node (L3) at (2.5,-0.866) {$3$};
\node (L4) at (1,-2*0.866) {$4$};
\draw[thick] (A) --(B)--(C)--(L3);
\draw[thick] (L5)--(A)--(L1);
\draw[thick] (L2)--(B);
\draw[thick] (C)--(L4);
\end{scope}

\begin{scope}[shift={(16,-0.866)}]
\coordinate (A) at (0,0);
\coordinate (B) at (1,0);
\coordinate (C) at (1.5,+0.866);
\node (L5) at (-0.5,-0.866) {$5$};
\node (L1) at (-0.5,+0.866) {$1$};
\node (L2) at (1,1.732) {$2$};
\node (L3) at (2.5,+0.866) {$3$};
\node (L4) at (1.5,-0.866) {$4$};
\draw[thick] (A) --(B)--(C)--(L3);
\draw[thick] (L5)--(A)--(L1);
\draw[thick] (L2)--(C);
\draw[thick] (B)--(L4);
\end{scope}

\begin{scope}[shift={(20,0)}]
\coordinate (A) at (0,0);
\coordinate (B) at (1,0);
\coordinate (C) at (1.5,-0.866);
\node (L4) at (-0.5,-0.866) {$4$};
\node (L5) at (-0.5,+0.866) {$5$};
\node (L1) at (1.5,+0.866) {$1$};
\node (L2) at (2.5,-0.866) {$2$};
\node (L3) at (1,-2*0.866) {$3$};
\draw[thick] (A) --(B)--(C)--(L2);
\draw[thick] (L5)--(A)--(L4);
\draw[thick] (L1)--(B);
\draw[thick] (C)--(L3);
\end{scope}
\end{tikzpicture}
$$
The six temporal Feynman diagrams give the expression
\begin{align*}
A_5^{\phi^3} &= \frac{1}{a_{12}(a_{12}+a_{13}+a_{23})} +  \frac{1}{a_{12}(a_{12}+a_{34})} +  \frac{1}{a_{34}(a_{12}+a_{34})} \\
&\;\;+  \frac{1}{a_{34}(a_{23}+a_{24}+a_{34})} +  \frac{1}{a_{23}(a_{23}+a_{24}+a_{34})} +  \frac{1}{a_{23}(a_{12}+a_{13}+a_{23})} \\
&=  \frac{1}{s_{12}s_{123}} +  \frac{1}{s_{12}(s_{12}+s_{34})} +  \frac{1}{s_{34}(s_{12}+s_{34})} +  \frac{1}{s_{34}s_{234}}+  \frac{1}{s_{23}s_{234}} +  \frac{1}{s_{23}s_{123}} 
\end{align*}
which equals the usual sum over five Feynman diagrams (or cubic planar trees)
\begin{align*}
A_5^{\phi^3} &= \frac{1}{a_{12}(a_{12}+a_{13}+a_{23})} +  \frac{1}{a_{12}a_{34}} +  \frac{1}{a_{23}(a_{23}+a_{24}+a_{34})} +  \frac{1}{a_{23}(a_{12}+a_{13}+a_{23})} +  \frac{1}{a_{12}(a_{12}+a_{13}+a_{23})} \\
&=  \frac{1}{s_{12}s_{123}} +  \frac{1}{s_{12}s_{34}} +  \frac{1}{s_{34}s_{234}}+  \frac{1}{s_{23}s_{234}} +  \frac{1}{s_{23}s_{123}} 
\end{align*}
\end{example}

\begin{lemma}
A flat $\pi$ of $L(M(K_n))$ is connected exactly when it has one non-singleton block.
\end{lemma}

Applying \cref{thm:deRhamfannested} with the choice of minimal building set (see \cref{sec:building}), we find that $\A(P_\id)$ is also given by a sum over cubic planar trees, giving the alternative formula
$$
A_{n+1}^{\phi^3} = \sum_{T \in \T^{(3)}_{n+1}} \prod_{e \in I(T)} \frac{1}{X_e}.
$$
This formula is the usual definition of the planar $\phi^3$-amplitude.  For brevity, we do not discuss the partial amplitudes $\A(P,Q)$.

\subsection{KLT matrix}
The KLT (Kawai-Lewellen-Tye) relation is an important identity in the physics of scattering amplitudes.  In its original formulation in \cite{KLT}, it expresses closed string amplitudes in terms of open string amplitudes via a \emph{string theory KLT kernel}.  Later, a \emph{field theory KLT kernel} was studied, relating gravity amplitudes to gauge theory amplitudes.  We show that our formula (\cref{thm:dRmain}) for the inverse of $\dRip{\cdot,\cdot}$ recovers the field theory KLT kernel in its formulation in \cite[(2.8)]{Frost}; see also \cite{BDSV,FMM}.  

\begin{theorem}\label{thm:Frost}
Let $\id = 1 2 \cdots (n-1) n$ and $\sigma = 1 \sigma(2) \cdots \sigma'(n-1) \sigma(n)$.  With the choices in \cref{lem:Knstar}, we have
\begin{equation}\label{eq:Frost}
\DdRip{P_\id,P_{\sigma}} =  \prod_{j = 2}^{n} (a_{(1,j)} + \sum_{k \in [2,j-1] \mid \sigma^{-1}(k) < \sigma^{-1}(j)} a_{(k,j)}).
\end{equation}
\end{theorem}

Each basis $B \in \B(M)$ can be identified with a spanning tree $T(B)$ on $[n]$ whose edges are the elements of $B$.  The following observation is straightforward.
\begin{lemma}\label{lem:treeB}
With the choices in \cref{lem:Knstar}, we have $P_\sigma \in \T^B$ if and only if for each distinct $i, j \in [2,n-1]$ such that $i$ belongs to the path from $j$ to $1$ in $T(B)$ we have $\sigma^{-1}(i) < \sigma^{-1}(j)$.
\end{lemma}

\begin{proof}[Proof of \cref{thm:Frost}]
Expanding the right hand side of \eqref{eq:Frost} each term is a monomial in $a_{(i,j)}$ of degree $n-1$.  Since each index $i \in [n]$ appears, this monomial corresponds to a base $B \in \B(M)$ and a tree $T(B)$.  By \cref{def:DdR}, we need to show that these are exactly the trees $T(B)$ such that $P_\id,P_{\sigma} \in \T^B$.  By \cref{lem:treeB} applied to $P_\id$, any such tree has the property that if the path from $j$ to $1$ contains $i$ then $i < j$.  Each such tree has the property that every vertex $i \in [2,n]$ is connected to exactly one vertex $j < i$, and therefore correspond exactly to the expansion of 
$$
 \prod_{j = 2}^{n} (a_{(1,j)} + \sum_{k \in [2,j-1] } a_{(k,j)}).
$$
Of these trees $T(B)$, the ones satisfying $P_{\sigma} \in \T^B$ are given by the right hand side of \eqref{eq:Frost}, again by \cref{lem:treeB}.
\end{proof}

Our \cref{thm:dRmain} recovers the well-known relation between partial scalar amplitudes and the KLT kernel; see for example \cite[Proposition 2.2]{Frost} or \cite{MizKLT}.  See \cref{ex:KLTexample}.

\subsection{String theory KLT matrix and $\alpha'$-corrected amplitudes} 
The Betti homology intersection pairings $\ip{P,Q}_B$ are called \emph{$\alpha'$-corrected amplitudes} in \cite{MizKLT}.  They were studied from a mathematical point of view in \cite{MHhom}.  
A \emph{temporal planar tree} is a pair $(T, \leq)$ where $T \in \T_{n+1}$ is a planar tree on $[n+1]$ and $\leq$ is a total ordering of the internal edges.  Let $\T_{n+1,\leq}$ denote the set of temporal planar trees on $[n+1]$.

Our formula \cref{def:Bettipair} for $\ip{P,P}_B$ is likely novel, and gives:
$$
\ip{P,P}_B = \sum_{(T,\leq) \in \T_{n+1,\leq}}\prod_{e \in I(T)} \frac{1}{b^2_{F_\leq(e)}-1}.
$$
The formula in \cite{MizKLT,MHhom} is obtained by taking the minimal building set for $L(M(K_n))$:
$$
\ip{P,P}_B =  \sum_{T \in \T_{n+1}}\prod_{e \in I(T)} \frac{1}{a_{F_\leq(e)}} = \sum_{T \in \T_{n+1}}\prod_{e \in I(T)} \frac{1}{b^2_{F(e)}-1}.
$$
Frost \cite{FroAss} interpreted this in terms of a sum over lattice points, analogous to our \cref{thm:Bettifan}.

The Betti cohomology intersection pairing $\bip{\cdot,\cdot}^B$ is known as the \emph{string theory KLT matrix}.  The intersection form matrix $\ip{\cdot,\cdot}^B_{\T^+}$ of \cref{thm:Bettiinverse} has dimensions $n!/2 \times n!/2$, and as such differs from the dimensions of the usual KLT matrix in the literature (which has dimension $(n-1)!$ or $(n-2)!$).  Our matrix $\ip{\cdot,\cdot}^B_{\T^+}$, for a one-dimensional arrangement, appeared in \cite[Section 6.4.3]{BMRS} under the name of ``overcomplete form of KLT relations".

\subsection{Determinants}
We list some of the determinantal formulae obtained for various amplitudes.

We have $\mu(M(K_n)) =(n-1)!$ and $\beta(M(K_n)) = (n-2)!$.  A set partition $\pi$ with multiple non-singleton blocks is a decomposable flat and $\beta(\pi) = 0$. For a set partition $\pi_S$ with a one non-singleton block $S$, the interval $[\hat 0, \pi_S]$ is isomorphic to $\Pi_{|S|}$, so $\beta(\pi) = (|S|-2)!$.  The interval $[\pi_S,\hat 1] \subset \Pi_n$ is isomorphic to $\Pi_{n-|S|+1}$, so $\mu(M_{\pi_S}) = (n-|S|)!$.  From \cref{thm:Aomotodet} and \cref{thm:SVdet} we obtain the following determinantal formulae.

\begin{corollary}\label{cor:det2}
The determinant of the $(n-2)! \times (n-2)!$ matrix $\A(P_\sigma,P_{\sigma'})$
%$\bdRip{P_\sigma,P_{\sigma'}}$ 
where $\sigma,\sigma'$ satisfy $\sigma(1) =\sigma'(1) = 1$ and $\sigma(n) =\sigma'(n) = n$ is given by 
$$
\pm \frac{\prod_{\{1,n\} \subset S \subsetneq [n]}{a_{F(S)}^{(n-|S|-1)! (|S|-2)!}}}{\prod_{\substack{S \subseteq [n] \\ |S| \geq 2 \text{ and } |S \cap \{1,n\}|\leq 1}}{a_{F(S)}^{(n-|S|-1)!(|S|-2)!}}}.
$$
\end{corollary}

\begin{corollary}\label{cor:det1}
The determinant of the $(n-1)! \times (n-1)!$ matrix $\dRip{\Omega_{P_\sigma},\Omega_{P_{\sigma'}}}$ where $\sigma,\sigma'$ satisfy $\sigma(1) =\sigma'(1) = 1$ is given by 
$$
 \frac{a^{(n-1)!-(n-2)!}_{F([n])} }{\prod_{S \subsetneq [n] \mid |S| \geq 2} a_{F(S)}^{(n-|S|)! (|S|-2)!}}.
$$
\end{corollary}

\begin{example}
Let $n = 4$, corresponding to $M_{0,5}$.  Then \cref{cor:det2} and \cref{cor:det1} give the determinants of a $2 \times 2$ and a $6 \times 6$ matrix.
$$
\begin{bmatrix}
 \frac{1}{s_{12} s_{123}}+\frac{1}{s_{12} s_{34}}+\frac{1}{s_{123} s_{23}}+\frac{1}{s_{23} s_{234}}+\frac{1}{s_{234} s_{34}} & -\frac{1}{s_{123} s_{23}}-\frac{1}{s_{23} s_{234}} \\
 -\frac{1}{s_{123} s_{23}}-\frac{1}{s_{23} s_{234}} & \frac{1}{s_{123} s_{13}}+\frac{1}{s_{123} s_{23}}+\frac{1}{s_{13} s_{24}}+\frac{1}{s_{23} s_{234}}+\frac{1}{s_{234} s_{24}} \\
\end{bmatrix}
$$
has determinant
$$
\frac{s_{134}s_{124}s_{14}}{s_{12}s_{13}s_{23}s_{24}s_{34}s_{123}s_{234}}.
 $$
 
 {\tiny \begin{multline*} \hspace{-25pt}
\left[
\begin{matrix}
 \frac{1}{s_{12} s_{123}}+\frac{1}{s_{12} s_{34}}+\frac{1}{s_{123} s_{23}}+\frac{1}{s_{23} s_{234}}+\frac{1}{s_{234} s_{34}} & -\frac{1}{s_{123} s_{23}}-\frac{1}{s_{23} s_{234}} & -\frac{1}{s_{234} s_{34}} \\
 -\frac{1}{s_{123} s_{23}}-\frac{1}{s_{23} s_{234}} & \frac{1}{s_{123} s_{13}}+\frac{1}{s_{123} s_{23}}+\frac{1}{s_{13} s_{24}}+\frac{1}{s_{23} s_{234}}+\frac{1}{s_{234} s_{24}} & -\frac{1}{s_{13} s_{24}}-\frac{1}{s_{234} s_{24}} \\
 -\frac{1}{s_{234} s_{34}} & -\frac{1}{s_{13} s_{24}}-\frac{1}{s_{234} s_{24}} & \frac{1}{s_{13} s_{134}}+\frac{1}{s_{13} s_{24}}+\frac{1}{s_{134} s_{34}}+\frac{1}{s_{234} s_{24}}+\frac{1}{s_{234} s_{34}} \\
 \frac{1}{s_{23} s_{234}}+\frac{1}{s_{234} s_{34}} & -\frac{1}{s_{23} s_{234}} & -\frac{1}{s_{134} s_{34}}-\frac{1}{s_{234} s_{34}} \\
 -\frac{1}{s_{23} s_{234}} & \frac{1}{s_{23} s_{234}}+\frac{1}{s_{234} s_{24}} & -\frac{1}{s_{234} s_{24}} \\ -\frac{1}{s_{12} s_{34}}-\frac{1}{s_{234} s_{34}} & -\frac{1}{s_{234} s_{24}} & \frac{1}{s_{234} s_{24}}+\frac{1}{s_{234} s_{34}}\\ 
\end{matrix}
\right.
\\
 \hspace{-25pt}
\left.
\begin{matrix}
 \frac{1}{s_{23} s_{234}}+\frac{1}{s_{234} s_{34}} & -\frac{1}{s_{23} s_{234}} & -\frac{1}{s_{12} s_{34}}-\frac{1}{s_{234} s_{34}} \\
  -\frac{1}{s_{23} s_{234}} & \frac{1}{s_{23} s_{234}}+\frac{1}{s_{234} s_{24}} & -\frac{1}{s_{234} s_{24}} \\
  -\frac{1}{s_{134} s_{34}}-\frac{1}{s_{234} s_{34}} & -\frac{1}{s_{234} s_{24}} & \frac{1}{s_{234} s_{24}}+\frac{1}{s_{234} s_{34}} \\
  \frac{1}{s_{134} s_{14}}+\frac{1}{s_{134} s_{34}}+\frac{1}{s_{14} s_{23}}+\frac{1}{s_{23} s_{234}}+\frac{1}{s_{234} s_{34}} & -\frac{1}{s_{14} s_{23}}-\frac{1}{s_{23} s_{234}} & -\frac{1}{s_{234} s_{34}} \\
  -\frac{1}{s_{14} s_{23}}-\frac{1}{s_{23} s_{234}} & \frac{1}{s_{124} s_{14}}+\frac{1}{s_{124} s_{24}}+\frac{1}{s_{14} s_{23}}+\frac{1}{s_{23} s_{234}}+\frac{1}{s_{234} s_{24}} & -\frac{1}{s_{124} s_{24}}-\frac{1}{s_{234} s_{24}} \\
 -\frac{1}{s_{234} s_{34}} & -\frac{1}{s_{124} s_{24}}-\frac{1}{s_{234} s_{24}} & \frac{1}{s_{12} s_{124}}+\frac{1}{s_{12} s_{34}}+\frac{1}{s_{124} s_{24}}+\frac{1}{s_{234} s_{24}}+\frac{1}{s_{234} s_{34}} 
\end{matrix}
\right]
\end{multline*}
}
has determinant
$$
\frac{s_{1234}^4}{s_{12}^2 s_{13}^2 s_{14}^2 s_{23}^2 s_{24}^2 s_{34}^2s_{123}s_{124}s_{134}s_{234}}.
$$
\end{example}

Other applications and directions to the physics of amplitudes and correlators will be explored in future work.

\bibliographystyle{alpha}
\bibliography{refs}

\end{document}